\documentclass{amsart} 

\pdfoutput=1

\usepackage{graphicx}

\usepackage{epsfig}

\usepackage{epstopdf}

\usepackage{amsmath}
\usepackage{amsthm}
\usepackage{amssymb}
\usepackage{amsfonts}
\usepackage{amscd}
\usepackage{mathrsfs}
\usepackage{mathdots}
\usepackage{tikz}
\usepackage[all]{xy}
\usepackage{ifpdf}

\newcommand\rank{\mathop{\rm rank}\nolimits}
\newcommand\im{\mathop{\rm im}\nolimits}
\newcommand\coker{\mathop{\rm coker}\nolimits}

\newcommand{\Tr}{\mathop{\rm Tr}\nolimits}

\newcommand\Hom{\mathop{\rm Hom}\nolimits}

\newcommand\End{\mathop{\rm End}\nolimits}

\newcommand\Spec{\mathop{\rm Spec}\nolimits}
\newcommand\Hilb{\mathop{\rm Hilb}\nolimits}

\newcommand{\length}{\mathop{\rm length}\nolimits}
\newcommand{\res}{\mathop{\sf res}\nolimits}

\newcommand{\balpha}{\boldsymbol \alpha}
\newcommand{\bbeta}{\boldsymbol \beta}
\newcommand{\bmu}{\boldsymbol \mu}
\newcommand{\bnu}{\boldsymbol \nu}
\newcommand{\blambda}{\boldsymbol \lambda}

\newcommand{\bvarphi}{\boldsymbol \varphi}
\newcommand{\sch}{\mathrm{Sch}}
\newcommand{\sets}{\mathrm{Sets}}
\newcommand{\Diag}{\mathrm{Diag}}

\newtheorem{lemma}{Lemma}[section]
\newtheorem{proposition}[lemma]{Proposition}
\newtheorem{assumption}[lemma]{Assumption}
\newtheorem{theorem}[lemma]{Theorem}
\newtheorem{remark}[lemma]{Remark}
\newtheorem{corollary}[lemma]{Corollary}
\newtheorem{example}[lemma]{Example}
\newtheorem{definition}[lemma]{Definition}
\newtheorem{claim}[lemma]{Claim}
\newtheorem{notation}[lemma]{Notation}

\begin{document}

\title{Unfolding of the unramified irregular singular generalized isomonodromic deformation}

\author{Michi-aki Inaba}
\address{Department of Mathematics, Kyoto University, Kyoto 606-8502, Japan}
\email{inaba@math.kyoto-u.ac.jp}

\subjclass[2010]{14D20, 14D15,  53D30, 34M56, 34M35, 34M40}

\maketitle

\begin{abstract}
We introduce an unfolded moduli space of connections,
which is an algebraic relative moduli space of connections on 
complex smooth projective curves, whose generic fiber is a moduli space
of regular singular connections and whose special fiber is a moduli space
of unramified irregular singular connections.
On the moduli space of unramified irregular singular connections,
there is a subbundle of the tangent bundle
defining the generalized isomonodromic deformation
produced by the Jimbo-Miwa-Ueno theory.
On an analytic open subset of the unfolded moduli space of connections,
we construct a non-canonical lift of this subbundle,
which we call an unfolding of the unramified irregular singular
generalized isomonodromic deformation. 
Our construction of an unfolding of the unramified irregular singular
generalized isomonodromic deformation
is not compatible with the asymptotic property in the unfolding
theory established by Hurtubise, Lambert and Rousseau
which gives unfolded Stokes matrices for an unfolded
linear differential equation in a general framework.
\end{abstract}


\section*{Introduction}

The intention of this paper is to produce a tool toward understanding
the confluence phenomena connecting the regular singular isomonodromic deformation
and the irregular singular generalized isomonodromic deformation.
In the case of connections on $\mathbb{P}^1$,
the regular singular isomonodromic deformation is the Schlesinger  equation
and the unramified irregular singular generalized isomonodromic deformation is
the Jimbo-Miwa-Ueno  equation which is completely given in
\cite{Jimbo-Miwa-Ueno}, \cite{Jimbo-Miwa-2}, \cite{Jimbo-Miwa-3}.
The most fundamental example of the confluence phenomena
will be the confluence of the classical hypergeometric functions,
though their isomonodromic deformations may not be mentioned
because of the rigidity.
There are extended results in 
\cite{Kimura-Haraoka-Takano} and
\cite{Kimura-Takano}.
The next important example of the confluence phenomena
will be the degeneration of Painlev\'e equations,
where the irregular singular generalized isomonodromic deformation
arises when we take a limit of the regular singular isomonodromic deformation.
Observation of confluence of Painlev\'e equations
via $\tau$ function is given in \cite{Jimbo-1}
and further study via confluent conformal blocks are given in
\cite{Lisovyy-Nagoya-Roussillon}.
There is an approach via monodromy manifolds in
\cite{Mazzocco} to the confluence of
Painlev\'e equations.
In \cite{Kimura-Tseveennamjil},
a generalization of the confluence phenomena to a general Schlesinger equation
is given.
An origin of confluence problems is given
by Ramis in \cite{Ramis}
and unfolding of Stokes data is one of  the important
problems.
Studies of confluence problem from this viewpoint are done in
\cite{Schafke}, \cite{Sternin-Shatalov}
and \cite{Glutsuk}.
A general framework of unfolded Stokes data
of an unfolded linear differential equation is established by Hurtubise, Lambert and Rousseau
in \cite{Hurtubise-Lambert-Rousseau} and \cite{Hurtubise-Rousseau}.
In \cite{Klimes}, confluence of unfolded Stokes data in rank two case is given
explicitly. 
One of the key ideas in 
the unfolding theory by Hurtubise, Lambert and Rousseau
in \cite{Hurtubise-Lambert-Rousseau} and \cite{Hurtubise-Rousseau}
is to adopt fundamental solutions with an asymptotic property,
which is estimated by
a flow of the vector field $v_{\epsilon}=p_{\epsilon}(x)\frac{\partial}{\partial x}$,
where $p_{\epsilon}(x)=0$ is a local unfolding equation.
They construct unfolded Stokes matrices of a linear differential equation
on $\mathbb{P}^1$ via connecting
fundamental solutions with an asymptotic property around points in the unfolding divisor
and that around  $\infty$.
In order to reconstruct an unfolded linear differential equation,
they consider another regular singular point,
whose monodromy reflects the analytic continuation along the `inner side'
of the unfolded divisor.
In \cite{Hurtubise-Rousseau},
they introduce a delicate condition called the `compatibility condition' in order that
the corresponding linear differential equation is a well-defined analytic family.

The author's early hope was to understand the unfolding theory by
Hurtubise, Lambert and Rousseau in a moduli theoretic way.
So we introduce in this paper an unfolded moduli space of connections,
whose generic fiber is a moduli space of regular singular connections
and whose special fiber is a moduli space of unramified irregular singular connections.

The Schlesinger type equation, or the regular singular isomonodromic deformation
is defined on a family of moduli spaces of regular singular connections
on smooth projective curves.
In order to get a good moduli space, we consider a  parabolic structure to the given connection 
and the moduli space is constructed in
\cite{Nakajima}, \cite{Arinkin-Lysenko},
\cite{IIS-1} and \cite{Inaba-1},
which is a smooth and quasi-projective moduli space.
The algebraic moduli construction is basically given by modifying the standard method
by Simpson in \cite{Simpson-1}, \cite{Simpson-2} or by Nitsure
in \cite{Nitsure}.
In \cite{IIS-1} and \cite{Inaba-1}, we formulate the regular singular isomonodromic
deformation and prove the geometric Painlev\'e property of the isomonodromic deformation
using the properness of the Riemann-Hilbert morphism.
In \cite{Yamakawa}, the moduli space of filtered local systems is introduced
by Yamakawa
and the Riemann-Hilbert isomorphism via the idea by Simpson
in \cite{Simpson-3} is given,
from which we can also prove the geometric Painlev\'e property
of the isomonodromic deformation.
Moduli theoretic descriptions of the regular singular isomonodromic deformation are
also given in \cite{Hurtubise-1}, \cite{Heu-1}, \cite{Heu-2},
\cite{Biswas-Heu-Hurtubise-1}, \cite{Biswas-Heu-Hurtubise-2} and \cite{Wong}.
We notice that we cannot forget the parabolic structure for the precise
formulation of the isomonodromic deformation given in \cite[Proposition 8.1]{Inaba-1}
on the locus where the parabolic structure is not completely determined by
the given connection.
Let us recall that 
the essential number of independent variables
of the regular singular isomonodromic deformation is
$3g-3+\deg D$,
where $D$ is the divisor consisting of all the regular singular points
and $g$ is the genus of base curves.

Moduli space of unramified irregular singular connections
is constructed in \cite{Biquard-Boalch} analytically
and in \cite{Inaba-Saito} algebraically.
The irregular singular generalized isomonodromic
deformation from the moduli theoretic viewpoint is given in
\cite{Boalch-1}, \cite{Boalch-2}, \cite{Fedorov},
\cite{Hurtubise-1}, 
\cite{Put-Saito}, \cite{Wong}, \cite{Bremer-Sage-2} and \cite{Inaba-Saito}
from various viewpoints, respectively.
In spite of the importance of parabolic structure in the regular singular case,
unfolding problem of the moduli space of irregular singular connections
does not seem to work well with parabolic structure,
especially for the deformation argument of ramified connections 
in \cite[Theorem 4.1]{Inaba-2}.
So we adopt another method of parameterizing the local exponents
in this paper.
If we fix distinct complex numbers
$\mu_1,\ldots,\mu_r$
and if we take generic unramified
local exponents
$\nu_1\dfrac{dz}{z^m},\ldots,\nu_r\dfrac{dz}{z^m}$
at a singular point $p$,
then we can observe that
there is a polynomial $\nu(T)\in\mathbb{C}[z]/(z^m)[T]$
satisfying
$\nu_k=\nu(\mu_k)$ for any $k$.
So we can regard $(\nu(T),\mu_1,\ldots,\mu_r)$ as a data of local exponents.
We can see that a connection $\nabla$ on a vector bundle $E$
has the local exponents
$\nu_1\dfrac{dz}{z^m},\ldots,\nu_r\dfrac{dz}{z^m}$ at $p$
if and only if there is an endomorphism
$N\in\End(E|_{mp})$ whose eigenvalues are $\mu_1,\ldots,\mu_r$
and $\nu(N)\dfrac{dz}{z^m}=\nabla|_{mp}$.

For the construction of the unfolded moduli space of connections,
we introduce a notion of $(\bnu,\bmu)$-connection.
Let $C$ be a complex smooth projective curve of genus $g$
and $D=D^{(1)}\sqcup\cdots\sqcup D^{(n)}$ be a divisor on $C$
locally given by the equation $D^{(i)}=\{ z^{m_i}-\epsilon^{m_i}=0 \}$.
The local exponents
$\bnu=(\nu^{(i)}(T))$ and $\bmu=(\mu^{(i)}_k)$
are given by
$\nu^{(i)}(T)\in {\mathcal O}_{D^{(i)}}[T]$
and distinct complex numbers $\mu^{(i)}_1,\ldots,\mu^{(i)}_r\in\mathbb{C}$.
The definition of
$(\bnu,\bmu)$-connection is given 
in Definition \ref {def-connection}
as a tuple $(E,\nabla,\{N^{(i)}\})$,
where $E$ is an algebraic vector bundle on $C$,
$\nabla$ is a connection on $E$ admitting poles along $D$
and $N^{(i)}\in \End(E|_{D^{(i)}})$ satisfies
$\nabla|_{D^{(i)}}=\nu(N^{(i)})\dfrac{dz} {z^{m_i}-\epsilon^{m_i}}$
and $\varphi^{(i)}_{\bmu}(N^{(i)})=0$,
where $\varphi^{(i)}_{\bmu}(T)=(T-\mu^{(i)}_1)\cdots(T-\mu^{(i)}_r)$.
In subsection \ref {subsection:moduli setting}, we define the relative moduli space 
$M^{\balpha}_{{\mathcal C},{\mathcal D}}(\tilde{\bnu},\bmu)
\longrightarrow {\mathcal T}_{\bmu,\blambda}$
of $\balpha$-stable
$(\bnu,\bmu)$-connections,
whose existence is provided by Theorem \ref{theorem:algebraic-moduli-unfolding}.
Here 
${\mathcal T}_{\bmu,\blambda}\longrightarrow \Delta_{\epsilon_0}$
is constructed in subsection \ref  {subsection:moduli setting},
on which there are a full family of pointed curves $(C,t_1,\ldots,t_n)$,
divisors $D^{(i)}$ given by the local equation
$z^{m_i}-\epsilon^{m_i}=0$
and a full family of exponents $\bnu$.
The fiber of the moduli space
$M^{\balpha}_{{\mathcal C},{\mathcal D}}(\tilde{\bnu},\bmu)$
over $\epsilon\neq 0$
is a moduli space of regular singular connections
and the fiber over $\epsilon=0$
is a moduli space of generic unramified irregular singular connections.

The fiber $M^{\balpha}_{{\mathcal C},{\mathcal D}}(\tilde{\bnu},\bmu)_{\epsilon=0}$
over $\epsilon=0\in \Delta_{\epsilon_0}$
is the moduli space of unramified irregular singular connections.
In \cite{Inaba-Saito}, we construct an algebraic splitting
\[
 \Psi_0\colon
 (\pi_{{\mathcal T}_{\bnu,\blambda,\epsilon=0}})^*
 T_{{\mathcal T}_{\bnu,\blambda,\epsilon=0}}
 \longrightarrow
 T_{M^{\balpha}_{{\mathcal C},{\mathcal D}}(\tilde{\bnu},\bmu)_{\epsilon=0}}
\]
of the surjection
$d\pi_{{\mathcal T}_{\bnu,\blambda,\epsilon=0}}\colon
T_{M^{\balpha}_{{\mathcal C},{\mathcal D}}(\tilde{\bnu},\bmu)_{\epsilon=0}}
\longrightarrow
(\pi_{{\mathcal T}_{\bnu,\blambda,\epsilon=0}})^*
T_{{\mathcal T}_{\bnu,\blambda,\epsilon=0}}$,
where
$T_{{\mathcal T}_{\bmu,\blambda,\epsilon=0}}$
and
$T_{M^{\balpha}_{{\mathcal C},{\mathcal D}}(\tilde{\bnu},\bmu)_{\epsilon=0}}$
are the tangent bundles of
${\mathcal T}_{\bmu,\blambda,\epsilon=0}$
and
$M^{\balpha}_{{\mathcal C},{\mathcal D}}(\tilde{\bnu},\bmu)_{\epsilon=0}$,
respectively.
The splitting
$\Psi_0$ is the irregular singular generalized isomonodromic deformation
arising from the theory by Jimbo, Miwa and Ueno in \cite{Jimbo-Miwa-Ueno}.
The idea of the construction of $\Psi_0$ is to construct a horizontal lift of the universal
relative connection, which is a first order infinitesimal extension of the relative connection
with an integrability condition.
We notice here that the complete description of the Jimbo-Miwa-Ueno equation
in \cite{Jimbo-Miwa-Ueno} says that the essential number of independent variables
of the unramified irregular singular generalized isomonodromic deformation is
$3g-3+\sum_{i=1}^n (r(m_i-1)+1)$.

One of the reasons of the difficulty in the confluence problem will be that 
the number $3g-3+\deg D$ of independent variables of the regular singular
isomonodromic deformation
is much smaller than the number
$3g-3+\sum_{i=1}^n (r(m_i-1)+1)$
of independent variables of the irregular singular generalized isomonodromic deformation.
Here we have $\deg D=\sum_{i=1}^n m_i$,
because the divisors are connected by a flat family.
In this paper, we try to extend the splitting $\Psi_0$ locally to the unfolded moduli space
$M^{\balpha}_{{\mathcal C},{\mathcal D}}(\tilde{\bnu},\bmu)$
via regarding ${\mathcal T}_{\bmu,\blambda}$ as the space of
independent variables.  
The main theorem of this paper is the following:
\begin{theorem}\label {thm:holomorphic-splitting}
For a general point
$x\in M^{\balpha}_{{\mathcal C},{\mathcal D}}(\tilde{\bnu},\bmu)_{\epsilon=0}$
satisfying Assumption \ref {assumption:irreducibility for universal family}
in subsection \ref {subsection:unfolded global horizontal lift},
there exist an analytic open neighborhood 
$M^{\circ}\subset M^{\balpha}_{{\mathcal C},{\mathcal D}}(\tilde{\bnu},\bmu)$
of $x$ whose image in ${\mathcal T}_{\bmu,\blambda}$ is denoted by
${\mathcal T}^{\circ}$,
blocks of local horizontal lifts
$\big(
\nabla^{flat}_{\mathbb{P}^1\times M^{\circ}[\bar{h}],v^{(i)}_{l,j}}
\big)$
defined in Definition \ref  {def:block of local horizontal lift}
and a holomorphic homomorphism
\[
  \Psi \colon
 (\pi_{{\mathcal T}^{\circ}})^*
 T^{hol}_{{\mathcal T}^{\circ}/\Delta_{\epsilon_0}}
  \longrightarrow
 T^{hol}_{M^{\circ}/\Delta_{\epsilon_0}}
\]
depending on 
$\big(
\nabla^{flat}_{\mathbb{P}^1\times M^{\circ}[\bar{h}],v^{(i)}_{l,j}}
\big)$,
which is a splitting of the canonical surjection of the tangent bundles
$T_{M^{\circ} /\Delta_{\epsilon_0}}
 \xrightarrow{ d \pi_{{\mathcal T}^{\circ}} }
 (\pi_{{\mathcal T}^{\circ}} )^*
 T_{{\mathcal T}^{\circ}/\Delta_{\epsilon_0}}$,
such that the restriction
$\Psi\big|_
{M^{\balpha}_{{\mathcal C},{\mathcal D}}(\tilde{\bnu},\bmu)_{\epsilon=0}\cap M^{\circ}}$
of $\Psi$ to the irregular singular locus coincides with
the  irregular singular generalized isomonodromic deformation
$\Psi_0^{hol}\big|_
{M^{\balpha}_{{\mathcal C},{\mathcal D}}(\tilde{\bnu},\bmu)_{\epsilon=0}\cap M^{\circ}}$.
\end{theorem}

The main idea of the construction of $\Psi$
in Theorem \ref {thm:holomorphic-splitting} is
to consider the restriction
$(\tilde{E},\tilde{\nabla},\{\tilde{N}^{(i)}\})|_{\Delta\times M^{\circ}}$
of the universal family of connections
to a local holomorphic disk $\Delta$
containing $D^{(i)}$ and to extend it 
to a family of connections on $\mathbb{P}^1$
admitting regular singularity along $\infty$.
We extend this family of connections on $\mathbb{P}^1$
to a family of integrable connections
$\nabla^{flat}_{\mathbb{P}^1\times M^{\circ}[\bar{h}],v^{(i)}_{l,j}}$
on $\mathbb{P}^1\times\Spec\mathbb{C}[h]/(h^2)$
depending on the data
$(\tilde{\Xi}^{(i)}_{l,j}(z))$ adjusting the residue part at $\infty$.
We glue the local integrable connections
$\nabla^{flat}_{\mathbb{P}^1\times M^{\circ}[\bar{h}],v^{(i)}_{l,j}}
\big|_{\Delta\times M^{\circ}}$
and obtain a global horizontal lift of
$(\tilde{E},\tilde{\nabla},\{\tilde{N}^{(i)}\})|_{{\mathcal C}_{M^{\circ}}}$,
which induces an unfolding in Theorem \ref {thm:holomorphic-splitting}.
In our unfolded generalized isomonodromic deformation determined by
$\Psi$, the monodromy along a loop surrounding whole the unfolding divisor $D^{(i)}$
is preserved constant, but the local monodromy
around each regular singular point in $D^{(i)}$ is not preserved constant,
because the local exponents are not constant.
So our unfolded generalized isomonodromic deformation does not mean
the usual regular singular isomonodromic deformation.
We notice that the splitting $\Psi$ in the theorem is not canonical because
it is essentially determined by the blocks of local horizontal lifts
$\big(\nabla^{flat}_{\mathbb{P}^1\times M^{\circ}[\bar{h}],v^{(i)}_{l,j}}\big)$
constructed in subsection \ref {subsection:local horizontal lift},
which depend on the data $(\tilde{\Xi}^{(i)}_{l,j}(z))$
adjusting the residue part
and also on a fundamental solution commuting with the monodromy
around $\infty$.
So we cannot expect the splitting $\Psi$
to be defined globally on
$M^{\balpha}_{{\mathcal C},{\mathcal D}}(\tilde{\bnu},\bmu)$.
Moreover, we cannot expect the integrability of the subbundle
$\im\Psi\subset T^{hol}_{M^{\circ}/\Delta_{\epsilon_0}}$.

The author's hope was to construct the unfolding $\Psi$
via adopting the asymptotic arguments in the unfolding theory
established by Hurtubise, Lambert and Rousseau
in a series of papers  \cite{Lambert-Rousseau-1}, \cite{Lambert-Rousseau-2},
\cite{Hurtubise-Lambert-Rousseau}, \cite{Hurtubise-Rousseau}.
Unfortunately we cannot achieve in such an easy way, because we do not know
that the unfolded Stokes matrices
defined in \cite{Hurtubise-Rousseau} are constant
for our generalized isomonodromic deformation $\Psi$.
This is another reason why the splitting $\Psi$
cannot be extended globally.
At the present, the framework of this paper is tentative because
the moduli space
$M^{\balpha}_{{\mathcal C},{\mathcal D}}(\tilde{\bnu},\bmu)$
dose not seem to be enough for the description of the
unfolded generalized isomonodromic deformation.
The author's hope is to find a good replacement of the moduli space
which describes our splitting $\Psi$ adequately.

The organization of this paper is the following.

In section \ref {section:linear-algebra}, we introduce a factorization
$V \xrightarrow {\kappa} V^{\vee} \xrightarrow{\theta} V$
of a given linear endomorphism $f\colon V\longrightarrow V$
whose minimal polynomial is of degree $\dim V$.
This gives the correspondence in Proposition \ref {prop:factorization-lemma}
and Proposition \ref  {prop:factorization-unique}
between
the linear endomorphisms $f\colon V\longrightarrow V$
whose minimal polynomial is of maximal degree
and the pairs $[(\theta,\kappa)]$ with $\theta,\kappa$ symmetric.
Using this correspondence, we can give in Proposition \ref{prop:kirillov-kostant-form}
a certain kind of expression of the Kirillov-Kostant symplectic form on
a $GL_r(\mathbb{C})$ adjoint orbit.

In section \ref  {section:algebraic moduli construction},
we introduce the notion of $(\bnu,\bmu)$-connection
which involves both a regular singular connection
and an unramified irregular singular connection.
We give a construction of the moduli space of
$(\bnu,\bmu)$-connections essentially using the
construction method in \cite{IIS-1}.
From the idea in section \ref {section:linear-algebra},
we can see that a $(\bnu,\bmu)$ connection corresponds to 
a tuple $(E,\nabla,\{\theta^{(i)},\kappa^{(i)}\})$.
Doing the deformation theory for this tuple,
we can get the smoothness of the moduli space
and a symplectic form.
These are summarized in Theorem \ref {theorem:algebraic-moduli-unfolding}.

In section \ref {section:theory of Hurtubise-Lambert-Rousseau},
we give an introduction to the unfolding theory constructed by
Hurtubise, Lambert and Rousseau 
by means of the restriction to a most easy case
when the perturbation of the singularity is given by the
equation $z^m-\epsilon^m=0$.
We need a consideration on the flows
given by $dz/dt=e^{\sqrt{-1}\theta}(z^m-\epsilon^m)$
in Proposition \ref {proposition:open-covering}.
One of the main tool in the unfolding theory is
a fundamental solution given in Theorem \ref {theorem:existence-of-fundamental-solutions}
which has an asymptotic property estimated by
flows given in Proposition \ref {proposition:open-covering}.

In section \ref {section:local horizontal lift},
we consider a family of connections $\nabla$
on a holomorphic disk $\Delta=\{z\in\mathbb{C} \, | \, |z|<1 \}$
admitting poles along $\{z^m-\epsilon^m=0\}$. 
Under some generic assumption on $\nabla$,
we give an extension of $\nabla$ as 
a family of connections on ${\mathcal O}_{\mathbb{P}^1}^{\oplus r}$
with a regular singularity along $\infty$,
whose connection matrix is given by
$A(z)dz/(z^m-\epsilon^m)$.
Using linear algebraic argument, we obtain
an adjusting data
$\tilde{\Xi}_{l,j}(z)$ such that
$\tilde{\Xi}_{l,j}(z)dz/(z^m-\epsilon^m)$
has no residue at $\infty$.
Then we can get a family of integrable connections
on $\mathbb{P}^1\times\Spec\mathbb{C}[h]/(h^2)$
given by a connection matrix
$(A(z)+\bar{h}\tilde{\Xi}_{l,j}(z))dz/(z^m-\epsilon^m)+B(z)d\bar{h}$
in Proposition \ref  {prop:local horizontal lift},
where $B(z)$ is a matrix of multivalued functions.

In section \ref  {section:construction of unfolding},
we give the setting of the relative moduli space of
$(\bnu,\bmu)$-connections whose generic fiber
is a moduli space of regular singular connections
and a special fiber is a moduli space of unramified irregular singular connections.
On the irregular singular fiber, we can define the generalized isomonodromic
deformation $\Psi_0$, which is basically determined by the Jimbo-Miwa-Ueno theory
and precisely given in \cite{Inaba-Saito}.
The integrability of the irregular singular generalized isomonodromic deformation
on $\mathbb{P}^1$ is proved in \cite{Jimbo-Miwa-Ueno}, which is
extended to ramified case in \cite{Bremer-Sage-2}.
We give in Theorem \ref {thm:integrability condition}
an alternative proof of its integrability involving the higher genus case
from the uniqueness property of its formulation.
Gluing the local integrable connections constructed in section \ref {section:local horizontal lift},
we construct a global horizontal lift in Proposition \ref {prop:existence-horizontal-lift},
which gives a local analytic lift of the unramified
irregular singular generalized isomonodromic deformation
and obtain Theorem \ref {thm:holomorphic-splitting}.

\section{An observation from linear algebra on a $GL_r(\mathbb{C})$ adjoint orbit}
\label {section:linear-algebra}

In this section, we give a small remark on an adjoint orbit of $GL_r(\mathbb{C})$
on $\mathfrak{gl}_r(\mathbb{C})$.
From the idea of the observation in this section, we will get in section
\ref{section:algebraic moduli construction} a convenient parametrization
of the local exponents of connections.
Furthermore, we will get a pertinent expression of the relative symplectic form on an unfolded
moduli space of connections on smooth projective curves
in section \ref {section:algebraic moduli construction}.

\subsection{Factorization of a linear endomorphism
whose minimal polynomial is of maximal degree}
\label {subsection:factorization of linear map}

Let $V$ be a vector space over $\mathbb{C}$ of dimension $r$
and $\mu_1,\ldots,\mu_r\in\mathbb{C}$ be mutually distinct
complex numbers.
If we consider the subvariety
\[
 C(\mu_1,\ldots,\mu_r):=
 \left\{ f \colon V \longrightarrow V\colon
 \text{linear map with the eigenvalues $\mu_1,\ldots,\mu_r$}
 \right\}
\]
of the affine space
$\Hom_{\mathbb{C}}(V,V)$,
then $C(\mu_1,\ldots,\mu_r)$
is isomorphic to the $GL_r(\mathbb{C})$-adjoint orbit of
the diagonal matrix
\[
 \begin{pmatrix}
  \mu_1 & \cdots & 0 \\
  \vdots & \ddots & \vdots \\
  0 & \cdots & \mu_r
 \end{pmatrix}.
\]
So $C(\mu_1,\ldots,\mu_r)$ has a symplectic structure
given by the Kirillov-Kostant symplectic form.
Indeed there is a canonical morphism from
$C(\mu_1,\ldots,\mu_r)$ to the complete flag variety
$F(V)$ by sending each $f$ to the flag of
$V$ induced by the eigen space decomposition of $f$.
The fiber is isomorphic to the set of upper triangular nilpotent matrices
which is also isomorphic to the cotangent space of
$F(V)$.
So $C(\mu_1,\ldots,\mu_r)$
is locally isomorphic over $F(V)$ to the cotangent bundle over $F(V)$
and the symplectic structure from the cotangent bundle
coincides with the Kirillov-Kostant symplectic form.
In  subsection \ref {subsection:expression of kirillov-kostant form}, 
we give another expression of the symplectic form
on the adjoint orbit
$C(\mu_1,\ldots,\mu_r)$.
For the construction of the symplectic form,
we extend to a slightly more general setting.

Let
$\varphi(T)\in\mathbb{C}[T]$ be a monic polynomial
of degree $r$
and $V$ be a vector space over $\mathbb{C}$ of dimension $r$.
We put
\[
 C_{\varphi(T)}:=\left\{
 f\colon V\longrightarrow V
 \left| \text{$f$ is a linear map whose minimal polynomial is $\varphi(T)$}
 \right\}\right..
\]
Recall that $\varphi(T)$ is a minimal polynomial of $f\colon V\longrightarrow V$
if and only if $\varphi(f)=0$ and the induced map
\[
 \mathbb{C}[T]/(\varphi(T)) \ni \overline{P(T)}
 \mapsto P(f)\in \End(V)
\]
is injective.

\begin{proposition}\label {prop:factorization-lemma}
 For each $f\in C_{\varphi(T)}$,
 there are an isomorphism $\theta\colon V^{\vee}\xrightarrow{\sim}V$
 and a linear map $\kappa\colon V\longrightarrow V^{\vee}$
 satisfying $f=\theta\circ\kappa$,
 $^t\theta=\theta$ and $^t\kappa=\kappa$.
 Here $V^{\vee}$ is the dual vector space of $V$,
 $^t\theta\colon V^{\vee}\longrightarrow(V^{\vee})^{\vee}=V$
 is the dual of $\theta$
 and
 $^t\kappa\colon V=(V^{\vee})^{\vee}\longrightarrow V^{\vee}$
 is the dual of $\kappa$.
\end{proposition}

\begin{proof}
The ring homomorphism
$\mathbb{C}[T]\ni P(T)\mapsto P(f)\in\End(V)$
induces a $\mathbb{C}[T]$-module structure on $V$.
By an elementary theory of linear algebra, there is  an isomorphism
\[
 V\xrightarrow{\sim}\mathbb{C}[T]/(\varphi(T)),
\]
of $\mathbb{C}[T]$-modules,
because the minimal polynomial $\varphi(T)$ of
$f$ has degree $r=\dim V$.
Since the minimal polynomial of $\,^tf$ coincides with $\varphi(T)$,
there is an isomorphism
\[
 V^{\vee}\xrightarrow{\sim} \mathbb{C}[T]/(\varphi(T))
\]
of $\mathbb{C}[T]$-modules.
So we can take an isomorphism
\[
 \theta\colon V^{\vee}\xrightarrow{\sim} V
\]
of $\mathbb{C}[T]$-modules.
If we put
\[
 \kappa:=\theta^{-1}\circ f\colon V\longrightarrow V^{\vee},
\]
then $\kappa$ becomes a homomorphism of $\mathbb{C}[T]$-modules
and $f=\theta\circ\kappa$.
We take a generator $v^*\in V^{\vee}$ of $V^{\vee}$
as a $\mathbb{C}[T]$-module.
Then $v:=\theta(v^*)\in V$
is a generator of $V$ as a $\mathbb{C}[T]$-module.
Take any $w^*_1,w^*_2\in V^{\vee}$.
Then we can write
$w^*_1=P_1(\, ^t f)v^*$ and $w^*_2=P_2(\, ^t f)v^*$ for certain polynomials
$P_1(T),P_2(T)\in\mathbb{C}[T]$.
For the dual pairing
$\langle \ , \  \rangle\colon V^{\vee}\times V\longrightarrow\mathbb{C}$,
we have
\begin{align*}
 \langle w^*_2, \ ^t\theta(w^*_1) \rangle 
 =\langle w^*_1\circ\theta, w^*_2 \rangle 
 &= \langle w^*_1, \ \theta(w^*_2) \rangle \\ 
 &=\langle P_1(\, ^t f)v^*, \theta (P_2(\, ^t f)v^*) \rangle \\
 &=\langle v^*\circ P_1(f), P_2(f) (\theta (v^*) ) \rangle \\
 &=\langle v^*, P_1(f) P_2(f) (\theta (v^*) ) \rangle \\
 &=\langle v^*,P_2(f)P_1(f) (\theta (v^*) ) \rangle \\ 
 &=\langle P_2(\,^t f)v^*,\theta(P_1(\,^t f)v^*) \rangle 
 =\langle w^*_2,\theta(w^*_1)\rangle.
\end{align*}
So we have $^t\theta(w^*_1)=\theta(w^*_1)$
and $^t\theta=\theta$.

Take any
$w_1,w_2\in V$.
Then there are polynomials $P_1(T),P_2(T)\in\mathbb{C}[T]$
satisfying $w_1=P_1(f)v$ and $w_2=P_2(f)v$.
We have
\begin{align*}
 \langle \,^t\kappa(w_1), w_2 \rangle 
 =
 \langle \kappa(w_2), w_1 \rangle  
 &=
 \langle \kappa(P_2(f)v), P_1(f)v \rangle \\ 
 &=
 \langle \theta^{-1}f P_2(f)v,P_1(f)v \rangle \\
 &=
 \langle ^t (f  P_2(f)) \theta^{-1}(v),P_1(f)v \rangle \\
 &=
 \langle \theta^{-1}(v), f P_2(f) P_1(f) v \rangle \\
 &=
 \langle \theta^{-1}(v), f P_1(f)P_2(f)v \rangle \\ 
 &=\langle \kappa(P_1(f)v), P_2(f)v \rangle 
 =\langle \kappa(w_1), w_2 \rangle.
\end{align*}
So we have $^t\kappa(w_1)=\kappa(w_1)$
and $^t\kappa=\kappa$ holds.
\end{proof}

\begin{proposition}\label {prop:factorization-unique}
For $f\in C_{\varphi(T)}$, assume that
$\theta_1,\theta_2\colon V^{\vee}\xrightarrow{\sim} V$
are isomorphisms
and
$\kappa_1,\kappa_2\colon V\longrightarrow V^{\vee}$
are linear maps satisfying
$f=\theta_1\circ\kappa_1=\theta_2\circ\kappa_2$,
$^t\theta_1=\theta_1$, $^t\theta_2=\theta_2$,
$^t\kappa_1=\kappa_1$ and $^t\kappa_2=\kappa_2$.
Then there exists $\overline{P(T)}\in(\mathbb{C}[T]/(\varphi(T)))^{\times}$
satisfying
$\theta_2=\theta_1\circ P(\,^t f)$ and $\kappa_2=(P(\,^t f))^{-1}\circ\kappa_1$.
\end{proposition}

\begin{proof}
Put $\sigma:=\theta_1^{-1}\circ\theta_2\colon V^{\vee}\longrightarrow V^{\vee}$.
Then
$^t f\circ \sigma=\,^t\kappa_1\circ\,^t\theta_1\circ\theta_1^{-1}\circ\theta_2
=\kappa_1\circ\theta_1\circ\theta_1^{-1}\circ\theta_2=\kappa_1\circ\theta_2$
and
$\sigma\circ \, ^t f=\theta_1^{-1}\circ\theta_2\circ\,^t\kappa_2\circ \,^t \theta_2
=\theta_1^{-1}\circ\theta_2\circ\kappa_2\circ\theta_2
=\theta_1^{-1}\circ f\circ\theta_2
=\theta_1^{-1}\circ\theta_1\circ\kappa_1\circ\theta_2
=\kappa_1\circ\theta_2$.
So $\sigma\circ\,^t f=\,^t f\circ\sigma$
and
$\sigma\colon V^{\vee}\xrightarrow{\sim} V^{\vee}$
becomes a $\mathbb{C}[T]$-isomorphism.
Since $\mathbb{C}[T]/(\varphi(T))\xrightarrow{\sim}\Hom_{\mathbb{C}[T]}(V^{\vee},V^{\vee})$,
there exists $\overline{P(T)}\in(\mathbb{C}[T]/(\varphi(T)))^{\times}$
satisfying $P(\,^t f)=\sigma=\theta_1^{-1}\circ\theta_2$.
So we have $\theta_1\circ P(\,^t f)=\theta_2$,
$\kappa_1=\theta_1^{-1}\circ f=\theta_1^{-1}\circ\theta_2\circ\kappa_2
=\sigma\circ\kappa_2$
and
$\kappa_2=\sigma^{-1}\circ\kappa_1=P(\,^t f)^{-1}\circ\kappa_1$.
\end{proof}

\subsection{An expression of the symplectic form on a $GL_r(\mathbb{C})$ adjoint orbit}
\label {subsection:expression of kirillov-kostant form}

Let the notations $V$, $\varphi(T)$, $r$ and $C_{\varphi(T)}$ be as in
subsection \ref {subsection:factorization of linear map}.
We set
\begin{align*}
 S(V^{\vee},V)&= \left\{ \theta\in\Hom_{\mathbb{C}}(V^{\vee},V)
 \left| \,^t\theta=\theta \right\}\right. \\
 S(V,V^{\vee})&= \left\{ \kappa\in\Hom_{\mathbb{C}}(V,V^{\vee})
 \left| \,^t\kappa=\kappa \right\}\right. 
\end{align*}
and
\[
 {\mathcal S}:=\left\{ (\theta,\kappa)\in  S(V^{\vee},V)\times S(V,V^{\vee})
 \left| 
 \begin{array}{l}
 \text {$\theta$ is isomorphic, $\varphi(\theta\circ\kappa)=0$ and the induced map} \\
 \text {$\mathbb{C}[T]/(\varphi(T))\ni\overline{P(T)}\mapsto
 P(\theta\circ\kappa) \in \End(V)$ is injective}
 \end{array}
 \right\}\right..
\]
Then there is an action of the commutative algebraic group
$(\mathbb{C}[T]/(\varphi(T)))^{\times}$ on ${\mathcal S}$ defined by
\[
 \overline{P(T)}\cdot (\theta,\kappa)
 =(\theta\circ P(\kappa\circ\theta) , \, P(\kappa\circ\theta)^{-1}\circ\kappa).
\]
for $\overline{P(T)} \in (\mathbb{C}[T]/(\varphi(T)))^{\times}$.
We can see by Proposition \ref {prop:factorization-lemma}
and Proposition \ref {prop:factorization-unique} that
the quotient of ${\mathcal S}$ by the action of $(\mathbb{C}[T]/(\varphi(T)))^{\times}$
is isomorphic to $C_{\varphi(T)}$:
\[
 {\mathcal S}/(\mathbb{C}[T]/(\varphi(T)))^{\times}\cong C_{\varphi(T)}.
\]
We describe the tangent space of $C_{\varphi(T)}$
at $f=\theta\circ\kappa$ via this isomorphism.
Let us consider the complex
\begin{equation} \label {equation:factorization complex}
 \mathbb{C}[T]/(\varphi(T)) \xrightarrow{d^0}
 S(V^{\vee},V)\oplus S(V,V^{\vee}) \xrightarrow{d^1}
 (\mathbb{C}[T]/(\varphi(T)))^{\vee}
\end{equation}
defined by
\begin{gather*}
  d^0(\overline{P(T)}) = (\theta\circ P(\,^tf) , \, -P(\,^tf)\circ\kappa)
 \hspace{30pt}
 \big( \overline{P(T)}\in \mathbb{C}[T]/(\varphi(T)) \big)  \\
  d^1(\tau,\xi) \colon \mathbb{C}[T]/(\varphi(T))\ni
 \overline{P(T)}\mapsto \Tr(P(f)\circ (\theta\circ\xi+\tau\circ\kappa))
 \in\mathbb{C}
 \hspace{20pt} 
 \big( (\tau,\xi)\in S(V^{\vee},V)\oplus S(V,V^{\vee}) \big).
\end{gather*}

\begin{proposition} \label {prop:tangent space of S}
 The tangent space $T_{\mathcal S}(\theta,\kappa)$
 of ${\mathcal S}$ at $(\theta,\kappa)$ is isomorphic to
 $\ker d^1$.
\end{proposition}

Before proving the proposition, we prove the following lemma.

\begin{lemma} \label {lemma:adjoint exact sequence}
For $f\in C_{\varphi(T)}$, the sequence
\[
 0 \longrightarrow \mathbb{C}[T]/(\varphi(T))
 \xrightarrow{\iota_f} \End_{\mathbb{C}}(V) \xrightarrow{\mathrm{ad}(f)}
 \End_{\mathbb{C}}(V) \xrightarrow{\pi_f}
 (\mathbb{C}[T]/(\varphi(T)))^{\vee}
 \longrightarrow 0
\]
is exact,
where $\iota_f$ is defined by $\iota_f(\overline{P(T)})=P(f)$
and $\pi_f$ is the dual of $\iota_f$.
\end{lemma}

{\it Proof of Lemma \ref {lemma:adjoint exact sequence}.}
The map $\iota_f$ is injective since $f$ belongs to
$C_{\varphi(T)}$.
Since the minimal polynomial of $f$ is of degree $r=\dim V$,
the linear map
\[
 \mathrm{ad}(f) \colon \End_{\mathbb{C}}(V) \ni g \mapsto f\circ g-g\circ f
 \in \End_{\mathbb{C}}(V)
\]
satisfies $\ker \mathrm{ad}(f)=\mathbb{C}[f]=\im \iota_f$.
In particular, we have $\rank\mathrm{ad}(f)=r^2-r$.
The map $\pi_f$ is given by
\[
 \pi_f(g) (\overline{P(T)})
 =
 \Tr ( g \circ P(f) )
\]
for $g\in\End_{\mathbb{C}}(V)$ and $\overline{P(T)}\in\mathbb{C}[T]/(\varphi(T))$.
So we have
\begin{align*}
 \pi_f(\mathrm{ad}(f)(g)) (\overline{P(T)})
 &=\Tr ( (f\circ g-g\circ f)(P(f)) \\
 &=\Tr (P(f)\circ f\circ g) - \Tr( g\circ f\circ P(f))
 =\Tr( f\circ P(f)\circ g) - \Tr (f\circ P(f)\circ g)
 =0
\end{align*}
for $g\in \End_{\mathbb{C}}(V)$ and $\overline{P(T)}\in \mathbb{C}[T]/(\varphi(T))$,
which means $\pi_f\circ\mathrm{ad}(f)=0$.
So we have
\[
 \im \mathrm{ad}(f)=\ker \pi_f
 =
 \left\{ g\in\End_{\mathbb{C}}(V) \left|
 \text{$\Tr(f^i\circ g)=0$ for $i=0,1,\ldots,r-1$} \right\}\right. ,
\]
because the right hand side is of dimension $r^2-r$.
Thus we have proved the lemma.
\hfill $\square$

\vspace{10pt}

\noindent
{\it Proof of Proposition \ref {prop:tangent space of S}.}
If we take $(\tau,\xi) \in \ker d^1$, we have
$\pi_f(\theta\circ\xi+\tau\circ\kappa)=d^1(\tau,\xi)=0$.
By Lemma \ref {lemma:adjoint exact sequence},
there is $g\in \End(V)$ satisfying
$\theta\circ\xi+\tau\circ\kappa=f\circ g-g\circ f$.
We write
$\varphi(T)=b_r T^r+b_{r-1}T^{r-1}+\cdots +b_1T+b_0$
with $b_r=1$.
Then the $\mathbb{C}[t]/(t^2)$-valued point
$(\theta+\tau\bar{t},\kappa+\xi\bar{t})$
of $S(V^{\vee},V)\times S(V,V^{\vee})$
satisfies
\begin{align*}
 \varphi((\theta+\tau\bar{t})\circ(\kappa+\xi\bar{t}))
 &=
 \varphi(f+(\theta\circ\xi+\tau\circ\kappa)\bar{t}) 
 =\varphi(f+(f\circ g-g\circ f)\bar{t}) 
 =
 \sum_{i=0}^r b_i \; (f+(f\circ g-g\circ f)\bar{t})^i \\
 &=
 \sum_{i=0}^r b_i \left( f^i +\sum_{j=0}^{i-1}f^j(f\circ g-g\circ f)f^{i-j-1}\bar{t}\right) 
 =
 \sum_{i=0}^r b_i \left( f^i + (f^i\circ g-g\circ f^i)\bar{t} \right) \\
 &=
 \varphi(f)+(\varphi(f)\circ g-g\circ\varphi(f))\bar{t}=0.
\end{align*}
So $(\theta+\tau\bar{t},\kappa+\xi\bar{t})$ gives a tangent vector of
${\mathcal S}$ at $(\theta,\kappa)$.

Conversely take a tangent vector of ${\mathcal S}$ 
and let $(\theta+\tau\bar{t},\kappa+\xi\bar{t})$ 
be the  corresponding $\mathbb{C}[t]/(t^2)$-valued point of ${\mathcal S}$.
Then we have
$\varphi((\theta+\tau\bar{t})\circ(\kappa+\xi\bar{t}))=0$ and
\begin{equation} \label {equation:injective ring homomorphism for adjoint orbit}
 \mathbb{C}[t]/(t^2)[T]/(\varphi(T))
 \ni \overline{P(T)}\mapsto P((\theta+\tau\bar{t})\circ(\kappa+\xi\bar{t}))
 \in \End_{\mathbb{C}[t]/(t^2)}(V\otimes_{\mathbb{C}}\mathbb{C}[t]/(t^2))
\end{equation}
is injective, whose cokernel is flat over $\mathbb{C}[t]/(t^2)$.
Recall that there is an isomorphism
$\sigma\colon \mathbb{C}[T]/(\varphi(T))\xrightarrow{\sim}V$.
So we can take a generator $v=\sigma(1)$ of $V$ as a $\mathbb{C}[T]$-module.
If we take a lift $\tilde{v}\in V\otimes\mathbb{C}[t]/(t^2)$ of
$v$, then $\tilde{v}$ becomes a generator of $V\otimes\mathbb{C}[t]/(t^2)$
as a $\mathbb{C}[t]/(t^2)[T]$-module with respect to the action of
$\mathbb{C}[t]/(t^2)[T]$ induced by the ring homomorphism 
(\ref {equation:injective ring homomorphism for adjoint orbit}).
So we have an isomorphism
\[
 \tilde{\sigma}\colon
 \mathbb{C}[t]/(t^2)[T]/(\varphi(T))\xrightarrow{\sim} V\otimes\mathbb{C}[t]/(t^2)
\]
satisfying $\tilde{\sigma}(1)=\tilde{v}$.
If we denote by $\mathrm{id}$ the identity map,
$\sigma\otimes\mathrm{id}\colon
\mathbb{C}[T]/(\varphi(T))\otimes\mathbb{C}[t]/(t^2)
\xrightarrow{\sim} V\otimes_{\mathbb{C}}\mathbb{C}[t]/(t^2)$
is another $\mathbb{C}[t]/(t^2)[T]$-isomorphism
with respect to the action of $\mathbb{C}[t]/(t^2)[T]$
on $V\otimes_{\mathbb{C}}\mathbb{C}[t]/(t^2)$ via the ring homomorphism
\[
 \mathbb{C}[t]/(t^2)[T] \ni P(T) \mapsto P(\theta\circ\kappa\otimes\mathrm{id})
 \in \End_{\mathbb{C}[t]/(t^2)}(V\otimes\mathbb{C}[t]/(t^2)).
\]
Composing $\tilde{\sigma}^{-1}$ with $\sigma\otimes\mathrm{id}$,
we obtain a $\mathbb{C}[t]/(t^2)$-automorphism of
$V\otimes\mathbb{C}[t]/(t^2)$
of the form
$\mathrm{id}+Q\bar{t}$ 
with $Q\in\End_{\mathbb{C}}(V)$
which makes the diagram
\[
 \begin{CD}
  V\otimes_{\mathbb{C}}\mathbb{C}[t]/(t^2) @> (\theta+\tau\bar{t})\circ(\kappa+\xi\bar{t}) >>
  V\otimes_{\mathbb{C}}\mathbb{C}[t]/(t^2)  \\
  @V \mathrm{id}+Q\bar{t}  V V     @V \mathrm{id}+Q\bar{t} V V \\
  V\otimes_{\mathbb{C}}\mathbb{C}[t]/(t^2) @> \theta\circ\kappa\otimes \mathrm{id} >>
  V\otimes_{\mathbb{C}}\mathbb{C}[t]/(t^2)
 \end{CD}
\]
commutative.
Then we have
\[
 (\theta\circ\xi+\tau\circ\kappa)\bar{t}
 =
 (\theta+\tau\bar{t})\circ(\kappa+\xi\bar{t}) - \theta\circ\kappa  
 =
 (\mathrm{id}-Q\bar{t})\circ(\theta\circ\kappa)\circ(\mathrm{id}+Q\bar{t})
 -\theta\circ\kappa
 =(f\circ Q-Q\circ f)\bar{t}
\]
and
\[
 \Tr(f^i\circ(\theta\circ\xi+\tau\circ\kappa))=
 \Tr(f^i(f\circ Q-Q\circ f))
 =\Tr(f^{i+1}\circ Q-Q\circ f^{i+1})=0
\]
for any $i\geq 0$.
Thus we have $(\tau,\xi)\in\ker d^1$.
By the correspondence $(\tau,\xi)\mapsto (\theta+\tau\bar{t},\kappa+\xi\bar{t})$,
we get the isomorphism from $\ker d^1$ to the tangent space of
${\mathcal S}$ at $(\theta,\kappa)$.
\hfill $\square$

We can see that
$\im (d^0)$ coincides with
the tangent space of the $(\mathbb{C}[T]/(\varphi(T))^{\times}$-orbit
of $(\theta,\kappa)$
in ${\mathcal S}$.
So the tangent space of
$C_{\varphi(T)}={\mathcal S}/(\mathbb{C}[T]/(\varphi(T))^{\times}$ at
$f=\theta\circ\kappa$
is isomorphic to
$T_{\mathcal S}(\theta,\kappa)/\im d^0$
which is the first cohomology of the complex (\ref {equation:factorization complex}):
\[
 T_{C_{\varphi(T)}}(f)\cong
 H^1\left(\mathbb{C}[T]/(\varphi(T)) \xrightarrow{d^0}
 S(V^{\vee},V)\oplus S(V,V^{\vee}) \xrightarrow{d^1}
 (\mathbb{C}[T]/(\varphi(T)))^{\vee}\right).
\]
We define a pairing
\[
 \omega_{C_{\varphi(T)}} \colon
 T_{C_{\varphi(T)}}(f) \times T_{C_{\varphi(T)}}(f)
 \longrightarrow \mathbb{C}
\]
by
\begin{equation} \label {equation:symplectic form for adjoint orbit}
 \omega_{C_{\varphi(T)}}( [(\tau,\xi)], [(\tau',\xi')]) =
 \frac{1}{2} \Tr(\tau\circ\xi'-\tau'\circ\xi).
\end{equation}
If $[(\tau,\xi)]=0$, then we can write
$\tau=\theta\circ P(\,^t f)$ and $\xi=-P(\,^t f)\circ\kappa$.
So we have
\[
 \Tr(\tau\circ\xi'-\tau'\circ\xi)=
 \Tr(\theta\circ P(\,^t f)\circ\xi'+\tau'\circ P(\,^t f)\circ\kappa)
 =\Tr(P(f)\circ(\theta\circ\xi'+\tau'\circ\kappa))=0.
\]
Similarly we can see that $ \Tr(\tau\circ\xi'-\tau'\circ\xi)=0$
if $[(\tau',\xi')]=0$.
Thus the pairing (\ref {equation:symplectic form for adjoint orbit})
is well-defined.
On the other hand, there is a well-known symplectic form so called
the Kirillov-Kostant form.
For two tangent vectors
$[(\tau,\xi)], [(\tau',\xi')] \in T_{C_{\varphi(T)}}(f)$
of $C_{\varphi(T)}$
at $f=\theta\circ\kappa$,
we can see by Lemma \ref {lemma:adjoint exact sequence}
that there exist $g, g' \in\Hom(V,V)$ satisfying
$f\circ g-g\circ f=\theta\circ\xi+\tau\circ\kappa$ and
$f\circ g'-g'\circ f=\theta\circ\xi'+\tau'\circ\kappa$.
The Kirillov-Kostant symplectic form $\omega_{\text{K-K}}$ is
defined in \cite[page 5, Definition 1]{Kirillov} by
\[
 \omega_{\text{K-K}}([(\tau,\xi)],[(\tau',\xi')])
 =\Tr(f\circ([g,g'])).
\]

\begin{proposition} \label {prop:kirillov-kostant-form}
 The pairing $\omega_{C_{\varphi(T)}}$ defined
 in (\ref {equation:symplectic form for adjoint orbit})
 coincides with the Kirillov-Kostant symplectic form $\omega_{\text{K-K}}$
 on the adjoint orbit $C_{\varphi(T)}$. 
\end{proposition}

\begin{proof}
Take any member
$(a,b)\in  S(V^{\vee},V)\oplus S(V,V^{\vee})$
satisfying
$\theta\circ b+ a\circ\kappa=0$.
Then we have
\[
 (\theta+a\, \bar{t})\circ(\kappa+b\, \bar{t})=\theta\circ\kappa
 =f
 \in \End_{\mathbb{C}[t]/(t^2)}(V\otimes_{\mathbb{C}}\mathbb{C}[t]/(t^2)),
\]
from which we can see
\[
 (\kappa+b\, \bar{t})\circ(\theta+a\, \bar{t})
 = \,^t(\kappa+b\, \bar{t})\circ\,^t(\theta+a\, \bar{t})
 =\,^t\left( (\theta+a\, \bar{t})\circ(\kappa+b\, \bar{t})\right)
 =\,^tf=\kappa\circ\theta.
\]
So we have
\begin{align*}
 (\mathrm{id}+\theta^{-1}a\, \bar{t})\circ \,^t f
 &= \theta^{-1}\circ(\theta+a\, \bar{t})\circ(\kappa+b\, \bar{t})\circ(\theta+a\, \bar{t}) \\
 &=\theta^{-1}\circ\theta\circ\kappa\circ(\theta+a\, \bar{t})
 =\kappa\circ\theta+\kappa\circ \theta\circ\theta^{-1}\circ a\, \bar{t}
 =\,^t f \circ (\mathrm{id}+\theta^{-1}\circ a\, \bar{t}).
\end{align*}
Then we have
$\theta^{-1}\circ a \in \End_{\mathbb{C}[T]}(V^{\vee})\cong\mathbb{C}[T]/(\varphi(T))$
and there exists $\overline{P(T)}\in\mathbb{C}[T]/(\varphi(T))$
satisfying $\theta^{-1}a=P(\,^tf)$.
So we have $a=\theta\circ P(\,^tf)$ and
$b=-\theta^{-1}\circ a\circ\kappa=-P(\,^tf)\circ\kappa$,
which mean that
$(a,b) \in \im (d^0)$.
Thus we have proved 
\begin{equation} \label {equation:kernel as tangent of orbit}
 \im (d^0)=
 \ker \Big( S(V^{\vee},V)\oplus S(V,V^{\vee})\ni (a,b)\mapsto
 \theta\circ b+ a\circ\kappa\in \Hom(V,V) \Big).
\end{equation}

Take two tangent vectors
$[(\tau,\xi)], [(\tau',\xi')] \in T_{C_{\varphi(T)}}(f)$
of $C_{\varphi(T)}$
at $f=\theta\circ\kappa$.
Since $(\tau,\xi), (\tau',\xi') \in \ker d^1$,
we can see from Lemma \ref {lemma:adjoint exact sequence}
that there exist $g, g' \in\Hom(V,V)$ satisfying
$f\circ g-g\circ f=\theta\circ\xi+\tau\circ\kappa$ and
$f\circ g'-g'\circ f=\theta\circ\xi'+\tau'\circ\kappa$.
Note that we have
\[
 \theta\circ(\kappa\circ g'+\,^t g'\circ\kappa)
 +(-g'\circ\theta-\theta\circ\,^t g')\circ\kappa
 =\theta\circ\kappa\circ g'-g'\circ\theta\circ\kappa
 =f\circ g'-g'\circ f=\theta\circ\xi'+\tau'\circ\kappa.
\]
By the equality (\ref {equation:kernel as tangent of orbit}),
we have
$[(\tau',\xi')]=[(-g'\circ\theta-\theta\circ\,^t g',\kappa\circ g'+\,^t g'\circ\kappa)]$
in $T_{C_{\varphi(T)}}(f)$
and we may assume that
$\tau'=-g'\circ\theta-\theta\circ\,^t g'$
and
$\xi'=\kappa\circ g'+\,^t g'\circ\kappa$.
We have
\begin{align*}
 \omega_{\text{K-K}}([(\tau,\xi)],[(\tau',\xi')])
 &=
 \Tr(f\circ([g,g'])) = \Tr( f\circ(g\circ g'-g'\circ g)) \\
 &=
 \Tr((f\circ g-g\circ f)\circ g'+(g\circ f\circ g'-f\circ g'\circ g)) \\
 &=
 \Tr((\theta\circ\xi+\tau\circ\kappa)\circ g')+\Tr(g\circ (f\circ g')-(f\circ g')\circ g) \\
 &=
 \Tr((\theta\circ\xi+\tau\circ\kappa)\circ g')
 =\Tr(g'\circ\theta\circ\xi)+\Tr(\tau\circ\kappa\circ g') \\
 &=
 \frac{1}{2}
 \left( \Tr(g'\circ \theta\circ\xi)+\Tr(\,^t\xi\circ\,^t\theta\circ\,^t g')
 +\Tr(\tau\circ\kappa\circ g')+\Tr(\,^t g'\circ\,^t\kappa\circ\,^t\tau) \right).
\end{align*}
\begin{claim}\label {claim:trace equality}
 $\Tr(u\circ v)=\Tr(v\circ u)$
 for any $u \in \Hom(V,V^{\vee})$ and any $v \in \Hom(V^{\vee},V)$.
\end{claim}

Using the above claim, we have
$\Tr(\,^t\xi\circ\,^t\theta\circ\,^t g')=\Tr(\,^t\theta\circ\,^t g'\circ\,^t\xi)
=\Tr(\theta\circ\,^t g'\circ\xi)$
and
$\Tr(\,^t g'\circ\,^t\kappa\circ\,^t\tau))=\Tr(\,^t\tau\circ \,^t g'\circ\,^t\kappa)
=\Tr(\tau\circ \,^t g'\circ\kappa)$.
So we have
\begin{align*}
 \omega_{\text{K-K}}([(\tau,\xi)],[(\tau',\xi')])
 &=
 \frac{1}{2}\left(\Tr(g'\circ \theta\circ\xi)+\Tr(\,^t\xi\circ\,^t\theta\circ\,^t g')
 +\Tr(\tau\circ\kappa\circ g')+\Tr(\,^t g'\circ\,^t\kappa\circ\,^t\tau)\right) \\
 &=
 \frac{1}{2}\left( \Tr((g'\circ\theta+\theta\circ\,^t g')\circ\xi)
 +\Tr(\tau\circ(\kappa\circ g'+\,^t g'\circ\kappa)) \right) \\
 &=\frac{1}{2} \left( \Tr(-\tau'\circ \xi)+\Tr(\tau\circ\xi') \right)
 =\omega_{C_{\varphi(T)}}([(\tau,\xi)],[(\tau',\xi')]).
\end{align*}

For the proof of Claim \ref {claim:trace equality},
we take a basis $e_1,\ldots,e_r$ of $V$
and its dual basis $e_1^*,\ldots,e_r^*$ of $V^{\vee}$.
If write $u(e_j)=\sum_{i=1}^r a_{i j} e^*_i$
and $v(e^*_l)=\sum_{k=1}^r b_{kl} e_k$,
then we have
\begin{align*}
 \Tr(u\circ v)
 &=
 \Tr\left( \sum_{i,l=1}^r \sum_{k=1}^r a_{i k} b_{kl} \, e^*_i\otimes e_l \right)
 =\sum_{k=1}^r \sum_{i=1}^r a_{i k}b_{k i} \\
 \Tr(v\circ u)
 &=
 \Tr\left( \sum_{j,k=1}^r\sum_{i=1}^r a_{i j} b_{k i} \, e_k\otimes e_j^*\right)
 =\sum_{i=1}^r \sum_{k=1}^r a_{i k} b_{k i}
\end{align*}
So we have $\Tr(u\circ v)=\Tr(v\circ u)$ and
Claim \ref {claim:trace equality} follows.
Thus we have proved $\omega_{\text{K-K}}=\omega_{C_{\varphi(T)}}$.
\end{proof}

\section{Algebraic construction of an unfolding of the moduli space
of unramified irregular singular connections}
\label {section:algebraic moduli construction}

\subsection{Regular singular and unramified irregular singular connections
as $(\bnu,\bmu)$-connections}

Let $C$ be a complex smooth projective irreducible curve of genus $g$.
We take an effective divisor $D\subset C$,
which has a decomposition
$D=D^{(1)}+D^{(2)}+\cdots+D^{(n)}
=D^{(1)}\sqcup\cdots\sqcup D^{(n)}$,
where each $D^{(i)}$ is an effective divisor of degree $m_i$
and $D^{(i)}\cap D^{(i')}=\emptyset$ for $i\neq i'$.
We write
$D^{(i)}=p^{(i)}_1+p^{(i)}_2+\cdots +p^{(i)}_{m_i}$  for $1\leq i\leq n$,
where each $p^{(i)}_j$ is a reduced point in $C$ and it may be possible
that $p^{(i)}_j=p^{(i)}_{j'}$ for $j\neq j'$.

Using the Chinese remainder theorem
\[
 {\mathcal O}_{2D^{(i)}}\cong
 \prod_{p\in D^{(i)}} {\mathcal O}_{2D^{(i)},p},
\]
we can choose $\bar{z}^{(i)}\in{\mathcal O}_{2D^{(i)}}$ satisfying
$\bar{z}^{(i)}(p^{(i)}_j)\neq \bar{z}^{(i)}(p^{(i)}_{j'})$
for $p^{(i)}_j\neq p^{(i)}_{j'}$ and
$d\bar{z}^{(i)}|_{p^{(i)}_j}\neq 0\in\Omega^1_C|_{p^{(i)}_j}$ for $j=1,\ldots,m_i$.
We write $\bar{z}^{(i)}_j:=\bar{z}^{(i)}-\bar{z}^{(i)}(p^{(i)}_j)$,
where $\bar{z}^{(i)}(p^{(i)}_j)\in\mathbb{C}$ is the value of $\bar{z}^{(i)}$ at $p^{(i)}_j$.
We take local lifts
$z^{(i)} \in {\mathcal O}_{\mathcal C}$ of $\bar{z}^{(i)}$,
put $z^{(i)}_j:=z^{(i)}-z^{(i)}(p^{(i)}_j)$ and define
\begin{equation} \label {equation: unfolding equation on finite scheme}
 \frac{d\bar{z}^{(i)}}{\bar{z}^{(i)}_1\bar{z}^{(i)}_2\cdots \bar{z}^{(i)}_{m_i}}
 :=
 \frac{dz^{(i)}}{z^{(i)}_1 z^{(i)}_2\cdots z^{(i)}_{m_i}} \bigg|_{D^{(i)}}
 \in \Omega^1_C(D)|_{D^{(i)}}
\end{equation}
which becomes a local basis of $\Omega^1_C(D)|_{D^{(i)}}$.
Note that the above definition is independent of the choice of representatives
$z^{(i)}$ of $\bar{z}^{(i)}$.
We denote the multiplicity of $D^{(i)}$ at each $p\in D^{(i)}$ by $m^{(i)}_p$.
If $l_1,\ldots,l_{m_i}$ are integers satisfying
$0\leq l_1,\ldots,l_{m_i}\leq 1$,
there is a unique decomposition
\begin{equation} \label {equation:factor decomposition}
 \frac{d\bar{z}^{(i)}}{(\bar{z}^{(i)}_1)^{l_1}(\bar{z}^{(i)}_2)^{l_2}\cdots (\bar{z}^{(i)}_{m_i})^{l_{m_i}}}
 =
 \sum_{p\in D^{(i)}}
 \sum_{1\leq j\leq m^{(i)}_p} \frac{ a^{(i)}_{p,j} \; d\bar{z}^{(i)} } { (\bar{z}^{(i)}-\bar{z}^{(i)}(p)) )^j }
\end{equation}
with $a^{(i)}_{p,j}\in\mathbb{C}$.
Since $a^{(i)}_{p,j}$ is determined by
\[
 a^{(i)}_{p,j}=
 \lim_{z^{(i)} \to p}  \; \frac {1} { (m^{(i)}_p-j) !} \;
 \frac { d^{m^{(i)}_p-j} } { d\, (z^{(i)})^{m^{(i)}_p-j} }
 \left( \frac { (z^{(i)}-z^{(i)}(p))^{m^{(i)}_p} } { z^{(i)}_1\cdots z^{(i)}_{m_i} } \right) ,
\]
we can see that $a^{(i)}_{p,j}$ is independent of
the choice of the lift $z^{(i)}$ of $\bar{z}^{(i)}$.
Then we define
\[
 \res_p \left( \frac{d\bar{z}^{(i)}}{(\bar{z}^{(i)}_1)^{l_1}\cdots(\bar{z}^{(i)}_{m_i})^{l_{m_i}}} \right)
 := a^{(i)}_{p,1}.
\]

\begin{lemma} \label {lemma:residue-identity}
If $l_1,\ldots,l_{m_i}$ are integers satisfying $0\leq l_1,\ldots,l_{m_i}\leq 1$ and
$l_1+\cdots+l_{m_i}\geq 2$,
the equality
\[
 \sum_{p\in D^{(i)}}\res_{p}\left(
 \frac{d\bar{z}^{(i)}}{(\bar{z}^{(i)}_1)^{l_1}\cdots(\bar{z}^{(i)}_{m_i})^{l_{m_i}}}
 \right)=0
\]
holds.
\end{lemma}

\begin{proof}
It is sufficient to prove the equality for the case
$l_1=l_2=\cdots=l_{m_i}=1$.
Since the equality which we want is a formal equality determined by
(\ref  {equation:factor decomposition}),
it is sufficient to prove the equality
\begin{equation}\label{equation:residue-identity-complex-plane}
 \sum_{p\in\{p_1,\ldots,p_m\}}\res_{z=p}\left( 
 \frac{dz}{(z-p_1)(z-p_2)\cdots(z-p_m)} \right)=0
\end{equation}
when $z$ is a coordinate of the complex plane $\mathbb{C}$,
$m\geq 2$ and $p_1\ldots,p_m\in\mathbb{C}$ may not be distinct.
If we take a circle $\gamma$ in $\mathbb{C}$
which is a boundary of a large disk containing all the points
$p_1,\ldots,p_m$ within,
then we have
\begin{align*}
 \sum_{p\in\{p_1,\ldots,p_m\}}\res_{z=p}\left( 
 \frac{dz}{(z-p_1)(z-p_2)\cdots(z-p_m)} \right)
 &=\frac{1}{2\pi\sqrt{-1}}\int_{\gamma}\frac{dz}{(z-p_1)(z-p_2)\cdots(z-p_m)} \\
 &=-\res_{z=\infty}\left( 
 \frac{dz}{(z-p_1)(z-p_2)\cdots(z-p_m)} \right)=0
\end{align*}
because $m\geq 2$.
Thus the equality (\ref{equation:residue-identity-complex-plane}) holds.
\end{proof}

We take 
$\bmu=(\mu^{(i)}_j)^{1\leq i\leq n}_{1\leq j\leq r}
\in H^0(D^{(i)}, {\mathcal O}_{D^{(i)}}^{nr})$
such that
$\mu^{(i)}_1|_p,\ldots,\mu^{(i)}_r|_p$
are mutually distinct at any point $p\in D^{(i)}$.
Then we define a polynomial
$\varphi_{\bmu}^{(i)}(T) \in
H^0 ( D^{(i)}, {\mathcal O}_{D^{(i)}} )[T]$
by setting
\[
 \varphi_{\bmu}^{(i)}(T):=\prod_{k=1}^r(T-\mu^{(i)}_k).
\]
We fix a tuple of complex numbers
$\blambda=(\lambda^{(i)}_k)^{1\leq i\leq n}_{1\leq k\leq r}\in\mathbb{C}^{nr}$
satisfying
$\sum_{i=1}^n\sum_{k=1}^r \lambda^{(i)}_k \in\mathbb{Z}$ and
put
\[
 a:=-\sum_{i=1}^n\sum_{k=1}^r \lambda^{(i)}_k .
\]
For each $i$, we take a polynomial 
$\nu^{(i)}(T)=c^{(i)}_0+c^{(i)}_1T+\cdots+c^{(i)}_{r-1}T^{r-1}
\in H^0( D^{(i)}, {\mathcal O}_{D^{(i)}} )[T]$
such that the expression
\[
 \nu^{(i)}(\mu^{(i)}_k)
 =
 \sum_{\genfrac{}{}{0pt}{}{ 0\leq l_1,\ldots,l_{m_i} \leq 1,} {0\leq l_1+\cdots+l_{m_i}<m_i} }
 a^{(i)}_{k,l_1,\ldots,l_{m_i}}
 (\bar{z}^{(i)}_1)^{l_1}(\bar{z}^{(i)}_2)^{l_2}\cdots(\bar{z}^{(i)}_{m_i})^{l_{m_i}}
\]
with $a^{(i)}_{k,l_1,\ldots,l_{m_i}}\in\mathbb{C}$ satisfies the equality
\begin{equation}\label{equation:residue-condition}
 \lambda^{(i)}_k=
 ( a^{(i)}_{k,0,1,\ldots,1}+a^{(i)}_{k,1,0,1,\ldots,1}+\cdots+a^{(i)}_{k,1,\ldots,1,0} )
\end{equation}
for any $i,k$.
We can see by Lemma \ref{lemma:residue-identity} that
\begin{align*}
 \sum_{p\in D^{(i)}}\res_{p}\left( \nu^{(i)}(\mu^{(i)}_k)
 \frac{d\bar{z}^{(i)}}{\bar{z}^{(i)}_1\cdots\bar{z}^{(i)}_{m_i}}
 \right)
 &=
 \sum_{ \genfrac{}{}{0pt}{}{ 0\leq  l_1,\ldots,l_{m_i}\leq 1}{0\leq l_1+\cdots+l_{m_i}<m_i}}
 a^{(i)}_{k,l_1,\ldots,l_{m_i}} 
 \sum_{p\in D^{(i)}} 
 \res_p \left( \frac{d\bar{z}^{(i)}}{(\bar{z}^{(i)}_1)^{1-l_1}\cdots(\bar{z}^{(i)}_{m_i})^{1-l_{m_i}}} \right) \\
 &=
 \sum_{s=1}^{m_i} a^{(i)}_{k,1,\ldots,l_s=0,\ldots,1}   
 \sum_{p\in D^{(i)}} \res_p\left( \frac{d\bar{z}^{(i)}}{\bar{z}^{(i)}_s} \right) \\
 &=
 a^{(i)}_{k,0,1,\ldots,1}+a^{(i)}_{k,1,0,1,\ldots,1}+\cdots+a^{(i)}_{k,1,\ldots,1,0} .
\end{align*}
So (\ref{equation:residue-condition}) means the equality
\begin{equation} \label {equation:residue-equality}
 \lambda^{(i)}_k=\sum_{p\in D^{(i)}}
 \res_p\left(\nu^{(i)}(\mu^{(i)}_k)
 \frac{d\bar{z}^{(i)}}{\bar{z}^{(i)}_1\bar{z}^{(i)}_2\cdots \bar{z}^{(i)}_{m_i}}\right)
\end{equation}
where $\sum_{p\in D^{(i)}}$ runs over the set theoretical points $p$ of $D^{(i)}$.

We assume the following assumption
on $\bnu=(\nu^{(i)}(T))^{1\leq i\leq n}$:
\begin{assumption} \label {assumption-generic} \rm
 For each $i$,
 $\nu^{(i)}(\mu^{(i)}_1)\big|_p,\ldots,\nu^{(i)}(\mu^{(i)}_r)\big|_p$
 are mutually distinct at any point $p\in D^{(i)}$.
\end{assumption}

\begin{definition} \label {def-connection} \rm
We say that a tuple
$(E,\nabla,\{N^{(i)}\}_{1\leq i\leq n})$ is a $(\bnu,\bmu)$-connection on $(C,D)$ if
\begin{itemize}
\item[(i)] $E$ is an algebraic vector bundle on $C$ of rank $r$ and degree $a$,
\item[(ii)] $\nabla\colon E\longrightarrow E\otimes\Omega^1_C(D)$
is an algebraic connection on $E$ admitting poles along $D$,
\item[(iii)] 
$N^{(i)}\colon E|_{D^{(i)}}\longrightarrow
E|_{D^{(i)}}$
is an ${\mathcal O}_{D^{(i)}}$-homomorphism satisfying
$\varphi^{(i)}_{\bmu}(N^{(i)})=0$,
the homomorphism
\begin{equation} \label {equation:injectivity in definition}
 {\mathcal O}_{D^{(i)}}[T]/(\varphi^{(i)}_{\bmu}(T))
 \ni \overline{P(T)}\mapsto P(N^{(i)}) \in \End(E|_{D^{(i)}})
\end{equation}
is injective  and
$\displaystyle\nu^{(i)}(N^{(i)})
\frac{d\bar{z}^{(i)}}{\bar{z}^{(i)}_1\bar{z}^{(i)}_2\cdots \bar{z}^{(i)}_{m_i}}
=\nabla|_{D^{(i)}}$
for $1\leq i\leq n$.
\end{itemize}
\end{definition}

\begin{remark} \label {remark:injectivity implies subbundle}
\rm
The injectivity of the homomorphism (\ref {equation:injectivity in definition})
in Definition \ref {def-connection} implies that
${\mathcal O}_{D^{(i)}}[T]/(\varphi^{(i)}_{\bmu}(T))$
becomes an ${\mathcal O}_{D^{(i)}}$-subbundle of
$\End(E|_{D^{(i)}})$.
\end{remark}

\begin{proposition}
Assume that $D$ is a reduced divisor on $C$.
In other words, we assume that
$p^{(i)}_j\neq p^{(i)}_{j'}$ for $j\neq j'$.
Then giving a $(\bnu,\bmu)$-connection on $(C,D)$
is equivalent to giving a regular singular connection
$(E,\nabla)$ on $C$ admitting poles along $D$
whose residue $\res_{p^{(i)}_j}(\nabla)$
at $p^{(i)}_j$ has the distinct eigenvalues
\[
 \left.\left\{
 \nu^{(i)}(\mu^{(i)}_k) \big| _{p^{(i)}_j}\res_{p^{(i)}_j}
 \left(\dfrac{d\bar{z}^{(i)}}{\bar{z}^{(i)}_1\bar{z}^{(i)}_2\cdots \bar{z}^{(i)}_{m_i}}\right)
 \right| 1\leq k \leq r \right\}.
\]
\end{proposition}

\begin{proof}
Let $(E,\nabla,\{N^{(i)}\})$ be a $(\bnu,\bmu)$-connection on $(C,D)$.
The restriction $N^{(i)}|_{p^{(i)}_j}\colon E|_{p^{(i)}_j}\longrightarrow E|_{p^{(i)}_j}$
of $N^{(i)}$ to the fiber $E|_{p^{(i)}_j}$ of $E$ at $p^{(i)}_j$
satisfies $\prod_{k=1}^r(N^{(i)}|_{p^{(i)}_j}-\mu^{(i)}_k\mathrm{id}_{E|_{p^{(i)}_j}})=0$,
because $\varphi^{(i)}_{\bmu}(N^{(i)})=0$.
From the  injectivity of the homomorphism (\ref {equation:injectivity in definition})
in Definition \ref {def-connection}, the induced homomorphism
\[
 \mathbb{C}[T]/(\varphi^{(i)}_{\bmu}(T))
 \ni \overline{P(T)} \mapsto P(N^{(i)}|_{p^{(i)}_j})
 \in \End(E|_{p^{(i)}_j})
\]
is injective.
So $N^{(i)}|_{p^{(i)}_j}$ has the distinct eigenvalues
$\mu^{(i)}_1|_{p^{(i)}_j},\ldots,\mu^{(i)}_r|_{p^{(i)}_j}$.
By Assumption \ref{assumption-generic}, the linear endomorphism on $E|_{p^{(i)}_j}$
\[
 \nu^{(i)}(N^{(i)})|_{p^{(i)}_j}=
 c^{(i)}_0|_{p^{(i)}_j}\mathrm{id}_{E|_{p^{(i)}_j}}+c^{(i)}_1|_{p^{(i)}_j} N^{(i)}|_{p^{(i)}_j}
 +\cdots+c^{(i)}_r |_{p^{(i)}_j}(N^{(i)}|_{p^{(i)}_j})^{m_ir-r}
 \colon E|_{p^{(i)}_j} \longrightarrow E|_{p^{(i)}_j}
\]
has the distinct eigenvalues
$\nu^{(i)}(\mu^{(i)}_1)|_{p^{(i)}_j},\ldots,\nu^{(i)}(\mu^{(i)}_r)|_{p^{(i)}_j}$.
Since $\displaystyle\nu^{(i)}(N^{(i)})
\frac{d\bar{z}^{(i)}}{\bar{z}^{(i)}_1\bar{z}^{(i)}_2\cdots \bar{z}^{(i)}_{m_i}}
=\nabla|_{D^{(i)}}$,
the residue homomorphism
$\res_{p^{(i)}_j}(\nabla)\colon E|_{p^{(i)}_j} \longrightarrow E|_{p^{(i)}_j}$
has the eigenvalues
\[
 \left.\left\{
 \nu^{(i)}(\mu^{(i)}_k) \big| _{p^{(i)}_j}\res_{p^{(i)}_j}
 \left(\dfrac{d\bar{z}^{(i)}}{\bar{z}^{(i)}_1\bar{z}^{(i)}_2\cdots \bar{z}^{(i)}_{m_i}}\right)
 \right| 1\leq k \leq r \right\}.
\]

Conversely let $E$ be a vector bundle on $C$ of rank $r$ and
$\nabla\colon E\longrightarrow E\otimes\Omega^1_C(D)$
be a connection whose residue $\res_{p^{(i)}_j}(\nabla)$ at $p^{(i)}_j$
has the distinct eigenvalues
$\left.\left\{
\nu^{(i)}(\mu^{(i)}_k) \big| _{p^{(i)}_j}\res_{p^{(i)}_j}
\left(\dfrac{d\bar{z}^{(i)}}{\bar{z}^{(i)}_1\bar{z}^{(i)}_2\cdots \bar{z}^{(i)}_{m_i}}\right)
\right| 1\leq k \leq r \right\}$.
Since the diagonal matrix
\[
 R=
 \begin{pmatrix}
  \nu^{(i)}(\mu^{(i)}_1) \big| _{p^{(i)}_j}\res_{p^{(i)}_j}
  \left(\dfrac{d\bar{z}^{(i)}}{\bar{z}^{(i)}_1\bar{z}^{(i)}_2\cdots \bar{z}^{(i)}_{m_i}}\right)
  & \cdots & 0 \\
  \vdots & \ddots & \vdots \\
  0 & \cdots &
 \nu^{(i)}(\mu^{(i)}_r) \big| _{p^{(i)}_j}\res_{p^{(i)}_j}
 \left(\dfrac{d\bar{z}^{(i)}}{\bar{z}^{(i)}_1\bar{z}^{(i)}_2\cdots \bar{z}^{(i)}_{m_i}}\right)
 \end{pmatrix}
\]
has the distinct eigenvalues and commutes with the diagonal matrix
$N=
\begin{pmatrix}
  \mu^{(i)}_1|_{p^{(i)}_j} & \cdots & 0 \\
  \vdots & \ddots & \vdots \\
  0 & \cdots & \mu^{(i)}_r|_{p^{(i)}_j}
 \end{pmatrix}$,
the matrix
$N$
can be written as a polynomial $\psi^{(i)}_j(R)$ in $R$ with coefficients in $\mathbb{C}$,
that is, $N=\psi^{(i)}_j(R)$.
Consider the linear map
\[
 \psi^{(i)}_j(\res_{p^{(i)}_j}(\nabla))\colon E|_{p^{(i)}_j}\longrightarrow E|_{p^{(i)}_j}.
\]
By the Chinese remainder theorem
${\mathcal O}_{D^{(i)}}\xrightarrow{\sim}
\bigoplus_{j=1}^{m_i}{\mathcal O}_{p^{(i)}_j}$,
we have an isomorphism
\[
 \Hom_{{\mathcal O}_{D^{(i)}}}(E|_{D^{(i)}},E|_{D^{(i)}})
 \xrightarrow{\sim}
 \bigoplus_{j=1}^{m_i}\Hom_{{\mathcal O}_{p^{(i)}_j}}(E|_{p^{(i)}_j},E|_{p^{(i)}_j}).
\]
So there is an endomorphism
$N^{(i)}\colon E|_{D^{(i)}}\longrightarrow E|_{D^{(i)}}$
satisfying
$N^{(i)}|_{p^{(i)}_j}=\psi^{(i)}_j(\res_{p^{(i)}_j}(\nabla))$ for $1\leq j\leq m_i$.
Since
\[
 R=\nu^{(i)}(N)|_{p^{(i)}_j}\res_{p^{(i)}_j}
 \left(\dfrac{d\bar{z}^{(i)}}{\bar{z}^{(i)}_1\bar{z}^{(i)}_2\cdots \bar{z}^{(i)}_{m_i}}\right)
 =\nu^{(i)}(\psi^{(i)}_j(R))\res_{p^{(i)}_j}
 \left(\dfrac{d\bar{z}^{(i)}}{\bar{z}^{(i)}_1\bar{z}^{(i)}_2\cdots \bar{z}^{(i)}_{m_i}}\right),
\]
we can see
\[
 \res_{p^{(i)}_j}(\nabla)
 =\nu^{(i)}\left(\psi^{(i)}_j(\res_{p^{(i)}_j}(\nabla)\right)
 \res_{p^{(i)}_j}\left(\dfrac{d\bar{z}^{(i)}}{\bar{z}^{(i)}_1\bar{z}^{(i)}_2\cdots  \bar{z}^{(i)}_{m_i}}\right)
 =\displaystyle\nu^{(i)}(N^{(i)}) \big| _{p^{(i)}_j}
 \res_{p^{(i)}_j}\left(\frac{d\bar{z}^{(i)}}{\bar{z}^{(i)}_1\bar{z}^{(i)}_2\cdots \bar{z}^{(i)}_{m_i}}\right)
\]
for $1\leq j\leq m_i$,
which is equivalent to
$\displaystyle\nu^{(i)}(N^{(i)})
\frac{d\bar{z}^{(i)}}{\bar{z}^{(i)}_1\bar{z}^{(i)}_2\cdots \bar{z}^{(i)}_{m_i}}
=\nabla|_{D^{(i)}}$.
From the definition, each $N^{(i)}|_{p^{(i)}_j}$ has the distinct eigenvalues
$\mu^{(i)}_1|_{p^{(i)}_j}, \ldots, \mu^{(i)}_r|_{p^{(i)}_j}$
and so the identity $\varphi_{\bmu}^{(i)}(N^{(i)})=0$ follows.
Thus $(E,\nabla,\{N^{(i)}\})$ becomes a
$(\bnu,\bmu)$-connection.
\end{proof}

The following definition of unramified irregular singular parabolic connection is 
given in \cite{Inaba-Saito}.
Here we restrict to the case of generic exponents
and a notation of suffix is slightly changed.

\begin{definition}\label{def:unramified-connection}\rm
Let $t_1,\ldots,t_n\in C$ be distinct points
and $m_1,\ldots,m_n$ be integers satisfying $m_i>1$ for any $i$.
Take a generator $z_i\in\mathfrak{m}_{t_i}$ of the maximal ideal
$\mathfrak{m}_{t_i}$ of ${\mathcal O}_{C,t_i}$.
Assume that $\nu^{(i)}_1,\ldots,\nu^{(i)}_r\in{\mathcal O}_{m_it_i}$
satisfy $\nu^{(i)}_k|_{t_i}\neq \nu^{(i)}_{k'}|_{t_i}$
for $k\neq k'$.
Then $(E,\nabla,\{l^{(i)}_k\})$ is said to be an unramified irregular singular parabolic connection
with the exponents $\nu^{(i)}_1\dfrac{dz_i}{z_i^{m_i}},\ldots,\nu^{(i)}_r\dfrac{dz_i}{z_i^{m_i}}$
at $t_i$ if $E$ is an algebraic vector bundle on $C$,
$\nabla\colon E\longrightarrow E\otimes\Omega^1_C(\sum_{i=1}^nm_it_i)$
is an algebraic connection,
$E|_{m_it_i}=l^{(i)}_1\supset l^{(i)}_2\supset \cdots\supset l^{(i)}_r\supset l^{(i)}_{r+1}=0$
is a filtration satisfying $l^{(i)}_k/l^{(i)}_{k+1}\cong{\mathcal O}_{m_it_i}$ and
$\Big(\nabla|_{m_it_i}-\nu^{(i)}_k\dfrac{dz_i}{z_i^{m_i}}\mathrm{id}\Big)(l^{(i)}_k)
\subset l^{(i)}_{k+1}\dfrac{dz_i}{z_i^{m_i}}$
for any $k$.
\end{definition}

\begin{remark} \label {remark: unramified exponent} \rm
Assume that $(E,\nabla,\{l^{(i)}_k\})$ is an unramified irregular singular parabolic connection
with the exponents $\nu^{(i)}_1\dfrac{dz_i}{z_i^{m_i}},\ldots,\nu^{(i)}_r\dfrac{dz_i}{z_i^{m_i}}$
in Definition \ref{def:unramified-connection}
satisfying $\nu^{(i)}_k|_{t_i}\neq\nu^{(i)}_{k'}|_{t_i}$ for $k\neq k'$.
Then we can see as in the proof of \cite[Proposition 2.3]{Inaba-Saito}
that there is a decomposition
\begin{equation} \label {equation:decomposition for unramified case}
 E|_{m_it_i}=\bigoplus_{k=1}^r\ker \left(\nabla|_{m_it_i}-\nu^{(i)}_k\dfrac{dz_i}{z_i^{m_i}}\right)
\end{equation}
which induces the filtration $l^{(i)}_*$
and the diagonal representation matrix
of $\nabla|_{m_it_i}$
\[
 \begin{pmatrix}
  \nu^{(i)}_1\dfrac{dz_i}{z_i^{m_i}} & \cdots & 0 \\
  \vdots & \ddots & \ddots \\
  0 & \cdots & \nu^{(i)}_r\dfrac{dz_i}{z_i^{m_i}}
 \end{pmatrix}
\]
with respect to a  basis of $E|_{m_it_i}$ obtained from the decomposition
(\ref {equation:decomposition for unramified case}).
\end{remark}

\begin{proposition}
Under Assumption \ref{assumption-generic},
suppose that each $D^{(i)}$ is a multiple divisor of degree $m_i$
for $1\leq i\leq n$.
In other words, we assume that $p^{(i)}_j=p^{(i)}_{j'}$ for any $j,j'$ and
$D^{(i)}=m_ip^{(i)}_1$.
Then giving a $(\bnu,\bmu)$-connection on $(C,D)$ is
equivalent to giving an unramified irregular singular parabolic connection
$(E,\nabla,\{l^{(i)}_k\})$ on $(C,D)$ with the exponents
$\left.\left\{ \nu^{(i)}(\mu^{(i)}_k)\dfrac{d\bar{z}^{(i)}_1}{(\bar{z}^{(i)}_1)^{m_i}} \right| 1\leq k\leq r \right\}$
at $p^{(i)}_1$.
\end{proposition}

\begin{proof}
Assume that a $(\bnu,\bmu)$ connection
$(E,\nabla,\{N^{(i)}\})$ on $(C,D)$ is given.
First note that there is a complex
\[
  E|_{D^{(i)}} \xrightarrow{N^{(i)}-\mu^{(i)}_k} E|_{D^{(i)}}
  \xrightarrow{ \prod_{k'\neq k}(N^{(i)}-\mu^{(i)}_{k'})} E|_{D^{(i)}} 
\]
which induces the homomorphism
\[
 \overline{\prod_{k'\neq k}(N^{(i)}-\mu^{(i)}_{k'})} \colon
 \coker(N^{(i)}-\mu^{(i)}_k) \longrightarrow E|_{D^{(i)}}.
\]
By Remark \ref  {remark:injectivity implies subbundle},
the restriction
$\mathbb{C}[T]/(\varphi_{\bmu}(T)|_{p^{(i)}_1})
\ni \overline{P(T)} \mapsto P(N^{(i)}|_{p^{(i)}_1})
\in\End(E|_{p^{(i)}_1})$
 of the homomorphism (\ref  {equation:injectivity in definition})
in Definition \ref {def-connection}
to the reduced point $p^{(i)}_1$ of $D^{(i)}=m_ip^{(i)}_1$
is also injective.
So
$N^{(i)}|_{p^{(i)}_1}\colon E|_{p^{(i)}_1}\longrightarrow E|_{p^{(i)}_1}$ has the distinct eigenvalues
$\mu^{(i)}_1|_{p^{(i)}_1},\ldots,\mu^{(i)}_r|_{p^{(i)}_1}$
and
\[
 \overline{ \prod_{k'\neq k}(N^{(i)}-\mu^{(i)}_{k'})|_{p^{(i)}_1} } \colon
 \coker((N^{(i)}-\mu^{(i)}_k)|_{p^{(i)}_1})
 \longrightarrow E|_{p^{(i)}_1}
\]
is an injection to the eigen subspace of $E|_{p^{(i)}_1}$
with respect to the eigenvalue $\mu^{(i)}_k|_{p^{(i)}_1}$ of $N^{(i)}|_{p^{(i)}_1}$.
Therefore we can see that
\[
 \overline{\prod_{k'\neq k}(N^{(i)}-\mu^{(i)}_{k'})}\colon
 \coker(N^{(i)}-\mu^{(i)}_k) \longrightarrow E|_{D^{(i)}}
\]
is also injective and its cokernel is a free ${\mathcal O}_{D^{(i)}}$-module of
rank $r-1$.
So
\[
 \coker(N^{(i)}-\mu^{(i)}_k)\xrightarrow{\sim}\ker (N^{(i)}-\mu^{(i)}_k)
 \subset E|_{D^{(i)}}
\]
is a rank one subbundle of $E|_{D^{(i)}}$
and we have a decomposition
\begin{equation}\label {equation:decomposition}
 E|_{D^{(i)}}=\bigoplus_{k=1}^r \ker (N^{(i)}-\mu^{(i)}_k). 
\end{equation}
By the equality
$\nu^{(i)}(N^{(i)})\dfrac{d\bar{z}^{(i)}}{\bar{z}^{(i)}_1\bar{z}^{(i)}_2\cdots \bar{z}^{(i)}_{m_i}}
=\nabla|_{D^{(i)}}$,
we can see that
the representation matrix of $\nabla|_{D^{(i)}}$
with respect to a basis giving the direct sum decomposition (\ref {equation:decomposition})
of $E|_{D^{(i)}}$ is
\[
 \begin{pmatrix}
 \nu^{(i)}(\mu^{(i)}_1)\dfrac{d\bar{z}^{(i)}}{(\bar{z}^{(i)}_1)^{m_i}} & \cdots & 0 \\
 \vdots & \ddots & \vdots \\
 0 & \cdots & \nu^{(i)}(\mu^{(i)}_r)\dfrac{d\bar{z}^{(i)}}{(\bar{z}^{(i)}_1)^{m_i}}
 \end{pmatrix}.
\]
If we choose the parabolic structure $\{l^{(i)}_k\}$
compatible with the decomposition (\ref{equation:decomposition}),
then $(E,\nabla,\{l^{(i)}_k\})$ becomes an unramified irregular singular parabolic connection
with the exponents 
$\bigg\{ \nu^{(i)}(\mu^{(i)}_k)\dfrac{d\bar{z}^{(i)}_1}{(\bar{z}^{(i)}_1)^{m_i}} \bigg\}_{1\leq k\leq r}$
at $p^{(i)}_1$ for $1\leq i\leq n$.

Conversely, let $(E,\nabla,\{l^{(i)}_k\})$ be an unramified irregular singular parabolic connection
with the exponents
$\bigg\{ \nu^{(i)}(\mu^{(i)}_k)\dfrac{d\bar{z}^{(i)}_1}{(\bar{z}^{(i)}_1)^{m_i}} \bigg\}_{1\leq k\leq r}$
at $p^{(i)}_1$.
Since $\nu^{(i)}(\mu^{(i)}_1)|_{p^{(i)}_1},\ldots,\nu^{(i)}(\mu^{(i)}_r)|_{p^{(i)}_1}$ are mutually distinct,
we have a decomposition
\[
 E|_{D^{(i)}}=\bigoplus_{k=1}^r
 \ker \left(\nabla|_{D^{(i)}}-\nu^{(i)}(\mu^{(i)}_k)
 \dfrac{d\bar{z}^{(i)}_1}{(\bar{z}^{(i)}_1)^{m_i}} \right)
\]
as in Remark \ref {remark: unramified exponent}
which is compatible with $\{ l^{(i)}_k \}$.
If we define a homomorphism $N^{(i)}\colon E|_{D^{(i)}}\longrightarrow E|_{D^{(i)}}$
by setting
\[
 N^{(i)}\big|_{\ker \big(\nabla|_{D^{(i)}}-\nu^{(i)}(\mu^{(i)}_k)
 \frac{d\bar{z}^{(i)}}{(z^{(i)}_1)^{m_i}} \big)}
 =\mu^{(i)}_k \cdot
 \mathrm{id}_{\ker \big(\nabla|_{D^{(i)}}-\nu^{(i)}(\mu^{(i)}_k)
 \frac{d\bar{z}^{(i)}}{(z^{(i)}_1)^{m_i}} \big)}
\]
for each $k$,
then $N^{(i)}$ satisfies $\varphi^{(i)}_{\bmu}(N^{(i)})=0$ and
$\nabla|_{D^{(i)}}=\nu^{(i)}(N^{(i)})\dfrac{d\bar{z}^{(i)}}{(\bar{z}^{(i)}_1)^{m_i}}$.
Since $N^{(i)}|_{p^{(i)}_1}$ has the distinct eigenvalues
$\mu^{(i)}_1|_{p^{(i)}_1},\ldots,\mu^{(i)}_r|_{p^{(i)}_1}$,
the homomorphism
\[
 {\mathcal O}_{D^{(i)}}[T]/(\varphi^{(i)}_{\bmu}(T))
 \ni \overline{P(T)} \mapsto P(N^{(i)})
 \in \End(E|_{D^{(i)}})
\]
is injective, because of the injectivity of its restriction 
to the reduced point $p^{(i)}_1$ of $D^{(i)}$.
So $(E,\nabla,\{N^{(i)}\})$ becomes a $(\bnu,\bmu)$-connection.
\end{proof}

Now we come back to the general setting in Definition \ref{def-connection}
and define a stability for a $(\bnu,\bmu)$-connection
$(E,\nabla,\{N^{(i)}\})$ which is necessary for the construction
of the moduli space.
By Assumption \ref{assumption-generic},
there is a unique filtration
\begin{equation} \label {equation:filtration}
 E|_{D^{(i)}}=l^{(i)}_1\supset l^{(i)}_2\supset\cdots\supset l^{(i)}_r\supset l^{(i)}_{r+1}=0
\end{equation}
such that $l^{(i)}_k/l^{(i)}_{k+1}\cong{\mathcal O}_{D^{(i)}}$,
$\Big( \nabla|_{D^{(i)}}-\nu^{(i)}(\mu^{(i)}_k)\dfrac{dz_i}{z_i^{m_i}}\mathrm{id} \Big) (l^{(i)}_k)
\subset l^{(i)}_{k+1}\dfrac{dz_i}{z_i^{m_i}}$
and $(N^{(i)}-\mu^{(i)}_k\mathrm{id})(l^{(i)}_k)\subset l^{(i)}_{k+1}$
for any $i,k$.

We take a tuple of positive rational numbers
$\balpha=(\alpha^{(i)}_k)^{1\leq i\leq n}_{1\leq k\leq r}$
satisfying
$0<\alpha^{(i)}_1<\alpha^{(i)}_2<\cdots<\alpha^{(i)}_r<1$
for any $i$
and $\alpha^{(i)}_k\neq\alpha^{(i')}_{k'}$
for $(i,k)\neq(i',k')$.
The following definition in fact depends on the ordering of
$\mu^{(i)}_1,\ldots,\mu^{(i)}_r$.

\begin{definition}\rm
A $(\bnu, \bmu)$-connection
$(E,\nabla,\{N^{(i)}\})$  on $(C,D)$ is
$\balpha$-stable (resp.\ $\balpha$-semistable) if the inequality
\[
 \frac{\displaystyle\deg F+\sum_{i=1}^n\sum_{k=1}^r
 \alpha^{(i)}_k\length ((F|_{D^{(i)}}\cap l^{(i)}_k)/(F|_{D^{(i)}}\cap l^{(i)}_{k+1}))}
 {\rank F}
 \genfrac{}{}{0pt}{}{<}{\text{(resp.\ $\leq$)}}
  \frac{\displaystyle\deg E+\sum_{i=1}^n\sum_{k=1}^r
 \alpha^{(i)}_k\length (l^{(i)}_k/l^{(i)}_{k+1})}{\rank E}
\]
holds for any subbundle $0\neq F\subsetneq E$
satisfying $\nabla(F)\subset F\otimes\Omega^1_C(D)$,
where $\{ l^{(i)}_k\}$ is the filtration (\ref {equation:filtration})
of $E|_{D^{(i)}}$ determined by $\nabla|_{D^{(i)}}$.
\end{definition}

\subsection{Relative moduli space of $(\tilde{\bnu},\tilde{\bmu})$-connections}
\label {subsection:former main theorem}

Let $S$ be an irreducible algebraic variety over $\Spec\mathbb{C}$ and
let ${\mathcal C}\longrightarrow S$ be a smooth projective morphism
whose geometric fibers are smooth projective irreducible curves of genus $g$.
Assume that ${\mathcal D}$ is an effective Cartier divisor on ${\mathcal C}$
flat over $S$, which has a decomposition
\[
 {\mathcal D}={\mathcal D}^{(1)}+\cdots+{\mathcal D}^{(n)}
 ={\mathcal D}^{(1)}\sqcup\cdots\sqcup{\mathcal D}^{(n)},
\]
where ${\mathcal D}^{(i)}$ is an effective Cartier divisor on ${\mathcal C}$
flat over $S$,
which also has a decomposition
\[
 {\mathcal D}^{(i)}={\mathcal D}^{(i)}_1+{\mathcal D}^{(i)}_2+\cdots+{\mathcal D}^{(i)}_{m_i}
\]
such that the composition ${\mathcal D}^{(i)}_j\hookrightarrow{\mathcal C}\longrightarrow S$
is isomorphic.
Here we assume that ${\mathcal D}^{(i)}\cap{\mathcal D}^{(i')}=\emptyset$
for $i\neq i'$  and $({\mathcal D}^{(i)}_j)_s\cap({\mathcal D}^{(i)}_{j'})_s=\emptyset$
for $j\neq j'$ if $({\mathcal D}^{(i)}_j)_s, ({\mathcal D}^{(i)}_{j'})_s$ are generic fibers but
${\mathcal D}^{(i)}_j$ and ${\mathcal D}^{(i)}_{j'}$ may intersect.

Assume that we can take a section
$\bar{z}^{(i)}\in{\mathcal O}_{2{\mathcal D}^{(i)}}$ such that
$\bar{z}^{(i)}-\bar{z}^{(i)}({\mathcal D}^{(i)}_j)=0$ is a defining equation
of ${\mathcal D}^{(i)}_j$ in $2{\mathcal D}^{(i)}$
and that
$d\bar{z}^{(i)}|_p$ gives a local basis of
$\Omega^1_{{\mathcal C}/S}\otimes{\mathcal O}_{{\mathcal D}^{(i)}}|_p$
for any point $p\in {\mathcal D}^{(i)}$,
where $\bar{z}^{(i)}({\mathcal D}^{(i)}_j)\in{\mathcal O}_S$
corresponds to $\bar{z}^{(i)}|_{{\mathcal D}^{(i)}_j}$
via the isomorphism ${\mathcal D}^{(i)}_j\xrightarrow{\sim}S$.
We denote
$\bar{z}^{(i)}-\bar{z}^{(i)}({\mathcal D}^{(i)}_j)\in
{\mathcal O}_{2{\mathcal D}^{(i)}}$ by $\bar{z}^{(i)}_j$.
Then we can define
\begin{equation} \label {equation: fundamental differential form}
 \frac{d\bar{z}^{(i)}}{\bar{z}^{(i)}_1\bar{z}^{(i)}_2\cdots \bar{z}^{(i)}_{m_i}}\in
 \Omega^1_{{\mathcal C}/S}({\mathcal D}^{(i)})|_{{\mathcal D}^{(i)}}
\end{equation}
similarly to (\ref {equation: unfolding equation on finite scheme})
which  is a local basis of
$\Omega^1_{{\mathcal C}/S}({\mathcal D}^{(i)})
|_{{\mathcal D}^{(i)}}$.

We fix
$\tilde{\bmu}=(\tilde{\mu}^{(i)}_j)^{1\leq i\leq n}_{1\leq j\leq r}
\in H^0( {\mathcal D}^{(i)}, {\mathcal O}_{{\mathcal D}^{(i)}}^{nr}) $
such that
$\tilde{\mu}^{(i)}_1|_p,\ldots,\tilde{\mu}^{(i)}_r|_p\in\mathbb{C}$
are mutually distinct at any point $p\in {\mathcal D}^{(i)}$.
Then we define a tuple
$\bvarphi_{\tilde{\bmu}}=(\varphi_{\tilde{\bmu}}^{(i)}(T))^{1\leq i\leq n}$
of polynomials by
\[
 \varphi_{\tilde{\bmu}}^{(i)}(T)=\prod_{k=1}^r(T-\tilde{\mu}^{(i)}_k)
 \in H^0( {\mathcal D}^{(i)}, {\mathcal O}_{{\mathcal D}^{(i)}} )[T].
\]
Assume that $a\in\mathbb{Z}$ 
and  $\tilde{\blambda}=(\tilde{\lambda}^{(i)}_k)\in H^0(S,{\mathcal O}_S)^{nr}$
satisfying
\[
 a+\sum_{i=1}^n\sum_{k=1}^r \tilde{\lambda}^{(i)}_k=0
\]
are given.
We also take a tuple $\tilde{\bnu}=(\tilde{\nu}^{(i)}(T))^{1\leq i\leq n}$
of polynomials
\[
 \tilde{\nu}^{(i)}(T)=c^{(i)}_0+c^{(i)}_1T+\cdots+c^{(i)}_{r-1} T^{r-1}\in
 H^0( {\mathcal D}^{(i)}, {\mathcal O}_{{\mathcal D}^{(i)}} )[T]
\]
such that the expression
\[
 \nu^{(i)}(\mu^{(i)}_k)
 =
 \sum_{\genfrac{}{}{0pt}{}{0\leq l_1,\ldots,l_{m_i}\leq 1}{0\leq l_1+\cdots+l_{m_i}<m_i}}
 a^{(i)}_{k,l_1,\ldots,l_{m_i}}
 (\bar{z}^{(i)}_1)^{l_1}\cdots(\bar{z}^{(i)}_{m_i})^{l_{m_i}}
\]
with $a^{(i)}_{k,l_1,\ldots,l_{m_i}}\in H^0(S,{\mathcal O}_S)$
satisfies the equality
\[
 \tilde{\lambda}^{(i)}_k
 =
 a^{(i)}_{k,0,1,\ldots,1}+a^{(i)}_{k,1,0,1,\ldots,1}+\cdots+a^{(i)}_{k,1,\ldots,1,0} 
\]
for any $i,k$.
Furthermore, we assume that
$\tilde{\nu}^{(i)}(\mu^{(i)}_1)|_p,\ldots,\tilde{\nu}^{(i)}(\mu^{(i)}_r)|_p$
are mutually distinct
for each $i$ and $p\in{\mathcal D}^{(i)}$.

Before the definition of a moduli functor,
we mention a convention of  notation used in this paper.
For a noetherian scheme $S'$ with a morphism $S'\longrightarrow S$,
we denote ${\mathcal C}\times_SS'$ by ${\mathcal C}_{S'}$
and denote ${\mathcal D}\times_SS'$ by ${\mathcal D}_{S'}$
and so on.
For a coherent sheaf $E$ on ${\mathcal C}$,
we denote the pull-back of $E$ under the morphism
${\mathcal C}\times_SS'\longrightarrow {\mathcal C}$
by $E_{S'}$ and so on.

\begin{definition}
\rm \label {def:moduli functor of (nu,mu)-connections}
We define a contravariant functor
${\mathcal M}^{\balpha}_{{\mathcal C},{\mathcal D}}(\tilde{\bnu},\tilde{\bmu})
\colon (\sch/S)^o \longrightarrow (\sets)$
from the category $(\sch/S)$ of noetherian schemes over $S$
to the category $(\sets)$ of sets by setting
\[
 {\mathcal M}^{\balpha}_{{\mathcal C},{\mathcal D}}(\tilde{\bnu},\tilde{\bmu})(S')
 =
 \left\{ (E,\nabla,\{N^{(i)}\}_{1\leq i\leq n}) \left|
 \text{ $(E,\nabla,\{N^{(i)}\})$ satisfies the following (a),(b),(c),(d) } \right\} \right/ \sim,
\]
for a noetherian scheme $S'$ over $S$, where
\begin{itemize}
\item[(a)] $E$ is a vector bundle on ${\mathcal C}_{S'}$ of rank $r$
and $\deg (E|_{{\mathcal C}_s})=a$ for any geometric point $s$ of $S$,
\item[(b)] $\nabla\colon E\longrightarrow E\otimes\Omega^1_{{\mathcal C}_{S'}/S'}({\mathcal D}_{S'})$
is an $S'$-relative connection, in other words, $\nabla(fa)=a\otimes df +f\nabla(a)$
for $f\in{\mathcal O}_{{\mathcal C}_{S'}}$ and $a\in E$,
\item[(c)] $N^{(i)} \colon E|_{{\mathcal D}^{(i)}_{S'}} \longrightarrow E|_{{\mathcal D}^{(i)}_{S'}}$
is an ${\mathcal O}_{{\mathcal D}^{(i)}_{S'}}$-homomorphim satisfying
$\varphi_{\bmu}^{(i)}(N^{(i)})=0$, the homomorphism
\[
 {\mathcal O}_{{\mathcal D}^{(i)}_{S'}}[T]/(\varphi_{\bmu}^{(i)}(T))
 \ni \overline{P(T)} \mapsto P(N^{(i)}) \in \End( E|_{{\mathcal D}^{(i)}_{S'}})
\]
is an injection whose cokernel is flat over $S'$,
$\nu^{(i)}(N^{(i)})\dfrac{d\bar{z}^{(i)}}{\bar{z}^{(i)}_1\bar{z}^{(i)}_2\cdots\bar{z}^{(i)}_{m_i}}
=\nabla|_{{\mathcal D}^{(i)}_{S'}}$
for $1\leq i\leq n$ and
\item[(d)] $(E|_{{\mathcal C}_s},\nabla|_{{\mathcal C}_s},\{N^{(i)}|_{{\mathcal D}^{(i)}_s}\})$
is $\balpha$-stable for any geometric point $s$ of $S'$.
\end{itemize}
Here $(E,\nabla,\{N^{(i)}\})\sim(E',\nabla',\{N'^{(i)}\})$
if there are a line bundle ${\mathcal L}$ on $S'$ and an isomorphism
$\sigma\colon E\xrightarrow{\sim}E'\otimes{\mathcal L}$
satisfying $(\mathrm{id}\otimes\sigma)\circ\nabla=\nabla'\circ\sigma$
and $\sigma|_{{\mathcal D}^{(i)}_{S'}}\circ N^{(i)}=(N'^{(i)}\otimes\mathrm{id})\circ\sigma|_{{\mathcal D}^{(i)}_{S'}}$
for any $i$.
\end{definition}

\begin{theorem} \label {theorem:algebraic-moduli-unfolding}
There exists a coarse moduli scheme
$M^{\balpha}_{{\mathcal C},{\mathcal D}}(\tilde{\bnu},\tilde{\bmu})$
of ${\mathcal M}^{\balpha}_{{\mathcal C},{\mathcal D}}(\tilde{\bnu},\tilde{\bmu})$.
The structure morphism
$M^{\balpha}_{{\mathcal C},{\mathcal D}}(\tilde{\bnu},\tilde{\bmu})
\longrightarrow S$
is a smooth and quasi-projective morphism
whose non-empty fiber is of dimension
$2r^2(g-1)+2+r(r-1)\sum_{i=1}^n m_i$.
Moreover, there is a relative symplectic form on
$M^{\balpha}_{{\mathcal C},{\mathcal D}}(\tilde{\bnu},\tilde{\bmu})$
over $S$.
\end{theorem}

We call $M^{\balpha}_{{\mathcal C},{\mathcal D}}(\tilde{\bnu},\tilde{\bmu})$
in Theorem \ref {theorem:algebraic-moduli-unfolding}
the relative moduli space of
$\balpha$-stable $(\tilde{\bnu},\tilde{\bmu})$ connections 
on $({\mathcal C},{\mathcal D})$ over $S$.
First we give a proof of the existence of the moduli space
$M^{\balpha}_{{\mathcal C},{\mathcal D}}(\tilde{\bnu},\tilde{\bmu})$.
We define a moduli functor
${\mathcal M}\colon (\sch/S)^o\longrightarrow (\sets)$
by
\[
 {\mathcal M}(S')=\left.\left\{ (E,\nabla,\{l^{(i)}_k\}) \right|
 \text{$(E,\nabla,\{l^{(i)}_k\})$ satisfies the following (i),(ii),(iii),(iv)} \right\}/\sim
\]
for a noetherian scheme $S'$ over $S$, where
\begin{itemize}
\item[(i)] $E$ is a vector bundle on ${\mathcal C}\times_SS'$ of rank $r$
and $\deg(E|_{{\mathcal C}_s})=a$ for any geometric point $s$ of $S'$,
\item[(ii)] $\nabla\colon E\longrightarrow 
E\otimes\Omega^1_{{\mathcal C}_{S'}/S'}({\mathcal D}_{S'})$
is a relative connection,
\item[(iii)] $E|_{{\mathcal D}^{(i)}_{S'}}=l^{(i)}_0\supset l^{(i)}_1\supset\cdots\supset l^{(i)}_{r-1}
\supset l^{(i)}_r=0$
is a filtration by coherent ${\mathcal O}_{{\mathcal D}^{(i)}_{S'}}$-submodules
such that each $l^{(i)}_k/l^{(i)}_{k+1}$ is flat over $S'$ and
$\length((l^{(i)}_k/l^{(i)}_{k+1})|_{{\mathcal D}^{(i)}_s})=m_i$ for any $s\in S'$,
\item[(iv)] for any geometric point $s$ of $S'$, the fiber
$(E,\nabla,\{l^{(i)}_k\})|_{{\mathcal C}_s}$ satisfies the stability condition
\begin{gather*}
 \frac{\deg F+\sum_{i=1}^n\sum_{k=1}^r
 \alpha^{(i)}_k\length ((F|_{{\mathcal D}^{(i)}_s}\cap l^{(i)}_k|_{{\mathcal D}^{(i)}_s})
 /(F|_{{\mathcal D}^{(i)}_s}\cap l^{(i)}_{k+1}|_{{\mathcal D}^{(i)}_s}))}
 {\rank F} \\
 <
 \frac{\deg E|_{{\mathcal D}^{(i)}_s}+\sum_{i=1}^n\sum_{k=1}^r
 \alpha^{(i)}_k\length (l^{(i)}_k|_{{\mathcal D}^{(i)}_s}/l^{(i)}_{k+1}|_{{\mathcal D}^{(i)}_s})}{\rank E}
\end{gather*}
for any subbundle $0\neq F\subsetneq E|_{{\mathcal C}_s}$
satisfying
$\nabla|_{{\mathcal C}_s}(F)\subset F\otimes\Omega^1_{{\mathcal C}_s}({\mathcal D}_s)$.
\end{itemize}
Here $(E,\nabla,\{l^{(i)}_k\})\sim(E',\nabla',\{l'^{(i)}_k\})$
if there are a line bundle ${\mathcal L}$ on $S'$ and an isomorphism
$(E,\nabla,\{l^{(i)}_k\})\xrightarrow{\sim}
(E',\nabla',\{l'^{(i)}_k\})\otimes_{{\mathcal O}_{S'}}{\mathcal L}$.
Note that  the parabolic structure $\{l^{(i)}_k\}$ in (iii) has no relationship
with the connection $\nabla$ in (ii).
The following lemma is already used in  \cite{IIS-1}, \cite{Inaba-1} and \cite{Inaba-Saito}.

\begin{lemma}\label{lemma:moduli-exists}
There exists a coarse moduli scheme $M$ of ${\mathcal M}$.
$M$ is quasi-projective over $S$ and represents the \'etale sheafification
of the moduli functor ${\mathcal M}$.
\end{lemma}

\begin{proof}
By \cite[Theorem 5.1]{IIS-1}, there exists a relative coarse
moduli scheme $\overline{M^{{\mathcal D},\balpha',\bbeta,\gamma}_{{\mathcal C}/S}(r,a,\{m_i\})}$
over $S$
of parabolic $\Lambda_{{\mathcal D}}^1$-triples
$(E_1,E_2,\phi,\nabla,\{l^{(i)}_k\})$,
where $E_1$ and $E_2$ are algebraic vector bundles of rank $r$
on a fiber of ${\mathcal C}$ over $S$,
$\phi\colon E_1\longrightarrow E_2$ is an ${\mathcal O}_{\mathcal C}$-homomorphism,
$\nabla\colon E_1\longrightarrow E_2\otimes\Omega^1_{{\mathcal C}/S}({\mathcal D})$
satisfies $\nabla(fa)=\phi(a)\otimes df +f\nabla(a)$
for $f\in{\mathcal O}_{\mathcal C}$, $a\in E_1$,
$E_1|_{{\mathcal D}^{(i)}_s}=l^{(i)}_0\supset l^{(i)}_1\supset\cdots\supset l^{(i)}_r=0$
is a filtration satisfying $\length(l^{(i)}_k/l^{(i)}_{k+1})=m_i$
and $(E_1,E_2,\phi,\nabla,\{l^{(i)}_k\})$
satisfies a stability condition with respect to $(\balpha',\bbeta,\gamma)$.
Furthermore, $\overline{M^{{\mathcal D},\balpha',\bbeta,\gamma}_{{\mathcal C}/S}(r,a,\{m_i\})}$
is quasi-projective over $S$.
The detail is written in \cite[section 5]{IIS-1}.
If we denote the moduli functor corresponding to
$\overline{M^{{\mathcal D},\balpha',\bbeta,\gamma}_{{\mathcal C}/S}(r,a,\{m_i\})}$
by
$\overline{{\mathcal M}^{{\mathcal D},\balpha',\bbeta,\gamma}_{{\mathcal C}/S}(r,a,\{m_i\})}$
and choose an appropriate stability parameter $(\balpha',\bbeta,\gamma)$
by a similar argument to that in \cite[section 5]{IIS-1},
then we can define a morphism of functors
\[
 {\mathcal M}\longrightarrow
 \overline{{\mathcal M}^{{\mathcal D},\balpha',\bbeta,\gamma}_{{\mathcal C}/S}(r,a,\{m_i\})}
\]
given by $(E,\nabla,\{l^{(i)}_k\})\mapsto (E,E,\mathrm{id}_E,\nabla,\{l^{(i)}_k\})$
which is represented by an open immersion.
So there is a Zariski open subset
$M\subset \overline{M^{{\mathcal D},\balpha',\bbeta,\gamma}_{{\mathcal C}/S}(r,a,\{m_i\})}$
satisfying
\[
 {\mathcal M} \cong
 M\times_{ \overline{M^{{\mathcal D},\balpha',\bbeta,\gamma}_{{\mathcal C}/S}(r,a,\{m_i\})} }
 \overline{{\mathcal M}^{{\mathcal D},\balpha',\bbeta,\gamma}_{{\mathcal C}/S}(r,a,\{m_i\})}.
\]
Then $M$ represents the \'etale sheafification of ${\mathcal M}$
and becomes a coarse moduli scheme of ${\mathcal M}$.
\end{proof}

\noindent
{\it Proof of the existence of
$M^{\balpha}_{{\mathcal C},{\mathcal D}}(\tilde{\bnu},\tilde{\bmu})$.}

For some quasi-finite \'etale covering
$\tilde{M}\longrightarrow M$, there is a universal family
$(\tilde{E},\tilde{\nabla},\{\tilde{l}^{(i)}_k\})$ on ${\mathcal C}\times_S\tilde{M}$.
Let $Y$ be the maximal locally closed subscheme of $\tilde{M}$
such that
$(l^{(i)}_k)_Y/(l^{(i)}_{k+1})_Y$ is a locally free ${\mathcal O}_{{\mathcal D}^{(i)}_Y}$-module
of rank one for $i=1,\ldots,n$
and
$\bigg(\tilde{\nabla}|_{{\mathcal D}^{(i)}_Y}
-\nu^{(i)}(\mu^{(i)}_k)\mathrm{id}
\dfrac{d\bar{z}^{(i)}}{\bar{z}^{(i)}_1\bar{z}^{(i)}_2\cdots \bar{z}^{(i)}_{m_i}}\bigg)
\big( (\tilde{l}^{(i)}_k)_Y \big)
\subset (l^{(i)}_{k+1})_Y\otimes\Omega^1_{{\mathcal C}_Y/Y}(D_Y)$
for $1\leq k\leq r$.
We set
\[
 P:=\prod_{i=1}^n
 \Spec S^*_Y\left({\mathcal Hom}(\tilde{E}|_{{\mathcal D}^{(i)}_Y},
 \tilde{E}|_{{\mathcal D}^{(i)}_Y})^{\vee}\right)
\]
and take universal families
$\tilde{N}^{(i)}\colon \tilde{E}|_{{\mathcal D}^{(i)}_P}\longrightarrow \tilde{E}|_{{\mathcal D}^{(i)}_P}$
for $i=1,\ldots,n$,
where $S^*_Y\left( {\mathcal Hom}(\tilde{E}|_{{\mathcal D}^{(i)}_Y},
 \tilde{E}|_{{\mathcal D}^{(i)}_Y})^{\vee}\right)$
denotes the symmetric algebra of
${\mathcal Hom}(\tilde{E}|_{{\mathcal D}^{(i)}_Y},
 \tilde{E}|_{{\mathcal D}^{(i)}_Y})^{\vee}$
over $Y$.
Let $Z$ be the maximal locally closed subscheme of $P$
satisfying
$\varphi_{\tilde{\bmu}}(\tilde{N}^{(i)})_Z=0\in\End(\tilde{E}|_{{\mathcal D}^{(i)}_Z})$,
$\tilde{\nu}^{(i)}(\tilde{N}^{(i)})\left.\dfrac{d\bar{z}^{(i)}}{\bar{z}^{(i)}_1\bar{z}^{(i)}_2\cdots \bar{z}^{(i)}_{m_i}}
\right|_{{\mathcal D}^{(i)}_Z}
=\tilde{\nabla}|_{{\mathcal D}^{(i)}_Z}$
and
\[
 {\mathcal O}_{{\mathcal D}^{(i)}_p}[T]/(\varphi^{(i)}_{\tilde{\bmu}}(T))
 \ni \overline { P(T) } \mapsto P((\tilde{N}^{(i)})_p)
 \in \End(\tilde{E}|_{{\mathcal D}^{(i)}_p}) 
\]
is injective for any $\mathbb{C}$-valued point $p$ of $Z$.
By construction, we can easily see that $Z$ descends to a quasi-projective scheme
$M^{\balpha}_{C,D}(\tilde{\bnu},\tilde{\bmu})$ over $M$,
which is the desired moduli space.
\hfill $\boxed{ }$

The proof of Theorem \ref {theorem:algebraic-moduli-unfolding}
will be completed at the end of subsection \ref {subsection:symplectic form}.

\subsection{Factorized $(\bnu,\bmu)$-connection}

For the rest of the proof of Theorem \ref{theorem:algebraic-moduli-unfolding},
we need to describe the tangent space of the moduli space.
We will describe the tangent space and give a symplectic structure
via the idea in section \ref{section:linear-algebra}.
So we introduce the notion of factorized $(\bnu,\bmu)$-connection
which comes from the idea of factorization of a linear map in 
subsection \ref {subsection:factorization of linear map}.

Let $C,D,D^{(i)},D^{(i)}_j,\bmu,\varphi^{(i)}_{\bmu},\bnu,\bar{z}^{(i)}$ and $\bar{z}^{(i)}_j$ be
as in Definition \ref{def-connection}.
The following notion of factorized connection is useful for
describing the deformation theory of $(\bnu,\bmu)$-connections
and the relative symplectic form on the moduli space.

\begin{definition}\rm
We say that a tuple $(E,\nabla,\{\theta^{(i)},\kappa^{(i)}\})$ 
is a factorized $(\bnu,\bmu)$-connection if
\begin{itemize}
\item[(1)] $E$ is an algebraic vector bundle on $C$
of rank $r$ and degree $a$,
\item[(2)] $\nabla\colon E\longrightarrow E\otimes\Omega^1_C(D)$
is an algebraic connection admitting poles along $D$,
\item[(3)] $\theta^{(i)}\colon E^{\vee}|_{D^{(i)}}
\stackrel{\sim}\longrightarrow E|_{D^{(i)}}$
is an ${\mathcal O}_{D^{(i)}}$-isomorphism satisfying $^t\theta^{(i)}=\theta^{(i)}$,
\item[(4)] $\kappa^{(i)}\colon E|_{D^{(i)}} \longrightarrow E^{\vee}|_{D^{(i)}}$
is an ${\mathcal O}_{D^{(i)}}$-homomorphism satisfying
$^t\kappa^{(i)}=\kappa^{(i)}$,
\item[(5)] the composition $N^{(i)}:=\theta^{(i)}\circ\kappa^{(i)}\colon
E|_{D^{(i)}} \longrightarrow E|_{D^{(i)}}$
satisfies
$\displaystyle\nu^{(i)}(N^{(i)})\frac{d\bar{z}^{(i)}}{\bar{z}^{(i)}_1\bar{z}^{(i)}_2\cdots \bar{z}^{(i)}_{m_i}}
=\nabla|_{D^{(i)}}$,
$\varphi^{(i)}_{\bmu}(N^{(i)})=0$
and the injectivity of the ring homomorphism
\[
 {\mathcal O}_{D^{(i)}}[T]/(\varphi^{(i)}_{\bmu}(T))
 \ni \overline{P(T)} \mapsto P(N^{(i)})\in \End_{{\mathcal O}_{D^{(i)}}}(E|_{D^{(i)}}).
\]
\end{itemize}
Two factorized $(\bnu,\bmu)$-connections
$(E,\nabla,\{\theta^{(i)},\kappa^{(i)}\})$ and
$(E',\nabla',\{\theta'^{(i)},\kappa'^{(i)}\})$ are isomorphic if
there is an isomorphism $\sigma\colon E\stackrel{\sim}\longrightarrow E'$
of algebraic vector bundles
such that
$(\sigma\otimes 1)\circ\nabla=\nabla'\circ\sigma$,
and
the diagrams
\[
 \begin{CD}
  E|_{D^{(i)}} @>\kappa^{(i)}>> E^{\vee}|_{D^{(i)}} \\
  @V \sigma|_{D^{(i)}} V \cong V 
  @V \cong V  ^t P^{(i)}(N^{(i)})\circ (\,^t\sigma|_{D^{(i)}})^{-1} V \\
  E'|_{D^{(i)}} @>\kappa'^{(i)}>> E'^{\vee}|_{D^{(i)}}
 \end{CD}
 \hspace{160pt}
 \begin{CD}
  E^{\vee}|_{D^{(i)}} @>\theta^{(i)}>> E|_{D^{(i)}}  \\
  @V ^t P^{(i)}(N^{(i)})\circ (\,^t\sigma|_{D^{(i)}})^{-1} V \cong V
  @V \sigma|_{D^{(i)}} V \cong V \\
  E'^{\vee}|_{D^{(i)}} @>\theta'^{(i)}>> E'|_{D^{(i)}} 
 \end{CD}
\]
are commutative
for some
$\overline{P^{(i)}(T)}\in
\left( {\mathcal O}_{D^{(i)}}[T]/(\varphi^{(i)}_{\bmu}(T))\right)^{\times}$.
\end{definition}

\begin{proposition}\label{prop:correspondence-factorized}
The correspondence
$(E,\nabla,\{\theta^{(i)},\kappa^{(i)}\})
\mapsto
(E,\nabla,\{\theta^{(i)}\circ\kappa^{(i)}\})$
gives a bijective correspondence between the isomorphism classes of
factorized $(\bnu,\bmu)$-connections and
the isomorphism classes of $(\bnu,\bmu)$-connections on $(C,D)$.
\end{proposition}

\begin{proof}
We will give the inverse correspondence.
Let $(E,\nabla,\{N^{(i)}\})$ be a 
$(\bnu,\bmu)$-connection on $(C,D)$.
We can define an ${\mathcal O}_{D^{(i)}}[T]$-module structure on
$E|_{D^{(i)}}$ by
\[
 {\mathcal O}_{D^{(i)}}[T] \times E|_{D^{(i)}}
 \ni (P(T) , v ) \mapsto P(N^{(i)})v \in E|_{D^{(i)}}.
\]
We also define  an ${\mathcal O}_{D^{(i)}}[T]$-module structure on
$E^{\vee}|_{D^{(i)}}$ by
\[
 {\mathcal O}_{D^{(i)}}[T] \times E^{\vee}|_{D^{(i)}}
 \ni ( P(T) , v ) \mapsto P(\,^tN^{(i)})v \in E^{\vee}|_{D^{(i)}}.
\]
For any point $x\in D^{(i)}$,
the homomorphism
$\mathbb{C}[T]/(\varphi^{(i)}_{\bmu}(T))\ni \overline{P(T)}
\mapsto P(N^{(i)}|_x)\in
\End_{\mathbb{C}}(E|_x)$
is injective by Remark \ref  {remark:injectivity implies subbundle}.
So the minimal polynomial of the endomorphism
$N^{(i)}|_x$ on the vector space $E|_x$ is $\varphi^{(i)}_{\bmu}|_x(T)$
whose degree is $r=\dim_{\mathbb{C}}E|_x$.
Thus an elementary theory of linear algebra implies that
there is an element $v_x\in E|_x$ such that
the homomorphism
$\mathbb{C}[T]/(\varphi^{(i)}_{\bmu}(T))\ni \overline{P(T)}
\mapsto P(N^{(i)})v_x\in E|_x$
is an isomorphism of $\mathbb{C}[T]$-modules.
If we take an element $v\in E|_{D^{(i)}}$
such that $v|_x=v_x$ for any $x\in D^{(i)}$,
then the homomorphism
\[
 {\mathcal O}_{D^{(i)}}[T]/(\varphi^{(i)}_{\bmu}(T))
 \ni \overline{P(T)}
 \mapsto P(N^{(i)})v\in E|_{D^{(i)}}
\]
is an isomorphism of
${\mathcal O}_{D^{(i)}}[T]$-modules.
Similarly $E^{\vee}|_{D^{(i)}}$
is isomorphic to
${\mathcal O}_{D^{(i)}}[T]/(\varphi^{(i)}_{\bmu}(T))$
as an ${\mathcal O}_{D^{(i)}}[T]$-module.
So we can take an ${\mathcal O}_{D^{(i)}}[T]$-isomorphism
$\theta^{(i)}\colon E^{\vee}|_{D^{(i)}}
 \stackrel{\sim}\longrightarrow E|_{D^{(i)}}$, which makes the diagram
\[
 \begin{CD}
  E^{\vee}|_{D^{(i)}} @>\theta^{(i)}>\sim> E|_{D^{(i)}}  \\
  @V ^tN^{(i)} V V
  @V N^{(i)} V  V \\
  E^{\vee}|_{D^{(i)}} @>\theta^{(i)}>\sim> E|_{D^{(i)}} 
 \end{CD}
\]
commutative.
If we define
\[
 \kappa^{(i)}:=(\theta^{(i)})^{-1}\circ N^{(i)}\colon
 E|_{D^{(i)}} \longrightarrow E^{\vee}|_{D^{(i)}},
\]
then $\kappa^{(i)}$ also becomes a homomorphism of
${\mathcal O}_{D^{(i)}}[T]$-modules.
By definition, we have
$\theta^{(i)}\circ\kappa^{(i)}=N^{(i)}$
and we can verify
the equalities
$^t\theta^{(i)}=\theta^{(i)}$ and
$^t\kappa^{(i)}=\kappa^{(i)}$
in the same way as Proposition \ref{prop:factorization-lemma}.
We can see by the same argument as
Proposition \ref{prop:factorization-unique}
that the ambiguity of the choice of $\theta^{(i)}$
is just a composition with the automorphism
of $E|_{{\mathcal D}^{(i)}_s}^{\vee}$
of the form $P(\,^t N^{(i)})$
for some $P(T)\in\mathbb{C}[T]$.
Thus we can define a correspondence
$(E,\nabla,\{N^{(i)}\})\mapsto
(E,\nabla,\{\theta^{(i)},\kappa^{(i)}\})$
which is the desired inverse correspondence
by its construction.
\end{proof}

We extend the above proposition to a relative setting
over a noetherian local scheme, that is, a scheme isomorphic to
$\Spec A$ for some noetherian local ring $A$.
Let ${\mathcal C},{\mathcal D},{\mathcal D}^{(i)},{\mathcal D}^{(i)}_j,
\tilde{\bnu},\tilde{\bmu},\varphi^{(i)}_{\tilde{\bmu}},\bar{z}^{(i)}$ and $\bar{z}^{(i)}_j$
be as in subsection \ref {subsection:former main theorem}.
Assume that $S':=\Spec A'$ is an noetherian local scheme with a morphism
$S' \longrightarrow S$.
We say that $(E,\nabla,\{N^{(i)}\})$ is a flat family of
$(\tilde{\bnu}_{S'},\tilde{\bmu}_{S'})$-connections
on $({\mathcal C}_{S'},{\mathcal D}_{S'})$ over $S'$ if
$E$ is a vector bundle on ${\mathcal C}_{S'}$ of rank $r$,
$\nabla\colon E\longrightarrow E\otimes\Omega^1_{{\mathcal C}_{S'}/S'}({\mathcal D}_{S'})$
is an $S'$-relative connection and
$N^{(i)}\colon E|_{{\mathcal D}^{(i)}_{S'}}\longrightarrow E|_{{\mathcal D}^{(i)}_{S'}}$
is an ${\mathcal O}_{{\mathcal D}^{(i)}_{S'}}$-homomorphism
such that $\varphi^{(i)}_{\tilde{\bmu}}(N^{(i)})=0$,
$\displaystyle\tilde{\nu}^{(i)}(N^{(i)})
\frac{d\bar{z}^{(i)}}{\bar{z}^{(i)}_1\bar{z}^{(i)}_2\cdots \bar{z}^{(i)}_{m_i}}
=\nabla|_{{\mathcal D}^{(i)}_{S'}}$
and the homomorphism
\[
 {\mathcal O}_{{\mathcal D}^{(i)}_{S'}}[T]/(\varphi^{(i)}_{\tilde{\bmu}}(T))
 \ni \overline{P(T)}\mapsto P(N^{(i)})\in\End(E|_{{\mathcal D}^{(i)}_{S'}})
\]
is an injection whose cokernel is flat over $S'$.
Similarly we say that
$(E,\nabla,\{\theta^{(i)},\kappa^{(i)}\})$
is a flat family of factorized $(\tilde{\bnu}_{S'},\tilde{\bmu}_{S'})$-connections 
on $({\mathcal C}_{S'},{\mathcal D}_{S'})$ over $S'$ if
$E$ is a vector bundle on ${\mathcal C}_{S'}$ of rank $r$,
$\nabla\colon E\longrightarrow E\otimes\Omega^1_{{\mathcal C}_{S'}/S'}({\mathcal D}_{S'})$
is an $S'$-relative connection,
$\theta^{(i)}\colon E^{\vee}|_{{\mathcal D}^{(i)}_{S'}}\longrightarrow E|_{{\mathcal D}^{(i)}_{S'}}$
is an isomorphism,
$\kappa^{(i)}\colon E|_{{\mathcal D}^{(i)}_{S'}}\longrightarrow E^{\vee}|_{{\mathcal D}^{(i)}_{S'}}$
is a homomorphism such that $^t\theta^{(i)}=\theta^{(i)}$,
$^t\kappa^{(i)}=\kappa^{(i)}$,
$\varphi^{(i)}(\theta^{(i)}\circ\kappa^{(i)})=0$,
$\displaystyle\nu^{(i)}(\theta^{(i)}\circ\kappa^{(i)})
\frac{d\bar{z}^{(i)}}{\bar{z}^{(i)}_1\bar{z}^{(i)}_2\cdots \bar{z}^{(i)}_{m_i}}
=\nabla|_{{\mathcal D}^{(i)}_{S'}}$
and the homomorphism
\[
 {\mathcal O}_{{\mathcal D}^{(i)}_{S'}}[T]/(\varphi^{(i)}_{\tilde{\bmu}}(T))
 \ni \overline{P(T)}\mapsto P(\theta^{(i)}\circ\kappa^{(i)})\in\End(E|_{{\mathcal D}^{(i)}_{S'}})
\]
is an injection
whose cokernel is flat over $S'$.

\begin{proposition}
Let ${\mathcal C},{\mathcal D},{\mathcal D}^{(i)},{\mathcal D}^{(i)}_j,
\tilde{\bnu},\tilde{\bmu},\varphi^{(i)}_{\tilde{\bmu}},\bar{z}^{(i)}$ and $\bar{z}^{(i)}_j$
be as in subsection \ref{subsection:former main theorem}
and let $S'$ be a noetherian local scheme with a morphism
$S' \longrightarrow S$.
Then the correspondence 
\[
 (E,\nabla,\{\theta^{(i)},\kappa^{(i)}\})
 \mapsto (E,\nabla,\{\theta^{(i)}\circ\kappa^{(i)}\})
\]
gives a bijective correspondence between the flat families of
factorized $(\tilde{\bnu}_{S'},\tilde{\bmu}_{S'})$-connections
on $({\mathcal C}_{S'},{\mathcal D}_{S'})$ over $S'$
and the flat families of
$(\tilde{\bnu}_{S'},\tilde{\bmu}_{S'})$-connections on
$({\mathcal C}_{S'},{\mathcal D}_{S'})$ over $S'$.
\end{proposition}

\begin{proof}
The proof is exactly the same as that of Proposition \ref{prop:correspondence-factorized}.
\end{proof}

\subsection{Tangent space of the moduli space of
$(\tilde{\bnu},\tilde{\bmu})$-connections}
\label {subsection:tangent}

We use the same notations as in subsection \ref {subsection:former main theorem}.
We take a $\mathbb{C}$-valued point $x$ of
$M^{\balpha}_{{\mathcal C},{\mathcal D}}(\tilde{\bnu},\tilde{\bmu})$
over a $\mathbb{C}$-valued point $s$ of $S$.
Let
$\big(E,\nabla,\{N^{(i)}\}\big)$
be the  $(\bnu,\bmu)$-connection 
on the fiber $({\mathcal C}_s,{\mathcal D}_s)$
corresponding to $x$,
where we put $(\bnu,\bmu):=(\tilde{\bnu}_s,\tilde{\bmu}_s)$.
By Proposition \ref {prop:correspondence-factorized},
we can take a factorized
$(\bnu,\bmu)$-connection
$\big(E,\nabla,\{ \theta^{(i)},\kappa^{(i)}\}\big)$
corresponding to $\big(E,\nabla,\{N^{(i)}\}\big)$.  
We will consider the deformation theory of
$\big(E,\nabla,\{N^{(i)}\}\big)$.

Recall that $\tilde{\nu}^{(i)}(T)$ is given by
\[
 \tilde{\nu}^{(i)}(T)=\sum_{j=0}^{r-1} c^{(i)}_j T^j \in
 H^0 ( {\mathcal D}^{(i)}, {\mathcal O}_{{\mathcal D}^{(i)}} )[T].
\]
We define homomorphisms
\begin{align*}
 \sigma^{(i) -}_{\theta^{(i)}} &\colon
 \End(E|_{{\mathcal D}^{(i)}_s})\oplus
 {\mathcal O}_{{\mathcal D}^{(i)}_s}[T]\big/\big(\varphi^{(i)}_{\bmu}(T)\big)
 \longrightarrow 
 \Hom(E|_{{\mathcal D}^{(i)}_s}^{\vee},E|_{{\mathcal D}^{(i)}_s}) \\
 \sigma^{(i)+}_{\kappa^{(i)}} &\colon
 \End(E|_{{\mathcal D}^{(i)}_s})\oplus
 {\mathcal O}_{{\mathcal D}^{(i)}_s}[T]\big/\big(\varphi^{(i)}_{\bmu}(T)\big)
 \longrightarrow 
 \Hom(E|_{{\mathcal D}^{(i)}_s},E|_{{\mathcal D}^{(i)}_s}^{\vee}) \\
 \delta^{(i)}_{\bnu,N^{(i)}} &\colon \End(E|_{{\mathcal D}^{(i)}_s})\longrightarrow
 \End(E|_{{\mathcal D}^{(i)}_s})\otimes\Omega^1_{{\mathcal C}_s}({\mathcal D}_s)
\end{align*}
by setting
\begin{align}
 \sigma^{(i)-}_{\theta^{(i)}} \big( u,\overline{P(T)} \big)
 &= 
 -u\circ\theta^{(i)}-\theta^{(i)}\circ \ ^t u +\theta^{(i)}\circ P(\ ^t N^{(i)})
 \label {equation:definition of sigma-} \\
 \sigma^{(i)+}_{\kappa^{(i)}} \big( u,\overline{P(T)} \big) 
 &=
 \kappa^{(i)}\circ u + \ ^t u\circ\kappa^{(i)}-P(\ ^t N^{(i)})\circ\kappa^{(i)}
 \label {equation:definition of sigma+} \\
 \delta^{(i)}_{\bnu,N^{(i)}}(u) &=
 \sum_{j=1}^{r-1}\sum_{l=1}^j c^{(i)}_j \: (N^{(i)})^{j-l}\circ u\ \circ (N^{(i)})^{l-1}
 \frac{d\bar{z}^{(i)}}{\bar{z}^{(i)}_1\bar{z}^{(i)}_2\cdots \bar{z}^{(i)}_{m_i}}
 \label {equation:definition of delta}
\end{align}
for $u\in\End(E|_{{\mathcal D}^{(i)}_s})$ and
$\overline{P(T)}\in {\mathcal O}_{{\mathcal D}^{(i)}_s}[T]
\big/\big(\varphi^{(i)}_{\bmu}(T)\big)$.
For each fixed $u\in\End(E|_{{\mathcal D}^{(i)}_s})$, we define a homomorphism
$\Theta_{u}^{(i)}\colon
{\mathcal O}_{{\mathcal D}^{(i)}_s}[T]\big/\big( \varphi^{(i)}_{\bmu}(T) \big)
\longrightarrow
\Omega^1_{{\mathcal C}_s}({\mathcal D}_s) \big|_{{\mathcal D}^{(i)}_s}$
by setting
\begin{equation} \label {equation:definition of Theta_u}
 \Theta_u^{(i)} \big(\overline{P(T)} \big)
 =
 \Tr \big( P(N^{(i)})\circ u \big)
 \frac{d\bar{z}^{(i)}}{\bar{z}^{(i)}_1\bar{z}^{(i)}_2\cdots \bar{z}^{(i)}_{m_i}}
\end{equation}
for $\overline{P(T)}\in {\mathcal O}_{{\mathcal D}^{(i)}_s}[T]/(\varphi^{(i)}_{\bmu}(T))$.
We put
\[
 {\mathcal G}^0 := {\mathcal End}(E), \quad
 {\mathcal G}^1 :={\mathcal End}(E)\otimes\Omega^1_{{\mathcal C}_s}({\mathcal D}_s), \quad
 G^1:=\bigoplus_{i=1}^n 
 \Hom \big( E|_{{\mathcal D}^{(i)}_s},E|_{{\mathcal D}^{(i)}_s}
 \otimes\Omega^1_{{\mathcal C}_s}({\mathcal D}^{(i)}_s) \big).
\]
Furthermore we put
\begin{align*}
 S(E|_{{\mathcal D}_s}^{\vee},E|_{{\mathcal D}_s})
 &=
 \Big\{ (\tau^{(i)})\in \bigoplus_{i=1}^n\
 \Hom \big( E|_{{\mathcal D}^{(i)}_s}^{\vee}, E|_{{\mathcal D}^{(i)}_s} \big) \Big|
 \,
 \text{$^t\tau^{(i)}=\tau^{(i)}$ for any $i$}
 \Big\} \\
 S(E|_{{\mathcal D}_s},E|_{{\mathcal D}_s}^{\vee})
 &=
 \Big\{
 (\xi^{(i)})\in\bigoplus_{i=1}^n 
 \Hom \big( E|_{{\mathcal D}^{(i)}_s}, E|_{{\mathcal D}^{(i)}_s}^{\vee} \big)
 \Big| \,
 \text{$^t\xi^{(i)}=\xi^{(i)}$ for any $i$}
 \Big\}
\end{align*}
and
\[
 Z^0 :=\bigoplus_{i=1}^n 
 {\mathcal O}_{{\mathcal D}^{(i)}_s}[T]\big/\big(\varphi^{(i)}_{\bmu}(T)\big), \quad
 Z^1:=\bigoplus_{i=1}^n\Hom_{{\mathcal O}_{{\mathcal D}^{(i)}_s}}
 \left({\mathcal O}_{{\mathcal D}^{(i)}_s}[T]\big/\big(\varphi^{(i)}_{\bmu}(T)\big),
 \Omega^1_{{\mathcal C}_s}({\mathcal D}^{(i)}_s)|_{{\mathcal D}^{(i)}_s} \right).
\]
We define sheaves ${\mathcal F}^0, {\mathcal F}^1, {\mathcal F}^2$
on ${\mathcal C}_s$ by
\begin{align*}
 {\mathcal F}^0
 &:=
 {\mathcal G}^0\oplus Z^0, \\
 {\mathcal F}^1
 &:=
 {\mathcal G}^1\oplus
 S(E|_{{\mathcal D}_s}^{\vee},E|_{{\mathcal D}_s})\oplus
 S(E|_{{\mathcal D}_s},E|_{{\mathcal D}_s}^{\vee}), \\
 {\mathcal F}^2
 &:=
 G^1\oplus Z^1
\end{align*}
and define homomorphisms
$d^0\colon{\mathcal F}^0\longrightarrow {\mathcal F}^1$,
$d^1\colon {\mathcal F}^1\longrightarrow {\mathcal F}^2$
by
\begin{align*}
 d^0\big(u,(\overline{P^{(i)}(T)})\big) 
 &= \left( \nabla\circ u - u\circ\nabla,
 \Big(\sigma^{(i)-}_{\theta^{(i)}} \big( u|_{{\mathcal D}^{(i)}_s},\overline{P^{(i)}(T)} \big) \Big),
 \Big(\sigma^{(i)+}_{\kappa^{(i)}} \big(u|_{{\mathcal D}^{(i)}_s},\overline{P^{(i)}(T)} \big) \Big) \right) \\
 d^1\big(v,(\tau^{(i)}),(\xi^{(i)})\big) 
 &=\left( \Big(
 v|_{{\mathcal D}^{(i)}_s}
 -\delta^{(i)}_{\bnu,N^{(i)}}\big(\tau^{(i)}\circ\kappa^{(i)}+\theta^{(i)}\circ\xi^{(i)}\big)\Big),
 \Big(\Theta^{(i)}_{(\tau^{(i)}\circ\kappa^{(i)}+\theta^{(i)}\circ\xi^{(i)})}\Big) \right).
\end{align*}

\begin{lemma}\label{lemma:define-complex}
 Under the above notation, $d^1\circ d^0=0$.
\end{lemma}

\begin{proof}
Take $(u,(\overline{P^{(i)}(T)}))\in{\mathcal F}^0={\mathcal G}^0\oplus Z^0$.
Note that
\begin{align*}
 &\quad
 \sigma^{(i)-}_{\theta^{(i)}}\big(u|_{{\mathcal D}^{(i)}_s},(\overline{P^{(i)}(T)})\big)\circ\kappa^{(i)}
 +\theta^{(i)}\circ\sigma^{(i)+}_{\kappa^{(i)}}\big(u|_{{\mathcal D}^{(i)}_s},(\overline{P^{(i)}(T)})\big) \\
 &=
\left (-u|_{{\mathcal D}^{(i)}_s}\circ\theta^{(i)}-\theta^{(i)}\circ\,^tu|_{{\mathcal D}^{(i)}_s} 
 +\theta^{(i)}\circ P^{(i)}(\,^tN^{(i)}) \right)\circ\kappa^{(i)}  \\
 &\hspace{100pt} +\theta^{(i)}\circ \left(\kappa^{(i)}\circ u|_{{\mathcal D}^{(i)}_s}
 +\,^tu|_{{\mathcal D}^{(i)}_s}\circ\kappa^{(i)}
 -P^{(i)}(\,^tN^{(i)})\circ\kappa^{(i)} \right) \\
 &=
 \theta^{(i)}\circ\kappa^{(i)}\circ u|_{{\mathcal D}^{(i)}_s}
 -u|_{{\mathcal D}^{(i)}_s}\circ\theta^{(i)}\circ\kappa^{(i)} \\
 &=
 N^{(i)}\circ u|_{{\mathcal D}^{(i)}_s}-u|_{{\mathcal D}^{(i)}_s}\circ N^{(i)} .
\end{align*}
So the first component of $d^1\big(d^0\big(u,(\overline{P^{(i)}(T)})\big)\big)$ is
\begin{align*}
 &\left( (\nabla\circ u-u\circ\nabla)|_{{\mathcal D}^{(i)}_s}
 -\delta^{(i)}_{\bnu,N^{(i)}}
 \left( \sigma^{(i)-}_{\theta^{(i)}}\big(u|_{{\mathcal D}^{(i)}_s},(\overline{P^{(i)}(T)})\big)
 \circ\kappa^{(i)}
 +\theta^{(i)}\circ
 \sigma^{(i)+}_{\kappa^{(i)}}\big(u|_{{\mathcal D}^{(i)}_s},(\overline{P^{(i)}(T)})\big)\right) \right) \\
 &=
 \left( (\nabla\circ u-u\circ\nabla)|_{{\mathcal D}^{(i)}_s}
 -\delta^{(i)}_{\bnu,N^{(i)}}( N^{(i)}\circ u|_{{\mathcal D}^{(i)}_s}-u|_{{\mathcal D}^{(i)}_s}\circ N^{(i)} )
 \right) \\
 &=
 \left( (\nabla\circ u-u\circ\nabla)|_{{\mathcal D}^{(i)}_s}
 -\sum_{j=1}^{r-1}\sum_{l=1}^j c^{(i)}_j \; (N^{(i)})^{j-l}\circ
 ( N^{(i)}\circ u|_{{\mathcal D}^{(i)}_s}-u|_{{\mathcal D}^{(i)}_s}\circ N^{(i)} )
 \circ (N^{(i)})^{l-1} \frac{d\bar{z}^{(i)}}{\bar{z}^{(i)}_1\cdots \bar{z}^{(i)}_{m_i}}
 \right) \\
 &=
 \left( (\nabla\circ u-u\circ\nabla)|_{{\mathcal D}^{(i)}_s}
 - \Big(\sum_{j=0}^{r-1} c^{(i)}_j \, (N^{(i)})^j\circ u|_{{\mathcal D}^{(i)}_s}
 - \sum_{j=0}^{r-1} c^{(i)}_j \, u|_{{\mathcal D}^{(i)}_s}\circ (N^{(i)})^j  \Big)
 \frac{d\bar{z}^{(i)}}{\bar{z}^{(i)}_1\cdots \bar{z}^{(i)}_{m_i}} \right) \\
 &=
 \left( (\nabla\circ u-u\circ\nabla)|_{{\mathcal D}^{(i)}_s}
 -\big( \nabla|_{{\mathcal D}^{(i)}_s}\circ u|_{{\mathcal D}^{(i)}_s} 
 -u|_{{\mathcal D}^{(i)}_s}\circ \nabla|_{{\mathcal D}^{(i)}_s}
 \big) \right) \\
 &=0.
\end{align*}
The second component of $d^1\big(d^0\big(u,(\overline{P^{(i)}(T)})\big)\big)$ is
\[
 \left( \Theta^{(i)}_{\sigma^{(i)-}_{\theta^{(i)}}
 (u|_{{\mathcal D}^{(i)}_s},(\overline{P^{(i)}(T)}))\circ\kappa^{(i)}
 +\theta^{(i)}\circ\sigma^{(i)+}_{\kappa^{(i)}}
 (u|_{{\mathcal D}^{(i)}_s},(\overline{P^{(i)}(T)}))} \right)
 =\left( \Theta^{(i)}_{ N^{(i)}\circ u|_{{\mathcal D}^{(i)}_s}-u|_{{\mathcal D}^{(i)}_s}\circ N^{(i)} } \right),
\]
which is zero because
\begin{align*}
 &\quad
 \Theta^{(i)}_{ N^{(i)}\circ u|_{{\mathcal D}^{(i)}_s}-u|_{{\mathcal D}^{(i)}_s}\circ N^{(i)} }
 ( \overline{Q(T)} )  \\
 &= \Tr\left( Q(N^{(i)})\circ N^{(i)}\circ u|_{{\mathcal D}^{(i)}_s}
 -Q(N^{(i)})\circ u|_{{\mathcal D}^{(i)}_s}\circ N^{(i)} \right)
 \frac{d\bar{z}^{(i)}}{\bar{z}^{(i)}_1\bar{z}^{(i)}_2\cdots \bar{z}^{(i)}_{m_i}} \\
 &= \left( \Tr\left( Q(N^{(i)})\circ N^{(i)} \circ u|_{{\mathcal D}^{(i)}_s} \right)
 -\Tr\left( N^{(i)}\circ Q(N^{(i)})\circ  u|_{{\mathcal D}^{(i)}_s}\right) \right)
 \frac{d\bar{z}^{(i)}}{\bar{z}^{(i)}_1\bar{z}^{(i)}_2\cdots \bar{z}^{(i)}_{m_i}} \\
 &=0
\end{align*}
for any 
$\overline{Q(T)}\in
{\mathcal O}_{{\mathcal D}^{(i)}_s}[T]\big/\big({\varphi^{(i)}_{\bmu}}(T)\big)$.
Thus we have proved
$d^1\big(d^0\big(u,(\overline{P^{(i)}(T)})\big)\big)=0$.
\end{proof}

By Lemma \ref{lemma:define-complex},
${\mathcal F}^{\bullet}=
[{\mathcal F}^0\xrightarrow{d^0}{\mathcal F}^1\xrightarrow{d^1}{\mathcal F}^2]$
becomes a complex.
Note that there is an exact commutative diagram
\[
 \begin{CD}
  0 \longrightarrow & 0 @>>> {\mathcal G}^0\oplus Z^0 @>>>
 {\mathcal G}^0\oplus Z^0 & \longrightarrow  0 \\
   & @VVV @V d^0 VV @VVV \\
  0 \longrightarrow & {\mathcal G}^1\oplus S(E|_{{\mathcal D}_s},E|_{{\mathcal D}_s}^{\vee}) @>>>
  {\mathcal G}^1\oplus S(E|_{{\mathcal D}_s}^{\vee},E|_{{\mathcal D}_s})\oplus
  S(E|_{{\mathcal D}_s},E|_{{\mathcal D}_s}^{\vee})
  @>>> S(E|_{{\mathcal D}_s}^{\vee},E|_{{\mathcal D}_s}) & \longrightarrow 0 \\
   & @VVV @V d^1 VV @VVV \\
  0 \longrightarrow & G^1\oplus Z^1 @>>> G^1\oplus Z^1 @>>> 0 & \longrightarrow 0.
 \end{CD}
\]
If we denote by ${\mathcal F}_0^{\bullet}$ the complex
${\mathcal G}^0\oplus Z^0\longrightarrow S(E|_{{\mathcal D}_s}^{\vee},E|_{{\mathcal D}_s})$
concentrated in degree $0$ and $1$
and if we  denote by ${\mathcal F}_1^{\bullet}$
the complex
${\mathcal G}^1\oplus S(E|_{{\mathcal D}_s},E|_{{\mathcal D}_s}^{\vee})
\longrightarrow G^1\oplus Z^1$
concentrated in degree $0$ and $1$,
then the above commutative diagram is a short exact sequence of complexes
\begin{equation} \label {equation:fundamental short exact sequence of complexes}
 0\longrightarrow {\mathcal F}_1^{\bullet}[-1]\longrightarrow{\mathcal F}^{\bullet}
 \longrightarrow {\mathcal F}_0^{\bullet}\longrightarrow 0
\end{equation}
which induces a long exact sequence of hyper cohomologies:
\begin{equation}\label{equation:fundamental exact sequence of cohomologies}
 0\longrightarrow \mathbf{H}^0({\mathcal F}^{\bullet})
 \longrightarrow
 \mathbf{H}^0({\mathcal F}_0^{\bullet})
 \longrightarrow
 \mathbf{H}^0({\mathcal F}_1^{\bullet})
 \longrightarrow
 \mathbf{H}^1({\mathcal F}^{\bullet})
 \longrightarrow
 \mathbf{H}^1({\mathcal F}_0^{\bullet})
 \longrightarrow
 \mathbf{H}^1({\mathcal F}_1^{\bullet})
 \longrightarrow
 \mathbf{H}^2({\mathcal F}^{\bullet})
 \longrightarrow 0.
\end{equation}

\begin{proposition} \label {prop:first deformation}
Let $A$ be an artinian local ring over $S$
with the maximal ideal $\mathfrak{m}$ satisfying $A/\mathfrak{m}=\mathbb{C}$
and let $I$ be an ideal of $A$ satisfying $\mathfrak{m}I=0$.
Assume that there exists a flat family
$(E',\nabla',\{N'^{(i)}\})\in{\mathcal M}^{\balpha}_{C,D}(\tilde{\bnu},\tilde{\bmu})(A)$
of $(\tilde{\bnu},\tilde{\bmu})$-connections over $A$ such that
$(E',\nabla',\{N'^{(i)}\})\otimes A/\mathfrak{m}\cong
(E,\nabla,\{N^{(i)}\})$.
Consider the restriction map
\[
 \rho_{A/I}\colon {\mathcal M}^{\balpha}_{C,D}(\tilde{\bnu},\tilde{\bmu})(A)
 \ni (\tilde{E},\tilde{\nabla},\{\tilde{N}^{(i)}\})
 \mapsto  (\tilde{E},\tilde{\nabla},\{\tilde{N}^{(i)}\})\otimes A/I
 \in {\mathcal M}^{\balpha}_{C,D}(\tilde{\bnu},\tilde{\bmu})(A/I).
\]
Then there exists a bijective correspondence
$\rho_{A/I}^{-1}((E',\nabla',\{N'^{(i)}\})\otimes A/I)
\cong \mathbf{H}^1({\mathcal F}^{\bullet})\otimes_{\mathbb{C}}I$.
\end{proposition}

\begin{proof}
We can take an affine open covering
${\mathcal C}_A=\bigcup_{\alpha}U_{\alpha}$ such that
$\sharp\{ i \,|\, {\mathcal D}^{(i)}_A \cap U_{\alpha}\neq\emptyset \}\leq 1$
for any $\alpha$
and $\sharp\{\alpha\,|\,{\mathcal D}^{(i)}_A\subset U_{\alpha}\}=1$
for any $i$.
We may assume that
$E'|_{U_{\alpha}}\cong{\mathcal O}_{U_{\alpha}}^{\oplus r}$ for any $\alpha$.
Take any member
$(\tilde{E},\tilde{\nabla},\{\tilde{N}^{(i)}\})
\in \rho_{A/I}^{-1}((E',\nabla',\{N'^{(i)}\})\otimes A/I)$.
Let $(E',\nabla',\{\theta'^{(i)},\kappa'^{(i)}\})$ and
$(\tilde{E},\tilde{\nabla},\{\tilde{\theta}^{(i)},\tilde{\kappa}^{(i)}\})$
be the flat families of factorized $(\tilde{\bnu},\tilde{\bmu})\otimes A$-connections
on $({\mathcal C}_A,{\mathcal D}_A)$ over $A$ corresponding to
$(E',\nabla',\{N'^{(i)}\})$ and $(\tilde{E},\tilde{\nabla},\{\tilde{N}^{(i)}\})$,
respectively.
We can take an isomorphism
$\sigma_{\alpha}\colon \tilde{E}|_{U_{\alpha}}
\xrightarrow{\sim} E'|_{U_{\alpha}}$
which is a lift of the given isomorphism
$\tilde{E}\otimes A/I|_{U_{\alpha}\otimes A/I}\xrightarrow{\sim}
E'\otimes A/I|_{U_{\alpha}\otimes A/I}$.
Then we put
\[
 u_{\alpha\beta}:=
 \sigma_{\alpha}\circ\sigma_{\beta}^{-1}
 -\mathrm{id}_{E'|_{U_{\alpha\beta}}}
 \in {\mathcal G}^0(U_{\alpha\beta})\otimes I,
 \quad
 v_{\alpha} := \sigma_{\alpha}\circ\tilde{\nabla}\circ\sigma_{\alpha}^{-1}-\nabla'
 \in {\mathcal G}^1(U_{\alpha})\otimes I 
\]
and
\[
  \tau^{(i)}_{\alpha} := \sigma_{\alpha}|_{{\mathcal D}^{(i)}_A}\circ\tilde{\theta}^{(i)}
 \circ\,^t\sigma_{\alpha}|_{{\mathcal D}^{(i)}_A}
  -\theta'^{(i)}, \quad
  \xi^{(i)}_{\alpha} := \,^t\sigma_{\alpha}|_{{\mathcal D}^{(i)}_A}^{-1}\circ \tilde{\kappa}^{(i)}
  \circ \sigma_{\alpha}|_{{\mathcal D}^{(i)}_A}^{-1}
  -\kappa'^{(i)}
\]
if ${\mathcal D}^{(i)}_A\subset U_{\alpha}$.
Note that we have
$((\tau^{(i)}_{\alpha}),(\xi^{(i)}_{\alpha})) \in
( S(E|_{{\mathcal D}_s}^{\vee},E|_{{\mathcal D}_s})\oplus
S(E|_{{\mathcal D}_s},E|_{{\mathcal D}_s}^{\vee}) ) (U_{\alpha})\otimes_{\mathbb{C}} I$.
We can easily check the equalities
\[
 u_{\beta\gamma}-u_{\alpha\gamma}+u_{\alpha\beta}=0, \quad
 \nabla\circ u_{\alpha\beta}-u_{\alpha\beta}\circ\nabla
 =v_{\beta}-v_{\alpha}.
\]
Since
\begin{align*}
 &\tau^{(i)}_{\alpha}\circ\kappa^{(i)}+\theta^{(i)}\circ\xi^{(i)}_{\alpha} \\
 &=
 \big(\sigma_{\alpha}|_{{\mathcal D}^{(i)}_A}\circ\tilde{\theta}^{(i)}
 \circ \,^t\sigma_{\alpha}|_{{\mathcal D}^{(i)}_A}-\theta'^{(i)}\big)
 \circ \,^t\sigma_{\alpha}|_{{\mathcal D}^{(i)}_A}^{-1}\circ\tilde{\kappa}^{(i)}
 \circ \sigma_{\alpha}|_{{\mathcal D}^{(i)}_A}^{-1}
 +\theta'^{(i)}\circ \big(\,^t\sigma_{\alpha}|_{{\mathcal D}^{(i)}_A}^{-1}
  \circ\tilde{\kappa}^{(i)}\circ \sigma_{\alpha}|_{{\mathcal D}^{(i)}_A}^{-1}
 -\kappa'^{(i)}\big)  \\
 &=
 \sigma_{\alpha}|_{{\mathcal D}^{(i)}_A}\circ\tilde{\theta}^{(i)}\circ\tilde{\kappa}^{(i)}
 \circ \sigma_{\alpha}|_{{\mathcal D}^{(i)}_A}^{-1}
 -\theta'^{(i)}\circ\kappa'^{(i)}  \\
 &=
 \sigma_{\alpha}|_{{\mathcal D}^{(i)}_A}\circ \tilde{N}^{(i)}
 \circ\sigma_{\alpha}|_{{\mathcal D}^{(i)}_A}^{-1} - N'^{(i)},
\end{align*}
we have
\begin{align*}
 &\delta^{(i)}_{\bnu,N^{(i)}}( \tau^{(i)}_{\alpha}\circ\kappa^{(i)}+\theta^{(i)}\circ\xi^{(i)}_{\alpha} ) \\
 &=
 \sum_{j=1}^{r-1}\sum_{l=1}^j c^{(i)}_j
 \big(\sigma_{\alpha}|_{{\mathcal D}^{(i)}_A}\circ\tilde{N}^{(i)}
 \circ\sigma_{\alpha}|_{{\mathcal D}^{(i)}_A}^{-1} \big)^{j-l}
 \circ \big( \sigma_{\alpha}|_{{\mathcal D}^{(i)}_A}\circ \tilde{N}^{(i)}
 \circ\sigma_{\alpha}|_{{\mathcal D}^{(i)}_A}^{-1} - N'^{(i)} \big)
 \circ (N'^{(i)})^{l-1} \frac{d\bar{z}^{(i)}}{\bar{z}^{(i)}_1\cdots\bar{z}^{(i)}_{m_i}} \\
 &=
 \sum_{j=1}^{r-1} c^{(i)}_j
 \Big( \big( \sigma_{\alpha}|_{{\mathcal D}^{(i)}_A}\circ\tilde{N}^{(i)}
 \circ\sigma_{\alpha}|_{{\mathcal D}^{(i)}_A}^{-1} \big)^j
 -\big( N'^{(i)} \big)^j \Big)
 \frac{d\bar{z}^{(i)}}{\bar{z}^{(i)}_1\cdots\bar{z}^{(i)}_{m_i}}  \\
 &=
 \sigma_{\alpha}|_{{\mathcal D}^{(i)}_A}\circ\tilde{\nabla}|_{{\mathcal D}^{(i)}_A}
 \circ\sigma_{\alpha}|_{{\mathcal D}^{(i)}_A}^{-1}-\nabla'|_{{\mathcal D}^{(i)}_A}.
\end{align*}
So the first component of $d^1(v_{\alpha},(\tau^{(i)}_{\alpha}),(\xi^{(i)}_{\alpha}))$
becomes
\begin{align*}
 v_{\alpha}|_{{\mathcal D}^{(i)}_A}
 -\delta^{(i)}_{\bnu,N^{(i)}}( \tau^{(i)}_{\alpha}\circ\kappa^{(i)}+\theta^{(i)}\circ\xi^{(i)}_{\alpha} ) 
 &=
 (\sigma_{\alpha}\circ\tilde{\nabla}\circ\sigma_{\alpha}^{-1}-\nabla')|_{{\mathcal D}^{(i)}_A}
 -(\sigma_{\alpha}|_{{\mathcal D}^{(i)}_A}\circ\tilde{\nabla}|_{{\mathcal D}^{(i)}_A}
 \circ\sigma_{\alpha}|_{{\mathcal D}^{(i)}_A}^{-1}-\nabla'|_{{\mathcal D}^{(i)}_A}) \\
 &=0.
\end{align*}
On the other hand,
$N'^{(i)}$ has a representation matrix
\[
 \begin{pmatrix}
  \mu^{(i)}_1 & \cdots & 0 \\
  \vdots & \ddots & \vdots \\
  0 & \cdots & \mu^{(i)}_r
 \end{pmatrix}
\] 
with respect to a basis $e'_1,\ldots,e'_r$ of $E'|_{{\mathcal D}^{(i)}_A}$
and $\tilde{N}^{(i)}$ has the same representation matrix
with respect to a basis $\tilde{e}_1,\ldots,\tilde{e}_r$
of $\tilde{E}|_{{\mathcal D}^{(i)}_A}$
from Definition \ref  {def:moduli functor of (nu,mu)-connections}, (c).
Moreover, we may assume that
$(e'_1,\ldots,e'_r)\otimes A/I=(\tilde{e}_1,\ldots,\tilde{e}_r)\otimes A/I$,
because $\tilde{N}^{(i)}\otimes A/I=N'^{(i)}\otimes A/I$.
So there exists $g \in I\End(E'|_{{\mathcal D}^{(i)}_A})$
satisfying
$(\mathrm{id}-g) \circ N'^{(i)} \circ (\mathrm{id}+g)
=\sigma_{\alpha}|_{{\mathcal D}^{(i)}_A}\circ\tilde{N}^{(i)}
\circ\,^t\sigma_{\alpha}|_{{\mathcal D}^{(i)}_A}$.
In other words,
$\sigma_{\alpha}|_{{\mathcal D}^{(i)}_A}\circ\tilde{N}^{(i)}
\circ\,^t\sigma_{\alpha}|_{{\mathcal D}^{(i)}_A}-N'^{(i)}
=N'^{(i)}\circ g-g\circ N'^{(i)}
=N^{(i)}\circ g-g\circ N^{(i)}$.
So the second component of $d^1(v_{\alpha},(\tau^{(i)}_{\alpha}),(\xi^{(i)}_{\alpha}))$
becomes
\[
 \Theta^{(i)}_{(\tau^{(i)}_{\alpha}\circ\kappa^{(i)}+\theta^{(i)}\circ\xi^{(i)}_{\alpha})}
 =\Theta^{(i)}_{(\sigma_{\alpha}|_{{\mathcal D}^{(i)}_A}\circ\tilde{N}^{(i)}
 \circ\sigma_{\alpha}|_{{\mathcal D}^{(i)}_A}^{-1}-N'^{(i)})} 
 =\Theta^{(i)}_{(N^{(i)}\circ g-g\circ N^{(i)})} 
 =0.
\]
Thus the element
\[
 \Phi(v):=\big[\big\{(u_{\alpha\beta},0)\big\},
 \big\{(v_{\alpha},(\tau^{(i)}_{\alpha}),(\xi^{(i)}_{\alpha}))\big\}\big]
 \in \mathbf{H}^1({\mathcal F}^{\bullet})\otimes I
\]
can be defined.

Conversely assume that
$w=\big[ \big\{(u_{\alpha\beta},0)\big\},
\big\{(v_{\alpha},(\tau^{(i)}_{\alpha}),(\xi^{(i)}_{\alpha}))\big\} \big]
 \in \mathbf{H}^1({\mathcal F}^{\bullet})\otimes I$
is given.
We put
$E_{\alpha}:=E'|_{U_{\alpha}}$ and define a connection
$\nabla_{\alpha}\colon E_{\alpha}\longrightarrow
E_{\alpha}\otimes\Omega^1_{{\mathcal C}_A/A}({\mathcal D}_A)$
by
$\nabla_{\alpha}=\nabla'+v_{\alpha}$.
Furthermore, we put
$\theta^{(i)}_{\alpha}:=\theta'^{(i)}+\tau^{(i)}_{\alpha}$,
$\kappa^{(i)}_{\alpha}:=\kappa'^{(i)}+\xi^{(i)}_{\alpha}$
if ${\mathcal D}^{(i)}_A\subset U_{\alpha}$.
We define the isomorphism
\[
 \varphi_{\beta\alpha}=\mathrm{id}+u_{\beta\alpha}\colon
 E_{\alpha}|_{U_{\alpha\beta}}\xrightarrow{\sim}
 E_{\beta}|_{U_{\alpha\beta}}.
\]
Since $(\{(u_{\alpha\beta},0)\},\{(v_{\alpha},(\tau^{(i)}_{\alpha}),(\xi^{(i)}_{\alpha}))\})$
satisfies the cocycle conditions
$\nabla\circ u_{\alpha\beta}-u_{\alpha\beta}\circ\nabla=v_{\beta}-v_{\alpha}$
and
$u_{\beta\alpha}-u_{\gamma\alpha}+u_{\gamma\beta}=0$,
we have the gluing condition
\[
 \varphi_{\gamma\alpha}=\varphi_{\gamma\beta}\circ\varphi_{\beta\alpha},
 \quad
 (\varphi_{\beta\alpha}\otimes 1)\circ\nabla_{\alpha}=\nabla_{\beta}\circ\varphi_{\beta\alpha}.
\]
So we can patch the local connections
$\{ (E_{\alpha},\nabla_{\alpha},\{ \theta^{(i)}_{\alpha},\kappa^{(i)}_{\alpha}\}) \}$
together via $\{ \varphi_{\beta\alpha} \}$
and obtain a flat family 
$(\tilde{E},\tilde{\nabla},\{\tilde{\theta}^{(i)},\tilde{\kappa}^{(i)}\})$
of factorized $(\tilde{\bnu},\tilde{\bmu})\otimes A$-connections over $A$,
which we denote by $\Psi(w)$.
By construction the correspondence
$\mathbf{H}^1({\mathcal F}^{\bullet})\otimes I \ni w
\mapsto \Psi(w)\in \rho_{A/I}^{-1} ((E',\nabla',\{N'^{(i)}\})\otimes A/I)$
gives the inverse of $\Phi$.
\end{proof}

As a corollary of Proposition \ref {prop:first deformation},
we get the following.

\begin{corollary}\label{cor:tangent-space}
The relative tangent space of the moduli space
$M^{\balpha}_{{\mathcal C},{\mathcal D}}(\tilde{\bnu},\tilde{\bmu})$ over $S$ at 
$(E,\nabla,\{N^{(i)}\})\in M^{\balpha}_{{\mathcal C},{\mathcal D}}(\tilde{\bnu},\tilde{\bmu})$
is isomorphic to
$\mathbf{H}^1({\mathcal F}^{\bullet})$.
\end{corollary}

\subsection{Nondegenerate pairing on the cohomologies}
\label {subsection:nondenerate pairing}

We use the same notations as in subsection \ref{subsection:tangent}.

If we denote
the complex
\[
 {\mathcal O}_{\mathcal C}
 \xrightarrow{d}
 \Omega^1_{{\mathcal C}/S}({\mathcal D})
 \longrightarrow
 \Omega^1_{{\mathcal C}/S}({\mathcal D})|_{\mathcal D}.
\]
by ${\mathcal L}^{\bullet}$, then there is a canonical quasi-isomorphism
$\Omega^{\bullet}_{{\mathcal C}/S}\longrightarrow
{\mathcal L}^{\bullet}$
and there is an isomorphism
\[
 \mathbf{H}^2({\mathcal L}^{\bullet}_s)\cong
 \mathbf{H}^2(\Omega^{\bullet}_{{\mathcal C}_s})
 \cong\mathbb{C},
\]
where ${\mathcal L}^{\bullet}_s:={\mathcal L}^{\bullet}|_{{\mathcal C}_s}$
is the restriction of the complex ${\mathcal L}^{\bullet}$
to the fiber ${\mathcal C}_s$.
We consider the modified complex
\[
 \tilde{\mathcal L}^{\bullet}_s \colon\quad
 {\mathcal L}^0_s\xrightarrow{\tilde{d}^0} {\mathcal L}^1_s\oplus Z^1
 \xrightarrow{\tilde{d}^1} {\mathcal L}^2_s\oplus Z^1,
\]
defined by
\[
 \tilde{d}^0(u)=(\,d u, \, 0 \,), \quad
 \tilde{d}^1(v,(Q^i))=
 \left( \big( v|_{{\mathcal D}^{(i)}_s}-Q^i\big( (\nu^{(i)})'(T)\big) \big), (Q^i) \right),
\]
where $(\nu^{(i)})'(T)$ is the derivative of
the polynomial $\nu^{(i)}(T)$ in $T$.
Then there is a canonical quasi-isomorphism
${\mathcal L}^{\bullet}_s\longrightarrow
\tilde{\mathcal L}^{\bullet}_s$.

We define a morphism of complexes
$\Tr\colon {\mathcal F}^{\bullet}
\longrightarrow \tilde{\mathcal L}^{\bullet}_s$
by
\begin{gather*}
 \Tr^0\big(u,(\overline{P^{(i)}(T)})\big)=\Tr(u), \quad
 \Tr^1\big(v,(\tau^{(i)}),(\xi^{(i)})\big)
 =\big( \Tr(v), ( \Theta_{\tau^{(i)}\circ\kappa^{(i)}+\theta^{(i)}\circ\xi^{(i)}} )\big), \\
 \Tr^2\big((g^{(i)}),(Q^{(i)})\big)=
 \big( ( \Tr ( g^{(i)} ) ) , (Q^{(i)}) \big).
\end{gather*}
Indeed  we can check the following commutative diagram:
\[
 \begin{CD}
  {\mathcal G}^0 \oplus Z^0 @>>>
  {\mathcal G}^1\oplus S(E|_{{\mathcal D}_s}^{\vee},E|_{{\mathcal D}_s})
  \oplus S(E|_{{\mathcal D}_s},E|_{{\mathcal D}_s}^{\vee})
  @>>> G^1\oplus Z^1 \\
  @V\Tr^0VV @V\Tr^1VV @V\Tr^2VV \\
  {\mathcal O}_{{\mathcal C}_s} @>d>>\Omega^1_{{\mathcal C}_s}({\mathcal D}_s)\oplus Z^1
  @>>> \Omega^1_{{\mathcal C}_s}({\mathcal D}_s)|_{{\mathcal D}_s}\oplus Z^1.
 \end{CD}
\]

For
$((\tau^{(i)}),(\xi^{(i)})),((\tau'^{(i)}),(\xi'^{(i)}))\in
S(E|_{{\mathcal D}_s}^{\vee},E|_{{\mathcal D}_s})\oplus
S(E|_{{\mathcal D}_s},E|_{{\mathcal D}_s}^{\vee})$,
we define
$\Xi^{(\tau^{(i)},\xi^{(i)})}_{(\tau'^{(i)},\xi'^{(i)})} \in
\Omega^1_{{\mathcal C}_s}({\mathcal D}^{(i)}_s)|_{{\mathcal D}^{(i)}_s}$ 
by setting
\begin{align}
 \Xi^{(\tau^{(i)},\xi^{(i)})}_{(\tau'^{(i)},\xi'^{(i)})}
 &=
 \frac{1}{2}\sum_{j=1}^{r-1} \sum_{l=0}^{j-1}  c^{(i)}_j 
 \Tr\Big(   
 \tau'^{(i)} \circ (\,^tN^{(i)})^l\circ\xi^{(i)} \circ (N^{(i)})^{j-l-1} \Big)
 \frac{d\bar{z}^{(i)}}{\bar{z}^{(i)}_1\bar{z}^{(i)}_2\cdots \bar{z}^{(i)}_{m_i}} 
 \label {equation:main part in the definition of symplectic form} \\
 & \hspace{30pt}
 - \frac{1}{2}\sum_{j=1}^{r-1} \sum_{l=0}^{j-1}  c^{(i)}_j 
 \Tr\Big(   \tau^{(i)} \circ (\,^tN^{(i)})^l \circ \xi'^{(i)} \circ (N^{(i)})^{j-1-l} \Big)
 \frac{d\bar{z}^{(i)}}{\bar{z}^{(i)}_1\bar{z}^{(i)}_2\cdots \bar{z}^{(i)}_{m_i}}.
 \notag
\end{align}

\begin{remark} \label{remark: kirillov-kostant factor in symplectic form}
\rm
In the extreme case when $\mu^{(i)}_k=\nu^{(i)}(\mu^{(i)}_k)$ for any $k$,
we have $c^{(i)}_1=1$ and $c^{(i)}_j=0$ for $j\neq 1$.
So we have
\[
 \Xi^{(\tau^{(i)},\xi^{(i)})}_{(\tau'^{(i)},\xi'^{(i)})}
 =
 \frac{1}{2} \Tr \Big(   
 \tau'^{(i)} \circ \xi^{(i)} - \tau^{(i)} \circ \xi'^{(i)}  \Big)
 \frac{d\bar{z}^{(i)}}{\bar{z}^{(i)}_1\bar{z}^{(i)}_2\cdots \bar{z}^{(i)}_{m_i}}
\]
which is almost the same form as the expression
in subsection \ref {subsection:expression of kirillov-kostant form},
(\ref {equation:symplectic form for adjoint orbit})
 of the Kirillov-Kostant form in Proposition \ref {prop:kirillov-kostant-form}.
\end{remark}

We define a bilinear pairing
\begin{equation*} 
 \omega_{(E,\nabla,\{N^{(i)}\})}\colon
 \mathbf{H}^1({\mathcal F}^{\bullet}) \times \mathbf{H}^1({\mathcal F}^{\bullet})
 \longrightarrow \mathbf{H}^2({\mathcal L}^{\bullet}_s)
 \cong\mathbb{C}
\end{equation*}
on $\mathbf{H}^1({\mathcal F}^{\bullet})$ by setting
\begin{align}
 &\omega_{(E,\nabla,\{N^{(i)}\})}
 \left( \big[ \big\{ (u_{\alpha\beta},0) \big\},
 \big\{ (v_{\alpha}, ((\tau^{(i)}_{\alpha}),(\xi^{(i)}_{\alpha})))\big\} \big],
 \big[ \big\{ (u'_{\alpha\beta},0) \big\} ,
 \big\{ (v'_{\alpha}, ((\tau'^{(i)}_{\alpha}),(\xi'^{(i)}_{\alpha}))) \big\} \big] \right)
 \label {equation:definition of nondegenerate pairing}  \\
 &=\left[
 \left\{ \Tr(u_{\alpha\beta}\circ u'_{\beta\gamma})\right\},
 -\left\{ \Tr(u_{\alpha\beta}\circ v'_{\beta}-v_{\alpha}\circ u'_{\alpha\beta}) \right\}, 
 \left\{ \Big(\Xi^{(\tau^{(i)}_{\alpha},\xi^{(i)}_{\alpha})}_{(\tau'^{(i)}_{\alpha},\xi'^{(i)}_{\alpha})}\Big)
 \right\} \right]
 \in \mathbf{H}^2({\mathcal L}^{\bullet}_s).  \notag
\end{align}
We will check that
the cohomology class (\ref {equation:definition of nondegenerate pairing}) 
in $\mathbf{H}^2({\mathcal L}^{\bullet}_s)$
is independent of the choice of the representatives
$\big(\big\{(u_{\alpha\beta},0)\big\},
\big\{(v_{\alpha},((\tau^{(i)}_{\alpha}),(\xi^{(i)}_{\alpha})))\big\}\big)$
and
$\big( \big\{(u'_{\alpha\beta},0)\big\},
\big\{(v'_{\alpha},((\tau'^{(i)}_{\alpha}),(\xi'^{(i)}_{\alpha})))\big\}\big)$,
respectively.
Indeed assume that
$\big[\big\{(u_{\alpha\beta},0)\big\},
\big\{(v_{\alpha},((\tau^{(i)}_{\alpha}),(\xi^{(i)}_{\alpha})))\big\}\big]=0$
in $\mathbf{H}^1({\mathcal F}^{\bullet})$.
Then there is
$\big\{ u_{\alpha}, \big(\overline{P^{(i)}_{\alpha}(T)}\big) \big\}\in
C^0(\{U_{\alpha}\},{\mathcal G}^0\oplus Z^0)$
satisfying
\begin{align*}
 &
 u_{\alpha\beta}=u_{\beta}-u_{\alpha}, \hspace{30pt}
 v_{\alpha}=\nabla\circ u_{\alpha}-u_{\alpha}\circ\nabla, \\
 & 
 \tau^{(i)}_{\alpha}=
 \sigma^{^{(i)}-}_{\theta^{(i)}}(u_{\alpha}|_{{\mathcal D}^{(i)}_s},\overline{P^{(i)}(T)})
 =-(u_{\alpha}|_{{\mathcal D}^{(i)}_s}\circ\theta^{(i)}
 +\theta^{(i)}\circ \,^t u_{\alpha}|_{{\mathcal D}^{(i)}_s})
 +\theta^{(i)}\circ P(\,^tN^{(i)}) \\
 &
 \xi^{(i)}_{\alpha}=
 \sigma^{^{(i)}+}_{\kappa^{(i)}}(u_{\alpha}|_{{\mathcal D}^{(i)}_s},\overline{P^{(i)}(T)})
 =\kappa^{(i)}\circ u_{\alpha}|_{{\mathcal D}^{(i)}_s}
 +\,^t u_{\alpha}|_{{\mathcal D}^{(i)}_s}\circ\kappa^{(i)}
 -P(\,^tN^{(i)})\circ\kappa^{(i)}.
\end{align*}
So we can write
\begin{align}
 &
 \omega_{(E,\nabla,\{N^{(i)}\})}
 \left( \big[
 \big\{ (u_{\alpha\beta},0) \big\},
 \big\{ (v_{\alpha},((\tau^{(i)}_{\alpha}),(\xi^{(i)}_{\alpha}))) \big\} \big],
 \big[ \big\{ (u'_{\alpha\beta},0) \big\},
 \big\{ (v'_{\alpha},((\tau'^{(i)}_{\alpha}),(\xi'^{(i)}_{\alpha}))) \big\} \big] \right)
  \label {equation: symplectic form coboundary} \\
 &=
 \bigg[ \big\{ \Tr((u_{\beta}-u_{\alpha})\circ u'_{\beta\gamma}) \big\},
 -\big\{ \Tr( (u_{\beta}-u_{\alpha})\circ v'_{\beta}
 - (\nabla\circ u_{\alpha}-u_{\alpha}\circ\nabla)\circ u'_{\alpha\beta}) \big\},  \notag \\
 &\hspace{200pt}
 \Big\{ \Big(
 \Xi^{\sigma^{^{(i)}-}_{\theta^{(i)}}(u_{\alpha}|_{{\mathcal D}^{(i)}_s},\overline{P^{(i)}(T)}),
  \sigma^{^{(i)}+}_{\kappa^{(i)}}(u_{\alpha}|_{{\mathcal D}^{(i)}_s},\overline{P^{(i)}(T)})}
 _{(\tau'^{(i)}_{\alpha},\xi'^{(i)}_{\alpha})}
  \Big) \Big\} \bigg]. \notag
\end{align}
If we put $c_{\alpha\beta}:=\Tr(u_{\alpha}\circ u'_{\alpha\beta})$, then
$\{c_{\alpha\beta}\}\in C^1(\{U_{\alpha}\}, L^0_s)$ and
\begin{equation} \label {equation: first component of coboundary of c}
 \big\{ \Tr((u_{\beta}-u_{\alpha})\circ u'_{\beta\gamma}) \big\}
 =\big\{ \Tr(u_{\beta}\circ u'_{\beta\gamma}
 -u_{\alpha}\circ(u'_{\alpha\gamma}-u'_{\alpha\beta})) \big\}
 =\big\{ c_{\beta\gamma}-c_{\alpha\gamma}+c_{\alpha\beta} \big\}.
\end{equation}
If we put $b_{\alpha}:=\Tr(u_{\alpha}\circ v'_{\alpha})$,
then $\{b_{\alpha}\}\in C^0(\{U_{\alpha}\},{\mathcal L}^1_s)$ and we have
\begin{align}
 d^0_{{\mathcal L}^{\bullet}_s} \big(\{c_{\alpha\beta} \}\big)
 &= \big\{ d \Tr(u_{\alpha}\circ u'_{\alpha\beta}) \big\}
 =
 \big\{ \Tr (\nabla\circ u_{\alpha}\circ u'_{\alpha\beta}
 -u_{\alpha}\circ u'_{\alpha\beta}\circ \nabla) \big\}  
  \label {equation: second component of coboundary of c}  \\
 &=
 \big\{ \Tr( (\nabla\circ u_{\alpha}-u_{\alpha}\circ\nabla)\circ u'_{\alpha\beta}
 +u_{\alpha}\circ( \nabla\circ u'_{\alpha\beta}-u'_{\alpha\beta}\circ\nabla)) \big\}
 \notag \\
 &=
 \big\{ \Tr( (\nabla\circ u_{\alpha}-u_{\alpha}\circ\nabla)\circ u'_{\alpha\beta}
 +u_{\alpha}\circ( v'_{\beta}-v'_{\alpha})) \big\}
 \notag \\
 &=
 \big\{ \Tr( (\nabla\circ u_{\alpha}-u_{\alpha}\circ\nabla)\circ u'_{\alpha\beta}
 +(u_{\alpha}-u_{\beta})\circ v'_{\beta}+(u_{\beta}\circ v'_{\beta}-u_{\alpha}\circ v'_{\alpha}))
 \big\} \notag \\
 &=
 -\big\{ \Tr( (u_{\beta}-u_{\alpha})\circ v'_{\beta}
 - (\nabla\circ u_{\alpha}-u_{\alpha}\circ\nabla)\circ u'_{\alpha\beta}) \big\}
 + \big\{ b_{\beta}-b_{\alpha} \big\}.   \notag
\end{align}
Since 
$\Tr\big(
 (\tau'^{(i)}_{\alpha}\circ\kappa^{(i)}+\theta^{(i)}\circ \xi'^{(i)}_{\alpha})
 \circ(N^{(i)})^l\circ P^{(i)}(N^{(i)})\circ(N^{(i)})^{j-1-l} \big)=0$
follows from
$\Theta^{(i)}_{\tau'^{(i)}\circ\kappa^{(i)}+\theta^{(i)}\circ\xi'^{(i)}}=0$,
\begin{align*}
 &\Tr\big( \tau'^{(i)}_{\alpha}\circ(\,^tN^{(i)})^l\circ
 \sigma^{^{(i)}+}_{\kappa^{(i)}}(u_{\alpha}|_{{\mathcal D}^{(i)}_s},
 \overline{P^{(i)}(T)})\circ(N^{(i)})^{j-1-l}\big) \\
 &\hspace{30pt}
 -\Tr\big( \sigma^{^{(i)}-}_{\theta^{(i)}}(u_{\alpha}|_{{\mathcal D}^{(i)}_s},\overline{P^{(i)}(T)})
 \circ(\,^tN^{(i)})^l \circ\xi'^{(i)}_{\alpha}\circ(N^{(i)})^{j-1-l}\big) \\
 &=
  \Tr \left( \tau'^{(i)}_{\alpha}\circ (\,^tN^{(i)})^l\circ \big(
 \kappa^{(i)}\circ u_{\alpha}|_{{\mathcal D}^{(i)}_s}
 +\,^tu_{\alpha}|_{{\mathcal D}^{(i)}_s}\circ\kappa^{(i)}
 -P^{(i)}(\,^tN^{(i)})\circ\kappa^{(i)} \big)\circ (N^{(i)})^{j-1-l} \right) \\
 & \hspace{30pt}
 -\Tr \left( \big(-u_{\alpha}|_{{\mathcal D}^{(i)}_s}\circ\theta^{(i)}
 -\theta^{(i)}\circ\,^tu_{\alpha}|_{{\mathcal D}^{(i)}_s}
 +\theta^{(i)}\circ P^{(i)}(\,^tN^{(i)})\big)\circ(\,^tN^{(i)})^l
 \circ \xi'^{(i)}_{\alpha}\circ(N^{(i)})^{j-1-l} \right) \\
 &=
  \Tr \left( 
 \tau'^{(i)}_{\alpha}\circ(\,^tN^{(i)})^l\circ\kappa^{(i)}\circ
 u_{\alpha}|_{{\mathcal D}^{(i)}_s}\circ(N^{(i)})^{j-1-l}
 +(\,^tN^{(i)})^{j-1-l}\circ\kappa^{(i)}\circ u_{\alpha}|_{{\mathcal D}^{(i)}_s}
 \circ (N^{(i)})^l \circ\tau'^{(i)}_{\alpha} \right) \\
 &\hspace{20pt}
  + \Tr\left(
 u_{\alpha}|_{{\mathcal D}^{(i)}_s}\circ\theta^{(i)}\circ
 (\,^tN^{(i)})^l\circ\xi'^{(i)}_{\alpha}\circ(N^{(i)})^{j-1-l}
 +(\,^tN^{(i)})^{j-1-l}\circ\xi'^{(i)}_{\alpha}\circ(N^{(i)})^l\circ
 u_{\alpha}|_{{\mathcal D}^{(i)}_s}\circ\theta^{(i)} \right) \\
  &\hspace{30pt}
  -\Tr\left(
 \big( \tau'^{(i)}_{\alpha}\circ\kappa^{(i)}+\theta^{(i)}\circ \xi'^{(i)}_{\alpha} \big)
 \circ(N^{(i)})^l\circ P^{(i)}(N^{(i)})\circ(N^{(i)})^{j-1-l} \right) \\
 &=
  \Tr \left( 
 u_{\alpha}|_{{\mathcal D}^{(i)}_s}\circ(N^{(i)})^{j-1-l} \circ
 \tau'^{(i)}_{\alpha}\circ\kappa^{(i)}\circ(N^{(i)})^l
 +u_{\alpha}|_{{\mathcal D}^{(i)}_s}\circ (N^{(i)})^l  \circ
 \tau'^{(i)}_{\alpha}\circ \kappa^{(i)}\circ (N^{(i)})^{j-1-l} \right) \\
 & \hspace{20pt}
  + \Tr\left(
 u_{\alpha}|_{{\mathcal D}^{(i)}_s}\circ(N^{(i)})^l\circ\theta^{(i)}
 \circ\xi'^{(i)}_{\alpha}\circ(N^{(i)})^{j-1-l}
 +u_{\alpha}|_{{\mathcal D}^{(i)}_s}\circ (N^{(i)})^{j-1-l}\circ
 \theta^{(i)}\circ\xi'^{(i)}_{\alpha}\circ(N^{(i)})^l \right).
\end{align*}
So we have
\begin{align*}
 & \Xi^{\sigma^{^{(i)}-}_{\theta^{(i)}}(u_{\alpha}|_{{\mathcal D}^{(i)}_s},\overline{P^{(i)}(T)}),
  \sigma^{^{(i)}+}_{\kappa^{(i)}}(u_{\alpha}|_{{\mathcal D}^{(i)}_s},\overline{P^{(i)}(T)})}
 _{(\tau'^{(i)}_{\alpha},\xi'^{(i)}_{\alpha})} \\
 &=
 \frac{1}{2}\sum_{j=1}^{r-1}\sum_{l=0}^{j-1} c^{(i)}_j
 \Tr \bigg( \tau'^{(i)}_{\alpha} \circ \big( \,^tN^{(i)} \big)^l \circ
 \sigma^{^{(i)}+}_{\kappa^{(i)}}\big( u_{\alpha}|_{{\mathcal D}^{(i)}_s},\overline{P^{(i)}(T)} \big)
 \circ \big( N^{(i)} \big)^{j-1-l} \\
 & \hspace{100pt}
 -
 \sigma^{^{(i)}-}_{\theta^{(i)}} \big( u_{\alpha}|_{{\mathcal D}^{(i)}_s},\overline{P^{(i)}(T)} \big)
 \circ(\,^tN^{(i)})^l \circ \xi'^{(i)}_{\alpha} \circ \big( N^{(i)} \big)^{j-1-l} \bigg)
  \frac{d\bar{z}^{(i)}}{\bar{z}^{(i)}_1\bar{z}^{(i)}_2\cdots \bar{z}^{(i)}_{m_i}} \\
 &=
  \sum_{j=1}^{r-1} \sum_{l=0}^{j-1}  c^{(i)}_j
 \Tr \bigg( 
 u_{\alpha}|_{{\mathcal D}^{(i)}_s}\circ \big( N^{(i)} \big)^l \circ
 \big( \tau'^{(i)}_{\alpha}\circ\kappa^{(i)} +\theta^{(i)}\circ\xi'^{(i)}_{\alpha} \big)
 \circ \big( N^{(i)} \big)^{j-1-l} \bigg)
 \frac{d\bar{z}^{(i)}}{\bar{z}^{(i)}_1\bar{z}^{(i)}_2\cdots \bar{z}^{(i)}_{m_i}} \\
 &=
 \Tr \Big( u_{\alpha}|_{{\mathcal D}^{(i)}_s}\circ
 \delta^{(i)}_{\bnu,N^{(i)}}
 \big( \tau'^{(i)}_{\alpha}\circ\kappa^{(i)}+\theta^{(i)}\circ\xi'^{(i)}_{\alpha} \big) \Big).
\end{align*}
Since
$v'_{\alpha}|_{{\mathcal D}^{(i)}_s}
=\delta^{(i)}_{\bnu,N^{(i)}}(\tau'^{(i)}_{\alpha}\circ\kappa^{(i)}+\theta^{(i)}\circ\xi'^{(i)}_{\alpha})$,
we have
\begin{align}
 d^1_{{\mathcal L}^{\bullet}_s} \left\{ (b_{\alpha}) \right\}
 =
 \left\{ \left(\Tr( u_{\alpha}\circ v'_{\alpha} )|_{{\mathcal D}^{(i)}_s}\right) \right\} 
 &= \left\{  \Big(\Tr
 \Big( u_{\alpha}|_{{\mathcal D}^{(i)}_s}\circ \delta^{(i)}_{\bnu,N^{(i)}}  (\tau'^{(i)}_{\alpha}\circ\kappa^{(i)}
 +\theta^{(i)}\circ\xi'^{(i)}_{\alpha}) \Big) \Big) \right\} 
  \label {equation: coboundary of b}  \\
 &=\Big\{  \Big(
  \Xi^{\sigma^{^{(i)}-}_{\theta^{(i)}}(u_{\alpha}|_{{\mathcal D}^{(i)}_s},\overline{P^{(i)}(T)}),
  \sigma^{^{(i)}+}_{\kappa^{(i)}}(u_{\alpha}|_{{\mathcal D}^{(i)}_s},\overline{P^{(i)}(T)})}
 _{(\tau'^{(i)}_{\alpha},\xi'^{(i)}_{\alpha})} \Big) \Big\}. \notag
\end{align}
The equalities (\ref {equation: first component of coboundary of c}),
(\ref {equation: second component of coboundary of c})
and (\ref  {equation: coboundary of b}) mean that
the cohomology class (\ref  {equation: symplectic form coboundary})
is represented as the coboundary of
$\big( \big\{ c_{\alpha\beta}\big\}, \big\{ b_{\alpha} \big\} \big)
\in C^0(\{U_{\alpha}\}, {\mathcal L}^{\bullet}_s)$,
which should be zero in $\mathbf{H}^2({\mathcal L}^{\bullet}_s)$.
Similarly (\ref {equation:definition of nondegenerate pairing}) becomes zero
when
$\big[ \big\{ (u'_{\alpha\beta},0) \big\} ,
\big\{ (v'_{\alpha}, (\tau'^{(i)}_{\alpha}),(\xi'^{(i)}_{\alpha})) \big\}] =0$
in $\mathbf{H}^1({\mathcal F}^{\bullet})$.
Thus we have proved that the bilinear pairing
$\omega_{(E,\nabla,\{N^{(i)}\})}$ is well-defined.

\begin{lemma} \label {lemma-non-degenerate}
 The bilinear pairing
$\omega_{(E,\nabla,\{N^{(i)}\})}\colon
\mathbf{H}^1({\mathcal F}^{\bullet}) \times \mathbf{H}^1({\mathcal F}^{\bullet})
 \longrightarrow \mathbf{H}^2({\mathcal L}^{\bullet}_s)
 \cong\mathbb{C}$
 defined in (\ref {equation:definition of nondegenerate pairing})  is a non-degenerate pairing.
\end{lemma}

\begin{proof}
Let 
$\sigma\colon \mathbf{H}^1({\mathcal F}^{\bullet})\longrightarrow\mathbf{H}^1({\mathcal F})^{\vee}$
be the homomorphism determined by the pairing
$\omega_{(E,\nabla,\{N^{(i)}\})}$.
We have to show that $\sigma$ is an isomorphism.
We can see that $\sigma$ induces the following exact commutative  diagram
\[
 \begin{CD}
  \mathbf{H}^0({\mathcal F}_0^{\bullet}) @>>> \mathbf{H}^0({\mathcal F}_1^{\bullet})
  @>>> \mathbf{H}^1({\mathcal F}^{\bullet}) @>>> \mathbf{H}^1({\mathcal F}_0^{\bullet})
  @>>> \mathbf{H}^1({\mathcal F}_1^{\bullet})  \\
  @V\sigma_1VV   @V\sigma_2VV    @V\sigma VV   @V\sigma_3VV    @V\sigma_4VV  \\
  \mathbf{H}^1({\mathcal F}_1^{\bullet})^{\vee} @>>> \mathbf{H}^1({\mathcal F}_0^{\bullet})^{\vee}
  @>>> \mathbf{H}^1({\mathcal F}^{\bullet})^{\vee} @>>> \mathbf{H}^0({\mathcal F}_1^{\bullet})^{\vee}
  @>>> \mathbf{H}^0({\mathcal F}_0^{\bullet})^{\vee}.
 \end{CD}
\]
Here
$\sigma_2 \colon  \mathbf{H}^0({\mathcal F}_1^{\bullet})
\longrightarrow
\mathbf{H}^1({\mathcal F}_0^{\bullet})^{\vee}$
and
$\sigma_3 \colon \mathbf{H}^1({\mathcal F}_0^{\bullet})
\longrightarrow \mathbf{H}^0({\mathcal F}_1^{\bullet})^{\vee}$
are given by the pairing
\begin{align*}
 \mathbf{H}^0({\mathcal F}_1^{\bullet})
 \times \mathbf{H}^1({\mathcal F}_0^{\bullet})
 \longrightarrow & \;
 \mathbf{H}^2({\mathcal L}^{\bullet}_s)
 \cong\mathbb{C} \\
 \left( \Big[ \big\{(v_{\alpha},(\xi^{(i)}_{\alpha})) \big\} \Big],
 \Big[ \big\{ (u'_{\alpha\beta},(\tau'^{(i)}_{\alpha}))\big\} \Big] \right) 
 \mapsto & \;
 \left[\left\{\Tr(v_{\alpha}\circ u'_{\alpha\beta})\right\},
 \Big\{ 
 \big(\Xi^{(0,\xi^{(i)}_{\alpha})}_{(\tau'^{(i)}_{\alpha},0)}\big)
 \Big\}\right]
\end{align*}
and 
$\sigma_1\colon \mathbf{H}^0({\mathcal F}_0^{\bullet})
\longrightarrow \mathbf{H}^1({\mathcal F}_1^{\bullet})^{\vee}$
and
$\sigma_4\colon \mathbf{H}^1({\mathcal F}_1^{\bullet})
\longrightarrow \mathbf{H}^0({\mathcal F}_0^{\bullet})^{\vee}$
are defined by the pairing
\begin{align*}
 & \hspace{20pt}
 \mathbf{H}^0({\mathcal F}_0^{\bullet})
 \times \mathbf{H}^1({\mathcal F}_1^{\bullet})
 \longrightarrow
 \mathbf{H}^2({\mathcal L}^{\bullet}_s)
 \cong\mathbb{C} \\
 &
 \left( \bigl[ \big\{(u_{\alpha},(\overline{P^{(i)}_{\alpha}}))\big\}\bigl],
 \big[ \big\{v'_{\alpha\beta}\big\},  \big\{ (g'^{(i)}_{\alpha}), (Q'^{(i)}_{\alpha})\big\}\big]
 \right) \\
 & \hspace{30pt}
 \mapsto 
 \left[\left\{-\Tr(u_{\alpha}\circ v'_{\alpha\beta})\right\},
 -\Big\{ \Big( \Tr\big(u_{\alpha}|_{{\mathcal D}^{(i)}_s}\circ g'^{(i)}_{\alpha}\big)
 \frac{d\bar{z}^{(i)}}{\bar{z}^{(i)}_1\bar{z}^{(i)}_2\cdots \bar{z}^{(i)}_{m_i}}
 +\frac{1}{2} Q'^{(i)}_{\alpha}\big(\overline{P^{(i)}_{\alpha}(T)(\nu^{(i)})'(T)}\big)\Big)
 \Big\}\right].
\end{align*}

We denote the short exact sequence of complexes
\[
 \begin{CD}
 0 @>>>  {\mathcal G}^1 @>>>
 {\mathcal G}^1\oplus S(E|_{{\mathcal D}_s},E|_{{\mathcal D}_s}^{\vee})
 @>>> S(E|_{{\mathcal D}_s},E|_{{\mathcal D}_s}^{\vee}) @>>> 0 \\
  &  & @VVV @VVV @VVV \\
 0 @>>>  G^1 @>>> G^1\oplus Z^1 @>>> Z^1 @>>> 0
\end{CD}
\]
simply by
$0 \longrightarrow [{\mathcal G}^1\rightarrow G^1] 
 \longrightarrow {\mathcal F}_1^{\bullet}
 \longrightarrow [S(E|_{{\mathcal D}_s},E|_{{\mathcal D}_s}^{\vee})\rightarrow Z^1]
 \longrightarrow 0$
and denote the short exact sequence of complexes
\[
 \begin{CD}
  0@>>> Z^0 @>>> {\mathcal G}^0\oplus Z^0 @>>> {\mathcal G}^0 @>>> 0 \\
  & & @VVV @VVV @VVV \\
  0 @>>> S(E|_{{\mathcal D}_s}^{\vee},E|_{{\mathcal D}_s}) @>>>
  S(E|_{{\mathcal D}_s}^{\vee},E|_{{\mathcal D}_s}) @>>> 0 @>>> 0
 \end{CD}
\]
simply by
$0\longrightarrow [Z^0\rightarrow  S(E|_{{\mathcal D}_s}^{\vee},E|_{{\mathcal D}_s})]
 \longrightarrow
 {\mathcal F}_0^{\bullet} \longrightarrow {\mathcal G}^0 \longrightarrow 0$.
These short exact sequences of complexes induce the exact commutative diagram
\[
 \begin{CD}
  0 \longrightarrow \; & H^0(\ker({\mathcal G}^1\rightarrow G^1)) & \; \longrightarrow \; &
  \mathbf{H}^0({\mathcal F}_1^{\bullet})
  & \; \longrightarrow \; & \ker(S(E|_{{\mathcal D}_s},E|_{{\mathcal D}_s}^{\vee})\rightarrow Z^1)
  & \; \longrightarrow \; & H^1(\ker({\mathcal G}^1\rightarrow G^1)) \\
  & @V\eta_1 VV @V\sigma_2 VV  @V\eta_2 VV @V\eta_3 VV \\
  0 \longrightarrow \; &   H^1({\mathcal G}^0)^{\vee} & \; \longrightarrow & \;
  \mathbf{H}^1({\mathcal F}_0^{\bullet})^{\vee}  & \; \longrightarrow \; &
  \coker(Z^0\rightarrow S(E|_{{\mathcal D}_s}^{\vee},E|_{{\mathcal D}_s}))^{\vee}
  & \; \longrightarrow \; & H^0({\mathcal G}^0)^{\vee}.
 \end{CD}
\]
Here $\eta_1$ and $\eta_3$ are induced by
the trace pairing
\[
 {\mathcal G}^0\otimes \ker({\mathcal G}^1\rightarrow G^1)
 \ni u\otimes v \mapsto \Tr(u\otimes v) \in \Omega^1_{{\mathcal C}_s}
\]
and the isomorphism
$H^1(\Omega^1_{{\mathcal C}_s})\xrightarrow{\sim}
\mathbf{H}^2(\tilde{L}^{\bullet}_s)\xrightarrow{\sim}\mathbb{C}$.
Since the above trace pairing induces the isomorphism
$\ker({\mathcal G}^1\rightarrow G^1)\xrightarrow{\sim}
({\mathcal G}^0)^{\vee}\otimes\Omega^1_{{\mathcal C}_s}$,
$\eta_1$, $\eta_3$ are the  isomorphisms induced by this isomorphism
and the Serre duality.
The homomorphism $\eta_2$ is induced by the  pairing
\begin{align}
 \label {equation: residue trace pairing}
 \ker\left(S(E|_{{\mathcal D}_s},E|_{{\mathcal D}_s}^{\vee})\rightarrow Z^1\right)\times
 \coker \left(Z^0\rightarrow S(E|_{{\mathcal D}_s}^{\vee},E|_{{\mathcal D}_s})\right)
 & \longrightarrow
 \mathbf{H}^2({\mathcal L}^{\bullet}_s)
 \cong \mathbb{C} \\
  ( (\xi^{(i)}) , (\tau^{(i)}) )  & \mapsto
  \left[\left\{ \big(\Xi^{(0,\xi^{(i)})}_{(\tau^{(i)},0)}\big) \right\}\right].
 \notag
\end{align}
Note that
$\Big[\big(\Xi^{(0,\xi^{(i)})}_{(\tau^{(i)},0)}\big) \Big]
\in \mathbf{H}^2({\mathcal L}^{\bullet}_s)$
corresponds to
\begin{gather}
 \frac{1}{2} \sum_{i=1}^n \res_{p\in{\mathcal D}^{(i)}_s} 
 \bigg(\sum_{j=1}^{r-1} \sum_{l=0}^{j-1} c^{(i)}_j
 \Tr \Big(   
 \tau^{(i)}\circ(\,^tN^{(i)})^l\circ\xi^{(i)}\circ(N^{(i)})^{j-l}\Big)
 \frac{d\bar{z}^{(i)}}{\bar{z}^{(i)}_1\bar{z}^{(i)}_2\cdots \bar{z}^{(i)}_{m_i}} \bigg) \notag
\end{gather}
via the isomorphism
$\mathbf{H}^2({\mathcal L}^{\bullet}_s)
\xrightarrow{\sim} \mathbb{C}$.
Let us consider the restriction to each point $p\in {\mathcal D}_s$ of the pairing
\begin{gather}
 \ker\left(S(E|_{{\mathcal D}_s},E|_{{\mathcal D}_s}^{\vee})\rightarrow Z^1\right)\times
 \coker \left(Z^0\rightarrow S(E|_{{\mathcal D}_s}^{\vee},E|_{{\mathcal D}_s})\right)
  \longrightarrow
  {\mathcal O}_{{\mathcal D}_s}
 \label {equation:nondegenerate trace pairing} \\
  ( (\xi^{(i)}) , (\tau^{(i)}) )   \mapsto
  \frac{1}{2} \sum_{i=1}^n  \sum_{j=1}^{r-1} \sum_{l=0}^{j-1} c^{(i)}_j
 \Tr\Big(   
 \tau^{(i)}\circ(\,^tN^{(i)})^l\circ\xi^{(i)}\circ(N^{(i)})^{j-1-l}\Big). \notag
\end{gather}
Assume that
$(\xi^{(i)})\in \ker\left(S(E|_{{\mathcal D}_s},E|_{{\mathcal D}_s}^{\vee})\rightarrow Z^1\right)_p$
satisfies
\[\sum_{i=1}^n  \sum_{j=1}^{r-1} \sum_{l=0}^{j-1} c^{(i)}_j
 \Tr\Big(   
 \tau^{(i)}\circ(\,^tN^{(i)})^l\circ\xi^{(i)}\circ(N^{(i)})^{j-1-l}\Big)=0
\]
for any
$(\tau^{(i)})\in \coker \left(Z^0\rightarrow S(E|_{{\mathcal D}_s}^{\vee},E|_{{\mathcal D}_s})\right)_p$.
Since the usual trace pairing is nondegenerate, we have
$\sum_{j=1}^{r-1}\sum_{l=0}^{j-1} c^{(i)}_j \, (\,^tN^{(i)})^l\circ\xi^{(i)}\circ (N^{(i)})^{j-1-l}=0$.
Recall that $\Theta^{(i)}_{(\theta^{(i)}\circ\xi^{(i)})}=0$
by the choice of $(\xi^{(i)})$, which is equivalent to the existence of
some $g\in \End(E|_p)$ satisfying
$\theta^{(i)}\circ \xi^{(i)}=N^{(i)}\circ g-g\circ N^{(i)}$.
So we have
$\sum_{j=1}^{r-1} c^{(i)}_j \: (\theta^{(i)})^{-1}\circ ((N^{(i)})^j\circ g-g\circ (N^{(i)})^j)=0$,
which means $\nu^{(i)}(N^{(i)})\circ g=g\circ\nu^{(i)}(N^{(i)})$.
Since $\nu^{(i)}$ satisfies Assumption \ref {assumption-generic},
we have $N^{(i)}\circ g=g\circ N^{(i)}$ and $\xi^{(i)}=0$.
Thus the pairing (\ref  {equation:nondegenerate trace pairing})
is nondegenerate because
$\rank_{{\mathcal O}_{\mathcal D}}
\ker\left(S(E|_{{\mathcal D}_s},E|_{{\mathcal D}_s}^{\vee})\rightarrow Z^1\right)
=\dfrac{r(r-1)}{2}
=\rank_{{\mathcal O}_{\mathcal D}} 
\coker \left(Z^0\rightarrow S(E|_{{\mathcal D}_s}^{\vee},E|_{{\mathcal D}_s})\right)$.
So the pairing (\ref   {equation: residue trace pairing}) becomes a nondegenerate pairing
of vector spaces over $\mathbb{C}$ and $\eta_2$ becomes isomorphic.
Thus the homomorphism
$\sigma_2\colon \mathbf{H}^0({\mathcal F}_1^{\bullet})
\xrightarrow{\sim} \mathbf{H}^1({\mathcal F}_0^{\bullet})$
becomes an isomorphism by the five lemma.
The homomorphism 
$\sigma_3\colon \mathbf{H}^1({\mathcal F}^{\bullet}_0)
\xrightarrow{\sim} \mathbf{H}^0({\mathcal F}^{\bullet}_1)$
is isomorphic because it is the dual of $\sigma_2$.

On the other hand, we have the exact commutative diagram
\[
 \begin{CD}
  \ker (Z^0\rightarrow S(E|_{{\mathcal D}_s}^{\vee},E|_{{\mathcal D}_s}))
  & \longrightarrow & \; \mathbf{H}^0({\mathcal F}_0^{\bullet})
  & \longrightarrow & \; H^0({\mathcal G}^0) & \longrightarrow & \; 
  \coker (Z^0\rightarrow S(E|_{{\mathcal D}_s}^{\vee},E|_{{\mathcal D}_s})) \\
  @VVV  @V\sigma_1 VV  @V\eta_4 VV  @V ^t\eta_2 V\cong V \\
  \coker (S(E|_{{\mathcal D}_s},E|_{{\mathcal D}_s}^{\vee})\rightarrow Z^1)^{\vee}
  & \longrightarrow & \; \mathbf{H}^1({\mathcal F}_1^{\bullet})^{\vee}
  & \longrightarrow & \; H^1(\ker({\mathcal G}^1\rightarrow G^1))^{\vee} 
  & \longrightarrow & \;
  \ker (S(E|_{{\mathcal D}_s},E|_{{\mathcal D}_s}^{\vee})\rightarrow Z^1) ^{\vee}.
 \end{CD}
\]
Note that
$\ker (Z^0\rightarrow S(E|_{{\mathcal D}_s}^{\vee},E|_{{\mathcal D}_s}))=0$
and
$\coker (S(E|_{{\mathcal D}_s},E|_{{\mathcal D}_s}^{\vee})\rightarrow Z^1)=0$.
The homomorphism $\eta_4$ is isomorphic
since it is induced by the isomorphism
$\ker({\mathcal G}^1\rightarrow G^1)^{\vee}\otimes\Omega^1_{{\mathcal C}_s}\cong{\mathcal G}^0$
and the Serre duality.
Thus the homomorphism $\sigma_1$ is an isomorphism.
The homomorphism
$\sigma_4\colon
\mathbf{H}^1({\mathcal F}_1^{\bullet})\longrightarrow\mathbf{H}^0({\mathcal F}_0^{\bullet})^{\vee}$
is isomorphic, because it is the dual of $\sigma_1$.

From all the above arguments, the homomorphism
$\sigma\colon\mathbf{H}^1({\mathcal F}^{\bullet})
\longrightarrow\mathbf{H}^1({\mathcal F}^{\bullet})^{\vee}$
is isomorphic by the five lemma,
because $\sigma_1,\sigma_2,\sigma_3,\sigma_4$ are all isomorphic.
\end{proof}

\begin{lemma}\label{lemma:trace-isomorphism}
 $\mathbf{H}^2(\Tr)\colon
 \mathbf{H}^2({\mathcal F}^{\bullet})\longrightarrow
 \mathbf{H}^2(\tilde{L}^{\bullet}_s)\cong\mathbb{C}$
 is an isomorphism.
\end{lemma}

\begin{proof}
From the proof of Lemma \ref{lemma-non-degenerate},
the exact commutative diagram
\[
 \begin{CD}
  \mathbf{H}^1({\mathcal F}^{\bullet}_0) @>>> \mathbf{H}^1({\mathcal F}^{\bullet}_1)
  @>>> \mathbf{H}^2({\mathcal F}^{\bullet}) @>>> 0 \\
  @V\sigma_3 VV @V\sigma_4 VV @V\sigma_5 VV \\
  \mathbf{H}^0({\mathcal F}^{\bullet}_1)^{\vee} @>>> \mathbf{H}^0({\mathcal F}^{\bullet}_0)^{\vee}
  @>>> \mathbf{H}^0({\mathcal F}^{\bullet})^{\vee} @>>> 0
 \end{CD}
\]
is induced
and
$\sigma_5\colon \mathbf{H}^2({\mathcal F}^{\bullet})
\xrightarrow{\sim} \mathbf{H}^0({\mathcal F}^{\bullet})^{\vee}$
is an isomorphism because $\sigma_3$ and $\sigma_4$ are isomorphic.
Note that $\mathbf{H}^0({\mathcal F}^{\bullet})=\mathbb{C}$
because $(E,\nabla,\{N^{(i)}\})$ is $\balpha$-stable
whose endomorphisms are only scalar multiplications.
We can see from the construction that the composition
\[
 \mathbf{H}^2({\mathcal F}^{\bullet})\xrightarrow[\sim]{\sigma_5}
 \mathbf{H}^0({\mathcal F}^{\bullet})^{\vee}
 \xrightarrow{\sim} \mathbf{H}^0(\tilde{\mathcal L}^{\bullet}_s)^{\vee}
 \xrightarrow{\sim} \mathbf{H}^2(\tilde{\mathcal L}^{\bullet}_s)
\]
coincides with $\mathbf{H}^2(\Tr)$
and the result follows.
\end{proof}

\begin{corollary}\label {cor:dimension of tangent space}
The dimension of the relative tangent space
of $M^{\balpha}_{{\mathcal C},{\mathcal D}}(\tilde{\bnu},\tilde{\bmu})$ over $S$
at $(E,\nabla,\{N^{(i)}\})$
is given by
\[
 \dim \mathbf{H}^1({\mathcal F}^{\bullet})=
 2r^2(g-1)+2+r(r-1)\sum_{i=1}^n m_i.
\]
\end{corollary}

\begin{proof}
Since we will prove the smoothness of the moduli space
$M^{\balpha}_{{\mathcal C},{\mathcal D}}(\tilde{\bnu},\tilde{\bmu})$ over $S$
in Proposition \ref {prop:smoothness of moduli},
we can deduce the corollary from \cite[Theorem 2.1]{Inaba-1}
and \cite[Theorem 2.2]{Inaba-Saito}.
We give here a direct proof using the proof of Lemma \ref {lemma-non-degenerate}.
Since $\mathbf{H}^0({\mathcal F}^{\bullet})\cong\mathbb{C}$
and $\mathbf{H}^2({\mathcal F}^{\bullet})\cong \mathbb{C}$,
the exact sequence (\ref {equation:fundamental exact sequence of cohomologies})
becomes
\[
 0\longrightarrow \mathbb{C}\longrightarrow
 \mathbf{H}^0({\mathcal F}_0^{\bullet})
 \longrightarrow \mathbf{H}^0({\mathcal F}_1^{\bullet})
 \longrightarrow \mathbf{H}^1({\mathcal F}^{\bullet})
 \longrightarrow \mathbf{H}^1({\mathcal F}_0^{\bullet})
 \longrightarrow \mathbf{H}^1({\mathcal F}_1^{\bullet})
 \longrightarrow \mathbb{C} \longrightarrow 0.
\]
Since $\mathbf{H}^0({\mathcal F}_1^{\bullet})\cong \mathbf{H}^1({\mathcal F}_0^{\bullet})^{\vee}$
and $\mathbf{H}^1({\mathcal F}_1^{\bullet})\cong  \mathbf{H}^0({\mathcal F}_0^{\bullet})^{\vee}$
by the proof of Lemma \ref {lemma-non-degenerate},  we have
\begin{align}
 \dim \mathbf{H}^1({\mathcal F}^{\bullet})&=
 \dim  \mathbf{H}^0({\mathcal F}_1^{\bullet}) + \dim \mathbf{H}^1({\mathcal F}_0^{\bullet})
 - \dim \mathbf{H}^0({\mathcal F}_0^{\bullet}) - \dim \mathbf{H}^1({\mathcal F}_1^{\bullet})
 +\dim\mathbb{C}+\dim\mathbb{C} 
  \label {equation:calculation of dimension} \\
 &= 2\dim \mathbf{H}^1({\mathcal F}_0^{\bullet})-2\mathbf{H}^0({\mathcal F}_0^{\bullet})+2 \notag \\
 &=-2\chi({\mathcal F}_0^{\bullet})+2 \notag 
\end{align}
Using the Riemann-Roch formula, we can see
\begin{align*}
 \chi({\mathcal F}_0^{\bullet})
 &=\chi({\mathcal G}^0)+\length Z^0-\length  S(E|_{{\mathcal D}_s}^{\vee},E|_{{\mathcal D}_s})) \\
 &=r^2(1-g)+\sum_{i=1}^n rm_i -\sum_{i=1}^n \frac{r(r+1)}{2}m_i.
\end{align*}
Substituting this in (\ref {equation:calculation of dimension})
we get the corollary.
\end{proof}

\subsection{Smoothness of the moduli space of $(\tilde{\bnu},\tilde{\bmu})$-connections}

We use the same notations as in subsection \ref{subsection:tangent}
and subsection \ref{subsection:nondenerate pairing}.

\begin{proposition}\label {prop-obstruction-class}
 Let $A$ be an artinian local ring over $S$
 with the maximal ideal $\mathfrak{m}$
 and $I$ be an ideal of $A$ satisfying $\mathfrak{m}I=0$
 and $A/\mathfrak{m}=\mathbb{C}$.
 Let $(E',\nabla',\{N'^{(i)}_j\})$ be a flat family of
 $(\tilde{\bnu},\tilde{\bmu})\otimes A/I$-connections on $({\mathcal C}_{A/I},{\mathcal D}_{A/I})$
 over $A/I$
 such that
 $(E',\nabla',\{N'^{(i)}\})\otimes A/\mathfrak{m}
 \cong (E,\nabla,\{N^{(i)}\})$.
 Then there is an obstruction class
 $o(E',\nabla',\{N'^{(i)}\})\in \mathbf{H}^2({\mathcal F}^{\bullet})\otimes I$
 whose vanishing is equivalent to the existence of a lift of
 $(E',\nabla',\{N'^{(i)}_j\})$ to a flat family of
 $(\tilde{\bnu},\tilde{\bmu})\otimes A$-connections on $({\mathcal C}_{A},{\mathcal D}_{A})$
 over $A$.
\end{proposition}

\begin{proof}
We can define the ${\mathcal O}_{{\mathcal D}^{(i)}_{A/I}}[T]$-module structures
on $E'|_{{\mathcal D}^{(i)}_{A/I}}$ and on $E'^{\vee}|_{{\mathcal D}^{(i)}_{A/I}}$
by $N'^{(i)}$ and $^tN'^{(i)}$, respectively.
Then we can take an ${\mathcal O}_{{\mathcal D}^{(i)}_{A/I}}[T]$-isomorphism
$\theta'^{(i)}\colon E'^{\vee}|_{{\mathcal D}^{(i)}_{A/I}}
 \xrightarrow{\sim} E'|_{{\mathcal D}^{(i)}_{A/I}}$ 
which is a lift of $\theta^{(i)}$.
If we put
$\kappa'^{(i)}:=(\theta'^{(i)})^{-1}\circ N'^{(i)}\colon
E'|_{{\mathcal D}^{(i)}_{A/I}} \longrightarrow E'^{\vee}|_{{\mathcal D}^{(i)}_{A/I}}$,
then 
$(E',\nabla',\{\theta'^{(i)},\kappa'^{(i)}\})$
is a flat family of factorized $(\tilde{\bnu},\tilde{\bmu})\otimes A/I$-connections on
$({\mathcal C}_{A/I},{\mathcal D}_{A/I})$ over $A/I$.

We can take an affine open covering
${\mathcal C}_A=\bigcup_{\alpha}U_{\alpha}$ such that
$\sharp\{i \, | \, {\mathcal D}^{(i)}_A\cap U_{\alpha}\neq\emptyset \}\leq 1$
for any $\alpha$
and
$\sharp\{\alpha \, | \, {\mathcal D}^{(i)}_A\subset U_{\alpha}\}=1$
for any $i$.
Furthermore, we may assume that
$E'|_{U_{\alpha}\otimes A/I}\cong{\mathcal O}_{U_{\alpha}\otimes A/I}^{\oplus r}$.
Take a free ${\mathcal O}_{U_{\alpha}}$-module $E_{\alpha}$
with an isomorphism
$\psi_{\alpha}\colon E_{\alpha}\otimes A/I\xrightarrow{\sim} E'|_{U_{\alpha}\otimes A/I}$
and a lift
$\sigma_{\beta\alpha}\colon E_{\alpha}|_{U_{\alpha\beta}}
\xrightarrow{\sim} E_{\beta}|_{U_{\alpha\beta}}$
of the composite
$\psi_{\beta}^{-1}\circ\psi_{\alpha}\colon
E_{\alpha}|_{U_{\alpha\beta}}\otimes A/I\xrightarrow[\sim]{\psi_{\alpha}}
E'|_{U_{\alpha\beta}\otimes A/I}
\xrightarrow[\sim]{\psi_{\beta}^{-1}}
E_{\beta}|_{U_{\alpha\beta}}\otimes A/I$.

If we write
$\varphi_{\tilde{\bmu}\otimes A}^{(i)}(T)=T^r+b_{r-1}T^{r-1}+\cdots+b_1T+b_0$
with $b_i\in{\mathcal O}_{{\mathcal D}^{(i)}_A}$
and define matrices $N,\Phi_1,\Phi_2$ by
\begin{gather*}
 N= 
 \begin{pmatrix}
  -b_{r-1} & 1 & 0 & \cdots & 0 \\
  -b_{r-2} & 0 & 1 & \cdots & 0 \\
  \vdots & \vdots & \ddots & \ddots & \vdots \\
  -b_1 & 0 & \cdots & 0 & 1 \\
  -b_0 & 0 & \cdots & \cdots & 0 
 \end{pmatrix},
 \quad
 \Phi_1=
 \begin{pmatrix}
  0 & 0 &  \cdots & 0  & 1 \\
  0 & 0 &  \cdots & 1 & b_{r-1} \\
  \vdots & \vdots  & \iddots & \iddots  & \vdots \\
  0 & 1 & b_{r-1} & \cdots  & b_2 \\
  1 & b_{r-1}  &  b_{r-2} & \cdots  &  b_1 
 \end{pmatrix} , \\
 \Phi_2=
 \begin{pmatrix}
  0 & 0 & \cdots & 0  & 1 & 0 \\
  0 & 0 & \cdots &  1 & b_{r-1} & 0 \\
  \vdots \vspace{2pt}& \vdots & \iddots  & \iddots & \vdots & \vdots  \\
  0 & 1 & b_{r-1} & \cdots & b_3 & 0 \\
  1 & b_{r-1} & b_{r-2} & \cdots &  b_2 & 0 \\
  0 & 0 & 0 &  \cdots & 0 & -b_0
 \end{pmatrix},
\end{gather*}
then $^t\Phi_1=\Phi_1$, $^t\Phi_2=\Phi_2$ and $\Phi_1$ is invertible.
We can check  $N\Phi_1=\Phi_2$,
which is equivalent to $N=\Phi_2 \Phi_1^{-1}$.
So there is a matrix factorization
\[
 \,^tN= \Phi_1^{-1} \Phi_2
 \colon {\mathcal O}_{{\mathcal D}^{(i)}_A}^{\oplus r}
 \xrightarrow{\Phi_2} 
 \Big({\mathcal O}_{{\mathcal D}^{(i)}_A}^{\oplus r}\Big)^{\vee}
 \xrightarrow{\Phi_1^{-1}} {\mathcal O}_{{\mathcal D}^{(i)}_A}^{\oplus r}.
\]
After  replacing the representative
$((\theta'^{(i)}),(\kappa'^{(i)}))$
by the action of an element of
$\left({\mathcal O}_{{\mathcal D}^{(i)}_{A/I}}[T]
/(\varphi_{\tilde{\bmu}\otimes A/I}^{(i)}(T))\right)^{\times}$,
we may assume that there is  an isomorphism
$g\colon {\mathcal O}_{{\mathcal D}^{(i)}_{A/I}}^{\oplus r}
\xrightarrow{\sim} E'|_{{\mathcal D}^{(i)}_{A/I}}$
satisfying
$\theta'^{(i)}=g\circ(\Phi_1^{-1}\otimes A/I)\circ \,^tg$
and
$\kappa'^{(i)}= \,^t g^{-1}\circ(\Phi_2\otimes A/I)\circ g^{-1}$.
We take a lift
$\tilde{g}\colon
{\mathcal O}_{{\mathcal D}^{(i)}_A}^{\oplus r}
\xrightarrow{\sim}E_{\alpha}|_{{\mathcal D}^{(i)}_A}$
of $g$, that is, $\psi_{\alpha}\circ(\tilde{g}\otimes A/I)=g$.
If we put
$\theta_{\alpha}^{(i)}:=\tilde{g}\circ \Phi_1^{-1}\circ\,^t\tilde{g}$
and
$\kappa^{(i)}_{\alpha}:=(\,^t\tilde{g})^{-1}\circ \Phi_2\circ\tilde{g}^{-1}$,
then $(\theta^{(i)}_{\alpha},\kappa^{(i)}_{\alpha})$
becomes a lift of
$(\theta'^{(i)},\kappa'^{(i)})$ and
$N^{(i)}_{\alpha}:=\theta^{(i)}_{\alpha}\circ\kappa^{(i)}_{\alpha}\colon
E_{\alpha}|_{{\mathcal D}^{(i)}_A}\longrightarrow
E_{\alpha}|_{{\mathcal D}^{(i)}_A}$
becomes a lift of $N'^{(i)}$.
We can take an $A$-relative local connection
$\nabla_{\alpha}\colon E_{\alpha}\longrightarrow
E_{\alpha}\otimes\Omega^1_{{\mathcal C}_A/A}({\mathcal D}_A)$
satisfying
$\displaystyle
\nu^{(i)}(N^{(i)}_{\alpha})\frac{d\overline{z}^{(i)}}{\bar{z}^{(i)}_1\bar{z}^{(i)}_2\cdots\bar{z}^{(i)}_{m_i}}
=\nabla_{\alpha}|_{{\mathcal D}^{(i)}_A}$
and
$\nabla_{\alpha}\otimes A/I
=\psi_{\alpha}^{-1}\circ\nabla'|_{U_{\alpha}\otimes A/I}\circ\psi_{\alpha}$.

If we put
\[
 u_{\alpha\beta\gamma}=
 \psi_{\alpha}\circ
 (\sigma_{\gamma\alpha}^{-1}\circ\sigma_{\gamma\beta}\circ\sigma_{\beta\alpha}
 -\mathrm{id}_{E_{\alpha}})\circ\psi_{\alpha}^{-1},
 \quad
 v_{\alpha\beta}=\psi_{\alpha}\circ(
 \sigma_{\beta\alpha}^{-1}\circ\nabla_{\beta}\circ\sigma_{\beta\alpha}
 -\nabla_{\alpha})
 \circ\psi_{\alpha}^{-1},
\]
then we have
\[
 v_{\beta\gamma}-v_{\alpha\gamma}+v_{\alpha\beta}=
 \nabla'\circ u_{\alpha\beta\gamma}-u_{\alpha\beta\gamma}\circ\nabla',
 \quad
 u_{\beta\gamma\delta}-u_{\alpha\gamma\delta}+u_{\alpha\beta\delta}-u_{\alpha\beta\gamma}=0
\]
and we can define an element
\[
 o(E',\nabla',\{N'^{(i)}\}):=
 [(\{(u_{\alpha\beta\gamma},0)\},\{(v_{\alpha\beta},(0,0))\},\{(0,0)\})]
 \in \mathbf{H}^2({\mathcal F}^{\bullet})\otimes I.
\]

Assume that $o(E',\nabla',\{N'^{(i)}\})=0$.
Then there are
\begin{align*}
 \{a_{\alpha\beta}\} 
 &\in
 I\otimes C^1(\{U_{\alpha}\},{\mathcal G}^0), \\
 \{b_{\alpha},(\tau^{(i)}_{\alpha}),(\xi^{(i)}_{\alpha})\}
 &\in
 I\otimes C^0(\{U_{\alpha}\},{\mathcal G}^1\oplus
 S(E^{\vee}|_{{\mathcal D}_s},E|_{{\mathcal D}_s})\oplus
 S(E|_{{\mathcal D}_s},E|_{{\mathcal D}_s}^{\vee}))
\end{align*}
satisfying
\begin{align*}
 & u_{\alpha\beta\gamma}
 =
 a_{\beta\gamma}-a_{\alpha\gamma}+a_{\alpha\beta},
 \hspace{30pt} 
 v_{\alpha\beta}
 =
 \nabla a_{\alpha\beta}-a_{\alpha\beta}\nabla
 -(b_{\beta}-b_{\alpha}), \\
 & b_{\alpha}|_{{\mathcal D}^{(i)}_s}
 =
 \delta^{(i)}_{\bnu,N^{(i)}}
 (\tau^{(i)}_{\alpha}\circ\kappa^{(i)}+\theta^{(i)}\circ\xi^{(i)}_{\alpha}),
 \hspace{20pt}
 \Theta^{(i)}_{\tau^{(i)}_{\alpha}\circ\kappa^{(i)}+\theta^{(i)}\circ\xi^{(i)}_{\alpha}}
 =0.
\end{align*}
If we put
$\tilde{\theta}^{(i)}_{\alpha}:=
\theta^{(i)}_{\alpha}+\psi_{\alpha}^{-1}\circ\tau^{(i)}_{\alpha}\circ\psi_{\alpha}$,
$\tilde{\kappa}^{(i)}_{\alpha}:=
\kappa^{(i)}_{\alpha}+\psi_{\alpha}^{-1}\circ\xi^{(i)}_{\alpha}\circ\psi_{\alpha}$,
then the composition
$\tilde{N}^{(i)}_{\alpha}:=\tilde{\theta}^{(i)}_{\alpha}\circ\tilde{\kappa}^{(i)}_{\alpha}
=N^{(i)}_{\alpha}+\psi_{\alpha}^{-1}\circ
(\tau^{(i)}_{\alpha}\circ\kappa^{(i)}+\theta^{(i)}\circ\xi^{(i)}_{\alpha})\circ\psi_{\alpha}$
satisfies
$\varphi^{(i)}_{\tilde{\bmu}}(\tilde{N}^{(i)}_{\alpha})=0$,
because there is
$g^{(i)}_{\alpha}\in \End(E|_{{\mathcal D}^{(i)}_s})\otimes I$
satisfying
$N^{(i)}\circ g^{(i)}_{\alpha}-g^{(i)}_{\alpha}\circ N^{(i)}
=\tau^{(i)}_{\alpha}\circ\kappa^{(i)}+\theta^{(i)}\circ\xi^{(i)}_{\alpha}$
from the condition
$\Theta^{(i)}_{\tau^{(i)}_{\alpha}\circ\kappa^{(i)}+\theta^{(i)}\circ\xi^{(i)}_{\alpha}} =0$.
We define a connection $\tilde{\nabla}_{\alpha}$ on $E_{\alpha}$ by
$\tilde{\nabla}_{\alpha}:=\nabla_{\alpha}+\psi_{\alpha}^{-1}\circ b_{\alpha}\circ \psi_{\alpha}$.
Then we have
\begin{align*}
 \tilde{\nabla}_{\alpha}|_{{\mathcal D}^{(i)}_A}
 &=
 \nabla_{\alpha}|_{{\mathcal D}^{(i)}_A}
 + ( \psi_{\alpha}^{-1}\circ b_ {\alpha} \circ \psi_{\alpha} ) |_{{\mathcal D}^{(i)}_A}
 =
 \tilde{\nu}^{(i)}(N^{(i)}_{\alpha}) \frac {d\bar{z}^{(i)}} {\bar{z}^{(i)}_1\cdots\bar{z}^{(i)}_{m_i}}
 +\delta^{(i)}_{\bnu,N^{(i)}} (\tilde{N}^{(i)}_{\alpha}-N^{(i)}_{\alpha}) \\
 &=
 \left( \tilde{\nu}^{(i)}(N^{(i)}_{\alpha}) 
 +\sum_{j=1}^{r-1}\sum_{l=1}^{j} c^{(i)}_j
 (\tilde{N}^{(i)}_{\alpha})^{j-l}(\tilde{N}^{(i)}_{\alpha}-N^{(i)}_{\alpha})
 (N^{(i)}_{\alpha})^{l-1}
 \right) \frac {d\bar{z}^{(i)}} {\bar{z}^{(i)}_1\cdots\bar{z}^{(i)}_{m_i}} \\
 &=
 \tilde{\nu}^{(i)}(N^{(i)}_{\alpha}) +
 \sum_{j=0}^{r-1} c^{(i)}_j \: (\tilde{N}^{(i)}_{\alpha})^j
 \frac {d\bar{z}^{(i)}} {\bar{z}^{(i)}_1\cdots\bar{z}^{(i)}_{m_i}}
 -\sum_{j=0}^{r-1} c^{(i)}_j \: (N^{(i)}_{\alpha})^j
 \frac {d\bar{z}^{(i)}} {\bar{z}^{(i)}_1\cdots\bar{z}^{(i)}_{m_i}} \\
 &=\tilde{\nu}^{(i)}(\tilde{N}^{(i)}_{\alpha})\frac {d\bar{z}^{(i)}} {\bar{z}^{(i)}_1\cdots\bar{z}^{(i)}_{m_i}}.
\end{align*}
If we put
$\tilde{\sigma}_{\beta\alpha}:=
\sigma_{\beta\alpha}\circ
(\mathrm{id}-\psi_{\alpha}^{-1}\circ a_{\alpha\beta}\circ\psi_{\alpha})$,
then
\begin{align*}
 &(\tilde{\sigma}_{\gamma\alpha})^{-1}\circ\tilde{\sigma}_{\gamma\beta}
 \circ\tilde{\sigma}_{\beta\alpha} \\
 &=
 (\mathrm{id}+\psi_{\alpha}^{-1}\circ a_{\alpha\gamma}\circ\psi_{\alpha})\circ
 \sigma_{\gamma\alpha}^{-1}
 \circ\sigma_{\gamma\beta}\circ
 (\mathrm{id}-\psi_{\beta}^{-1}\circ a_{\beta\gamma}\circ\psi_{\beta})
 \circ\sigma_{\beta\alpha}\circ
 (\mathrm{id}-\psi_{\alpha}^{-1}\circ a_{\alpha\beta}\circ\psi_{\alpha}) \\
 &=
 (\mathrm{id}+\psi_{\alpha}^{-1}\circ a_{\alpha\gamma}\circ\psi_{\alpha})
 \circ\sigma_{\gamma\alpha}^{-1}
 \circ\sigma_{\gamma\beta}\circ\sigma_{\beta\alpha}\circ
 (\mathrm{id}-\psi_{\alpha}^{-1}\circ a_{\beta\gamma}\circ\psi_{\alpha})
 \circ(\mathrm{id}+\psi_{\alpha}^{-1}\circ a_{\alpha\beta}\circ\psi_{\alpha}) \\
 &=
 \sigma_{\gamma\alpha}^{-1} \circ\sigma_{\gamma\beta}\circ\sigma_{\beta\alpha}
 \circ(\mathrm{id}+\psi_{\alpha}^{-1}\circ a_{\alpha\gamma}\circ\psi_{\alpha})
 \circ (\mathrm{id}-\psi_{\alpha}^{-1}\circ a_{\beta\gamma}\circ \psi_{\alpha})
 \circ (\mathrm{id}-\psi_{\alpha}^{-1}\circ a_{\alpha\beta}\circ\psi_{\alpha})  \\
 &=
 (\mathrm{id}+\psi_{\alpha}^{-1}\circ u_{\alpha\beta\gamma}\circ\psi_{\alpha})\circ
 (\mathrm{id}+\psi_{\alpha}^{-1}\circ
 (a_{\alpha\gamma}-a_{\beta\gamma}-a_{\alpha\beta})\circ\psi_{\alpha}) \\
 &=\mathrm{id}+\psi_{\alpha}^{-1}\circ ( u_{\alpha\beta\gamma}
 -(a_{\beta\gamma}-a_{\alpha\gamma}+a_{\alpha\beta})) \circ\psi_{\alpha}
 =\mathrm{id}
\end{align*}
because $\sigma_{\beta\alpha}\otimes A/I=\mathrm{id}$.
We also have
\begin{align*}
 \tilde{\sigma}_{\beta\alpha}^{-1}\circ\tilde{\nabla}_{\beta}\circ\tilde{\sigma}_{\beta\alpha}
 &=
 (\mathrm{id}+\psi_{\alpha}^{-1}\circ a_{\alpha\beta}\circ\psi_{\alpha})
 \circ\sigma_{\beta\alpha}^{-1}
 \circ(\nabla_{\beta}+\psi_{\beta}^{-1}\circ b_{\beta}\circ\psi_{\beta})
 \circ\sigma_{\beta\alpha}\circ
 (\mathrm{id}-\psi_{\alpha}^{-1}\circ a_{\alpha\beta}\circ\psi_{\alpha}) \\
 &=
 \sigma_{\beta\alpha}^{-1}\circ\nabla_{\beta}\circ\sigma_{\beta\alpha}
 -\psi_{\alpha}^{-1}\circ\nabla'\circ a_{\alpha\beta}\circ\psi_{\alpha}
 +\psi_{\alpha}^{-1}\circ a_{\alpha\beta}\circ\nabla' \circ\psi_{\alpha}
 +\psi_{\alpha}^{-1}\circ b_{\beta}\circ\psi_{\alpha} \\
 &=
 \nabla_{\alpha}+\psi_{\alpha}^{-1}\circ v_{\alpha\beta}\circ\psi_{\alpha}
 -\psi_{\alpha}^{-1}\circ
 (\nabla'\circ a_{\alpha\beta}-a_{\alpha\beta}\circ\nabla'-b_{\beta})\circ\psi_{\alpha}  \\
 &=
 \nabla_{\alpha}+\psi_{\alpha}^{-1}\circ b_{\alpha}\circ\psi_{\alpha}
 =\tilde{\nabla}_{\alpha}.
\end{align*}
Thus we can patch
$(E_{\alpha},\tilde{\nabla}_{\alpha}, \{ \tilde{\theta}^{(i)}_{\alpha}, \tilde{\kappa}^{(i)}_{\alpha} \})$
together via the gluing isomorphisms $\{\tilde{\sigma}_{\beta\alpha}\}$
and obtain a flat family 
$(\tilde{E},\tilde{\nabla},\{ \tilde{\theta}^{(i)},\tilde{\kappa}^{(i)}\})$
of factorized $(\tilde{\bnu},\tilde{\bmu})\otimes A$-connections over $A$
which is a lift of $(E',\nabla',\{\theta'^{(i)},\kappa'^{(i)}\})$.
Conversely, we can immediately see that
$o(E',\nabla',\{N'^{(i)}\})=0$ if there is a lift of
$(E',\nabla',\{\theta'^{(i)},\kappa'^{(i)}\})$ over $A$,
which corresponds to a lift of
$(E',\nabla',\{N'^{(i)}\})$ over $A$.
Thus the proposition is proved.
\end{proof}

\begin{lemma}\label{lemma-trace-obstruction}
The isomorphism $\mathbf{H}^2(\Tr)\colon \mathbf{H}^2({\mathcal F}^{\bullet}\otimes I)
\xrightarrow{\sim}\mathbf{H}^2(\tilde{\mathcal L}^{\bullet}_s\otimes I)
=\mathbf{H}^2({\mathcal L}^{\bullet}_s\otimes I)$
in Lemma \ref{lemma:trace-isomorphism}
sends the obstruction class $o(E',\nabla',\{N'^{(i)}\})$
defined in the proof of Proposition \ref{prop-obstruction-class}
to an element of $\mathbf{H}^2({\mathcal L}^{\bullet}_s\otimes I)$
whose vanishing is equivalent to the existence of an extension of
$(\det(E',\nabla'))$ to a pair $(L,\nabla_L)$
of a line bundle $L$ on $C\times\Spec A$
and a connection
$\nabla_L\colon L\longrightarrow
L\otimes\Omega^1_{{\mathcal C}_A/A}({\mathcal D}_A)$
satisfying
$(L,\nabla_L)\otimes A/I\cong\det(E',\nabla')$
and
$\nabla_L|_{{\mathcal D}^{(i)}_A}=
\sum_{k=1}^r \tilde{\nu}^{(i)}(\tilde{\mu}^{(i)}_k)_A$.
\end{lemma}

\begin{proof}
Take the same affine open covering $\{U_{\alpha}\}$ of ${\mathcal C}_A$
and the lifts $(E_{\alpha},\nabla_{\alpha})$ of $(E',\nabla')|_{U_{\alpha}\times \Spec A/I}$
as in the proof of Proposition \ref {prop-obstruction-class}.
Then $\det(E_{\alpha},\nabla_{\alpha})$ is a lift of
$\det(E',\nabla')|_{U_{\alpha}\times \Spec A/I}$
and the class
\begin{align*}
 o(\det(E',\nabla'))
 &:=
 \Big[ \big\{ \det(\psi_{\alpha})\circ(\det(\sigma_{\gamma\alpha}^{-1}
 \circ\sigma_{\gamma\beta}\circ\sigma_{\beta\alpha})-\mathrm{id}_{\det E_{\alpha}})
 \circ\det(\psi_{\alpha}^{-1}) \big\},  \\
 & \hspace{20pt}
 \big\{ \det(\psi_{\alpha})\circ(
 \det(\sigma_{\beta\alpha}^{-1})\circ\det(\nabla_{\beta})\circ
 \det(\sigma_{\beta\alpha})
 -
 \det(\nabla_{\alpha}) )
 \circ\det(\psi_{\alpha}^{-1})  \big\} \Big]
 \in \mathbf{H}^2({\mathcal L}^{\bullet}\otimes I)
\end{align*}
is nothing but the obstruction for the existence of a lift $(L,\nabla_L)$ of
$\det(E',\nabla')$ over $A$ satisfying
$\nabla_L|_{{\mathcal D}^{(i)}_A}=\sum_{k=1}^r \tilde{\nu}^{(i)}(\tilde{\mu}^{(i)}_k)_A$.
Here
$\det\nabla_{\alpha}\colon \det E_{\alpha}\longrightarrow 
\det E_{\alpha}\otimes\Omega^1_{{\mathcal C}_A/A}(D_A)$
is the $A$-relative connection on $\det(E_{\alpha})$
induced from $\nabla_{\alpha}$, which is defined by
\[
 (\det(\nabla_{\alpha}))(v_1\wedge v_2\wedge\cdots v_r)=
 \nabla_{\alpha}(v_1)\wedge v_2\wedge \cdots \wedge v_r
 +\cdots+ v_1\wedge\cdots\wedge v_{r-1}\wedge\nabla_{\alpha}(v_r)
\]
for $v_1,\ldots,v_r\in E_{\alpha}$.
For the notations $\{u_{\alpha\beta\gamma}\}, \{v_{\alpha\beta}\}$ in the proof of
Proposition \ref {prop-obstruction-class}, we have
\begin{align*}
 \Tr(u_{\alpha\beta\gamma}) &=
 \det(\psi_{\alpha})\circ(\det(\sigma_{\gamma\alpha}^{-1}
 \circ\sigma_{\gamma\beta}\circ\sigma_{\beta\alpha})-\mathrm{id}_{\det E_{\alpha}})
 \circ\det(\psi_{\alpha}^{-1}) \\
 \Tr(v_{\alpha\beta}) &=\det(\psi_{\alpha})\circ(
 \det(\sigma_{\beta\alpha}^{-1})\circ\det(\nabla_{\beta})\circ
 \det(\sigma_{\beta\alpha})
 -\det(\nabla_{\alpha}))
 \circ\det(\psi_{\alpha}^{-1}).
\end{align*}
So $o(\det(E',\nabla'))$ is nothing but
the image of the obstruction class
$o(E',\nabla',\{l'^{(i)}_j,N'^{(i)}_j\})\in \mathbf{H}^2({\mathcal F}^{\bullet}\otimes I)$
under the isomorphism
$\mathbf{H}^2(\Tr)\colon \mathbf{H}^2({\mathcal F}^{\bullet}\otimes I)
\xrightarrow{\sim}\mathbf{H}^2({\mathcal L}^{\bullet}_s\otimes I)$.
\end{proof}

\begin{proposition}\label {prop:smoothness of moduli}
The moduli space $M^{\balpha}_{{\mathcal C},{\mathcal D}}(\tilde{\bnu},\tilde{\bmu})$
is smooth over $S$.
\end{proposition}

\begin{proof}
Consider the $S$-relative moduli space
$M_{{\mathcal C},{\mathcal D}}(\Tr(\tilde{\bnu}),\Tr(\tilde{\bmu}))$
whose $S'$-valued points are the pairs $(L,\nabla_L)$ of a line bundle
$L$ on ${\mathcal C}_{S'}$ and a relative connection
$\nabla_L\colon L\longrightarrow L\otimes\Omega^1_{{\mathcal C}_{S'}/S'}({\mathcal D}_{S'})$
satisfying
$\nabla_L|_{{\mathcal D}^{(i)}_{S'}}=\sum_{k=1}^r \tilde{\nu}^{(i)}(\tilde{\mu}^{(i)}_k)_{S'}$.
Then $M_{{\mathcal C},{\mathcal D}}(\Tr(\tilde{\bnu}),\Tr(\tilde{\bmu}))$
is an affine space bundle over the Jacobian
variety of ${\mathcal C}$ over $S$ whose fiber is isomorphic to
$H^0(\Omega^1_{{\mathcal C}_s})$.
So we can prove by the same method as in the proof of
\cite[Theorem 2.1]{Inaba-1} that
$M_{{\mathcal C},{\mathcal D}}(\Tr(\tilde{\bnu}),\Tr(\tilde{\bmu}))$
is smooth over $S$ and the obstruction class $o(\det(E',\nabla'))$ should vanish.
Thus the obstruction class $o(E',\nabla',\{N'^{(i)}_j\})$
also vanishes by Lemma \ref{lemma-trace-obstruction}
and the moduli space $M^{\balpha}_{{\mathcal C},{\mathcal D}}(\tilde{\bnu},\tilde{\bmu})$
is smooth over $S$.
\end{proof}

\subsection{Relative symplectic form on the moduli space}
\label {subsection:symplectic form}

\begin{proposition} \label {prop:symplectic form}
There exists an $S$-relative symplectic form
$\omega\in H^0(M^{\balpha}_{{\mathcal C},{\mathcal D}}(\tilde{\bnu},\tilde{\bmu}),
\Omega^2_{M^{\balpha}_{{\mathcal C},{\mathcal D}}(\tilde{\bnu},\tilde{\bmu})/S})$
on the moduli space
$M^{\balpha}_{{\mathcal C},{\mathcal D}}(\tilde{\bnu},\tilde{\bmu})$.
\end{proposition}

\begin{proof}
For some quasi-finite \'etale covering
$\tilde{M}\longrightarrow M^{\balpha}_{{\mathcal C},{\mathcal D}}(\tilde{\bnu},\tilde{\bmu})$,
there is a universal flat family of $(\tilde{\bnu},\tilde{\bmu})$-connections
$(\tilde{E},\tilde{\nabla},\{\tilde{N}^{(i)}\})$
on ${\mathcal C}\times_S\tilde{M}$ over $\tilde{M}$.
Replacing $\tilde{M}$ by a refinement,
there is a corresponding flat family
$(\tilde{E},\tilde{\nabla},\{\tilde{\theta}^{(i)},\tilde{\kappa}^{(i)}\})$
of factorized $(\tilde{\bnu},\tilde{\bmu})$-connections
on ${\mathcal C}\times_S\tilde{M}$ over $\tilde{M}$.
We define homomorphisms
\begin{align*}
 \sigma^{(i) -}_{\theta^{(i)}} &\colon
 {\mathcal End}(E|_{{\mathcal D}^{(i)}_{\tilde{M}}})
 \oplus{\mathcal O}_{\tilde{M}}[T]/(\varphi^{(i)}_{\tilde{\bmu}}(T))
 \longrightarrow
 {\mathcal Hom}(E|_{{\mathcal D}^{(i)}_{\tilde{M}}}^{\vee},E|_{{\mathcal D}^{(i)}_{\tilde{M}}}) \\
 \sigma^{(i)+}_{\kappa^{(i)}} &\colon
 {\mathcal End}(E|_{{\mathcal D}^{(i)}_{\tilde{M}}})
 \oplus{\mathcal O}_{\tilde{M}}[T]/(\varphi^{(i)}_{\tilde{\bmu}}(T))
 \longrightarrow
 {\mathcal Hom}(E|_{{\mathcal D}^{(i)}_{\tilde{M}}},E|_{{\mathcal D}^{(i)}_{\tilde{M}}}^{\vee}) \\
 \delta^{(i)}_{\bnu,N^{(i)}} &\colon
 {\mathcal End}(E|_{{\mathcal D}^{(i)}_{\tilde{M}}})\longrightarrow
 {\mathcal End}(E|_{{\mathcal D}^{(i)}_{\tilde{M}}})\otimes
 \Omega^1_{{\mathcal C}_{\tilde{M}}/\tilde{M}}({\mathcal D}_{\tilde{M}})
\end{align*}
by the same formulas as in
subsection \ref{subsection:tangent},
(\ref{equation:definition of sigma-}), (\ref{equation:definition of sigma+})
and (\ref{equation:definition of delta}).
For each $u\in{\mathcal End}(E|_{{\mathcal D}^{(i)}_{\tilde{M}}})$,
we define a homomorphism
\[
 \Theta_{u}^{(i)}\colon
 {\mathcal O}_{{\mathcal D}^{(i)}_{\tilde{M}}}[T]/(\varphi^{(i)}_{\tilde{\bmu}}(T))
 \longrightarrow
 \Omega^1_{{\mathcal C}_{\tilde{M}}/\tilde{M}}({\mathcal D}_{\tilde{M}})|_{{\mathcal D}^{(i)}_{\tilde{M}}}
\]
by the same formula as subsection \ref {subsection:tangent},
 (\ref {equation:definition of Theta_u}).
We put
\begin{align*}
 &
 \tilde{\mathcal G}^0
 :=
 {\mathcal End} (\tilde{E}), \quad
 \tilde{\mathcal G}^1:=
 {\mathcal End} (\tilde{E})\otimes
 \Omega^1_{ {\mathcal C}\times_S \tilde{M}/\tilde{M} } ({\mathcal D}_{\tilde{M}}), \quad
 \tilde{G}^1  := \tilde{\mathcal G}^1|_{{\mathcal D}_{\tilde{M}}}, \\
 &
 S(\tilde{E}|_{{\mathcal D}_{\tilde{M}}}^{\vee},\tilde{E}|_{{\mathcal D}_{\tilde{M}}})
 :=
 \left.\left\{ (\tau^{(i)}) \in \bigoplus_{i=1}^n 
 {\mathcal Hom} ( \tilde{E}|_{{\mathcal D}^{(i)}_{\tilde{M}}}^{\vee},
 \tilde{E}|_{{\mathcal D}^{(i)}_{\tilde{M}}})
 \right| \; \text{$^t\tau^{(i)}=\tau^{(i)}$ for any $i$} \right\}, \\
 &
 S(\tilde{E}|_{{\mathcal D}_{\tilde{M}}},\tilde{E}|_{{\mathcal D}_{\tilde{M}}}^{\vee})
 :=
 \left.\left\{ (\xi^{(i)}) \in \bigoplus_{i=1}^n 
 {\mathcal Hom}(\tilde{E}|_{{\mathcal D}^{(i)}_{\tilde{M}}},
 \tilde{E}|_{{\mathcal D}^{(i)}_{\tilde{M}}}^{\vee})
 \right| \ \text{$^t\xi^{(i)}=\xi^{(i)}$ for any $i$} \right\}, \\
 &
 \tilde{Z}^0
 := \bigoplus_{i=1}^n {\mathcal O}_{{\mathcal D}^{(i)}_{\tilde{M}}}[T]
 /(\varphi_{\tilde{\bmu}}^{(i)}(T)),
 \quad
 \tilde{Z}^1 := \bigoplus_{i=1}^n {\mathcal Hom}_{{\mathcal O}_{{\mathcal D}^{(i)}_{\tilde{M}}}}
 ({\mathcal O}_{{\mathcal D}^{(i)}_{\tilde{M}}} [T]/(\varphi_{\tilde{\bmu}}^{(i)}(T)) ,
 \Omega^1_{{\mathcal C}_{\tilde{M}}/\tilde{M}} ({\mathcal D}_{\tilde{M}})
 \big| _{{\mathcal D}_{\tilde{M}}}).
\end{align*}
We define a complex
$\tilde{\mathcal F}^{\bullet}=
[\tilde{\mathcal F}^0 \xrightarrow{d^0} \tilde{\mathcal F}^1 \xrightarrow{d^1} \tilde{\mathcal F}^2]$
in the same way as subsection \ref {subsection:tangent};
\begin{align*}
 &\tilde{\mathcal F}^0 = \tilde{\mathcal G}^0 \oplus \tilde{Z}^0, \quad
 \tilde{\mathcal F}^1 = \tilde{\mathcal G}^1 \oplus
 S(\tilde{E}|_{{\mathcal D}_{\tilde{M}}}^{\vee},\tilde{E}|_{{\mathcal D}_{\tilde{M}}})
 \oplus
 S(\tilde{E}|_{{\mathcal D}_{\tilde{M}}},\tilde{E}|_{{\mathcal D}_{\tilde{M}}}^{\vee}), \quad
 \tilde{\mathcal F}^2 = \tilde{G}^1 \oplus \tilde{Z}^1 \\
 & d^0(u,(\overline{P^{(i)}(T)})) 
 = \left( \nabla\circ u - u\circ\nabla,
 \left(\sigma^{(i)-}_{\theta^{(i)}} \left( u|_{{\mathcal D}^{(i)}_s},\overline{P^{(i)}(T)} \right) \right),
 \left(\sigma^{(i)+}_{\kappa^{(i)}}\left(u|_{{\mathcal D}^{(i)}_s},\overline{P^{(i)}(T)} \right)\right) \right) \\
 & d^1(v,(\tau^{(i)}),(\xi^{(i)})) 
 =\left( \left(
 v|_{{\mathcal D}^{(i)}_s}
 -\delta^{(i)}_{\bnu,N^{(i)}}(\tau^{(i)}\circ\kappa^{(i)}+\theta^{(i)}\circ\xi^{(i)})\right),
 \left(\Theta^{(i)}_{(\tau^{(i)}\circ\kappa^{(i)}+\theta^{(i)}\circ\xi^{(i)})}\right) \right).
\end{align*}
Then we can see by the same proof as Proposition \ref {prop:first deformation}
that the relative tangent bundle
$T_{\tilde{M}/S}$ of $\tilde{M}$ over $S$ is isomorphic to
$\mathbf{R}^1(p_{\tilde{M}})_*(\tilde{\mathcal F}^{\bullet})$,
where $p_{\tilde{M}}\colon {\mathcal C}\times_S\tilde{M}\longrightarrow\tilde{M}$
is the structure morphism.
We define
$\big( \Xi^{(\tau^{(i)},\xi^{(i)})}_{(\tau'^{(i)},\xi'^{(i)})} \big)
 \in \Omega^1_{{\mathcal C}_{\tilde{M}}/\tilde{M}}({\mathcal D}_{\tilde{M}})
 |_{{\mathcal D}_{\tilde{M}}}$
for
$((\tau^{(i)}),(\xi^{(i)})), ((\tau'^{(i)}),(\xi'^{(i)})) \in 
S(\tilde{E}|_{{\mathcal D}_{\tilde{M}}}^{\vee},\tilde{E}|_{{\mathcal D}_{\tilde{M}}})
\oplus
S(\tilde{E}|_{{\mathcal D}_{\tilde{M}}},\tilde{E}|_{{\mathcal D}_{\tilde{M}}}^{\vee})$
in the same way as (\ref {equation:main part in the definition of symplectic form})
in subsection \ref {subsection:nondenerate pairing}.
We take an affine open covering
$\{U_{\alpha}\}$ of ${\mathcal C}$ and define a pairing
\[
 \omega_{\tilde{M}} \colon
 \mathbf{R}^1(p_{\tilde{M}})_*(\tilde{\mathcal F}^{\bullet})
 \times \mathbf{R}^1(p_{\tilde{M}})_*(\tilde{\mathcal F}^{\bullet})
 \longrightarrow \mathbf{R}^2(p_{\tilde{M}})_*
 ({\mathcal L}^{\bullet}_M)
 \cong {\mathcal O}_{\tilde{M}}
\]
by
\begin{align*}
 &\omega_{\tilde{M}}
 \left( \big[ \big\{ (u_{\alpha\beta},0) \big\},
 \big\{ (v_{\alpha},((\tau^{(i)}_{\alpha}),(\xi^{(i)}_{\alpha}))) \big\} \big],
 \big[ \big\{ (u'_{\alpha\beta},0) \big\},
 \big\{ (v'_{\alpha},((\tau'^{(i)}_{\alpha}),(\xi'^{(i)}_{\alpha}))) \big\} \big] \right) \\
 &=\left[
 \left\{ \Tr(u_{\alpha\beta}\circ u'_{\beta\gamma})\right\},
 -\left\{ \left( \Tr(u_{\alpha\beta}\circ v'_{\beta}-v_{\alpha}\circ u'_{\alpha\beta}) , 0 \right) \right\}, 
 \Big\{ \big( \Xi^{(\tau^{(i)}_{\alpha},\xi^{(i)}_{\alpha})}_{(\tau'^{(i)}_{\alpha},\xi'^{(i)}_{\alpha})}
 \big) \Big\} \right]
\end{align*}
using the \u{C}ech cohomology with respect to the covering $\{U_{\alpha}\times_S\tilde{M}\}$.
Then the restriction $\omega_{\tilde{M}}\big|_x$ at a point
$x$ of $\tilde{M}$ whose image in
$M^{\balpha}_{{\mathcal C},{\mathcal D}}(\tilde{\bnu},\tilde{\bmu})$
corresponds to $(E,\nabla,\{l^{(i)}\})$ is nothing but the pairing
$\omega_{(E,\nabla,\{l^{(i)}\})}$ in Lemma \ref{lemma-non-degenerate},
which is nondegenerate.
We can easily see that $\omega_{\tilde{M}}$ descends to a pairing
\[
 \omega_{M^{\balpha}_{{\mathcal C},{\mathcal D}}(\tilde{\bnu},\tilde{\bmu})}\colon
 T_{M^{\balpha}_{{\mathcal C},{\mathcal D}}(\tilde{\bnu},\tilde{\bmu})/S}\times
 T_{M^{\balpha}_{{\mathcal C},{\mathcal D}}(\tilde{\bnu},\tilde{\bmu})/S}
 \longrightarrow {\mathcal O}_{M^{\balpha}_{{\mathcal C},{\mathcal D}}(\tilde{\bnu},\tilde{\bmu})}
\]
which is nondegenerate.
If we take a tangent vector
$v\in T_{M^{\balpha}_{{\mathcal C},{\mathcal D}}(\tilde{\bnu},\tilde{\bmu})/S}(x)$
at a point $x\in M^{\balpha}_{{\mathcal C},{\mathcal D}}(\tilde{\bnu},\tilde{\bmu})$
corresponding to a $(\tilde{\bnu}_s,\tilde{\bmu}_s)$-connection $(E,\nabla,\{l^{(i)}\})$,
$v$ corresponds to a $\mathbb{C}[t]/(t^2)$-valued point
$(E',\nabla',\{l'^{(i)}\})$ of
${\mathcal M}^{\balpha}_{{\mathcal C},{\mathcal D}}(\tilde{\bnu},\tilde{\bmu})$
which is a lift of $(E,\nabla,\{l^{(i)}\})$.
Then we can check that
$\omega_{M^{\balpha}_{{\mathcal C},{\mathcal D}}(\tilde{\bnu},\tilde{\bmu})} (v,v)$
coincides with the image by
$\Tr\colon \mathbf{H}^2({\mathcal F}^{\bullet})\xrightarrow{\sim}
\mathbf{H}^2({\mathcal L}^{\bullet}_s)$
of the obstruction class
$o(E',\nabla',\{l'^{(i)}\})$ for the lifting of $(E',\nabla',\{l'^{(i)}\})$ to
a $\mathbb{C}[t]/(t^3)$-valued point of
${\mathcal M}^{\balpha}_{{\mathcal C},{\mathcal D}}(\tilde{\bnu},\tilde{\bmu})$
which is given in Proposition \ref {prop-obstruction-class}.
Since $M^{\balpha}_{{\mathcal C},{\mathcal D}}(\tilde{\bnu},\tilde{\bmu})$
is smooth over $S$ by Proposition \ref {prop:smoothness of moduli},
we have
$\omega_{M^{\balpha}_{{\mathcal C},{\mathcal D}}(\tilde{\bnu},\tilde{\bmu})} (v,v)=0$.
So the pairing $\omega_{M^{\balpha}_{{\mathcal C},{\mathcal D}}(\tilde{\bnu},\tilde{\bmu})}$
is skew-symmetric and define a relative $2$-form
$ \omega_{M^{\balpha}_{{\mathcal C},{\mathcal D}}}
\in H^0(M^{\balpha}_{{\mathcal C},{\mathcal D}}(\tilde{\bnu},\tilde{\bmu}),
\Omega^2_{M^{\balpha}_{{\mathcal C},{\mathcal D}}(\tilde{\bnu},\tilde{\bmu})/S})$.

A generic geometric fiber
$M^{\balpha}_{{\mathcal C},{\mathcal D}}(\tilde{\bnu},\tilde{\bmu})_s$
over $S$ is the moduli space of regular singular connections on ${\mathcal C}_s$
along the reduced divisor ${\mathcal D}_s$.
If we put
$\tilde{M}_s:=\tilde{M}\times_{M^{\balpha}_{{\mathcal C},{\mathcal D}}(\tilde{\bnu},\tilde{\bmu})}
M^{\balpha}_{{\mathcal C},{\mathcal D}}(\tilde{\bnu},\tilde{\bmu})_s$,
there is a universal parabolic structure
$\tilde{E}_{\tilde{M}_s}|_{(\tilde{\mathcal D}^{(i)}_j)_{\tilde{M}_s}}=\tilde{l}^{(i)}_{j,0}
\supset\cdots\supset \tilde{l}^{(i)}_{j,r-1}\supset\tilde{l}^{(i)}_{j,r}=0$
determined by $\tilde{\nabla}_{\tilde{M}_s}$.
If we put
\begin{align*}
 \tilde{\mathcal F}_{par}^0
 &:=
 \left\{ u\in
 \tilde{\mathcal G}^0_{M^{\balpha}_{{\mathcal C},{\mathcal D}}(\tilde{\bnu},\tilde{\bmu})_s}
 \left|
 \text{$u|_{({\mathcal D}^{(i)}_j)_{\tilde{M}_s}}(\tilde{l}^{(i)}_{j,k})\subset \tilde{l}^{(i)}_{j,k}$
 for any $i,j,k$}
 \right\}\right. \\
 \tilde{\mathcal F}_{par}^1
 &:=
 \left\{ v\in
 \tilde{\mathcal G}^1_{M^{\balpha}_{{\mathcal C},{\mathcal D}}(\tilde{\bnu},\tilde{\bmu})_s}
 \left|
 \text{$v|_{({\mathcal D}^{(i)}_j)_{\tilde{M}_s}}(\tilde{l}^{(i)}_{j,k})\subset
 \tilde{l}^{(i)}_{j,k+1}\otimes
 \Omega^1_{{\mathcal C}_{\tilde{M}_s}/\tilde{M}_s}({\mathcal D}_{\tilde{M}_s})$
 for any $i,j,k$}
 \right\}\right. \\
 \nabla_{\tilde{\mathcal F}_{par}^{\bullet}} 
 &\colon
 \tilde{\mathcal F}^0_{par} \ni u\mapsto
 \tilde{\nabla}\circ u-u\circ\tilde{\nabla} \in \tilde{\mathcal F}^1_{par},
\end{align*}
then the canonical inclusions
$\tilde{\mathcal F}^0_{par}\hookrightarrow \tilde{\mathcal G}^0_{\tilde{M}_s}$
and $\tilde{\mathcal F}^1_{par} \hookrightarrow \tilde{\mathcal G}^1_{\tilde{M}_s}$ 
induce a morphism
$\tilde{\mathcal F}^{\bullet}_{par}
\longrightarrow
\tilde{\mathcal F}^{\bullet}_{\tilde{M}_s}$
of complexes
which induces an isomorphism
\[
 \mathbf{R}^1(\pi_{\tilde{M}_s})_*(\tilde{\mathcal F}^{\bullet}_{par})
 \xrightarrow{\sim}
 \mathbf{R}^1(\pi_{\tilde{M}_s})_*(\tilde{\mathcal F}^{\bullet}_{\tilde{M}_s})
\]
because they are both isomorphic to the tangent bundle
of $\tilde{M}_s$.
A symplectic form $\omega_{\tilde{M}_s}$ on $\tilde{M}_s$
is defined in \cite[Proposition 7.2]{Inaba-1},
which satisfies $d\omega_{\tilde{M}_s}=0$ by \cite[Porposition 7.3]{Inaba-1}.
By construction, we can see that
$\omega_{\tilde{M}_s}=\omega_{\tilde{M}}|_{\tilde{M}_s}$.
So we have  $d\omega_{M^{\balpha}_{{\mathcal C},{\mathcal D}}(\tilde{\bnu},\tilde{\bmu})}|
_{M^{\balpha}_{{\mathcal C},{\mathcal D}}(\tilde{\bnu},\tilde{\bmu})_s}=0$,
which implies that $\omega_{M^{\balpha}_{{\mathcal C},{\mathcal D}}(\tilde{\bnu},\tilde{\bmu})}$
is relatively $d$-closed on $M^{\balpha}_{{\mathcal C},{\mathcal D}}(\tilde{\bnu},\tilde{\bmu})$
over $S$.
\end{proof}

Eventually Theorem \ref {theorem:algebraic-moduli-unfolding}
follows from Corollary \ref  {cor:dimension of tangent space},
Proposition \ref {prop:smoothness of moduli}
and Proposition \ref  {prop:symplectic form}.
.

\section{Fundamental solution of an unfolded linear differential equation with
an asymptotic property}
\label {section:theory of Hurtubise-Lambert-Rousseau}

In this section, we introduce the existence theorem of fundamental solutions
with an asymptotic property
of an unfolded linear differential equation,
which is one of the main tools in the unfolding theory of linear differential equations
established by Hurtubise, Lambert and Rousseau
in \cite{Hurtubise-Lambert-Rousseau} and \cite{Hurtubise-Rousseau}.
Unfortunately, the unfolded generalized isomonodromic deformation
in Theorem \ref {thm:holomorphic-splitting} is not compatible with the asymptotic property
given in the unfolding theory in \cite{Hurtubise-Lambert-Rousseau}, \cite{Hurtubise-Rousseau}.
However, it will be worth pointing out what is the difficulty in
adopting the asymptotic property
in \cite{Hurtubise-Lambert-Rousseau}, \cite{Hurtubise-Rousseau}
to our moduli theoretic setting constructed 
in section \ref  {section:algebraic moduli construction}.
Since the unfolding theory in
\cite{Hurtubise-Lambert-Rousseau}, \cite{Hurtubise-Rousseau}
are written in a very general setting and hard to follow all of them,
we restrict to the easy case when the unfolding of the singular divisor
is given by the equation $z^m-\epsilon^m=0$.

\subsection{Flows for an asymptotic estimate}
\label{subsection:flow}

Let $\Delta=\{ z\in\mathbb{C} \, | \, |z|<1 \}$ be a unit disk in the complex plane $\mathbb{C}$.
For an integer $m$ with $m\geq 2$,
we put
$\zeta_m:=\exp\left(\dfrac{2\pi\sqrt{-1}}{m}\right)$.
Then we have
$z^m-\epsilon^m= (z-\epsilon \zeta_m)(z-\epsilon \zeta_m^2)\cdots(z-\epsilon \zeta_m^m)$
for $z,\epsilon\in\Delta$.
We set
\[
 D:=\left\{ (z,\epsilon) \in \Delta\times\Delta \left| \,
 z^m-\epsilon^m=0 \right\} \right..
\]
Note that there is an equality
\[
 \frac{1}{z^m-\epsilon^m}=
 \frac{1}{(z-\epsilon \zeta_m)\cdots(z-\epsilon \zeta_m^m)}
 =
 \sum_{j=1}^m \frac{1}{\prod_{j\neq i} \epsilon ( \zeta_m^i-\zeta_m^j)}
 \frac{1}{z-\epsilon \zeta_m^i}
\]
for $(z,\epsilon)\in (\Delta\times\Delta)\setminus D$.
By Lemma \ref {lemma:residue-identity}, we have
\[
 \sum_{i=1}^m\frac{1}{\prod_{j\neq i}(\epsilon \zeta_m^i-\epsilon \zeta_m^j)}=
 \sum_{i=1}^m\res_{z=\epsilon \zeta_m^i}
 \left(\dfrac{dz}{(z-\epsilon \zeta_m)(z-\epsilon \zeta_m^2)
 \cdots(z-\epsilon \zeta_m^m)}\right)=0
\]
for $\epsilon\neq 0$, since $m\geq 2$.

For a fixed $\theta\in\mathbb{R}$,
we consider a holomorphic differential equation
\begin{equation}\label{equation:original-holomorphic-differential-equation}
 \frac{dz}{d\tau}= e^{\sqrt{-1}\theta}(z^m-\epsilon^m)
 =e^{\sqrt{-1}\theta}
 (z-\epsilon \zeta_m)(z-\epsilon \zeta_m^2)\cdots(z-\epsilon \zeta_m^m).
\end{equation}
Under the above equation, we can regard $\tau$ as a multi-valued function in
$z\in (\Delta\times\Delta)\setminus D$.
We substitute into $\tau\in\mathbb{C}$ a real variable
$t\in\mathbb{R}$ and consider the restricted differential equation
\begin{equation} \label {equation:original-differential-equation}
 \frac{dz}{dt}= e^{\sqrt{-1}\theta}(z^m-\epsilon^m)
 =e^{\sqrt{-1}\theta}
 (z-\epsilon \zeta_m)(z-\epsilon \zeta_m^2)\cdots(z-\epsilon \zeta_m^m).
\end{equation}
Note that giving a solution
$z(t)=x(t)+\sqrt{-1}y(t)$ of
the differential equation (\ref {equation:original-differential-equation})
is equivalent to giving a flow of the vector field
\begin{equation}\label{equation:vector-field}
 v_{\epsilon,\theta}
 =\mathrm{Re}\left(e^{\sqrt{-1}\theta}
 (z^m-\epsilon^m)\right)
 \frac{\partial}{\partial x}
 +
 \mathrm{Im}\left(e^{\sqrt{-1}\theta}
 (z^m-\epsilon^m)\right)
 \frac{\partial}{\partial y}.
\end{equation}

For the investigation of the flow of the vector field $v_{\epsilon,\theta}$,
we consider the surjective morphism
\[
 \varpi\colon \Delta\times [0,1)\times S^1 \longrightarrow
 \Delta\times\Delta
\]
defined by
\[
 \varpi (z,s,e^{\sqrt{-1}\psi})=(z,s e^{\sqrt{-1}\psi})
\]
and we call $\varpi$ a polar blow up of $\Delta\times\Delta$ along $\Delta\times\{0\}$.
Here we denote $\{t\in \mathbb{R} \,| \,  a\leq t <b\}$
by $[a,b)$ for real numbers $a,b$ satisfying $a<b$.

We consider the following proposition which treats
an easy restricted case of the analysis of
flows in a series of papers \cite{Lambert-Rousseau-1}, \cite{Lambert-Rousseau-2},
\cite{Hurtubise-Lambert-Rousseau},
\cite{Hurtubise-Rousseau}.
We give here just an elementary proof in an easy restricted case
for the purpose of the author's understanding.
So it may seem trivial for experts.

\begin{proposition} \label {proposition:open-covering}
There is an open neighborhood $U$ of $\{0\}\times\{0\}\times S^1$
in $\Delta\times[0,1)\times S^1$
and an open covering
\begin{equation} \label {equation: open covering}
 U\setminus (U\cap\varpi^{-1}(D))
 =\bigcup_{j=1}^m\bigcup_{0\leq\psi_0\leq 2\pi}\bigcup_{\xi=1,2}
 W^{(j)}_{\psi_0,\xi}
\end{equation}
such that any flow of the vector field
\[
 v_{\epsilon,\theta^{(j)}_{\psi_0,\xi}}
 =\mathrm{Re}\left(e^{\sqrt{-1}\theta^{(j)}_{\psi_0,\xi}}
 (z^m-\epsilon^m)\right)
 \frac{\partial}{\partial x}
 +
 \mathrm{Im}\left(e^{\sqrt{-1}\theta^{(j)}_{\psi_0,\xi}}
 (z^m-\epsilon^m)\right)
 \frac{\partial}{\partial y}
\]
starting at a point of $W^{(j)}_{\psi_0,\xi}$ converges to
a point in $\varpi^{-1}(D)$,
where $\theta^{(j)}_{\psi_0,\xi}$ is determined by $j,\psi_0,\xi$.
\end{proposition}

\begin{proof}
We take a point
$(z_0,s_0,e^{\sqrt{-1}\psi_0}) \in
(\Delta\setminus \{0\})\times\left[0,\,\frac{1}{3}\right)\times S^1$
satisfying $0<|z_0|<\frac{1}{4}$.
We can choose an integer $j$ with $1\leq j\leq m$ satisfying
\[
 -\frac{\pi}{m} \leq \arg(z_0)-\psi_0-\frac{2j\pi}{m} \leq \frac{\pi}{m}.
\]
We divide into two cases:
\begin{gather*}
 0\leq\arg(z_0)-\psi_0-\dfrac{2j\pi}{m}\leq\dfrac{\pi}{m}, \quad
  -\dfrac{\pi}{m} \leq \arg(z_0)-\psi_0-\dfrac{2j\pi}{m} < 0.
\end{gather*}

\noindent
{\bf Case 1.}
$0\leq \arg(z_0)-\dfrac{2j\pi}{m}-\psi_0 \leq \dfrac{\pi}{m}$. \\
In this case we choose small $\delta>0$ satisfying
$\delta<\dfrac{\pi}{24m}$
and put 
\begin{equation} \label {equation:definition of theta case 1}
 \theta^{(j)}_{\psi_0,1}:=-\dfrac{2j(m-1)\pi}{m}-(m-1)\psi_0+\pi+\delta.
\end{equation}
We simply denote $\theta^{(j)}_{\psi_0,1}$ by $\theta$ in the following.
So $\theta$ is given by
\[
 \frac{\theta-\pi}{m-1}=-\frac{2j\pi}{m}-\psi_0+\frac{\delta}{m-1}.
\]
Note that we have 
\begin{gather*}
 \dfrac{\delta}{m-1}\leq
 \arg(z_0)+\dfrac{\theta-\pi}{m-1}
 \leq \dfrac{\pi}{m}+\dfrac{\delta}{m-1},
 \hspace{30pt}
 \psi_0+\frac{2j\pi}{m}+\frac{\theta-\pi}{m-1}
 =\frac{\delta}{m-1}.
\end{gather*}
If we replace $\delta>0$ sufficiently smaller,
we may assume that the two segments
\begin{align*}
 l_1&=
 \left\{ z\in\mathbb{C} \left|
 \arg\left(e^{\sqrt{-1}\frac{\pi}{3m}}-e^{\sqrt{-1}\frac{\theta-\pi}{m-1}}z\right)
 = \dfrac{(2m+1)\delta}{m-1} , \;
 |z| <1 , \; \mathrm{Re}(z)>0 \right\}\right. \\ 
 l_2&=
 \left\{ z\in\mathbb{C} \left|
 \arg(z)+\dfrac{\theta-\pi}{m-1}=\dfrac{\pi}{m}+\dfrac{2\delta}{m-1} , \;
 |z| < 1 , \; \mathrm{Re}(z)>0  \right\}\right.
\end{align*}
intersects at a point
$s_1 e^{\sqrt{-1}(\frac{\pi}{m}-\frac{\theta-\pi}{m-1}+\frac{2\delta}{m-1})}$
satisfying $\dfrac{1}{4}<s_1<1$.
Then we put
\[
 P^{(j)}_{\psi_0,1}=
 \left\{(z,(s,e^{\sqrt{-1}\psi}))\in
 \Delta\times \Big[0,\frac{1}{3}\Big)\times S^1\left|
 \begin{array}{l}
 \displaystyle -\dfrac{3\delta}{2m-2}<\psi+\dfrac{2j\pi}{m}+\frac{\theta-\pi}{m-1}
 <\frac{3\delta}{2m-2}, \; z\neq 0 \\
 \displaystyle
 \text{
 $\dfrac{(2m+1)\delta}{m-1}<
 \arg\left( e^{\sqrt{-1}\frac{\pi}{3m}}-e^{\sqrt{-1}\frac{\theta-\pi}{m-1}}z\right)
 <\dfrac{\pi}{2}+\dfrac{\pi}{3m}$} \\
 \text{and
 $-\dfrac{\pi}{3m}<\arg(z)+\dfrac{\theta-\pi}{m-1}
 <\dfrac{\pi}{m}+\dfrac{2\delta}{m-1}$
 } 
 \end{array}
 \right\}\right..
\]
A picture of the region
$\left\{ \tilde{z}=e^{\sqrt{-1}\frac{\theta-\pi}{m-1}}z \, \left| \,
(z,s,e^{\sqrt{-1}\psi}) \in
P^{(j)}_{\psi_0,1}\cap \Big(\Delta\times \big\{(s,e^{\sqrt{-1}\psi})\big\}\Big)
\right\}\right.$
looks like [figure 1].
Since
$\arg\left( e^{\sqrt{-1}\frac{\theta-\pi}{m-1}}e^{\sqrt{-1}\theta}
 \big(  e^{\sqrt{-1}\psi} \zeta_m^j \big)^m \right)
 =\dfrac{\theta-\pi}{m-1}+\theta+m\psi
 =\dfrac{m(\theta-\pi)}{m-1}+m\psi+\pi$,
we have
\begin{equation}\label{equation:inequality-epsilon}
 \pi-\frac{3m\delta}{2m-2}<
 \arg\left( e^{\sqrt{-1}\frac{\theta-\pi}{m-1}}e^{\sqrt{-1}\theta} 
 \big( e^{\sqrt{-1}\psi} \zeta_m^j \big)^m \right)
 <\pi+\frac{3m\delta}{2m-2}
\end{equation}
if
$-\dfrac{3\delta}{2m-2}<\psi+\dfrac{2j\pi}{m}+\dfrac{\theta-\pi}{m-1}
<\dfrac{3\delta}{2m-2}$.
So we can take $\eta>0$ depending on $m,j,\theta,\delta$ such that
\begin{equation} \label {equation: choice of delta}
 -\frac{2m\delta}{m-1}<
 \arg \left( e^{\sqrt{-1}\frac{\theta-\pi}{m-1}}e^{\sqrt{-1}\theta}
 (w^m-(e^{\sqrt{-1}\psi}\zeta_m^j)^m) \right)<\frac{2m\delta}{m-1}
\end{equation}
holds for any $w\in\Delta$ satisfying $|w| \leq \eta$,
when
$-\dfrac{3\delta}{2m-2}<\psi+\dfrac{2j\pi}{m}+\dfrac{\theta-\pi}{m-1}
<\dfrac{3\delta}{2m-2}$.
We put
\[
 Q^{(j)}_{\psi_0,1}:=
 \left\{ (z,s,e^{\sqrt{-1}\psi})\in
 \Delta\times \Big[ 0,\frac{1}{3} \Big) \times S^1
 \left|
 \begin{array}{l}
 e^{\sqrt{-1}\frac{\theta-\pi}{m-1}}z-\eta s e^{\sqrt{-1}\pi}\neq 0, \\
 -\dfrac{\pi}{6m} <\arg\left( e^{\sqrt{-1}\frac{\theta-\pi}{m-1}}z
 -\eta s e^{\sqrt{-1}\pi}\right)
 <\dfrac{\pi}{6m} \\
 \text{$-\dfrac{3\delta}{2m-2}<\psi+\dfrac{2j\pi}{m}+\dfrac{\theta-\pi}{m-1}
 <\dfrac{3\delta}{2m-2}$ and} \\
 \text{
 $\dfrac{\pi}{m}+\dfrac{2\delta}{m-1}\leq\arg(z)+\dfrac{\theta-\pi}{m-1}
 \leq 2\pi-\dfrac{\pi}{3m}$
 if $z\neq 0$}
 \end{array}
 \right\}\right..
\]
and set
\[
 R^{(j)}_{\psi_0,1}:= 
 P^{(j)}_{\psi_0,1} \cup Q^{(j)}_{\psi_0,1}.
\]
We may assume $\eta<\dfrac{1}{4}$ and then the segment
\[
 l_3^+=\left\{ z\in\mathbb{C} \left|
 \arg\left( e^{\sqrt{-1}\frac{\theta-\pi}{m-1}}z
 -\eta s e^{\sqrt{-1}\pi} \right)
 =\dfrac{\pi}{6m}, \;
 -\eta < \mathrm{Re}(z) < 1 \right\}\right.
\]
intersects with the segment $l_2$
at a point $s_2 e^{\sqrt{-1}(\frac{\pi}{m}-\frac{\theta-\pi}{m-1}+\frac{2\delta}{m-1})}$
satisfying $0<s_2<\dfrac{1}{4}<s_1$ if $s>0$.
A picture of the region
$\left\{ \tilde{z}=e^{\sqrt{-1}\frac{\theta-\pi}{m-1}}z \, \left| \, 
(z,s,e^{\sqrt{-1}\psi})\in
R^{(j)}_{\psi_0,1}\cap \Big(\Delta\times \big\{(s,e^{\sqrt{-1}\psi})\big\}\Big)
\right\}\right.$
looks like [figure 2]


\includegraphics[width=480pt] {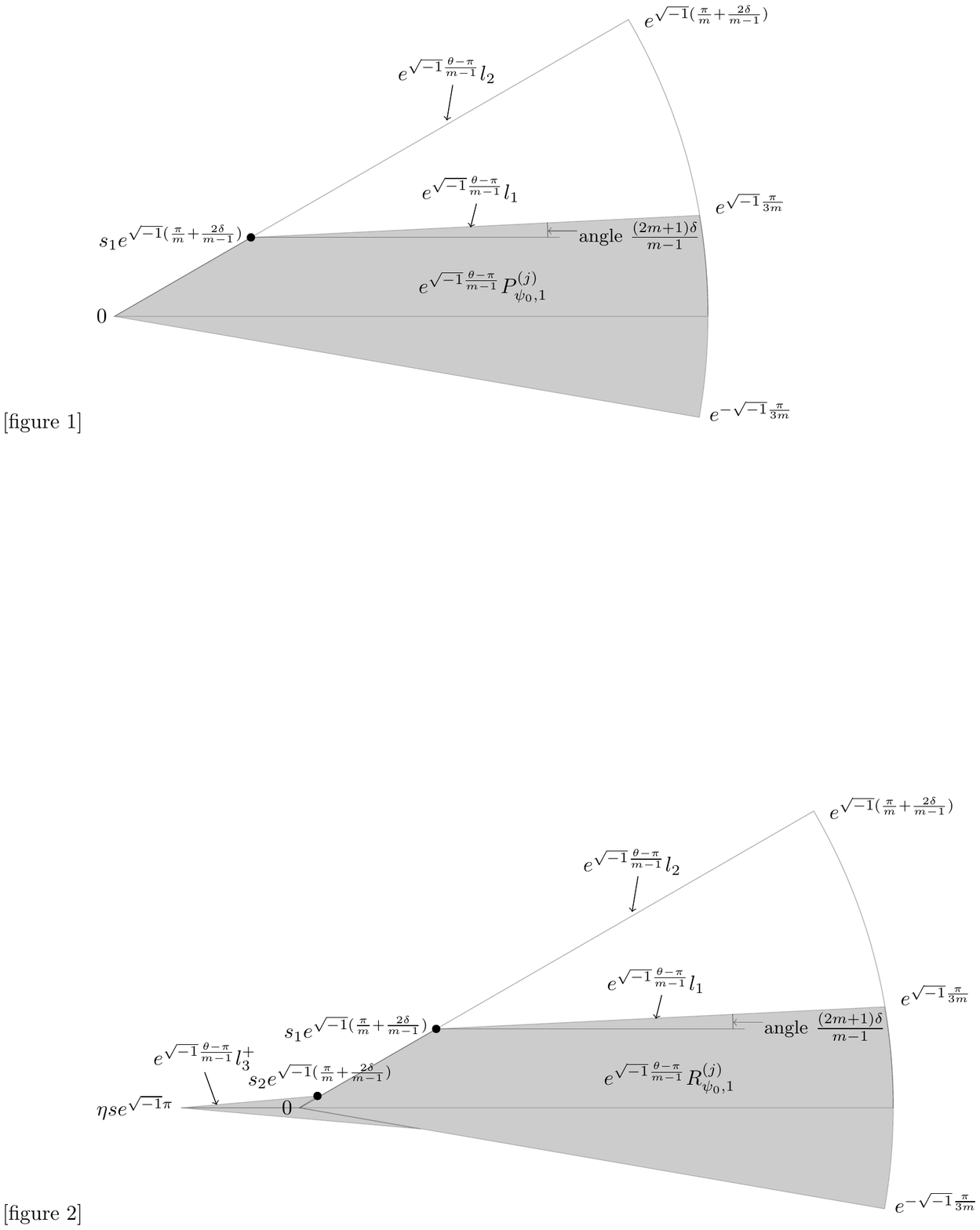}


In the case of $\epsilon=s e^{\sqrt{-1}\psi}=0$, we can see
$Q^{(j)}_{\psi_0,1}\cap
\Big(\Delta\times \big\{ \big( 0,e^{\sqrt{-1}\psi} \big) \big\}\Big)
=\emptyset$
by the definition of $Q^{(j)}_{\psi_0,1}$,
from which we have
\[
 R^{(j)}_{\psi_0,1}\cap
 \Big(\Delta\times \big\{ \big( 0 , e^{\sqrt{-1}\psi} \big) \big\} \Big)
 =
 P^{(j)}_{\psi_0,1}\cap
 \Big(\Delta\times \big\{ \big( 0 , e^{\sqrt{-1}\psi} \big ) \big\} \Big).
\]
In any case, $\big(z_0,s_0,e^{\sqrt{-1}\psi_0}\big)$
lies in $R^{(j)}_{\psi_0,1}$
and
\[
 R^{(j)}_{\psi_0,1}\cap \varpi^{-1}(D)=
 \left.\left\{ \big( \epsilon \zeta_m^j, s, e^{\sqrt{-1}\psi}\big)
 \in R^{(j)}_{\psi_0,1} \right|
 \epsilon=s e^{\sqrt{-1}\psi}\right\}.
\]

Consider the differential equation
\[
 \frac{dz(t)}{dt}=e^{\sqrt{-1}\theta}(z(t)^m-\epsilon^m)
 =e^{\sqrt{-1}\theta}
 (z(t)-\epsilon \zeta_m)(z(t)-\epsilon \zeta_m^2)\cdots(z(t)-\epsilon \zeta_m^m)
\]
with respect to a real time variable $t$ and the initial point
$z(0)\in R^{(j)}_{\psi_0,1}\setminus
(\varpi^{-1}(D)\cap R^{(j)}_{\psi_0})$. 
The solution of the above differential equation
is equivalent to
the flow of the vector field
\[
 v_{\epsilon,\theta}
 =\mathrm{Re}\left(e^{\sqrt{-1}\theta}
 (z^m-\epsilon^m)\right)
 \frac{\partial}{\partial x}
 +
 \mathrm{Im}\left(e^{\sqrt{-1}\theta}
 (z^m-\epsilon^m)\right)
 \frac{\partial}{\partial y}
\]
starting at a point in
$R^{(j)}_{\psi_0,1}\setminus
(\varpi^{-1}(D)\cap R^{(j)}_{\psi_0,1})$.
Notice that the direction of the vector $v_{\epsilon,\theta}$
is given by
$\arg \left( e^{\sqrt{-1}\theta}(z(t)^m-\epsilon^m) \right)$.
We investigate the direction of the vector $v_{\epsilon,\theta}$
at each boundary point of
the fiber
$R^{(j)}_{\psi_0,1}\cap
\Big(\Delta\times\big\{ (s,e^{\sqrt{-1}\psi}) \big\} \Big)$
of $R^{(j)}_{\psi_0,1}$
over $(s,e^{\sqrt{-1}\psi})\in [0,\frac{1}{3})\times S^1$.

First take
a boundary point $(z,s,e^{\sqrt{-1}\psi})$ of
$R^{(j)}_{\psi_0,1}\cap
\Big(\Delta\times \big\{(s,e^{\sqrt{-1}\psi})\big\}\Big)$
satisfying 
$\arg(z)+\dfrac{\theta-\pi}{m-1}=\dfrac{\pi}{m}+\dfrac{2\delta}{m-1}$.
Then we have
\[
 \arg \left(e^{\sqrt{-1}\frac{\theta-\pi}{m-1}}e^{\sqrt{-1}\theta}z^m \right)
 =\frac{\theta-\pi}{m-1}+\theta+m\arg(z) 
 =\frac{m(\theta-\pi)}{m-1}+m\arg(z)+\pi =2\pi+\frac{2m\delta}{m-1}. 
\]
Combined with the inequality (\ref{equation:inequality-epsilon}),
we have
\[
 -\frac{3m\delta}{2m-2}<
 \arg\left( e^{\sqrt{-1}\frac{\theta-\pi}{m-1}}
 e^{\sqrt{-1}\theta}(z^m-(\epsilon \zeta_m^j)^m)\right)
 <\frac{2m\delta}{m-1}
 <\frac{\pi}{m}+\frac{2\delta}{m-1},
\]
from which we can see that the vector $v_{\epsilon,\theta}$ 
faces toward the interior of the region 
$R^{(j)}_{\psi_0,1}\cap
\Big(\Delta\times \big\{(s,e^{\sqrt{-1}\psi})\big\}\Big)$.

Secondly take 
a boundary point $(z,s,e^{\sqrt{-1}\psi})$
of $R^{(j)}_{\psi_0,1}\cap
\Big(\Delta\times \big\{(s,e^{\sqrt{-1}\psi})\big\}\Big)$
satisfying 
$\arg(z)+\dfrac{\theta-\pi}{m-1}=-\dfrac{\pi}{3m}$.
Then we have
\[
 \arg \left(e^{\sqrt{-1}\frac{\theta-\pi}{m-1}}e^{\sqrt{-1}\theta}z^m \right)
 =\frac{\theta-\pi}{m-1}+\theta+m\arg(z) 
 =\frac{m(\theta-\pi)}{m-1}+m\arg(z)+\pi =\frac{2\pi}{3}. 
\]
Combined with (\ref{equation:inequality-epsilon}), we have
\[
 -\frac{\pi}{3m}<
 -\frac{3m\delta}{2m-2}<
 \arg\left( e^{\sqrt{-1}\frac{\theta-\pi}{m-1}}
 e^{\sqrt{-1}\theta}(z^m-(\epsilon \zeta_m^j)^m)\right)
 <\frac{2\pi}{3}.
\]
So the vector $v_{\epsilon,\theta}$ 
faces toward the interior of the region 
$R^{(j)}_{\psi_0,1}\cap
\Big(\Delta\times \big\{(s,e^{\sqrt{-1}\psi})\big\}\Big)$.

Thirdly we take a boundary point $(z,s,e^{\sqrt{-1}\psi})$ of
$R^{(j)}_{\psi_0,1}\cap
\Big(\Delta\times \big\{(s,e^{\sqrt{-1}\psi})\big\}\Big)$
which satisfies the equality
$\arg\left(e^{\sqrt{-1}\frac{\pi}{3m}}-e^{\sqrt{-1}\frac{\theta-\pi}{m-1}}z\right)=
 \dfrac{(2m+1)\delta}{m-1}$,
which means that $z$ lies on the segment $l_1$.
Since
$\dfrac{\pi}{3}+\pi \leq
\arg\left( e^{\sqrt{-1}\frac{\theta-\pi}{m-1}}e^{\sqrt{-1}\theta}z^m \right)
\leq 2\pi+\dfrac{2m\delta}{m-1}$,
we can see by the inequality (\ref{equation:inequality-epsilon})
that the inequality
\[
 -\frac{2\pi}{3}\leq
 \arg\left( e^{\sqrt{-1}\frac{\theta-\pi}{m-1}}e^{\sqrt{-1}\theta}
 (z^m-(\epsilon \zeta_m^j)^m)\right)
 \leq \frac{2m\delta}{m-1}<\frac{(2m+1)\delta}{m-1}
\]
holds.
So
the vector $v_{\epsilon,\theta}$ faces toward the interior of the region
$R^{(j)}_{\psi_0,1}\cap
\Big(\Delta\times \big\{(s,e^{\sqrt{-1}\psi})\big\}\Big)$
at this point.
A picture of the direction of the vector $v_{\epsilon,\theta}$
is [figure 3].

Fourthly we take a boundary point
$\big(z,s,e^{\sqrt{-1}\psi}\big)$ of
$R^{(j)}_{\psi_0,1}\cap\Big(\Delta\times
\big\{\big(s,e^{\sqrt{-1}\psi}\big)\big\}\Big)$
satisfying  $|z|=1$.
Note that we have
$-\dfrac{\pi}{3m}\leq\arg(z)+\dfrac{\theta-\pi}{m-1}\leq\dfrac{\pi}{3m}$.
If $\dfrac{\pi}{6m}\leq\arg(z)+\dfrac{\theta-\pi}{m-1}\leq\dfrac{\pi}{3m}$,
then 
\[
 \frac{7\pi}{6}\leq
 \arg \left(e^{\sqrt{-1}\frac{\theta-\pi}{m-1}}e^{\sqrt{-1}\theta}z^m \right)
 =\frac{\theta-\pi}{m-1}+\theta+m\arg(z) 
 =\frac{m(\theta-\pi)}{m-1}+m\arg(z)+\pi \leq \frac{4\pi}{3}.
\]
Since $|\epsilon^m|\leq s < \dfrac{1}{3}=\dfrac{1}{3}|z|^m$
and
$\dfrac{5\pi}{6}\leq
\arg\left(e^{\sqrt{-1}\frac{\theta-\pi}{m-1}}e^{\sqrt{-1}\theta}\epsilon^m\right)
\leq\dfrac{7\pi}{6}$
by (\ref{equation:inequality-epsilon}),
we have a rough estimate
\[
 \dfrac{7\pi}{6} \leq
 \arg\left( e^{\sqrt{-1}\frac{\theta-\pi}{m-1}}
 e^{\sqrt{-1}\theta}(z^m-(\epsilon \zeta_m^j)^m)\right)
 < \dfrac{3\pi}{2}.
\]
So the vector $v_{\epsilon,\theta}$ faces toward the interior of the region
$R^{(j)}_{\psi_0,1}\cap\Big(\Delta\times
\big\{\big(s,e^{\sqrt{-1}\psi}\big)\big\}\Big)$ 
at this point.
If $-\dfrac{\pi}{3m}\leq\arg(z)+\dfrac{\theta-\pi}{m-1} \leq -\dfrac{\pi}{6m}$,
then we have
\[
 \frac{2\pi}{3} \leq
 \arg \left(e^{\sqrt{-1}\frac{\theta-\pi}{m-1}}e^{\sqrt{-1}\theta}z^m \right)
 =\frac{m(\theta-\pi)}{m-1}+m\arg(z)+\pi \leq \dfrac{5\pi}{6}
\]
and we have, from (\ref{equation:inequality-epsilon}) and
$|\epsilon^m|<\frac{1}{3}=\frac{1}{3}|z^m|$, a rough estimate
\[
 \frac{\pi}{2} <
 \arg \left(e^{\sqrt{-1}\frac{\theta-\pi}{m-1}}e^{\sqrt{-1}\theta}(z^m -\epsilon^m) \right)
 =\frac{m(\theta-\pi)}{m-1}+m\arg(z)+\pi \leq \dfrac{5\pi}{6}.
\]
So the vector $v_{\epsilon,\theta}$ faces toward the interior of  the region
$R^{(j)}_{\psi_0,1}\cap
\Big(\Delta\times \big\{\big(s,e^{\sqrt{-1}\psi}\big)\big\}\Big)$
at this point.
If $-\dfrac{\pi}{6m}\leq \arg(z)+\dfrac{\theta-\pi}{m-1}\leq \dfrac{\pi}{6m}$,
then we have
$\dfrac{5\pi}{6}\leq
 \arg \left(e^{\sqrt{-1}\frac{\theta-\pi}{m-1}}e^{\sqrt{-1}\theta}z^m  \right)
 \leq \dfrac{7\pi}{6}$,
from which we obtain a rough estimate
\[
 \dfrac{2\pi}{3} <
 \arg \left(e^{\sqrt{-1}\frac{\theta-\pi}{m-1}}e^{\sqrt{-1}\theta}
 (z^m -\epsilon^m) \right)
 <\dfrac{4\pi}{3}
\]
using  (\ref{equation:inequality-epsilon})
and $|\epsilon^m|< \frac{1}{3}=\frac{1}{3}|z|$.
So $v_{\epsilon,\theta}$ faces toward the interior of the region
$R^{(j)}_{\psi_0}\cap\Big(\Delta\times
\big\{\big(s,e^{\sqrt{-1}\psi}\big)\big\}\Big)$.

Finally we take a boundary point
$\big(z,s,e^{\sqrt{-1}\psi}\big)$ of
$R^{(j)}_{\psi_0,1}\cap\Big(\Delta\times
\big\{\big(s,e^{\sqrt{-1}\psi}\big)\big\}\Big)$
satisfying
$\big(z,s,e^{\sqrt{-1}\psi}\big)\in \overline{ Q^{(j)}_{\psi_0,1} }$
and
$\arg\left( e^{\sqrt{-1}\frac{\theta-\pi}{m-1}}z
-\eta s e^{\sqrt{-1}\pi}\right)
=\pm\dfrac{\pi}{6m}$.
Then we have $|z|\leq s \eta$ and
\[
 -\frac{\pi}{6m}<-\frac{2m\delta}{m-1}<
 \arg \left( e^{\sqrt{-1}\frac{\theta-\pi}{m-1}}e^{\sqrt{-1}\theta}
 \big(z^m-\big(s e^{\sqrt{-1}\psi}\big)^m\big) \right)
 <\frac{2m\delta}{m-1}<\frac{\pi}{6m}
\]
because of the inequality (\ref {equation: choice of delta})
and the assumption
$0<\delta<\dfrac{\pi}{24m}$.
Thus the vector $v_{\epsilon,\theta}$ faces toward the interior of the region
$R^{(j)}_{\psi_0,1}\cap\Big(\Delta\times \big\{(s,e^{\sqrt{-1}\psi})\big\}\Big)$
at this point.
A picture of the direction of $v_{\epsilon,\theta}$
is [figure 4].


\includegraphics[width=480pt] {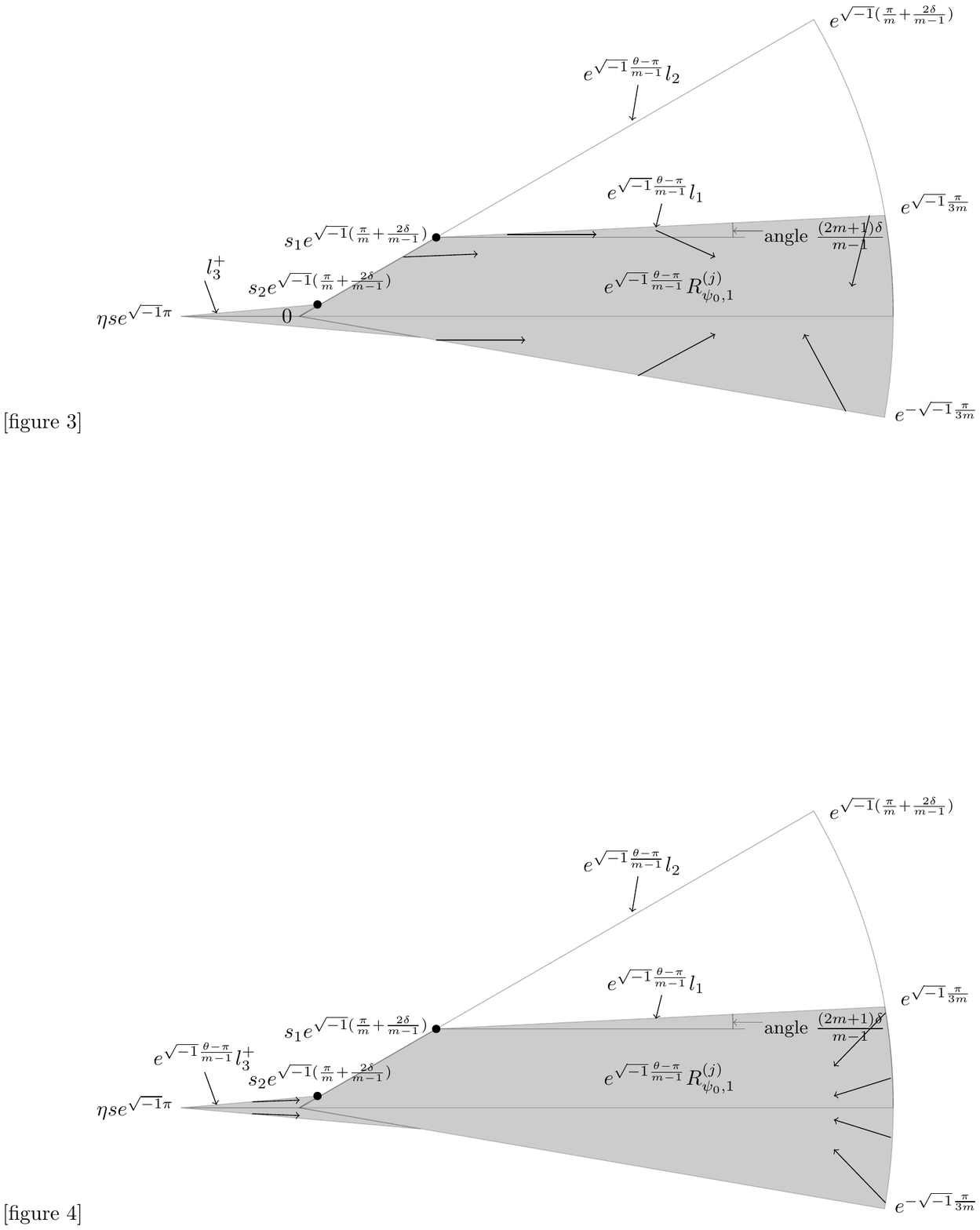}


From all the above arguments, we can see that
the flows of the vector field $v_{\epsilon,\theta}$ stay inside the region
$R^{(j)}_{\psi_0,1}\setminus(\varpi^{-1}(D)\cap R^{(j)}_{\psi_0,1})$.
Take a flow $\{(z(t),(s,e^{\sqrt{-1}\psi}))|t\geq 0\}$ inside
$R^{(j)}_{\psi_0,1}\setminus(\varpi^{-1}(D)\cap R^{(j)}_{\psi_0,1})$.
If we set
\[
 R':=\left\{ (z,s,e^{\sqrt{-1}\psi})\in
 \Delta\times\Big[ 0,\frac{1}{3}\Big)\times S^1  \left| 
 \begin{array}{l}
 -\dfrac{3\delta}{2m-2}<\psi+\dfrac{2j\pi}{m}+\dfrac{\theta-\pi}{m-1}<\dfrac{3\delta}{2m-2}, \\
 z\neq 0, \;
 -\dfrac{\pi}{3m}<\arg(z)+\dfrac{\theta-\pi}{m-1}<\dfrac{\pi}{3m}
 \end{array}
 \right\}\right. ,
\]
then we have
$R'\subset R^{(j)}_{\psi_0,1}$
and we can see by the argument similar to the former analysis on the direction of
$v_{\epsilon,\theta}$
that flows of $v_{\epsilon,\theta}$ starting at points in
$R'\setminus(\varpi^{-1}(D)\cap R')$ stay inside 
$R'\setminus(\varpi^{-1}(D)\cap R')$.
Take any point
$(z,s,e^{\sqrt{-1}\psi})\in R^{(j)}_{\psi_0,1}\setminus R'$.
If $z\neq 0$, then
we have either 
$(z,s,e^{\sqrt{-1}\psi})\in Q^{(j)}_{\psi_0,1}$ or
$\dfrac{\pi}{3m}<\arg(z)+\dfrac{\theta-\pi}{m-1}<\dfrac{\pi}{m}+\dfrac{2\delta}{m-1}$.
So we have either
$|z|<\eta s$ or
\begin{equation} \label {equation:range of angles in R'}
 \frac{4\pi}{3}<
 \arg \left( e^{\sqrt{-1}\frac{\theta-\pi}{m-1}}e^{\sqrt{-1}\theta} z^m \right)
 <2\pi+\frac{2m\delta}{m-1}.
\end{equation}
Combined with (\ref{equation:inequality-epsilon}),
we have $e^{\sqrt{-1}\theta}(z^m-\epsilon^m)\neq 0$.
If $z=0$, then $s>0$ and we have $e^{\sqrt{-1}\theta}(z^m-\epsilon^m)\neq 0$
again.
So
$v_{\epsilon,\theta}$ does not vanish on 
$R^{(j)}_{\psi_0,1}\setminus R'$
and there is no limit point
$\lim_{t\to\infty}z(t)$ inside  $R^{(j)}_{\psi_0,1}\setminus R'$.
Since the inequality (\ref {equation:range of angles in R'})
holds as long as $(z,s,e^{\sqrt{-1}\psi})$ lies in 
$P^{(j)}_{\psi_0,1}\setminus R'$,
flows of $v_{\epsilon,\theta}$ do not stay inside
$R^{(j)}_{\psi_0,1}\setminus R'$ and there exists $t_0>0$ such that
$(z(t_0),s,e^{\sqrt{-1}\psi})$ is contained in the region 
$R'\setminus(\varpi^{-1}(D)\cap R')$.

If
$(z(t),(s,e^{\sqrt{-1}\psi}))\in
R'\setminus(\varpi^{-1}(D)\cap R')$, then we have
\[
 -\frac{(m-1)\pi}{3m} \leq
 \arg\left(\sum_{l=0}^{m-1}\left(z(t)e^{\sqrt{-1}\frac{\theta-\pi}{m-1}}\right)^{m-1-l}
 \left(e^{\sqrt{-1}\frac{\theta-\pi}{m-1}}\epsilon \zeta_m^j \right)^l\right)
 \leq \frac{(m-1)\pi}{3m}.
 \]
By the calculation
\begin{align*}
 \frac{d}{dt}\frac{1}{|z(t)-\epsilon \zeta_m^j |^{2m}} 
 &=
 \frac{1}{\overline{(z(t)-\epsilon \zeta_m^j)^m}}\frac{d}{dt}
 \left(\frac{1}{(z(t)-\epsilon \zeta_m^j)^m} \right)
 + \frac{1}{(z(t)-\epsilon \zeta_m^j)^m} \frac{d}{dt}
 \left(\frac{1}{\overline{(z(t)-\epsilon \zeta_m^j)^m}} \right)
 \notag \\
 &=
 \frac{1}{\overline{(z(t)-\epsilon \zeta_m^j)^m}}
 \frac{-m}{(z(t)-\epsilon \zeta_m^j)^{m+1}}\frac{d z(t)}{dt}
 +\frac{1}{(z(t)-\epsilon \zeta_m^j)^m}
 \frac{-m}{\overline{(z-\epsilon \zeta_m^j)^{m+1}}}\frac{d\overline{z(t)}}{dt}
 \notag  \\
 &=
 -\frac{me^{\sqrt{-1}\theta}(z(t)^m-(\epsilon \zeta_m^j)^m)}
 {(z(t)-\epsilon \zeta_m^j)^{m+1}\overline{(z(t)-\epsilon \zeta_m^j)^m}}
 -\frac{m\overline{e^{\sqrt{-1}\theta}(z(t)^m-(\epsilon \zeta_m^j)^m)}}
 {(z(t)-\epsilon \zeta_m^j)^m\overline{(z(t)-\epsilon \zeta_m^j)^{m+1}}} 
 \notag \\
 &=
 \frac{2m } {|z(t)-\epsilon \zeta_m^j|^{2m}}
 \mathrm{Re}\left( -e^{\sqrt{-1}\theta}
 \frac{z(t)^m-(\epsilon \zeta_m^j)^m}{z(t)-\epsilon \zeta_m^j} \right)
 \notag \\
 &=\frac{2m\, \mathrm{Re}\left(- e^{\sqrt{-1}\theta} 
 \big(z(t)^{m-1}+\epsilon \zeta_m^j z(t)^{m-2}+\cdots+(\epsilon \zeta_m^j)^{m-2}z(t)
 +(\epsilon \zeta_m^j)^{m-1}\big) \right) }
 {|z(t)-\epsilon \zeta_m^j |^{2m}},
\end{align*}
we can see
\begin{align*}
 \frac{d}{dt}\frac{1}{ | z(t)-\epsilon \zeta_m^j |^{2m} } 
 &=
 \frac{2m\, \mathrm{Re}\left(e^{\sqrt{-1}(\theta-\pi)} 
 \big(z(t)^{m-1}+\epsilon \zeta_m^j z(t)^{m-2}+\cdots+(\epsilon \zeta_m^j)^{m-2}z(t)
 + ( \epsilon \zeta_m^j )^{m-1}\big) \right) }
 { | z(t)-\epsilon \zeta_m^j |^{2m} } \\
 &=
 \frac{2m}{ | z(t)-\epsilon \zeta_m^j |^{2m}} \;
 \mathrm{Re} \left(
 \sum_{l=0}^{m-1}\left(z(t)e^{\sqrt{-1}\frac{\theta-\pi}{m-1}}\right)^{m-1-l}
 \left( e^{\sqrt{-1}\frac{\theta-\pi}{m-1}}\epsilon \zeta_m^j \right)^l \right) \\
 &\geq
 \frac{2m}{ |z(t)-\epsilon \zeta_m^j |^{2m} }\left( \max\{|z(t)|,|\epsilon|\} \right)^{m-1}
 \cos\left( \frac{(m-1)\pi}{3m} \right) \\
 &\geq
 \frac{2m}{ |z(t)-\epsilon \zeta_m^j |^{2m} }
 \left( \frac{|z(t)-\epsilon\zeta_m^j|}{2} \right)^{m-1}
 \frac{1}{2} \\
 &=
 \frac{m}{2^{m-1}|z(t)-\epsilon\zeta_m^j|^{m+1}}
 \geq \frac{m}{4^m}>0.
\end{align*}
So we have
$\dfrac{1} { |z(t)-\epsilon \zeta_m^j |^{2m} } 
\geq \dfrac{m}{4^m}  t-C$
for some constant $C>0$.
Thus we have
\[
 \lim_{t\to\infty} z(t)=\epsilon \zeta_m^j.
\]
and
the flow of $v_{\theta,\epsilon}$ starting at any point of
$R^{(j)}_{\psi_0,1}\setminus(\varpi^{-1}(D)\cap R^{(j)}_{\psi_0,1})$
converges to $(\epsilon \zeta_m^j, s, e^{\sqrt{-1}\psi})\in \varpi^{-1}(D)$.

\noindent
{\bf Case 2.} $-\dfrac{\pi}{m}\leq \arg(z_0)-\dfrac{2j\pi}{m}-\psi_0 < 0$. \\
In this case, we take $\delta>0$ satisfying
$\delta<\dfrac{\pi}{24m}$ and put
\begin{equation} \label {equation:definition of theta case 2}
 \theta^{(i)}_{\psi_0,2}:=-\frac{2j(m-1)\pi}{m}-(m-1)\psi_0+\pi-\delta.
\end{equation}
If we simply write
$\theta:=\theta^{(i)}_{\psi_0,2}$,
then we have
\[
 -\frac{\pi}{m}-\frac{\delta}{m-1}\leq
 \arg(z_0)+\frac{\theta-\pi}{m-1}
 \leq -\frac{\delta}{m-1}.
\]
We take $\frac{1}{4}<s_1<1$ and $\eta>0$ similarly to Case 1 and put
\begin{align*}
 P^{(j)}_{\psi_0,2} 
 &=
 \left\{(z,(s,e^{\sqrt{-1}\psi}))\in \Delta\times \Big[0,\frac{1}{3}\Big)\times S^1
 \left|
 \begin{array}{l}
 \displaystyle -\dfrac{3\delta}{2m-2}<\psi+\dfrac{2j\pi}{m}+\frac{\theta-\pi}{m-1}
 <\frac{3\delta}{2m-2}, \; z\neq 0 \\
 \displaystyle
 \text{
 $-\dfrac{\pi}{2}-\dfrac{\pi}{3m} <
 \arg\left(e^{-\sqrt{-1}\frac{\pi}{3m}}-e^{\sqrt{-1}\frac{\theta-\pi}{m-1}}z\right)
 < -\dfrac{(2m+1)\delta}{m-1}$} \\
 \text{and $-\dfrac{\pi}{m}-\dfrac{2\delta}{m-1}<\arg(z)+\dfrac{\theta-\pi}{m-1}<\dfrac{\pi}{3m}$} 
 \end{array}
 \right\}\right. \\
 Q^{(j)}_{\psi_0,2}
 &:=\left\{ (z,s,e^{\sqrt{-1}\psi})\in\Delta\times\Big[0,\frac{1}{3}\Big)\times S^1
 \left|
 \begin{array}{l}
 e^{\sqrt{-1}\frac{\theta-\pi}{m-1}}z-\eta s e^{\sqrt{-1}\pi}\neq 0, \\
 -\dfrac{\pi}{6m} <\arg\left( e^{\sqrt{-1}\frac{\theta-\pi}{m-1}}z-\eta s e^{\sqrt{-1}\pi}\right)
 <\dfrac{\pi}{6m}, \\
 \text{$-\dfrac{3\delta}{2m-2}<\psi+\dfrac{2k\pi}{m}+\dfrac{\theta-\pi}{m-1}
 <\dfrac{3\delta}{2m-2}$ and} \\
 \text{$\dfrac{\pi}{3m}\leq\arg(z)+\dfrac{\theta-\pi}{m-1}
 \leq 2\pi-\dfrac{\pi}{m}-\dfrac{2\delta}{m-1}$ for $z\neq 0$}
 \end{array}
 \right\}\right.. \\
 R^{(j)}_{\psi_0,2}
 &:=
 P^{(j)}_{\psi_0,2} \cup Q^{(j)}_{\psi_0,2}.
\end{align*}
By the similar argument to Case 1, we can see that
$(z_0,s_0,e^{\sqrt{-1}\psi_0})\in R^{(j)}_{\psi_0,2}$
and
the flow $(z(t),s,e^{\sqrt{-1}\psi})_{t\geq 0}$ of $v_{\epsilon,\theta}$
starting at a point in
$R^{(j)}_{\psi_0,2}\setminus(\varpi^{-1}(D)\cap R^{(j)}_{\psi_0,2})$
satisfies
\[
 \lim_{t\to\infty} z(t)=\epsilon \zeta_m^j.
\]

If we put
\[
 U:= (\{0\}\times\{0\}\times S^1) \cup \bigcup R^{(j)}_{\psi_0,2},
\]
then we can see by the construction of $R^{(j)}_{\psi_0,2}$
that 
$\left\{z\in\Delta\left| \, |z|<\frac{1}{4}\right\}\right.
\times \big[ 0,\frac{1}{3}\big) \times S^1$
is contained in $U$.
So we can write
\[
 U=\left(\Big\{z\in\Delta \Big| \, |z|<\frac{1}{4}\Big\}
 \times \Big[ 0,\frac{1}{3}\Big) \times S^1\right)
 \cup  \bigcup R^{(j)}_{\psi_0,2}
\]
and we can see that
$U$ is an open neighborhood of $\{0\}\times\{0\}\times S^1$
in $\Delta\times[0,1)\times S^1$.
If we put
\[
 W^{(j)}_{\psi_0,2}:=
 R^{(j)}_{\psi_0,2}\setminus(\varpi^{-1}(D)\cap R^{(j)}_{\psi_0,2}),
\]
then we have an open covering
\[
 U\setminus(U\cap\varpi^{-1}(D))=
 \bigcup W^{(j)}_{\psi_0,2}.
\]
This covering satisfies the statement of the proposition.
\end{proof}

\subsection {Fundamental solution with an asymptotic property}

We use the same notations as in subsection \ref{subsection:flow}
Take a point
$p_0\in W^{(j)}_{\psi_0,\xi}$
and consider the holomorphic solution
$\big(z(\tau),s,e^{\sqrt{-1}\psi}\big)$ of the differential equation
\begin{equation} \label {equation:equation of holomorphic flow}
 \frac{d z(\tau)}{d\tau}=e^{\sqrt{-1}\theta}(z(\tau)^m-\epsilon^m)
\end{equation}
satisfying $\big(z(0),s,e^{\sqrt{-1}\psi}\big)=p_0$,
where $\epsilon=s e^{\sqrt{-1}\psi}$
and $\theta=\theta^{(j)}_{\psi_0,\xi}$.
If we take $t_1,u_1\in\mathbb{R}$
and if we fix $t_1+\sqrt{-1}u_1$ constant,
$\big(z(t+t_1+\sqrt{-1}u_1),s,e^{\sqrt{-1}\psi}\big)_{t\geq 0}$
coincides with the flow
$\big(z_{t_1+\sqrt{-1}u_1}(t),s,e^{\sqrt{-1}\psi}\big)$ of 
$v_{\epsilon,\theta}$
satisfying $z_{t_1+\sqrt{-1}u_1}(0)=z(t_1+\sqrt{-1}u_1)$.
So we can extend the solution $(z(\tau),s,e^{\sqrt{-1}\psi})$
by an analytic continuation to a holomorphic function
in $\tau$ on an open neighborhood of $\mathbb{R}_{\geq 0}$
whose image by $z(\tau)$ is an open neighborhood of the flow
of $v_{\epsilon,\theta}$ starting at the point $p_0$.
Note that we have
\[
 \lim_{t\to\infty} z(t+\sqrt{-1}u_1)=\epsilon \zeta_m^j
\]
and
$z_{t_1+\sqrt{-1}u_1}(t)=z(t+t_1+\sqrt{-1}u_1)=z_{\sqrt{-1}u_1}(t+t_1)$.

The following theorem is a weak unfolded analogue
of the existence theorem of fundamental solutions
with an asymptotic property \cite[Theorem 12.1]{Wasow}
in the irregular singular case.
It is an easy restricted case of
a more general theorem in \cite{Hurtubise-Lambert-Rousseau}
and \cite{Hurtubise-Rousseau},
which is one of the main tools in the unfolding theory
by Hurtubise, Lambert and Rousseau.

\begin{theorem}[{\cite[Theorem 5.3]{Hurtubise-Lambert-Rousseau}},
{\cite[Theorem 2.5]{Hurtubise-Rousseau}}]
\label {theorem:existence-of-fundamental-solutions}
Consider the linear differential equation
\begin{equation} \label {equation:unfolded-differential-equation}
 \begin{pmatrix} \dfrac{df_1}{dz} \\ \vdots \\ \dfrac{df_r}{dz} \end{pmatrix}
 =
 \frac{A(z,\epsilon,w)} { (z^m-\epsilon^m) }
 \begin{pmatrix} f_1 \\ \vdots \\ f_r \end{pmatrix}
\end{equation}
on the polydisk $\Delta\times\Delta \times\Delta^s$,
where $A(z,\epsilon,w)$ is an $r\times r$ matrix of holomorphic functions
in $(z,\epsilon,w)=(z,\epsilon,w_1,\ldots,w_s)\in\Delta\times\Delta\times\Delta^s$
such that 
\[
 A(z,\epsilon,w)-
 \begin{pmatrix}
   \nu_1(z,\epsilon,w) & \cdots & 0 \\
   \vdots & \ddots & \vdots \\
   0 & \cdots & \nu_r(z,\epsilon,w)
 \end{pmatrix}
 \in (z^m-\epsilon^m) M_r({\mathcal O}_{\Delta\times\Delta\times\Delta^s}^{hol}),
\]
where $\nu_1(z,\epsilon,w),\ldots,\nu_r(z,\epsilon,w)$ are polynomials in $z$
whose coefficients are holomorphic functions in $\epsilon,w$
and $\nu_1(\epsilon\zeta_m^j,\epsilon,w),\ldots,\nu_r(\epsilon\zeta_m^j,\epsilon,w)$
are mutually distinct for any fixed $j$, $\epsilon$ and $w$.
Then for a certain choice of the open covering
$\{ W^{(j)}_{\psi_0,\xi} \}$ of $U\setminus (\varpi^{-1}(D)\cap U)$
in Proposition \ref{proposition:open-covering},
there are an open covering
\[
 W^{(j)}_{\psi_0,\xi}\ \times \Delta^s =
 \bigcup_{p\in W^{(j)}_{\psi_0,\xi}}  S^{(j)}_{\psi_0,\xi,p},
\]
and a matrix
$Y_{\vartheta}(z)=
\left(y_1^{\vartheta}(z),\ldots,y_r^{\vartheta}(z)\right)$
of solutions on
$S_{\vartheta}:=S^{(j)}_{\psi_0,\xi,p}$
of the differential equation (\ref{equation:unfolded-differential-equation}),
that is,
\[
 \frac{d \, Y_{\vartheta}(z) }{dz}
 =\frac{A(z,\epsilon,w)}{z^m-\epsilon^m} \; Y_{\vartheta}(z)
\]
such that for the solution $z(\tau)$ of the holomorphic differential equation 
(\ref  {equation:equation of holomorphic flow}) with the initial value
$z(0)=p\in S^{(j)}_{\psi_0,\xi,p}$,
the limit
\begin{align*}
 &
 \lim_{t\to\infty}  
 Y_{\vartheta}(z(t+u))
 \exp \left( -
 \begin{pmatrix}
  \int_{t_0}^t \nu_1(z(t+u))e^{\sqrt{-1}\theta} dt & \cdots & 0 \\
  \vdots & \ddots & \vdots \\
  0 & \cdots & \int_{t_0}^t \nu_r(z(t+u))e^{\sqrt{-1}\theta}dt
 \end{pmatrix}
 \right) \\
 &=C^{\vartheta}_u(s,e^{\sqrt{-1}\psi},w) 
\end{align*}
along the flow $(z(t+u))_{t\geq 0}$
exists and the limit 
$C^{\vartheta}_u(s e^{\sqrt{-1}\psi} ,w)$
is a diagonal matrix of functions
continuous in $s,e^{\sqrt{-1}\psi},w,t_1,u_1$ 
and holomorphic in $w$ and $\epsilon=s e^{\sqrt{-1}\psi}\neq 0$.
\end{theorem}

\begin{proof}
For the solution $z(\tau)$ of the differential equation
(\ref {equation:equation of holomorphic flow})
with an initial value $(z(0),s,e^{\sqrt{-1}\psi})=p$ in 
$W^{(j)}_{\psi_0,\xi}$,
we consider $z(t+u)$
for $u\in\mathbb{C}$ with $|u|\ll 1$.
If we write $\epsilon:=s e^{\sqrt{-1}\psi}$,
the restriction of the differential equation
(\ref{equation:unfolded-differential-equation})
to the flow $z(t+u)$ of $v_{\epsilon,\theta}$ becomes
\begin{equation*}
 \begin{pmatrix} \dfrac{df_1(z(t+u),\epsilon,w)}{dt} \\ \vdots \\
 \dfrac{df_r(z(t+u),\epsilon,w)}{dt} \end{pmatrix}
 =
 e^{\sqrt{-1}\theta}A(z(t+u),\epsilon,w)
 \begin{pmatrix} f_1(z(t+u),\epsilon,w) \\ \vdots
  \\ f_r(z(t+u),\epsilon,w)
 \end{pmatrix}.
\end{equation*}
Since the flow $(z(t+u),s,e^{\sqrt{-1}\psi},w)$ is contained in
$W^{(j)}_{\psi_0,\xi}\times\Delta^s$,
we have
$\lim_{t\to\infty} z(t+u)= \epsilon \zeta_m^j$ and
\begin{align*}
 \lim_{t\to\infty} e^{\sqrt{-1}\theta} A(z(t+u),\epsilon,w)
 &=e^{\sqrt{-1}\theta} A(\epsilon \zeta_m^j,\epsilon, w) \\
 &=
 \begin{pmatrix}
  e^{\sqrt{-1}\theta}\nu_1(\epsilon \zeta_m^j,\epsilon,w) & \cdots & 0 \\
  \vdots & \ddots & \vdots \\
  0 & \cdots & e^{\sqrt{-1}\theta}\nu_r(\epsilon \zeta_m^j,\epsilon,w)
 \end{pmatrix}.
\end{align*}
We may assume by a suitable choice of $\delta>0$ for defining
$\theta=\theta^{(j)}_{\psi_0,\xi}$ in
(\ref {equation:definition of theta case 1}) and
(\ref {equation:definition of theta case 2}) 
that the real parts
$\mathrm{Re}\left( e^{\sqrt{-1}\theta}\nu_1(\epsilon \zeta_m^j,\epsilon,w) \right),
\ldots,\mathrm{Re}\left( e^{\sqrt{-1}\theta}\nu_r(\epsilon \zeta_m^j,\epsilon,w) \right)$
of the eigenvalues of the matrix $e^{\sqrt{-1}\theta} A(\epsilon \zeta_m^j,\epsilon, w)$
are mutually distinct.
Moreover we may assume by replacing the order of a holomorphic frame that
\begin{equation} \label {equation:assumption of ordering inequality}
 \mathrm{Re}\left( e^{\sqrt{-1}\theta}\nu_1(\epsilon \zeta_m^j,\epsilon,w) \right)
 <\cdots <
 \mathrm{Re}\left( e^{\sqrt{-1}\theta}\nu_r(\epsilon \zeta_m^j,\epsilon,w) \right)
\end{equation}
holds.
As in the proof of Proposition \ref{proposition:open-covering},
we have
\[
 -\frac{(m-1)\pi}{3m} \leq
 \arg\left(\left(e^{\sqrt{-1}\frac{\theta-\pi}{m-1}}z(t+u)\right)^{m-1}
 \right)
 \leq \frac{(m-1)\pi}{3m}
\]
for sufficiently large $t>0$.
So we have
\begin{align*} 
 \dfrac{d}{dt}|z(t+u)^m-\epsilon^m| 
 &=
 \frac{1}{2 (|z(t+u)^m-\epsilon^m|^2)^{\frac{1}{2}}}
 \dfrac{d}{dt}\left( (z(t+u)^m-\epsilon^m)
 \overline{(z(t+u)^m-\epsilon^m)} \right) \\
 &=
 \frac  { 2 \: \mathrm{Re}\left(
 m z(t+u)^{m-1}z'(t+u)
 \overline{(z(t+u)^m-\epsilon^m)} \right) } 
 {2|z(t+u)^m-\epsilon^m|} \\
 &=
 \frac  { \mathrm{Re}\left(
 m e^{\sqrt{-1}\theta}z(t+u)^{m-1}(z(t+u)^m-\epsilon^m)
 \overline{(z(t+u)^m-\epsilon^m)} \right) } 
 {|z(t+u)^m-\epsilon^m|} \\
 &=
 \mathrm{Re} \left( -m \left(e^{\sqrt{-1}\frac{\theta-\pi}{m-1}}z(t+u)\right)^{m-1} \right)
 \left | z(t+u)^m-\epsilon^m \right |  \\
 &\leq
 -m\cos\left(\frac{(m-1)\pi}{3\pi}\right)
 \left|z(t+u)^{m-1}\right|  \left | z(t+u)^m-\epsilon^m \right | \\
 &\leq
 -\frac{m}{2}\left | z(t+u)^m-\epsilon^m \right |^{\frac{m-1}{m}}
 \left | z(t+u)^m-\epsilon^m \right |
\end{align*}
for sufficiently large $t>0$, from which we have
\begin{align*}
 \dfrac{d}{dt}
 \left( \left | z(t+u)^m-\epsilon^m \right |^{-\frac{m-1}{m}} \right)
 &=
 -\frac{m-1}{m}  \left | z(t+u)^m-\epsilon^m \right |^{-\frac{m-1}{m}-1}
 \dfrac{d}{dt}|z(t+u)^m-\epsilon^m|  \\
 &\geq
 \frac{m-1}{2}.
\end{align*}
So there exists a constant $C>0$ such that
\[
 \left| z(t+u)^m-\epsilon^m \right|^{-\frac{m-1}{m}}
 \geq \frac{m-1}{2}t-C
\]
holds for sufficiently large $t>0$.
If we write $\nu_k=\sum_{l=0}^q b_l(\epsilon,w)z^l$, we have
\[
 \frac{d}{dt} \ 
 e^{\sqrt{-1}\theta}\nu_k(z(t+u),\epsilon,w)
  =e^{\sqrt{-1}\theta}\sum_{l=0}^q l \: b_l(\epsilon,w)z(t+u)^{l-1}
  e^{\sqrt{-1}\theta}(z(t+u)^m-\epsilon^m).
\]
So there is a constant $C'>0$ satisfying
$\left| \dfrac{d}{dt} \ 
 e^{\sqrt{-1}\theta}\nu_j(z(t+u),\epsilon,w) \right| 
 \leq C' \left| z(t+u)^m-\epsilon^m \right|$
 and
\begin{align*}
 \int_{a_0}^{\infty} \left| \frac{d}{dt} \ 
 e^{\sqrt{-1}\theta}\nu_k(z(t+u),\epsilon,w) \right| dt 
 &\leq
 C' \int_{a_0}^{\infty} \left| z(t+u)^m-\epsilon^m \right| dt \\
 &\leq
 C'\int_{a_0}^{\infty} \left( \frac{m-1}{2}t-C \right)^{-1-\frac{1}{m-1}} dt < \infty
\end{align*}
for a reference point $a_0\in\mathbb{R}_{>0}$.
Similarly we have
\[
 \int_{a_0}^{\infty} \left\| A(z(t+u),\epsilon,w)-
 \begin{pmatrix}
   \nu_1(z(t+u),\epsilon,w) & \cdots & 0 \\
   \vdots & \ddots & \vdots \\
   0 & \cdots & \nu_r(z(t+u),\epsilon,w)
 \end{pmatrix} \right\| dt <\infty
\]
because the absolute values of the entries of the matrix
\[
 A(z(t+u),\epsilon,w)-
 \begin{pmatrix}
   \nu_1(z(t+u),\epsilon,w) & \cdots & 0 \\
   \vdots & \ddots & \vdots \\
   0 & \cdots & \nu_r(z(t+u),\epsilon,w)
 \end{pmatrix}
\]
are bounded by $C'' \left| z(t+u)^m-\epsilon^m \right|$
for some constant $C''>0$.
Thus, by the theorem of Levinson (\cite[Theorem 1]{Levinson}),
there are $t_0>0$ and a matrix
\[
 Y^u(t,s,e^{\sqrt{-1}\psi},w)
 =\left( y^u_1(t,s,e^{\sqrt{-1}\psi},w) \, , \, \ldots \, , \,
 y^u_r(t,s,e^{\sqrt{-1}\psi},w) \right)
\]
of solutions
$y^u_1(t,s,e^{\sqrt{-1}\psi},w),\ldots,y^u_r(t,s,e^{\sqrt{-1}\psi},w)$
of the differential equation
\begin{equation} \label {equation:restricted differential equation}
 \frac{d y(t)}{dt}= e^{\sqrt{-1}\theta} A(z(t+u),\epsilon,w) \; y(t)
\end{equation}
defined for $t>t_0-b$ for some $b>0$, which satisfies
\begin{align}
 &\lim_{t\to\infty}
 Y^u(t,s,e^{\sqrt{-1}\psi},w)
  \; \exp \left( -
 \begin{pmatrix}
  \int_{t_0}^t \nu_1(z(t+u))e^{\sqrt{-1}\theta}dt & \cdots & 0 \\
  \vdots & \ddots & \vdots \\
  0 & \cdots & \int_{t_0}^t \nu_r(z(t+u))e^{\sqrt{-1}\theta}dt
 \end{pmatrix} \right) 
  \label {equation:asymptotic limit of modified fundamental solution} \\
 &=
 C_u(s,e^{\sqrt{-1}\psi},w)
 =
 \begin{pmatrix}
  c_1(u) & \cdots & 0 \\
  \vdots & \ddots & \vdots \\
  0 & \cdots & c_r(u)
 \end{pmatrix}
 \notag 
\end{align}
with $C_u(\epsilon,w)$
constant in $z$
satisfying
\[
 A(\epsilon \zeta_m^j,\epsilon,w) \; C_u(s,e^{\sqrt{-1}\psi},w)
 =
 C_u(s,e^{\sqrt{-1}\psi},w) \;
 \begin{pmatrix}
  \nu_1(\epsilon \zeta_m^j,\epsilon,w) & \cdots & 0 \\
  \vdots & \ddots & \vdots \\
  0 & \cdots & \nu_r(\epsilon \zeta_m^j,\epsilon,w)
 \end{pmatrix}.
\]
Notice that $y^u_k(t,s,e^{\sqrt{-1}\psi},w)$
is constructed in \cite{Levinson} by applying 
an infinite sum and integrations of the form $\int_a^t$ or $\int_t^{\infty}$
to given functions in $t,s,e^{\sqrt{-1}\psi},w,u$ constructed from $A(z,\epsilon,w)$.
So we can see by their construction in \cite{Levinson}
that the solutions
$y^u_k(t,s,e^{\sqrt{-1}\psi},w)$
are functions continuous in $s,e^{\sqrt{-1}\psi},w,u$
and holomorphic in $w,u$ and $\epsilon\neq 0$.
Furthermore,
$C_u(s,e^{\sqrt{-1}\psi}, w)$
is a matrix of functions continuous in
$s,e^{\sqrt{-1}\psi},w,u$
and holomorphic  in $w$, $u$ and $\epsilon\neq 0$.
Since $A(\epsilon \zeta_m^j,\epsilon,w)$ is a diagonal matrix
with the distinct eigenvalues 
by the assumption,
$C_u(\epsilon,w)$
becomes a diagonal matrix.

By the fundamental theorem of ordinary linear differential equations,
there exists a fundamental solution
\[
 Y_{\vartheta}(z,s,e^{\sqrt{-1}\psi},w)
 =\left( y_1^{\vartheta}(z,s,e^{\sqrt{-1}\psi},w),\ldots,
 y_r^{\vartheta}(z,s,e^{\sqrt{-1}\psi},w) \right)
\]
of the differential equation
(\ref {equation:unfolded-differential-equation}),
that is to say, 
\[
 \frac{d Y_{\vartheta}}{ dz}=\frac{A(z)}{z^m-s^me^{\sqrt{-1}m\psi}}Y_{\vartheta}
\]
in a neighborhood of $(z(t_0),s,e^{\sqrt{-1}\psi},w)$ which satisfies the initial condition
\[
 Y_{\vartheta}(z(t_0),s,e^{\sqrt{-1}\psi},w)=Y^0(t_0,s,e^{\sqrt{-1}\psi},w).
 \]
Here the suffix $\vartheta$ means the data
$p,j,\psi_0,\xi$.
Since the solutions of the linear differential equation
(\ref {equation:unfolded-differential-equation})
form a local system on $U\setminus(D\cap U)$,
we can extend $Y_{\vartheta}(z)$ to a matrix of holomorphic functions
in a neighborhood of $\{z(t)|t\geq t_0\}$
by an analytic continuation.
We fix $u\in\mathbb{C}$ close to the origin $0$.
Since both 
$Y_{\vartheta}(z(t+u),s,e^{\sqrt{-1}\psi},w):=
 \Big(
  y^{\vartheta}_1(z(t+u),s,e^{\sqrt{-1}\psi},w) \, ,  \ldots ,
  y^{\vartheta}_r(z(t+u),s,e^{\sqrt{-1}\psi},w)$
and $Y^u(t,s,e^{\sqrt{-1}\psi},w)$
satisfy the same linear differential equation
\[
 \frac{dY}{dt}=e^{\sqrt{-1}\theta}A(z(t+u)) \: Y,
\]
there is a matrix
$P(u)$
of functions continuous in
$s,e^{\sqrt{-1}\psi},w,u$
and holomorphic in $w$ and $u$ satisfying
\[
 Y_{\vartheta}(z(t+u),s,e^{\sqrt{-1}\psi},w)
 =
 Y^u(t,s,e^{\sqrt{-1}\psi},w) \: P(u)
\]
for $t$ close to $t_0$.
We put
$\Lambda_k(t,u):=
\exp\left( \int_{t_0}^t \nu_k(z(t+u))e^{\sqrt{-1}\theta}dt \right)$.
By (\ref {equation:assumption of ordering inequality}),
$\lim_{t\to\infty}\Lambda_k(t)^{-1}\Lambda_{k'}(t)$
is divergent if $k<k'$.
If $u\in\mathbb{R}$ is a real number,
we can see by the property 
(\ref {equation:asymptotic limit of modified fundamental solution})
for $u=0$ that 
\begin{align*}
 &\lim_{t\to\infty} Y_{\vartheta}(z(t+u))
 \; \exp  
 \begin{pmatrix}
 -\int_{t_0}^t  \nu_1(z(t+u)) e^{\sqrt{-1}\theta} dt & \cdots & 0 \\
  \vdots & \ddots & \vdots \\
  0 & \cdots &  -\int_{t_0}^t \nu_r(z(t+u)) e^{\sqrt{-1}\theta} dt
 \end{pmatrix}     \\
 &=
 \lim_{t\to\infty} Y_{\vartheta}(z(t+u))
 \; \exp  
 \begin{pmatrix}
  -\int_{t_0+u}^{t+u}  \nu_1(z(t')) e^{\sqrt{-1}\theta} dt' & \cdots & 0 \\
  \vdots & \ddots & \vdots \\
  0 & \cdots &  -\int_{t_0+u}^{t+u} \nu_r(z(t')) e^{\sqrt{-1}\theta} dt'
 \end{pmatrix}     \\
 &=
 C_u(s,e^{\sqrt{-1}\psi},w) 
 \;
 \begin{pmatrix}
   \exp  \left(\int_{t_0}^{t_0+u}  \nu_1(z(t')) e^{\sqrt{-1}\theta} dt'\right) & \cdots & 0 \\
  \vdots & \ddots & \vdots \\
  0 & \cdots &   \exp  \left(\int_{t_0}^{t_0+u} \nu_r(z(t')) e^{\sqrt{-1}\theta} dt'\right)
 \end{pmatrix}   
\end{align*}
is convergent and its limit is a diagonal matrix.
If we put
\[
 Y^u(t)=
 \begin{pmatrix} y^u_1(t),\ldots, y^u_r(t) \end{pmatrix},
 \hspace{20pt}
 P(u)
 =
 \begin{pmatrix}
  p_{1,1}(u) & \cdots & p_{1,r}(u) \\
  \vdots & \ddots & \vdots \\
  p_{r,1}(u) & \cdots & p_{r,r}(u)
 \end{pmatrix},
\]
then, for $u\in\mathbb{R}$,
\begin{align*}
 &Y_{\vartheta}(z(t+u))
 \begin{pmatrix}
 \Lambda_1(t)^{-1} & \cdots & 0 \\
  \vdots & \ddots & \vdots \\
  0 & \cdots & \Lambda_r(t)^{-1}
 \end{pmatrix}
 =
 Y^u(t) \: P(u) \: 
 \begin{pmatrix}
 \Lambda_1(t)^{-1} & \cdots & 0 \\
  \vdots & \ddots & \vdots \\
  0 & \cdots & \Lambda_r(t)^{-1}
 \end{pmatrix} \\
 &=
 \begin{pmatrix} y^u_1(t),\ldots, y^u_r(t) \end{pmatrix}
 \begin{pmatrix}
  p_{1,1}(u) & \cdots & p_{1,r}(u) \\
  \vdots & \ddots & \vdots \\
  p_{r,1}(u) & \cdots & p_{r,r}(u)
 \end{pmatrix}
 \begin{pmatrix}
 \Lambda_1(t)^{-1} & \cdots & 0 \\
  \vdots & \ddots & \vdots \\
  0 & \cdots & \Lambda_r(t)^{-1}
 \end{pmatrix} \\
 &=
 \left( \sum_{k=1}^r p_{k,1}(u)\Lambda_1(t)^{-1}y^u_k(t),
 \ldots,  \sum_{k=1}^r p_{k,r}(u)\Lambda_r(t)^{-1}y^u_k(t) \right)
\end{align*}
is bounded when $t\to\infty$.
Note that
\[
 \Lambda_l(t)^{-1}y^u_k(t)=
 \left(\Lambda_l(t)^{-1}\Lambda_k(t)\right) \;
 \left(\Lambda_k(t)^{-1} y^u_k(t) \right)
\]
is divergent for $l<k$ when $t\to\infty$,
because $\lim_{t\to\infty}  \Lambda_l(t)^{-1}\Lambda_k(t)$
is divergent
and $\lim_{t\to\infty} \Lambda_k(t)^{-1} y^u_k(t) =c_k(u)e_k\neq 0$.
So we should have $p_{k,l}(u)=0$ for $k>l$ and $u\in\mathbb{R}$ with $|u|\ll 1$.
Since $p_{k,l}(u)$ is holomorphic in $u$, we have $p_{k,l}(u)=0$
for $u\in\mathbb{C}$ with $|u|\ll 1$.
In other words, $P(u)$ is an upper triangular matrix of holomorphic functions in $u$.
Then we have, for $u\in\mathbb{C}$ with $|u| \ll 1$, that
\begin{align*}
 &\lim_{t\to\infty} Y_{\vartheta}(z(t+u))
 \; \exp  
 \begin{pmatrix}
 -\int_{t_0}^t  \nu_1(z(s+u)) e^{\sqrt{-1}\theta} ds & \cdots & 0 \\
  \vdots & \ddots & \vdots \\
  0 & \cdots &  -\int_{t_0}^t \nu_r(z(s+u)) e^{\sqrt{-1}\theta} ds
 \end{pmatrix}     \\
 &=
 \lim_{t\to\infty}
 \begin{pmatrix} y^u_1(t),\ldots, y^u_r(t) \end{pmatrix}
 \begin{pmatrix}
  p_{1,1}(u) & \cdots & p_{1,r}(u) \\
  \vdots & \ddots & \vdots \\
  0 & \cdots & p_{r,r}(u)
 \end{pmatrix}
 \begin{pmatrix}
 \Lambda_1(t)^{-1} & \cdots & 0 \\
  \vdots & \ddots & \vdots \\
  0 & \cdots & \Lambda_r(t)^{-1}
 \end{pmatrix} 
\end{align*}
converges to a diagonal matrix
$C^{\vartheta}_u(s,e^{\sqrt{-1}\psi},w)$.
\end{proof}

\begin{remark} \rm
 Although a formal solution transforming
 an unfolded linear differential equation to
 a normal form is given in \cite[Theorem 3.2]{Hurtubise-Lambert-Rousseau},
 we cannot expect to construct a fundamental solution
 of (\ref{equation:unfolded-differential-equation})
 with an asymptotic property with respect to the formal solution
 as in the irregular singular case
 (\cite[Theorem 12.1]{Wasow}).
\end{remark}

\section{Construction of a local horizontal lift}
\label {section:local horizontal lift}

In this section,
we construct an integrable connection which is  a first order
infinitesimal extension of a given local relative connection.
We call this extension a local horizontal lift,
or a block of local horizontal lifts in section \ref {section:construction of unfolding},
which is a key part in the construction  of an unfolding
of the unramified irregular singular generalized isomonodromic deformation.
A basic idea in this section is to extend a local connection to a global
connection on $\mathbb{P}^1$
with regular singularity at $\infty$.
Unfortunately, our construction of a local horizontal lift
is not canonical but it is systematically determined.
So it enables us to construct a non-canonical
global horizontal lift in section \ref  {section:construction of unfolding},
which induces an unfolded
generalized isomonodromic deformation.

\subsection {Extension of a local connection to a global connection
on $\mathbb{P}^1$}
\label {subsection:global connection on P^1}

Consider the divisor
\[
 D:= \{ (z,\epsilon,w)\in\Delta\times\Delta\times\Delta^s
 | z^m-\epsilon^m=0 \}
\]
on the polydisk $\Delta\times \Delta\times\Delta^s$,
where $\Delta=\{z\in\mathbb{C} \, | \, |z|<1\}$.
If we put
\[
 D_j:=
 \left\{ (z,\epsilon,w)\in\Delta\times\Delta\times\Delta^s
 \left| z-\epsilon\zeta_m^j=0 \right\}\right.
\]
for $j=1,\ldots, m$ with
$\zeta_m=\exp(\frac{2\pi\sqrt{-1}}{m})$,
then we can write
\[
 D=D_1+\cdots+D_m
\]
as an effective divisor on $\Delta\times\Delta\times\Delta^s$.
We consider a family of intervals
\[
 \Gamma_{\Delta,j}=
 \left.\left\{ (s \zeta_m^j \epsilon , \epsilon, w)
 \in \Delta\times\Delta\times\Delta^s \right|
 0\leq s\leq 1 \right\}
\]
which join the origin $0$ and $\zeta_m^j\epsilon$ and consider their union
\[
 \Gamma_{\Delta}:=\bigcup_{j=1}^m \Gamma_{\Delta,j}.
\]

We consider the embedding
$\Delta\times\Delta\times\Delta^s
\hookrightarrow\mathbb{P}^1\times\Delta\times\Delta^s=
\mathbb{P}^1_{\Delta\times\Delta^s}$
and regard $D$ as an effective divisor on
$\mathbb{P}^1\times\Delta\times\Delta^s$.

We prepare a notation of diagonal matrix.
\begin{notation}
\rm
We denote the diagonal matrix whose $(k,k)$ entry is $a_k$ by
$\mathrm{Diag}_{(a_k)}$;
\[
 \mathrm{Diag}_{(a_k)}=
 \begin{pmatrix}
  a_1 & \cdots & 0 \\
  \vdots & \ddots & \vdots \\
  0 & \cdots & a_r
 \end{pmatrix}.
\]
\end{notation}

Take mutually distinct complex numbers
$\mu_1,\ldots,\mu_r$
and a polynomial
$\nu(T)\in{\mathcal O}_D[T]$
given by
\begin{align} \label {equation:expression of nu}
 \nu(T) 
 &=
 \sum_{l=0}^{r-1}\Big(  \sum_{j=0}^{m-1} c_{l,j} z^j \Big) T^l 
\end{align}
with $c_{l,j} \in {\mathcal O}_{\Delta\times\Delta^s}$
such that
$\nu(\mu_1)|_p,\ldots,\nu(\mu_r)|_p$ are distinct complex numbers
at any point $p\in D$.

We denote the closed interval
$\{t\in\mathbb{R} \, | \, 0\leq t\leq 1\}$
by $[0,1]$.
We take a continuous map
\[
 \tilde{\gamma}\colon [0,1]\times \Delta\times\Delta^s
 \longrightarrow \Delta\times\Delta\times\Delta^s
\]
and an open subset $W\subset \Delta\times\Delta\times\Delta^s$
such that 
$\tilde{\gamma}(0,b)=\tilde{\gamma}(1,b)$
for any $b\in \Delta\times\Delta^s$,
each fiber $W_b$ over $b\in \Delta\times\Delta^s$
is a disk containing $D_b$ and that the boundary
$\partial W_b$ coincides with the image
$\tilde{\gamma}([0,1]\times\{b\})$.

Let
\begin{equation} \label {equation: local connection}
 \nabla_{\Delta}\colon
 {\mathcal O}_{ \Delta\times \Delta\times\Delta^s }^{\oplus r}
 \ni
 \begin{pmatrix} f_1 \\ \vdots \\ f_r \end{pmatrix}
 \mapsto
 \begin{pmatrix} d f_1 \\ \vdots \\  d f_r \end{pmatrix}
 +
 A(z,\epsilon,w) \dfrac{dz}{z^m-\epsilon^m}
 \begin{pmatrix} f_1 \\ \vdots \\ f_r \end{pmatrix}
 \in
 \Omega^1_{\Delta\times\Delta\times\Delta^s/\Delta\times\Delta^s}(D) ^{\oplus r} 
\end{equation}
be a relative connection on
$\Delta\times\Delta\times\Delta^s$ over $\Delta\times\Delta^s$
satisfying
\begin{equation} \label {equation:condition of local exponent}
 A(z,\epsilon,w)\big|_D
 =
 \Diag_{(\nu(\mu_k))}\big|_D
 =
 \left.
 \begin{pmatrix}
  \nu(\mu_1) & \cdots & 0 \\
  \vdots & \ddots & \vdots \\
  0 & \cdots & \nu(\mu_r)
 \end{pmatrix}
 \right|_D.
\end{equation}
For each point $b\in\Delta\times\Delta^s$,
we consider the restriction
$\nabla_{\Delta_b}:=\nabla_{\Delta}|_{\Delta\times\{b\}}$
and its associated connection
\[
 \nabla_{\Delta_b}^{\dag}\colon
 {\mathcal End} \big( {\mathcal O}_{\Delta\times\{b\}}^{\oplus r} \big)
 \ni u \mapsto \nabla_{\Delta_b} \circ u- u \circ \nabla_{\Delta_b} 
 \in
 {\mathcal End} \big( {\mathcal O}_{\Delta\times\{b\}}^{\oplus r} \big)
 \otimes \Omega^1_{\Delta\times\{b\}} (D_b).
\]

We assume the following condition for $\nabla_{\Delta}$:
\begin{assumption} \label {assumption:local irreducible}
\begin{itemize}
\item[(i)] the monodromy of $\nabla_{\Delta_b}$ along $\tilde{\gamma}_b$
has a diagonal representation matrix of holomorphic functions
over $\Delta\times\Delta^s$ with $r$ distinct eigenvalues
for any $b\in\Delta\times\Delta^s$
and
\item[(ii)]
$H^0\big(\Delta\times\{b\}, \ker \big(\nabla_{\Delta_b}^{\dag}\big)\big)
=\mathbb{C}$
for each $b\in \Delta\times\Delta^s$.
\end{itemize}
\end{assumption}

\begin{proposition} \label {prop:normalization of local connection via extension to P^1}
There exist an open neighborhood
${\mathcal V}$ of $(0,0)$ in $\Delta\times\Delta^s$
and a relative connection
\[
 \nabla^{\mathbb{P}^1} \colon
 \big({\mathcal O}^{hol}_{\mathbb{P}^1\times{\mathcal V}}\big)^{\oplus r}
 \longrightarrow
 \big({\mathcal O}^{hol}_{\mathbb{P}^1\times{\mathcal V}}\big)^{\oplus r}
 \otimes \Omega^1_{\mathbb{P}^1\times{\mathcal V}/{\mathcal V}}
 \big( (D\cap(\Delta\times{\mathcal V}))\cup
 (\{\infty\}\times {\mathcal V}) \big)^{hol}
\]
on $\mathbb{P}^1\times{\mathcal V}$ over ${\mathcal V}$
admitting poles along
$(D\cap(\Delta\times{\mathcal V}))\cup
(\{\infty\}\times {\mathcal V})$
such that the restriction
$\nabla^{\mathbb{P}^1}|_{\Delta\times {\mathcal V}}$ is isomorphic to
the restriction $\nabla_{\Delta}|_{\Delta\times{\mathcal V}}$
of $\nabla_{\Delta}$
in (\ref {equation: local connection}).
\end{proposition}

\begin{proof}

Let $\mathrm{Mon}_{\tilde{\gamma}}(\nabla_{\Delta})$
be the monodromy matrix of $\nabla_{\Delta}$ along $\tilde{\gamma}$
with respect to a local basis of $\ker \nabla_{\Delta}$.
We can take a contractible open subset $W'\subset\Delta\times\Delta\times\Delta^s$
with $\overline{W'}\subset W$
such that the fiber $\overline{W'}_b$
is a closed disk for each $b\in \Delta\times\Delta^s$
and that the fundamental group
$\pi_1((\Delta\times\Delta\times \Delta^s)\setminus\overline{W'},*)$
is isomorphic to $\mathbb{Z}$ which is generated by $\tilde{\gamma}$.
We can take a regular singular relative connection
\[
 \nabla_{\infty}\colon
 \big(
 {\mathcal O}_{\mathbb{P}^1\times \Delta\times\Delta^s \setminus \overline{W'}}^{hol}
 \big)^{\oplus r}
 \longrightarrow
 \big(
 {\mathcal O}_{\mathbb{P}^1\times \Delta\times\Delta^s \setminus \overline{W'}}^{hol}
 \big)^{\oplus r}
 \otimes
 \Omega^1_{(\mathbb{P}^1\times \Delta\times\Delta^s \setminus \overline{W'})
 / \Delta\times\Delta^s }
 (\{\infty\}\times \Delta\times\Delta^s)
\]
such that  the monodromy of $\nabla_{\infty}$
along $\tilde{\gamma}$ 
is given by $\mathrm{Mon}_{\tilde{\gamma}}(\nabla_{\Delta})$
and that
the set of eigenvalues of
$\res_{(\infty,b')}\big(\nabla_{\infty}\big|_
{(\mathbb{P}^1\times\{b'\})\setminus(\overline{W}\cap(\mathbb{P}^1\times\{b'\}))}\big)$
is contained in
$\{ z\in\mathbb{C} \, | \, 0\leq \mathrm{Re}(z)<1\}$
for any $b'\in \Delta\times\Delta^s$.
Note that
$\big(\big(
{\mathcal O}^{hol}_{(\Delta\times\Delta\times\Delta^s)\setminus\overline{W'}}\big)^{\oplus r},
\nabla_{\Delta}\big|_{(\Delta\times\Delta\times\Delta^s)\setminus\overline{W'}}\big)$
and
$\big(\big(
{\mathcal O}^{hol}_{(\Delta\times\Delta\times\Delta^s)\setminus\overline{W'}}\big)^{\oplus r},
\nabla_{\infty}\big|_{(\Delta\times\Delta\times\Delta^s)\setminus\overline{W'}}\big)$
are isomorphic,
because their corresponding representations of the fundamental group
$\pi_1\big((\Delta\times\Delta\times\Delta^s)\setminus\overline{W'},*\big)
\cong\mathbb{Z}$
are given by the
same monodromy matrix $\mathrm{Mon}_{\tilde{\gamma}}(\nabla_{\Delta})$.
So we can patch $\nabla_{\infty}$,
$\nabla_{\Delta}|_{\Delta\times \Delta\times\Delta^s}$
and obtain a global relative connection
\[
 \nabla_0\colon E_0 \longrightarrow 
 E_0 \otimes
 \Omega^1_{(\mathbb{P}^1\times \Delta\times\Delta^s)/\Delta\times\Delta^s}
 \big( D\cup \big( \{\infty\}\times \Delta\times\Delta^s \big) \big)
\]
on $\mathbb{P}^1\times \Delta\times\Delta^s$ over $\Delta\times\Delta^s$.
We can write
\[
 E_0|_{\mathbb{P}^1\times\{(0,0)\}}\cong
 \bigoplus_{k=1}^r {\mathcal O}_{\mathbb{P}^1}(a_k)
\]
with $a_1\geq a_2\geq \cdots\geq a_r$.
Assume that $a_1>a_r$.
For some choice of $k$, the projection
\begin{align*}
 \psi'_0\colon E_0 \longrightarrow
 E_0|_ { \{\infty\}\times \Delta\times\Delta^s}
 &=
 \bigoplus_{k=1}^r \ker \Big(\nabla_0|_{\{\infty\}\times \Delta\times\Delta^s}
 -\nu(\mu_k)\frac{dz}{z^m-\epsilon^m}\Big|_{\{\infty\}\times\Delta\times\Delta^s} \Big)
 \\
 &\longrightarrow
 \ker \Big(\nabla_0|_{\{\infty\}\times \Delta\times\Delta^s}
 -\nu(\mu_k)\frac{dz}{z^m-\epsilon^m}\Big|_{\{\infty\}\times\Delta\times\Delta^s} \Big)
\end{align*}
satisfies
$\psi'_0 |_{\{(\infty ,(0,0))\}}({\mathcal O}_{\mathbb{P}^1}(a_1))
=\ker \Big(\nabla_0|_{\{(\infty,(0,0))\}}
 -\nu(\mu_k) \dfrac{dz}{z^m-\epsilon^m}\Big|_{\{(\infty,(0,0))\}} \Big)
\cong {\mathcal O}_{\{(\infty ,(0,0))\}}$.
Then there is an open neighborhood ${\mathcal V}$ of $(0,0)$ in
$\Delta\times\Delta^s$ such that
\[
 \psi_0:=\psi'_0|_{\mathbb{P}^1\times{\mathcal V}} \colon
 E_0|_{\mathbb{P}^1\times{\mathcal V}} \longrightarrow
 \ker \Big(\nabla_0|_{\{\infty\}\times {\mathcal V}}
  -\nu(\mu_k)\frac{dz}{z^m-\epsilon^m}
 \Big|_{\{\infty\}\times{\mathcal V}} \Big)
\]
is surjective.
If we put
$(E_1 , \nabla_1):=(\ker \psi_0,\nabla_0|_{\ker \psi_0})$,
then $\nabla_1$ is a relative connection on 
$\mathbb{P}^1\times {\mathcal V}$ over ${\mathcal V}$
admitting poles along 
$(D\cap(\Delta\times{\mathcal V}))\cup
(\{\infty\}\times {\mathcal V})$
and  we have
\[
 E_1|_{\mathbb{P}^1\times \{(0,0)\}}\cong
 {\mathcal O}_{\mathbb{P}^1}(a_1-1)\oplus
 \bigoplus_{k=2}^{r}{\mathcal O}_{\mathbb{P}^1}(a_k).
\]
Similarly we can choose a surjection
$\psi_1\colon E_1\longrightarrow {\mathcal O}_{\{\infty\}\times{\mathcal V}}$
after shrinking ${\mathcal V}$
such that $\ker\psi_1$ is preserved by $\nabla_1$ and that
$\psi_1({\mathcal O}(\tilde{a}_1))={\mathcal O}_{\{\infty\}\times{\mathcal V}}$
for $\tilde{a}_1:=\max\{a_1-1,a_2\}$.
Then we put $(E_2,\nabla_2):=(\ker\psi_1,\nabla_1|_{\ker\psi_1})$.
Repeating this procedure, we finally obtain $(E_N,\nabla_N)$ such that
$E_N|_{\mathbb{P}^1\times {\mathcal V}}\cong 
{\mathcal O}_{\mathbb{P}^1_{{\mathcal V}}}(N_0)^{\oplus r}$.
So the connection
$\nabla_N\otimes{\mathcal O}(-N_0)$
satisfies the condition of the proposition.
\end{proof}

\subsection {The construction of  a local horizontal lift}
\label {subsection:local horizontal lift}

We use the same notations as in subsection \ref {subsection:global connection on P^1}.
We consider the non-reduced analytic space
$\mathbb{P}^1\times\Delta\times\Delta^s\times\Spec\mathbb{C}[h]/(h^2)$.
For an analytic open subset $U\subset\mathbb{P}^1\times\Delta\times\Delta^s$,
we denote by
$U[\bar{h}]$
the analytic open subspace of
$\mathbb{P}^1\times\Delta\times\Delta^s\times\Spec\mathbb{C}[h]/(h^2)$
whose underlying set of points coincides with $U$.
In this subsection, we will construct an extension of
the relative connection $\nabla^{\mathbb{P}^1}$
constructed in Proposition \ref {prop:normalization of local connection via extension to P^1}
to an integrable connection on
$\mathbb{P}^1\times{\mathcal V}[\bar{h}]$
over ${\mathcal V}$.
This produces a block of local horizontal lifts defined in
Definition \ref {def:block of local horizontal lift},
which is a key concept in the construction of a global horizontal lift
in subsection \ref {subsection:unfolded global horizontal lift}.

Recall that the sheaf of holomorphic differential forms
$\big( \Omega^1_{ \left. \left(\mathbb{P}^1_{\Delta\times\Delta^s}
\setminus \Gamma_{\Delta} \right)[\bar{h}]
\right/\Delta\times\Delta^s} \big)^{hol}$
on
$\big(\mathbb{P}^1_{\Delta\times\Delta^s}\setminus
\Gamma_{\Delta} \big)[\bar{h}]$
is given by
\[
 \big(\Omega^1_{\left.\left(\mathbb{P}^1_{\Delta\times\Delta^s}\setminus
 \Gamma_{\Delta} \right)[\bar{h}] \right/\Delta\times\Delta^s} \big)^{hol}
 =I_{\Delta_{ \left.\left( \mathbb{P}^1_{\Delta\times\Delta^s}
 \setminus \Gamma_{\Delta} \right)[\bar{h}]  \right/\Delta\times\Delta^s}} ^{hol}
 \big/
 \big(I_{\Delta_{ \left. \left( \mathbb{P}^1_{\Delta\times\Delta^s}
 \setminus \Gamma_{\Delta} \right)[\bar{h}] 
 \right/\Delta\times\Delta^s}}^{hol} \big)^2,
\]
where
$I_{\Delta_{ \left.\left( \mathbb{P}^1_{\Delta\times\Delta^s}
 \setminus \Gamma_{\Delta} \right)[\bar{h}]  \right/\Delta\times\Delta^s}} ^{hol}$
is the ideal sheaf of 
${\mathcal O}_{\left( \mathbb{P}^1_{\Delta\times\Delta^s}
\setminus \Gamma_{\Delta} \right)[\bar{h}]   \times_{\Delta\times\Delta^s}
 \left( \mathbb{P}^1_{\Delta\times\Delta^s}\setminus
\Gamma_{\Delta} \right)[\bar{h}]  }^{hol}$
which defines the diagonal
\[
 \big(\mathbb{P}^1_{\Delta\times\Delta^s}\setminus
 \Gamma_{\Delta}\big)[\bar{h}] 
 \hookrightarrow
 \left( \mathbb{P}^1_{\Delta\times\Delta^s}\setminus \Gamma_{\Delta}\right)[\bar{h}]  
 \times_{\Delta\times\Delta^s}
 \left( \mathbb{P}^1_{\Delta\times\Delta^s} \setminus \Gamma_{\Delta}\right)[\bar{h}] .
\]
Let
\[
 \iota_{\left(\mathbb{P}^1_{\Delta\times\Delta^s}
 \setminus \Gamma_{\Delta}\right)[\bar{h}]}
 \colon \big(\mathbb{P}^1_{\Delta\times\Delta^s}\setminus 
 \Gamma_{\Delta}\big)[\bar{h}]
 \hookrightarrow \mathbb{P}^1_{\Delta\times\Delta^s}[\bar{h}]
\]
be the inclusion.
We put
${\mathcal V}[\bar{h}]:={\mathcal V}\times\Spec\mathbb{C}[h]/(h^2)$.
We denote
$D\times_{\Delta\times\Delta^s}{\mathcal V}$,
$\Gamma\times_{\Delta\times\Delta^s}{\mathcal V}$
by $D_{{\mathcal V}}$, $\Gamma_{\mathcal V}$, respectively
and denote
$D\times_{\Delta\times\Delta^s}{\mathcal V}[\bar{h}]$
by
$D_{{\mathcal V}}[\bar{h}]$.
We first construct an extension of the relative connection
$\nabla^{\mathbb{P}^1}$
to a relative connection
on $\mathbb{P}^1\times{\mathcal V}[\bar{h}]$
over ${\mathcal V}[\bar{h}]$.
We need the following lemma:

\begin{lemma} \label {lemma:summation of adjoint images}
Let $A_1,\ldots,A_m$ be elements of $\End_{\mathbb{C}}(\mathbb{C}^r)$
satisfying
\[
 \bigcap_{j=1}^m \ker \mathrm{ad}(A_j) =\mathbb{C}\cdot\mathrm{id},
\] 
where $\mathrm{ad}(A_j)\colon
\End_{\mathbb{C}}(\mathbb{C}^r)
\ni X \mapsto A_j X-X A_j \in \End_{\mathbb{C}}(\mathbb{C}^r)$
is the adjoint map.
Then we have
\[
 \sum_{j=1}^m \im (\mathrm{ad}(A_j)) =
 \ker \left( \End_{\mathbb{C}} ( \mathbb{C}^r )
 \xrightarrow {\Tr} \mathbb{C} \right).
\]
\end{lemma}

\begin{proof}
In general we have
$^t\mathrm{ad}(A_j)=-\mathrm{ad}(A_j)$, because
\begin{align*}
 \Tr ( \,^t\mathrm{ad}(A_j)(X)\cdot B)
 =\Tr( X\cdot \mathrm{ad}(A_j)(B) ) 
 &=\Tr ( X \cdot (A_j B-B A_j) ) \\
 &= \Tr ((X A_j-A_j X)B +A_jXB-XBA_j) \\
 &=\Tr((XA_j-A_jX)B) +\Tr(A_jXB)-\Tr(XBA_j) \\
 &=\Tr(-\mathrm{ad}(A_j)(X)\cdot B)
\end{align*}
for any $X,B\in\End_{\mathbb{C}}(\mathbb{C}^r)$.
So there are exact sequences
\[
 0 \longrightarrow \ker \mathrm{ad}(A_j)
 \longrightarrow \End_{\mathbb{C}}(\mathbb{C}^r)
 \xrightarrow { \mathrm{ad}(A_j) }
 \End_{\mathbb{C}} (\mathbb{C}^r)
 \longrightarrow 
 (\ker \mathrm{ad}(A_j))^{\vee} \longrightarrow 0.
\]
for $j=1,\ldots,m$.
Since
$\End_{\mathbb{C}}(\mathbb{C}^r)\xrightarrow{\pi}
\End_{\mathbb{C}}(\mathbb{C}^r)\big/\sum_{j=1}^m\im(\mathrm{ad}(A_j))$
is the largest quotient vector space
satisfying $\pi\circ\mathrm{ad}(A_j)=0$
for $j=1,\ldots,m$,
its dual is given by
\[
 \Big( \End_{\mathbb{C}}(\mathbb{C}^r)\Big/\sum_{j=1}^m\im(\mathrm{ad}(A_j)) \Big)^{\vee}
 =
 \bigcap_{j=1}^m  \ker \,^t\mathrm{ad}(A_j) 
 =
 \bigcap_{j=1}^m  \ker \mathrm{ad}(A_j)
 =\mathbb{C}\cdot\mathrm{id}
 \subset \End_{\mathbb{C}}(\mathbb{C}^r).
\]
Taking the dual again, we obtain
\[
 \End_{\mathbb{C}}(\mathbb{C}^r)\Big/\sum_{j=1}^m\im(\mathrm{ad}(A_j)) 
 =
 (\mathbb{C}\cdot\mathrm{id})^{\vee}
 =
 \End_{\mathbb{C}} (\mathbb{C}^r)\big/
 \ker \big(\End_{\mathbb{C}}(\mathbb{C}^r)\xrightarrow{\Tr}\mathbb{C})\big).
\]
Thus we have
$\sum_{j=1}^m \im (\mathrm{ad}(A_j))=
\ker\big(\End_{\mathbb{C}}(\mathbb{C}^r)\xrightarrow{\Tr}\mathbb{C}\big)$.
\end{proof}

For the relative connection
\begin{equation} \label {equation:extension of local connections to P^1}
 \nabla^{\mathbb{P}^1} \colon
 \big({\mathcal O}_{\mathbb{P}^1\times {\mathcal V}}^{hol}\big)^{\oplus r}
 \longrightarrow
 \big({\mathcal O}_{\mathbb{P}^1\times {\mathcal V}}^{hol}\big)^{\oplus r} \otimes
 \Omega^1_{\mathbb{P}^1\times {\mathcal V}/{\mathcal V}}
 \big(D_{{\mathcal V}}\cup
 (\{\infty\}\times {\mathcal V})\big)^{hol}
\end{equation}
constructed in Proposition \ref {prop:normalization of local connection via extension to P^1},
let $A_{\infty}(z,\epsilon)\dfrac{dz}{z^m-\epsilon^m}$
be the connection matrix of
$\nabla^{\mathbb{P}^1}$.
Since $\nabla^{\mathbb{P}^1}$ is regular singular at $z=\infty$,
we can write
\[
 A_{\infty}(z,\epsilon)=A_{\infty,0}(\epsilon)+A_{\infty,1}(\epsilon) z+\cdots
 +A_{\infty,m-1}(\epsilon) z^{m-1}
\]
with matrices $A_{\infty,0}(\epsilon),\ldots,A_{\infty,m-1}(\epsilon)$
of holomorphic functions in  $(\epsilon,w)\in {\mathcal V}$.
Using
$\nabla_{\mathbb{P}^1}|_{\Delta\times{\mathcal V}}
=\nabla_{\Delta}|_{\Delta\times{\mathcal V}}$
and (\ref {equation:condition of local exponent}),
we can see that
there exists an invertible matrix $P(z,\epsilon)$
of holomorphic functions on a neighborhood of $D_{\mathcal V}$
such that
\begin{equation} \label {equation: restriction to 2D}
 \left( P(z,\epsilon)^{-1} d P(z,\epsilon)
 +
 P(z,\epsilon)^{-1} A_{\infty}(z,\epsilon)\frac{dz}{z^m-\epsilon^m} P(z,\epsilon) \right)
 \Big|_{2D_{{\mathcal V}}}
 =
 \Diag_{(\nu(\mu_k))} \frac{dz}{z^m-\epsilon^m} 
 \Big|_{2D_{{\mathcal V}}}.
\end{equation}
Since $\nu(\mu_1)|_p,\ldots,\nu(\mu_r)|_p$ are distinct
at any point $p\in D_{\mathcal V}$,
there exists a polynomial
$\bar{\psi}(T)=\bar{a}_{r-1}T^{r-1}+\cdots+\bar{a}_1T+\bar{a}_0
\in {\mathcal O}^{hol}_{D_{{\mathcal V}}}[T]$
satisfying
\[
 \bar{\psi} \left( \Diag_{(\nu(\mu_k))}\right) \frac{dz} {z^m-\epsilon^m} \Big|_{D_{{\mathcal V}}}
 =
 \Diag_{(\mu_k)} \frac{dz}{z^m-\epsilon^m} \Big|_{D_{{\mathcal V}}}.
\]
After shrinking ${\mathcal V}$,
we can take lifts
$a_0(z,\epsilon),a_1(z,\epsilon),\ldots,a_{r-1}(z,\epsilon)
\in{\mathcal O}^{hol}_{\mathcal V}[z]$
of $\bar{a}_0,\bar{a}_1,\ldots,\bar{a}_{r-1}$ and put
\[
 \psi(T):=a_{r-1}(z,\epsilon)T^{r-1}+a_{r-2}(z,\epsilon)T^{r-2}
 +\cdots+a_1(z,\epsilon)T+a_0(z,\epsilon)
 \in{\mathcal O}_{\mathcal V}[z][T].
\]
Here we may assume that
$a_0(z,\epsilon),\ldots,a_{r-1}(z,\epsilon)$ are polynomials
in $z$ of degree less than $m$.
Then $\psi(A_{\infty}(z,\epsilon))$ is a matrix of polynomials in $z$
and we have
\[
 P(z,\epsilon)^{-1} \psi ( A_{\infty}(z,\epsilon) ) P(z,\epsilon)
 \frac{dz}{z^m-\epsilon^m} \Big|_{D_{{\mathcal V}}}
 =
 \Diag_{(\mu_k)} \frac{dz}{z^m-\epsilon^m} \Big|_{D_{{\mathcal V}}}.
\]
For $l=0,1,\ldots,r-1$ and for $j'=0,1,\ldots,m-2$, we have
\begin{align*}
 \res_{z=\infty} \left( \Tr \left(\psi(A_{\infty}(z,\epsilon))^l 
 \frac{  z^{j'}  dz} {z^m-\epsilon^m} \right)  \right) 
 &=-\sum_{j=1}^m \res_{z=\epsilon\zeta_m^j} 
 \left( \Tr \left( \psi(A_{\infty}(z,\epsilon))^l \frac{z^{j'} dz} {z^m-\epsilon^m} \right)  \right) \\
 &=
 -\sum_{j=1}^m \res_{z=\epsilon\zeta_m^j}
 \left( \Tr \left( P(z,\epsilon)^{-1} \psi(A_{\infty}(z,\epsilon))^l P(z,\epsilon)
 \frac{z^{j'} dz} {z^m-\epsilon^m}\right)\right) \\
 & =
 -\sum_{j=1}^m \res_{z=\epsilon\zeta_m^j}
 \left( \Tr \left( \left( \Diag_{(\mu_k)} \right)^l
 \frac {z^{j'} dz} {z^m-\epsilon^m} \right)\right)  \\
 &=
 \res_{z=\infty}
 \left( \Tr \left( \Diag_{(\mu_k^l)} \frac {z^{j'} dz} {z^m-\epsilon^m} \right)\right) 
 =0.
\end{align*}
We can write
\[
 \psi(A_{\infty}(z,\epsilon))^l=\sum_{q=0}^Q C^{(l)}_{q}(\epsilon) z^{q}
\]
for matrices $C^{(l)}_{q}(\epsilon)$ constant in $z$.
We define
\begin{equation} \label {equation:definition of Xi}
 \Xi_{l,j}(z,\epsilon):=
 \sum_{j'=0}^{m-1}\sum_{\genfrac{}{}{0pt}{}{p\geq 0}{0 \leq pm+j'-j \leq Q}} 
 \epsilon^{pm} z^{j'} C^{(l)}_{pm+j'-j}(\epsilon)
\end{equation}
for $j=0,1,\ldots,m-1$
and $l=0,1\ldots,r-1$.
In other words, $\Xi_{l,j}(z,\epsilon)$
is obtained from $z^j\psi(A_{\infty}(z,\epsilon))^l$
by substituting $\epsilon^m$ in $z^m$.
Then we have
\begin{align*}
 A_{\infty}(z,\epsilon)\frac{dz}{z^m-\epsilon^m} \Big|_{D_{{\mathcal V}}}
 &=
 P(z,\epsilon) \: \nu\left( \Diag_{(\mu_k)} \right) P(z,\epsilon)^{-1}
 \frac{dz}{z^m-\epsilon^m} \Big|_{D_{{\mathcal V}}} \\
 &=
 \sum_{l=0}^{r-1} \sum_{j=0}^{m-1}
 c_{l,j}z^j \psi(A_{\infty}(z,\epsilon))^l \frac{dz}{z^m-\epsilon^m} \Big|_{D_{{\mathcal V}}}
 =
 \sum_{l=0}^{r-1} \sum_{j=0}^{m-1} c_{l,j}
 \Xi_{l,j}(z,\epsilon) \frac{dz}{z^m-\epsilon^m} \Big|_{D_{{\mathcal V}}},
\end{align*}
from which we have
\[
 A_{\infty}(z,\epsilon)=\sum_{l=0}^{r-1}\sum_{j=0}^{m-1}
 c_{l,j} \Xi_{l,j}(z,\epsilon).
\]
Note that we have
\begin{align}
 \res_{z=\infty} 
 \left( \Tr \left(
 \Xi_{l,j}(z,\epsilon) \frac{dz}{z^m-\epsilon^m} \right) \right)  
 &=
 -\Tr \Big( \sum_{0\leq pm+m-1-j \leq Q} \epsilon^{pm} C^{(l)}_{pm+m-1-j}(\epsilon) \Big)
 \label {equation:trace of residue is zero}  \\
 &=
 \res_{z=\infty}
 \left( \Tr \left(
 z^j \psi(A_{\infty}(z,\epsilon))^l \frac {dz} {z^m-\epsilon^m} \right) \right)
 =0 \notag
\end{align}
for $j=0,1,\ldots,m-2$.

We put
${\mathcal V}_{\epsilon^m}:={\mathcal V}\times_{\Delta\times\Delta^s}
(\Spec\mathbb{C}[\epsilon]/(\epsilon^m)\times\Delta^s)$
and
${\mathcal V}_{\epsilon^m}[\bar{h}]
:={\mathcal V}_{\epsilon^m}\times\Spec\mathbb{C}[h]/(h^2)$.
Then the restriction 
\begin{align}
 \nabla^{\mathbb{P}^1}|_{\mathbb{P}^1\times{\mathcal V}_{\epsilon^m}}
 &\colon
 ({\mathcal O}^{hol}_{\mathbb{P}^1\times{\mathcal V}_{\epsilon^m}})^{\oplus r}
 \longrightarrow
({\mathcal O}^{hol}_{\mathbb{P}^1\times{\mathcal V}_{\epsilon^m}})^{\oplus r}
 \otimes \Omega^1_{\mathbb{P}^1\times{\mathcal V}_{\epsilon^m}/{\mathcal V}_{\epsilon^m}}
 \left(D_{{\mathcal V}_{\epsilon^m}}\cup (\infty\times{\mathcal V}_{\epsilon^m})\right) 
 \label {equation:restriction of connections on P^1 to irregular locus} \\
 & \begin{pmatrix} f_1 \\ \vdots \\ f_r \end{pmatrix} 
 \mapsto \begin{pmatrix} df_1 \\ \vdots \\ df_r \end{pmatrix}
 +
 A_{\infty}(z,\bar{\epsilon})\frac{dz}{z^m} \begin{pmatrix} f_1 \\ \vdots \\ f_r \end{pmatrix}
 \notag
\end{align}
of the relative connection $\nabla^{\mathbb{P}^1}$ given in
(\ref {equation:extension of local connections to P^1})
to $\mathbb{P}^1\times{\mathcal V}_{\epsilon^m}$
becomes a relative irregular singular connection,
where $A_{\infty}(z,\bar{\epsilon})$ is the restriction of
$A_{\infty}(z,\epsilon)$ to
$\mathbb{P}^1\times{\mathcal V}_{\epsilon^m}$.
If we put
\[
 B_{0,l,j}(z):= P(z,\bar{\epsilon}) \: \Diag_{\big(\int \mu_k^l z^j\frac{dz}{z^m}\big)}
 P(z,\bar{\epsilon})^{-1} 
\]
for $j=0,1,\ldots,m-2$ and $l=0,1,\ldots,r-1$,
then $B_{0,l,j}(z)$ becomes a matrix of single valued meromorphic forms
whose pole order at $z=0$ is at most $m-1$,
because $\mu_k^l\dfrac{z^j dz}{z^m}$ has no residue part.
If we put
\begin{equation} \label {equation:definition of A_{epsilon^m,h,v}}
 A_{\epsilon^m,\bar{h},v_{l,j}}(z)\frac{dz}{z^m}:= dB_{0,l,j}(z)
 +\left[ A_{\infty}(z,\bar{\epsilon}), B_{0,l,j}(z) \right] \frac{dz}{z^m},
\end{equation}
then we can see that
$P(z,\bar{\epsilon})^{-1} A_{\epsilon^m,\bar{h},v_{l,j}}(z)
P(z,\bar{\epsilon}) \big|_{D_{{\mathcal V}_{\epsilon^m}}}
=
\Diag_{(\mu_k^l z^j)} \big|_{D_{{\mathcal V}_{\epsilon^m}}}$
because of (\ref {equation: restriction to 2D}).
Let us consider the connection
\begin{align}
 \nabla_{\Delta\times{\mathcal V}_{\epsilon^m}[\bar{h}],v_{l,j}}^{flat}
 &\colon
 ({\mathcal O}^{hol}_{\Delta\times{\mathcal V}_{\epsilon^m}[\bar{h}]})^{\oplus r}
 \longrightarrow
 ({\mathcal O}^{hol}_{\Delta\times{\mathcal V}_{\epsilon^m}[\bar{h}]})^{\oplus r}
 \otimes 
 \Omega_{\Delta\times{\mathcal V}_{\epsilon^m}[\bar{h}]/{\mathcal V}_{\epsilon^m}}
 (D_{{\mathcal V}_{\epsilon^m}[\bar{h}]}) 
 \label {equation:local horizontal lift of irregular connections} \\
 &
 \begin{pmatrix} f_1 \\ \vdots \\ f_r \end{pmatrix}
 \mapsto
 \begin{pmatrix} df_1 \\ \vdots \\ df_r \end{pmatrix}
 +
 \left( (A_{\infty}(z,\bar{\epsilon}) +\bar{h} A_{\epsilon^m,\bar{h},v_{l,_j}}(z) )\frac{dz}{z^m}
 +B_{0,l,j}(z) d\bar{h} \right)
 \begin{pmatrix} f_1 \\ \vdots \\ f_r \end{pmatrix}.
 \notag
\end{align}

\begin{lemma}
The connection
$\nabla^{flat}_{\Delta\times{\mathcal V}_{\epsilon}[\bar{h}],v_{l,j}}$
given in (\ref  {equation:local horizontal lift of irregular connections})
satisfies the integrability condition
\begin{align*}
 & d \left( (A_{\infty}+\bar{h}A_{\epsilon^m,\bar{h},v_{l,j}}) \frac{dz}{z^m} 
 +B_{0,l,j} d\bar{h} \right)  \\
 & \hspace{20pt}
 + \left[ \left( (A_{\infty}+\bar{h}A_{\epsilon^m,\bar{h},v_{l,j}}) \frac{dz}{z^m}
 +B_{0,l,j} d\bar{h} \right)  ,
 \left( (A_{\infty}+\bar{h}A_{\epsilon^m,\bar{h},v_{l,j}}) \frac{dz}{z^m}
 +B_{0,l,j} d\bar{h} \right) \right]
 =0.
\end{align*}
\end{lemma}

\begin{proof}
The lemma follows from the immediate calculation
\begin{align*}
 & d \left( (A_{\infty}+\bar{h}A_{\epsilon^m,\bar{h},v_{l,j}}) \frac{dz}{z^m} 
 +B_{0,l,j} d\bar{h} \right)  \\
 & \hspace{20pt}
 + \left[ \left( (A_{\infty}+\bar{h}A_{\epsilon^m,\bar{h},v_{l,j}}) \frac{dz}{z^m}
 +B_{0,l,j} d\bar{h} \right)  ,
 \left( (A_{\infty}+\bar{h}A_{\epsilon^m,\bar{h},v_{l,j}}) \frac{dz}{z^m}
 +B_{0,l,j} d\bar{h} \right) \right] \\
 &=
 d\bar{h}\wedge A_{\epsilon^m,\bar{h},v_{l,j}}\frac{dz}{z^m}
 +dB_{0,l,j}\wedge d\bar{h}
 + \left[ A_{\infty} \frac{dz}{z^m} , B_{0,l,j} d\bar{h} \right]
 =0 
\end{align*}
using (\ref{equation:definition of A_{epsilon^m,h,v}}).
\end{proof}

We choose a fundamental solution $Y_{0,\infty}(z)$ of
$\nabla^{\mathbb{P}^1}_{\Delta\times{\mathcal V}_{\epsilon^m}}$
and put
$\tilde{Y}_{0,\infty}(z,\bar{h}):=Y_{0,\infty}(z) - \bar{h} B_{0,l,j}(z)Y_{0,\infty}(z)$.

\begin{lemma} \label {lemma:fundamental solution of irregular relative connection}
$\tilde{Y}_{0,\infty}(z,\bar{h})=Y_{0,\infty}(z) - \bar{h} B_{0,l,j}(z)Y_{0,\infty}(z)$
is a fundamental solution of the relative connection
\begin{equation} \label {equation:relative irregular connection induced by local horizontal lift}
 \overline { \nabla^{flat}_{\Delta\times{\mathcal V}_{\epsilon^m}[\bar{h}],v_{l,j} } }
 \colon
 ({\mathcal O}^{hol}_{\Delta\times{\mathcal V}_{\epsilon^m}[\bar{h}]})^{\oplus r}
 \longrightarrow
 ({\mathcal O}^{hol}_{\Delta\times{\mathcal V}_{\epsilon^m}[\bar{h}]})^{\oplus r}
 \otimes 
 \Omega_{\Delta\times{\mathcal V}_{\epsilon^m}[\bar{h}]/{\mathcal V}_{\epsilon^m}[\bar{h}]}
 (D_{{\mathcal V}_{\epsilon^m}[\bar{h}]}) 
\end{equation}
induced by
$\nabla^{flat}_{\Delta\times{\mathcal V}_{\epsilon^m}[\bar{h}],v_{l,j} }$,
whose connection matrix is
$\displaystyle (A_{\infty}+\bar{h}A_{\epsilon^m,\bar{h},v_{l,j}}) \frac{dz}{z^m}$.
\end{lemma}

\begin{proof}
The lemma follows from the calculation
\begin{align}
 \frac{\partial} {\partial z}
 \left( Y_{0,\infty}
 -\bar{h} B_{0,l,j}(z) Y_{0,\infty} \right) dz 
  &=
 d Y_{0,\infty}(z)
 -\bar{h} \big( dB_{0,l,j}(z) \: Y_{0,\infty}
 +B_{0,l,j}(z) \: dY_{0,\infty} \big)  
 \label {equation:relative fundamental solution by integrability} \\
 &= 
 -\frac{ A_{\infty}(z,\bar{\epsilon}) dz}{z^m}  Y_{0,\infty} 
 -
 \bar{h} A_{\epsilon^m,\bar{h},v_{l,j}}(z) \frac{dz}{z^m}
 \: Y_{0,\infty} \notag \\
 & \hspace{20pt}
 +\bar{h}\Big(
 \big[ A_{\infty}(z,\bar{\epsilon}), B_{0,l,j}(z) \big] 
 +B_{0,l,j}(z) A_{\infty}(z,\bar{\epsilon}) \Big) \frac{dz}{z^m} 
 \: Y_{0,\infty} \notag \\
 &=
 -\big( A_{\infty}(z,\bar{\epsilon})+\bar{h}A_{\epsilon^m,\bar{h},v_{l,j}}(z) \big)
 \frac{dz}  {z^m}
 \: \big( Y_{0,\infty}
 -\bar{h} B_{0,l,j}(z) Y_{0,\infty} \big) . \notag
\end{align}
\end{proof}

Let $\mathrm{Mon}_{\tilde{\gamma}}$ be
the monodromy matrix of $Y_{0,\infty}(z)$
along $\tilde{\gamma}$.
Then
$\tilde{Y}_{0,\infty}(z,\bar{h})=Y_{0,\infty}(z) - \bar{h} B_{0,l,j}(z)Y_{0,\infty}(z)$
has the monodromy matrix
$\mathrm{Mon}_{\tilde{\gamma}}$
along $\tilde{\gamma}$,
because $B_{0,l,j}(z)$
is single valued on
$(\Delta\times{\mathcal V}_{\epsilon^m})\setminus D_{{\mathcal V}_{\epsilon^m}}$.
By the similar method to that in the proof of Proposition
\ref {prop:normalization of local connection via extension to P^1},
we can construct a global connection
\begin{align*}
 \nabla_{\mathbb{P}^1\times{\mathcal V}_{\epsilon^m}[\bar{h}],v_{l,j}}
 &\colon
 ({\mathcal O}^{hol}_{\mathbb{P}^1\times{\mathcal V}_{\epsilon^m}[\bar{h}]})^{\oplus r}
 \longrightarrow
({\mathcal O}^{hol}_{\mathbb{P}^1\times{\mathcal V}_{\epsilon^m}[\bar{h}]})^{\oplus r}
 \otimes \Omega^1_{\mathbb{P}^1\times{\mathcal V}_{\epsilon^m}[\bar{h}]
 /{\mathcal V}[\bar{h}]}
 \left(D_{{\mathcal V}_{\epsilon^m}[\bar{h}]} \cup
 (\infty\times{\mathcal V}_{\epsilon^m}[\bar{h}])\right) \\
 & \begin{pmatrix} f_1 \\ \vdots \\ f_r \end{pmatrix} 
 \mapsto \begin{pmatrix} df_1 \\ \vdots \\ df_r \end{pmatrix}
 +
 (A_{\infty}(z,\bar{\epsilon})+\bar{h}\tilde{A}'_{\epsilon^m,\bar{h},v_{l,j}}(z))
 \frac{dz}{z^m-\epsilon^m} \begin{pmatrix} f_1 \\ \vdots \\ f_r \end{pmatrix}
\end{align*}
satisfying
\[
 \res_{z=\infty} \left( \tilde{A}'_{\epsilon^m,\bar{h},v_{l,j}}(z)
 \frac{dz}{z^m} \right)=0
\]
such that the restriction of
$\nabla_{\mathbb{P}^1\times{\mathcal V}_{\epsilon^m}[\bar{h}],v_{l,j}}$
to
$\mathbb{P}^1\times{\mathcal V}_{\epsilon^m}$
coincides with the restriction
$\nabla^{\mathbb{P}^1}|_{\mathbb{P}^1\times{\mathcal V}_{\epsilon^m}}$
given in (\ref {equation:restriction of connections on P^1 to irregular locus})
and that the restriction of
$\nabla_{\mathbb{P}^1\times{\mathcal V}_{\epsilon^m}[\bar{h}],v_{l,j}}$
to
$\Delta\times{\mathcal V}_{\epsilon^m}[\bar{h}]$
is isomorphic to the irregular singular relative connection
$\overline{ \nabla_{\Delta\times{\mathcal V}_{\epsilon^m}[\bar{h}],v_{l,j}}^{flat} }$
given in (\ref {equation:relative irregular connection induced by local horizontal lift}).
By construction, there is a convergent power series
\[
 \sum_{l'=0}^{\infty}R'^{(l)}_{0,j,l'}z^{l'}
\]
such that
\begin{align*}
 &(A_{\infty}(z,\bar{\epsilon})+\bar{h}A_{\epsilon^m,\bar{h},v_{l,j}}(z))
 \frac{dz}{z^m} \\
 &=
 \bar{h}\sum_{l'=1}^{\infty} l' R'^{(l)}_{0,j,l'} z^{l'-1} dz  
 +\Big( 1-\bar{h}\sum_{l'=0}^{\infty}R'^{(l)}_{0,j,l'}z^{l'}\Big)
 \big(A_{\infty}(z,\bar{\epsilon})+\bar{h}\tilde{A}'_{\epsilon^m,\bar{h},v_{l,j}}(z)\big)
 \frac{dz}{z^m}
 \Big( 1+\bar{h}\sum_{l'=0}^{\infty}R'^{(l)}_{0,j,l'}z^{l'}\Big),
\end{align*}
which implies
\begin{align*}
 \Xi_{l,j}(z,\bar{\epsilon}) \big|_{D_{{\mathcal V}_{\epsilon^m}}}
 &=
 \psi(A_{\infty}(z,\bar{\epsilon}))^l z^j  \big|_{D_{{\mathcal V}_{\epsilon^m}}}
 =
 P(z,\bar{\epsilon}) \: \Diag_{ (\mu_k^l z^j)} \: P(z,\bar{\epsilon})^{-1}
 \big|_{D_{{\mathcal V}_{\epsilon^m}}}
 =
 A_{\epsilon^m,\bar{h},v_{l,j}}(z) 
 \big|_{D_{{\mathcal V}_{\epsilon^m}}} \\
 &=
 \bigg(\tilde{A}'_{\epsilon^m,\bar{h},v_{l,j}}(z)
 + \sum_{j'=0}^{m-1} \sum_{l'=0}^{j'}
  \left[ A_{\infty,j'-l'}(\bar{\epsilon}) , R'^{(l)}_{0,j,l'} \right] z^{j'} \bigg) \bigg)
  \bigg|_{D_{{\mathcal V}_{\epsilon^m}}}.
\end{align*}
So we have
\begin{equation} \label {equation:adjusting in irregular case}
 \Xi_{l,j}(z,\bar{\epsilon})
 =
 \tilde{A}'_{\epsilon^m,\bar{h},v_{l,j}}(z)
 +
 \sum_{j'=0}^{m-1}\sum_{l'=0}^{j'}
 \Big[ A_{\infty,j'-l'}(\bar{\epsilon}) , R'^{(l)}_{0,j,l'} \Big] z^{j'} .
\end{equation}
We put
\[
 B'_{0,l,j}(z):=
 B_{0,l,j}(z)-\sum_{l'=0}^{\infty} R'^{(l)}_{0,j,l'} z^{l'}.
\]

\begin{lemma} \label {lemma:local horizontal lift of irregular connection}
The connection on
$({\mathcal O}_{\Delta\times{\mathcal V}_{\epsilon^m}[\bar{h}]}^{hol})^{\oplus r}$
given by the connection matrix
\[
 (A_{\infty}(z,\bar{\epsilon})+\bar{h}\tilde{A}'_{\epsilon^m,\bar{h},v_{l,j}}(z))
 \frac{dz}{z^m}
 +B'_{0,l,j}(z) d\bar{h}
\]
is isomorphic to the connection
$\nabla_{\Delta\times{\mathcal V}_{\epsilon^m}[\bar{h}],v_{l,j}}^{flat}$
given in (\ref {equation:local horizontal lift of irregular connections})
and satisfies the integrability condition.

\end{lemma}

\begin{proof}
Indeed the isomorphism is given by
$ I_r+\bar{h}\sum_{l'=0}^{\infty} B'_{0,j,l'} z^{l'}$
and the integrability follows from that of
$\nabla_{\Delta\times{\mathcal V}_{\epsilon^m}[\bar{h}],v_{l,j}}^{flat}$.
\end{proof}

We will give a lift of the connection given in
Lemma \ref {lemma:local horizontal lift of irregular connection}
as a connection on $\Delta\times{\mathcal V}[\bar{h}]$,
by means of extending the data
$(R'^{(l)}_{0,j,l'})$.

\begin{definition} \rm
We say that
$\big( R^{(l)}_{j,l'}(\epsilon) \big)^{0\leq l \leq r-1}
_{0\leq j\leq m-1, 0\leq l'\leq r-1}$
is an adjusting data for the connection
$\nabla^{\mathbb{P}^1}$ given in (\ref  {equation:extension of local connections to P^1})
if each $R^{(l)}_{j,l'}(\epsilon)$ is a matrix whose entries belong to
${\mathcal O}_{\mathcal V}^{hol}$ such that
$R^{(l)}_{j,l'}(\epsilon) \big|_{\epsilon^m=0}=R'^{(l)}_{0,j,l'}$
and that
the $z^{m-1}$-coefficient of
$\Xi_{l,j}(z,\epsilon)$ given in (\ref {equation:definition of Xi})
is expressed by
\begin{equation} \label {equation:adjustment at residue part}
 \sum_{0\leq pm+m-1-j\leq Q} \epsilon^{pm} C^{(l)}_{pm+m-1-j}(\epsilon)
 = \sum_{l'=0}^{m-1} \left[ A_{\infty,m-l'-1}(\epsilon) , R^{(l)}_{j,l'}(\epsilon) \right].
\end{equation}
\end{definition}

\begin{lemma} \label {lemma:existence of adjusting data}
There exists an adjusting data
$\big( R^{(l)}_{j,l'}(\epsilon) \big)^{0\leq l \leq r-1}
_{0\leq j\leq m-1, 0\leq l'\leq r-1}$
for the connection
$\nabla^{\mathbb{P}^1}$.
\end{lemma}

\begin{proof}
For each
$u\in \bigcap_{j=0}^{m-1} \ker (\mathrm{ad}(A_{\infty,j}(\epsilon)) )$,
we have
$u \cdot A_{\infty}(z,\epsilon)\dfrac{dz}{z^m-\epsilon^m}
- A_{\infty}(z,\epsilon)\dfrac{dz}{z^m-\epsilon^m} \cdot u=0$.
So $u|_{\Delta\times\{ b\}}$ is a section of
$\ker \nabla^{\dag}_{\Delta_b}$
on $\Delta\times \{b\} $
for each $b\in {\mathcal V}$,
which is a scalar endomorphism by
Assumption  \ref {assumption:local irreducible}, (ii).
Then we have
$u\in {\mathcal O}_{\mathcal V}^{hol}\cdot\mathrm{id}$ 
and
\begin{equation} \label {equation:intersection of adjoint kernel}
 \bigcap_{j=0}^{m-1} \ker \left(\mathrm{ad}(A_{\infty,j}(\epsilon)) \right)
 ={\mathcal O}_{\mathcal V}^{hol} \cdot\mathrm{id}
\end{equation}
follows.
So we can see
\[
 \sum_{j=0}^{m-1} \im (\mathrm{ad}(A_{\infty,j}(\epsilon)))
 =
 \ker \left( {\mathcal End}_{{\mathcal O}^{hol}_{\mathcal V}}
 \left( \big({\mathcal O}_{\mathcal V}^{hol}\big)^{\oplus r}\right)
 \xrightarrow{\Tr} {\mathcal O}_{\mathcal V}^{hol} \right),
\]
because the equality for the restriction to each $b'\in{\mathcal V}$
holds by Lemma \ref {lemma:summation of adjoint images}.
Then, after shrinking ${\mathcal V}$,
there are  matrices $R^{(l)}_{j,0}(\epsilon),\ldots,R^{(l)}_{j,m-1}(\epsilon)$
constant in $z$ such that
\[
 \sum_{0\leq pm+m-1-j\leq Q} \epsilon^{pm} C^{(l)}_{pm+m-1-j}(\epsilon)
 = \sum_{l'=0}^{m-1} \left[ A_{\infty,m-l'-1}(\epsilon) , R^{(l)}_{j,l'}(\epsilon) \right].
\]
because of (\ref {equation:trace of residue is zero}).
Here we may assume
$R^{(l)}_{j,l'}(\epsilon) \big|_{\epsilon^m=0}=R'^{(l)}_{0,j,l'}$
by using (\ref {equation:adjusting in irregular case}).
\end{proof}

For $l=0,1,\ldots,r-1$ and for $j=0,1,\ldots,m-2$,
we take an adjusting data
$\big( R^{(l)}_{j,l'}(\epsilon) \big)^{0\leq l \leq r-1}
_{0\leq j\leq m-1, 0\leq l'\leq r-1}$
for the connection
$\nabla^{\mathbb{P}^1}$
and define
\begin{align} 
 \tilde{\Xi}_{l,j} (z,\epsilon)
 &:=
 \Xi_{l,j} (z,\epsilon)-
 \sum_{q=0}^{m-1}\sum_{0 \leq l' \leq m-1-q}
 \Big[ A_{\infty,q}(\epsilon) , R^{(l)}_{j,l'}(\epsilon) \Big] z^{q+l'}
  \label {equation:definition of adjustment of Xi}  \\
 &\hspace{70pt}
 -\sum_{q=0}^{m-1} \sum_{m-q \leq l' \leq m-1} 
 \Big[ A_{\infty,q}(\epsilon) , R^{(l)}_{j,l'}(\epsilon) \Big] \epsilon^m z^{q+l'-m}.
 \notag
\end{align}
Then, using (\ref {equation:adjustment at residue part}), we have the equality
\begin{align}
 \res_{z=\infty} \left( \tilde{\Xi}_{l,j} (z,\epsilon) \frac{dz}{z^m-\epsilon^m} \right)
 &=\res_{z=\infty} \left( \Xi_{l,j} (z,\epsilon) \frac{dz}{z^m-\epsilon^m} 
 -\sum_{l'=0}^{m-1} \left[ A_{\infty,m-l'-1}(\epsilon) , R^{(l)}_{j,l'}(\epsilon) \right] 
 \frac { z^{m-1}dz } { z^m-\epsilon^m }
 \right)  \label {equation:adjusted residue part is zero} \\
 &=
 -\sum_{0\leq pm+m-1-j\leq Q} \epsilon^{pm} C^{(l)}_{pm+m-1-j}(\epsilon)
 +\sum_{l'=0}^{m-1} \left[ A_{\infty,m-l'-1}(\epsilon) , R^{(l)}_{j,l'}(\epsilon) \right]
 \notag \\
 &=0 \notag
\end{align}
for $j=0,1,\ldots,m-2$, $l=0,1,\ldots,r-1$ and we have
\begin{align*}
 \tilde{\Xi}_{l,j} (z,\epsilon) \frac {dz} {z^m-\epsilon^m} \bigg|_{D_{{\mathcal V}}}
 &=
 P(z,\epsilon) z^j \Diag_{(\mu_k^l)} P(z,\epsilon)^{-1}
 \frac{dz}{z^m-\epsilon^m} \Big|_{D_{{\mathcal V}}}  \\
 & \hspace{40pt}
 -\Big[A_{\infty}(z,\epsilon),\sum_{l'=0}^{m-1} R^{(l)}_{j,l'}(\epsilon) z^{l'} \Big]
 \frac{dz}{z^m-\epsilon^m} \Big|_{D_{{\mathcal V}}}.
\end{align*}

Let
\begin{equation} \label {equation: lifted relative connection on projective line}
 \nabla_{\mathbb{P}^1\times{\mathcal V},v_{l,j}} \colon
 ({\mathcal O}_{\mathbb{P}^1\times{\mathcal V}[\bar{h}]}^{hol})^{\oplus r}
 \longrightarrow
 ({\mathcal O}_{\mathbb{P}^1\times{\mathcal V}[\bar{h}]}^{hol})^{\oplus r}
 \otimes\Omega^1_{\mathbb{P}^1\times{\mathcal V}[\bar{h}]
 /{\mathcal V}[\bar{h}]}
 \left(D_{{\mathcal V}[\bar{h}]} \cup
 (\infty\times {\mathcal V}[\bar{h}]) \right)^{hol}
\end{equation}
be the relative connection defined by
\[
 \nabla_{\mathbb{P}^1\times{\mathcal V}[\bar{h}],v_{l,j}}
 \begin{pmatrix} f_1 \\ \vdots \\ f_r \end{pmatrix}
 =
  \begin{pmatrix} df_1 \\ \vdots \\ df_r \end{pmatrix}
  +\Big(A_{\infty}(z,\epsilon)+ \bar{h} \tilde{\Xi}_{l,j}(z,\epsilon) \Big)
  \frac{dz}{z^m-\epsilon^m}
  \begin{pmatrix} f_1 \\ \vdots \\ f_r \end{pmatrix}.
\]
Then
$\nabla_{\mathbb{P}^1\times{\mathcal V}[\bar{h}],v_{l,j}} 
\big|_{\Delta\times{\mathcal V}_{\epsilon^m}[\bar{h}]}$
is isomorphic to
$\overline {\nabla^{flat}_{\Delta\times{\mathcal V}_{\epsilon^m}[\bar{h}],v_{l,j}} }$
by the construction.
Using (\ref  {equation:adjusted residue part is zero}), we can see the equality
$\res_{z=\infty}(\nabla_{\mathbb{P}^1\times{\mathcal V}[\bar{h}],v_{l,j}})
=\res_{z=\infty}\big(\nabla^{\mathbb{P}^1}\big)$.
By construction, there is an invertible matrix
$\tilde{P}(z,\bar{h})$ such that
\begin{align*}
 &\Big(
 \tilde{P}(z,\bar{h})^{-1} d\tilde{P}(z,\bar{h})
 +\tilde{P}(z,\bar{h})^{-1}\Big( A_{\infty}(z,\epsilon)+\bar{h} \tilde{\Xi}_{l,j}(z,\epsilon) \Big)
 \frac{dz}{z^m-\epsilon^m}\tilde{P}(z,\bar{h})
 \Big)\bigg|_{D_{{\mathcal V}[\bar{h}]}} \\
 &=
 \Diag_{(\nu(\mu_k)+\bar{h}\mu_k^l z^j)}\frac{dz}{z^m-\epsilon^m}
 \bigg|_{D_{{\mathcal V[\bar{h}]}}}.
\end{align*}
We may further assume that
\[
 \tilde{P}(z,\bar{h})P(z)^{-1} \big|_{D_{{\mathcal V}[\bar{h}]}}
 =
 \Big( I_r+\bar{h}\sum_{l'=0}^{m-1} R^{(l)}_{j,l'} z^{l'}\Big) 
 \Big|_{D_{{\mathcal V}[\bar{h}]}}.
\]

We will construct an integrable connection on
$\mathbb{P}^1\times{\mathcal V}[\bar{h}]$
over ${\mathcal V}$ which is an extension of
(\ref {equation: lifted relative connection on projective line}).

\begin{definition} \rm
We say that a connection
\begin{align} 
 \nabla_{\mathbb{P}^1\times{\mathcal V}[\bar{h}],v_{l,j}}^{flat}  \colon
 &({\mathcal O}_{\mathbb{P}^1\times\tilde{\mathcal V}[\bar{h}]}^{hol})^{\oplus r}
 \longrightarrow
 ({\mathcal O}_{\mathbb{P}^1\times\tilde{\mathcal V}[\bar{h}]}^{hol})^{\oplus r}
 \otimes
 (\iota_{{\mathcal V}[\bar{h}]})_*
 \Omega^1_{ (\mathbb{P}^1\times{\mathcal V}
 \setminus\Gamma_{\mathcal V})[\bar{h}] 
 \big/{\mathcal V} }
 \left( \infty\times{\mathcal V}[\bar{h}] \right)^{hol}
 \label {equation:fundamental local flat connection} \\
 &\begin{pmatrix} f_1 \\ \vdots \\ f_r \end{pmatrix}
 \mapsto
 \begin{pmatrix} df_1 \\ \vdots \\  df_r \end{pmatrix}
 +
 \left( (A_{\infty}(z,\epsilon)+\bar{h} \tilde{\Xi}_{l,j}(z,\epsilon))\frac{dz}{z^m-\epsilon^m}
 + B_{l,j}(z,\epsilon)d\bar{h} \right)
 \begin{pmatrix} f_1 \\ \vdots \\ f_r \end{pmatrix} \notag
\end{align}
is a horizontal lift of
$\nabla_{\mathbb{P}^1\times{\mathcal V},v_{l,j}}$
if
$B_{l,j}(z,\epsilon)|_{\epsilon^m=0}=B'_{0,l,j}(z)$
and
$\nabla_{\mathbb{P}^1\times{\mathcal V}[\bar{h}],v_{l,j}}^{flat}$
is integrable in the sense that
\begin{align*}
 &
 d \left( \big(A_{\infty}(z)+\bar{h} \tilde{\Xi}_{l,j}(z)\big)\frac{dz}{z^m-\epsilon^m}
 + B_{l,j}(z)d\bar{h} \right) \\
 & \;
 + \left[ \left( \big(A_{\infty}(z)+\bar{h} \tilde{\Xi}_{l,j}(z)\big)\frac{dz}{z^m-\epsilon^m}
 + B_{l,j}(z)d\bar{h} \right),
 \left( \big(A_{\infty}(z)+\bar{h} \tilde{\Xi}_{l,j}(z) \big) \frac{dz}{z^m-\epsilon^m}
 + B_{l,j}(z)d\bar{h} \right) \right] \\
 &=0.
\end{align*}
\end{definition}

\begin{proposition} \rm \label {prop:local horizontal lift}
There exists a horizontal lift
\begin{equation*} 
 \nabla_{\mathbb{P}^1\times{\mathcal V}[\bar{h}],v_{l,j}}^{flat}  \colon
 ({\mathcal O}_{\mathbb{P}^1\times\tilde{\mathcal V}[\bar{h}]}^{hol})^{\oplus r}
 \longrightarrow
 ({\mathcal O}_{\mathbb{P}^1\times\tilde{\mathcal V}[\bar{h}]}^{hol})^{\oplus r}
 \otimes
 (\iota_{{\mathcal V}[\bar{h}]})_*
 \Omega^1_{ (\mathbb{P}^1\times{\mathcal V}
 \setminus\Gamma_{\mathcal V})[\bar{h}] 
 \big/{\mathcal V} }
 \left( \infty\times{\mathcal V}[\bar{h}] \right)^{hol}
\end{equation*}
of the relative connection
$\nabla_{\mathbb{P}^1\times{\mathcal V},v_{l,j}}$
given in
(\ref {equation: lifted relative connection on projective line})
after shrinking ${\mathcal V}$,
where
$\iota_{{\mathcal V}[\bar{h}]}\colon
(\mathbb{P}^1\times{\mathcal V}\setminus\Gamma_{\mathcal V})[\bar{h}]
\hookrightarrow
\mathbb{P}^1\times{\mathcal V}[\bar{h}]$
is the canonical inclusion..
\end{proposition}

\begin{proof}
After shrinking ${\mathcal V}$, we can take a local basis
$\tilde{Y}_{\infty}(z,\epsilon,\bar{h})$ of
$\ker (\nabla_{\mathbb{P}^1\times{\mathcal V}[\bar{h}],v_{l,j}})$
on
$(U_{\infty}\setminus
\Gamma_{\infty})\times{\mathcal V}[\bar{h}]$
for some open neighborhood $U_{\infty}$ of
$\infty$ in $\mathbb{P}^1$
and a slit $\Gamma_{\infty}\subset U_{\infty}$
which is a  simple path joining $\infty$ and a boundary point 
$b_{\infty}\in \partial U_{\infty}$ of $U_{\infty}$.
Here we may assume that
the restriction $Y_{\infty}(z,\bar{\epsilon})$
of $\tilde{Y}_{\infty}(z,\epsilon,\bar{h})$
to $(U_{\infty}\setminus\Gamma_{\infty})\times{\mathcal V}_{\epsilon^m}$
coincides with
$Y_{0,\infty}(z)$
which is chosen before
Lemma \ref {lemma:fundamental solution of irregular relative connection}.
We may further assume that the monodromy matrix
$\mathrm{Mon}_{\infty}(\epsilon)$ of
$\tilde{Y}_{\infty}(z,\epsilon,\bar{h})$ around
$\infty\times{\mathcal V}[\bar{h}]$
coincides with that of
$Y_{\infty}(z,\epsilon):=\tilde{Y}_{\infty}(z,\epsilon,0)$,
because the residue part of the connection matrix of
$\nabla_{\mathbb{P}^1\times{\mathcal V}[\bar{h}],v_{l,j}}$
at $z=\infty$ is constant in $\bar{h}$.
Consider the restriction
$\tilde{Y}_{\infty}(z,\bar{\epsilon},\bar{h})$
of $\tilde{Y}_{\infty}(z,\epsilon,\bar{h})$ to
$(U_{\infty}\setminus\Gamma_{\infty})
\times{\mathcal V}_{\epsilon^m}[\bar{h}]$.
Using  the integrability condition
of $\nabla_{\Delta\times{\mathcal V}_{\epsilon^m}[\bar{h}],v_{l,j}}^{flat}$,
we can see in the same way as
(\ref {equation:relative fundamental solution by integrability})
that
$Y_{\infty}(z,\bar{\epsilon})-\bar{h} B'_{0,l,j}(z) Y_{\infty}(z,\bar{\epsilon})$
is a fundamental solution of
$\nabla_{\mathbb{P}^1\times{\mathcal V}[\bar{h}]}
\big|_{(U_{\infty}\setminus\Gamma_{\infty})
\times{\mathcal V}_{\epsilon^m}[\bar{h}]}$
after an analytic continuation.
So we can write
\[
 Y_{\infty}(z,\bar{\epsilon})-\bar{h} B'_{0,l,j}(z) Y_{\infty}(z,\bar{\epsilon})
 =
 \tilde{Y}_{\infty}(z,\bar{\epsilon},\bar{h}) \: C(\bar{\epsilon},\bar{h})
\]
for a matrix $C(\bar{\epsilon},\bar{h})$ constant in $z$.
Since both 
$Y_{\infty}(z,\bar{\epsilon})-\bar{h} B'_{0,l,j}(z) Y_{\infty}(z,\bar{\epsilon})$
and
$\tilde{Y}_{\infty}(z,\bar{\epsilon},\bar{h})$
have the same monodromy
$\mathrm{Mon}_{\infty}(\bar{\epsilon})
:=\mathrm{Mon}_{\infty}(\epsilon)|_{\epsilon^m=0}$,
we should have
\[
 \left(
 Y_{\infty}(z,\bar{\epsilon})-\bar{h} B'_{0,l,j}(z) Y_{\infty}(z,\bar{\epsilon})
 \right)
 \mathrm{Mon}_{\infty}(\bar{\epsilon})
 =
 \tilde{Y}_{\infty}(z,\bar{\epsilon},\bar{h}) \: 
 \mathrm{Mon}_{\infty}(\bar{\epsilon}) \: C(\bar{\epsilon},\bar{h})
\]
from which we have
\[
 C(\bar{\epsilon},\bar{h}) \: \mathrm{Mon}_{\infty}(\bar{\epsilon})
 =
 \mathrm{Mon}_{\infty}(\bar{\epsilon}) \: C(\bar{\epsilon},\bar{h}).
\]
So we can write
\[
 C(\bar{\epsilon},\bar{h})=\sum_{l=0}^{r-1} b_l(\bar{\epsilon},\bar{h}) 
 \mathrm{Mon}_{\infty}(\bar{\epsilon})^l,
\]
because $\mathrm{Mon}_{\infty}(\bar{\epsilon})|_b$
has the $r$ distinct eigenvalues at each $b\in {\mathcal V}_{\epsilon^m}$.
Shrinking ${\mathcal V}$,
we can take lifts $b_l(\epsilon,\bar{h})$ of $b_l(\bar{\epsilon},\bar{h})$
as holomorphic functions in $\epsilon$.
If we replace
$\tilde{Y}_{\infty}(z,\epsilon,\bar{h})$
by
$\tilde{Y}_{\infty}(z,\epsilon,\bar{h})
\sum_{l=0}^{r-1} b_l(\epsilon,\bar{h}) 
\mathrm{Mon}_{\infty}(\epsilon)^l$,
then the restriction of
$\tilde{Y}_{\infty}(z,\epsilon,\bar{h})$ to
$(U_{\infty}\times{\mathcal V}_{\epsilon^m}[\bar{h}])\setminus
(\Gamma_{\infty}\times{\mathcal V}_{\epsilon^m}[\bar{h}])$
coincides with
$Y_{\infty}(z,\bar{\epsilon})-\bar{h} B'_{0,l,j}(z) Y_{\infty}(z,\bar{\epsilon})$.

If we define
\begin{equation} \label {equation:definition of B(z) }
 B_{l,j}(z,\epsilon):=- \frac{\partial \tilde{Y}_{\infty}(z,\epsilon,\bar{h})} {\partial \bar{h}}
 Y_{\infty}(z,\epsilon)^{-1},
\end{equation}
we have
$B_{l,j}(z,\epsilon)|_{\epsilon^m=0}
=B'_{0,l,j}(z)$.
Since both $\tilde{Y}_{\infty}(z,\epsilon,\bar{h})$
and $Y_{\infty}(z,\epsilon)$ have the same monodromy matrix around $\infty$,
we can regard
$B_{l,j}(z,\epsilon)$ as a matrix of single valued holomorphic functions on
$(\mathbb{P}^1\times{\mathcal V})\setminus \Gamma_{\mathcal V}$
after an analytic continuation.
Let us consider the connection
\begin{align*} 
 \nabla_{\mathbb{P}^1\times{\mathcal V}[\bar{h}],v_{l,j}}^{flat}  \colon
 &({\mathcal O}_{\mathbb{P}^1\times\tilde{\mathcal V}[\bar{h}]}^{hol})^{\oplus r}
 \longrightarrow
 ({\mathcal O}_{\mathbb{P}^1\times\tilde{\mathcal V}[\bar{h}]}^{hol})^{\oplus r}
 \otimes
 (\iota_{{\mathcal V}[\bar{h}]})_*
 \Omega^1_{ (\mathbb{P}^1\times{\mathcal V}
 \setminus\Gamma_{\mathcal V})[\bar{h}] 
 \big/{\mathcal V} }
 \left( \infty\times{\mathcal V}[\bar{h}] \right)^{hol}
  \\
 &\begin{pmatrix} f_1 \\ \vdots \\ f_r \end{pmatrix}
 \mapsto
 \begin{pmatrix} df_1 \\ \vdots \\  df_r \end{pmatrix}
 +
 \left( (A_{\infty}(z,\epsilon)+\bar{h} \tilde{\Xi}_{l,j}(z,\epsilon))\frac{dz}{z^m-\epsilon^m}
 + B_{l,j}(z,\epsilon)d\bar{h} \right)
 \begin{pmatrix} f_1 \\ \vdots \\ f_r \end{pmatrix} .
\end{align*}
The curvature form of
$\nabla_{\mathbb{P}^1\times{\mathcal V}[\bar{h}],v_{l,j}}^{flat}$
becomes
\begin{align*}
 &
 d \left( \big(A_{\infty}(z)+\bar{h} \tilde{\Xi}_{l,j}(z)\big)\frac{dz}{z^m-\epsilon^m}
 + B_{l,j}(z)d\bar{h} \right) \\
 & \;
 + \left[ \left( \big(A_{\infty}(z)+\bar{h} \tilde{\Xi}_{l,j}(z)\big)\frac{dz}{z^m-\epsilon^m}
 + B_{l,j}(z)d\bar{h} \right),
 \left( \big(A_{\infty}(z)+\bar{h} \tilde{\Xi}_{l,j}(z) \big) \frac{dz}{z^m-\epsilon^m}
 + B_{l,j}(z)d\bar{h} \right) \right] \\
 &=
 \tilde{\Xi}_{l,j}(z)\: d\bar{h} \wedge \frac{dz}{z^m-\epsilon^m}
  +\frac{\partial B_{l,j}(z)}{\partial z} dz\wedge d\bar{h}
 +(A_{\infty}(z)B_{l,j}(z)-B_{l,j}(z)A_{\infty}(z))\frac{dz}{z^m-\epsilon}\wedge d\bar{h} \\
 &=-\tilde{\Xi}_{l,j}(z) \frac{dz}{z^m-\epsilon^m}\wedge d\bar{h}
  -\frac{\partial}{\partial z} \Big(
 \frac{\partial \tilde{Y}_{\infty}}{\partial \bar{h}}(z,\bar{h}) \;
 Y_{\infty}(z)^{-1} \Big) dz\wedge d\bar{h}
 \\
 & \quad
 + \Big( -A_{\infty}(z) \frac{\partial \tilde{Y}_{\infty}}{\partial \bar{h}}(z,\bar{h}) \;
 Y_{\infty}(z)^{-1}
 +\frac{\partial \tilde{Y}_{\infty}}{\partial \bar{h}}(z,\bar{h}) \;
 Y_{\infty}(z)^{-1} A_{\infty}(z) \Big)
 \frac{dz}{z^m-\epsilon^m} \wedge d\bar{h} \\
 &=
 -
 \tilde{\Xi}_{l,j}(z) \frac{dz}{z^m-\epsilon^m}\wedge d\bar{h}
 -\bigg( \frac{\partial^2 \tilde{Y}_{\infty}}{\partial \bar{h} \partial z} \;
 Y_{\infty}^{-1} \bigg) dz\wedge d\bar{h}
 +
 \bigg(\frac{\partial \tilde{Y}_{\infty}}{\partial \bar{h}} \;
 Y_{\infty}^{-1} \: 
 \frac{\partial Y_{\infty}}{\partial z} \:
 Y_{\infty}^{-1} \bigg) dz\wedge d\bar{h} \\
 & \hspace{20pt}
 - \Big( A_{\infty}(z) \frac{\partial \tilde{Y}_{\infty}}{\partial \bar{h}} \;
 Y_{\infty}^{-1}
 -\frac{\partial \tilde{Y}_{\vartheta}}{\partial \bar{h}} \;
 Y_{\infty}^{-1} A_{\infty}(z) \Big)
 \frac{dz}{z^m-\epsilon^m} \wedge d\bar{h} \\
 &=
 -\frac{ \tilde{\Xi}_{l,j}(z) dz}{z^m-\epsilon^m}\wedge d\bar{h} 
 - 
 \frac{\partial}{\partial \bar{h}}
 \bigg( -\frac{A_{\infty}(z)+\bar{h} \tilde{\Xi}_{l,j}(z) } {z^m-\epsilon^m}
 \tilde{Y}_{\infty} \bigg)
 \; Y_{\infty}^{-1} \: dz\wedge d\bar{h}  
 -
 \frac{\partial \tilde{Y}_{\infty}}{\partial \bar{h}} \;
 Y_{\infty}^{-1} \: 
 \frac{ A_{\infty}(z) dz}{z^m-\epsilon^m}\wedge d\bar{h}  \\
 &\hspace{20pt}
 - \Big( A_{\infty}(z) \frac{\partial \tilde{Y}_{\infty}}{\partial \bar{h}} \;
 Y_{\infty}^{-1}
 -\frac{\partial \tilde{Y}_{\infty}}{\partial \bar{h}} \;
 Y_{\infty}^{-1} A_{\infty}(z) \Big)
 \frac{dz}{z^m-\epsilon^m} \wedge d\bar{h} \\
 &=
 - \frac{ \tilde{\Xi}_{l,j}(z)dz } {z^m-\epsilon^m}\wedge d\bar{h} 
 +\frac{ \tilde{\Xi}_{l,j}(z) } {z^m-\epsilon^m} dz\wedge d\bar{h}
 +\frac{A_{\infty}(z)}{z^m-\epsilon^m}
 \frac{\partial \tilde{Y}_{\infty}}{\partial \bar{h}}
 Y_{\infty}^{-1} dz\wedge d\bar{h}
 -\frac{\partial \tilde{Y}_{\infty}} {\partial \bar{h}}
 Y_{\infty}^{-1}
 \frac{A_{\infty}(z)dz}{z^m-\epsilon^m}\wedge d\bar{h}  \\
 &\hspace{20pt}
 - A_{\infty}(z) \frac{\partial \tilde{Y}_{\infty}}{\partial \bar{h}} \; Y_{\infty}^{-1}
 \frac{dz}{z^m-\epsilon^m} \wedge d\bar{h} 
 +\frac{\partial \tilde{Y}_{\infty}} {\partial \bar{h}} 
 Y_{\infty}^{-1}
 A_{\infty}(z)\frac{dz}{z^m-\epsilon^m} \wedge d\bar{h} \\
 &=0.
\end{align*}
So
$\nabla_{\mathbb{P}^1\times{\mathcal V}[\bar{h}],v_{l,j}}^{flat}$
is an integrable connection
and becomes a horizontal lift of
$\nabla_{\mathbb{P}^1\times{\mathcal V},v_{l,j}}$.
\end{proof}

\subsection{Comparison  with the asymptotic property
in the unfolding theory by Hurtubise, Lambert and Rousseau}

In the unfolding theory by Hurtubise, Lambert and Rousseau in
\cite{Hurtubise-Lambert-Rousseau}, \cite{Hurtubise-Rousseau},
unfolded Stokes matrices for unfolded linear differential equations
are defined.
So our integrable connection
$\nabla^{flat}_{\mathbb{P}^1\times{\mathcal V}[\bar{h}],v_{l,j}}$
constructed in Proposition \ref {prop:local horizontal lift}
induces unfolded Stokes matrices
but we cannot expect that
these matrices are constant in $\bar{h}$.
Although we cannot produce any positive result
on the asymptotic property concerned with the integrable connection
$\nabla^{flat}_{\mathbb{P}^1\times{\mathcal V}[\bar{h}],v_{l,j}}$
defined by (\ref{equation:fundamental local flat connection})
in subsection \ref {subsection:local horizontal lift},
it will be worth pointing out what is the difficulty.

We use the same notations as in subsection \ref {subsection:global connection on P^1}
and in subsection \ref {subsection:local horizontal lift}.
We consider the multivalued function
\[
 \tau_{\epsilon}(z):= \int \frac{dz}{z^m-\epsilon^m}
\]
which is single valued on
$\mathbb{P}^1_{\Delta\times\Delta^s}\setminus\Gamma_{\Delta}$.
Under a suitable choice of path integral, we may assume that
$\tau_{\epsilon}(z)$ does not vanish on
$\Gamma_{\Delta}\setminus(\Gamma_{\Delta}\cap D)$.
Let
\begin{align*}
 \varpi \colon
 & [0,1) \times S^1\longrightarrow \Delta \\
 & (s,e^{\sqrt{-1}\psi})\mapsto s e^{\sqrt{-1}\psi}
\end{align*}
be the polar blow up.
We can regard
$\Delta\times[0,1)\times S^1\times\Delta^s\subset
\mathbb{C}\times[0,1)\times S^1\times\Delta^s
\subset\mathbb{P}^1\times[0,1)\times S^1\times\Delta^s$.

By Proposition \ref{proposition:open-covering},
we can take an open neighborhood $U$ of
$\{0\}\times\{0\}\times S^1\times \Delta^s$ in
$\Delta\times [0,1)\times S^1\times\Delta^s$
and an open covering
\[
 U \setminus((\mathrm{id}\times\varpi\times\mathrm{id})^{-1}(D)\cap U))
 =\bigcup_{j=1}^m\bigcup_{0\leq\psi_0\leq 2\pi}\bigcup_{\xi=1}^2 W^{(j)}_{\psi_0,\xi}
\]
such that
any flow of the vector field
\[
 v_{\epsilon,\theta}
 =\mathrm{Re}\left(e^{\sqrt{-1}\theta}
 (z^m-\epsilon^m)\right)
 \frac{\partial}{\partial x}
 +
 \mathrm{Im}\left(e^{\sqrt{-1}\theta}
 (z^m-\epsilon^m)\right)
 \frac{\partial}{\partial y}
\]
starting at a point in
$W^{(j)}_{\psi_0,\xi}$
has an accumulation point in
$(\mathrm{id}\times\varpi^{-1}\times\mathrm{id})^{-1}(D)\cap U$,
where $x=\mathrm{Re}(z)$, $y=\mathrm{Im}(z)$.
Here $\theta=\theta^{(j)}_{\psi_0,\xi}\in\mathbb{R}$ is determined by
$j,\psi_0,\xi$ as in the proof of Proposition \ref {proposition:open-covering}.

We take an open covering
\[
 (\varpi\times\mathrm{id}_{\Delta^s})^{-1}({\mathcal V})
 =
 \bigcup_{b\in  (\varpi\times\mathrm{id}_{\Delta^s})^{-1}({\mathcal V})}
 \tilde{\mathcal V}'_b
\]
by small contractible open subsets ${\mathcal V}'_b$
of $(\varpi\times\mathrm{id}_{\Delta^s})^{-1}({\mathcal V})$.
By Theorem \ref {theorem:existence-of-fundamental-solutions},
we can see that there are an open covering
\[
 (\Delta\times \tilde{\mathcal V}'_b) \cap W^{(j)}_{\psi_0,\xi}
 =\bigcup_{p\in W^{(j)}_{\psi_0,\xi}} S^{(j)}_{\psi_0,\xi,p}
 =\bigcup_{\vartheta} S_{\vartheta}
\]
with $\vartheta=(j,\psi_0,\xi,p)$
and a matrix
\[
 Y_{\vartheta}(z,s,e^{\sqrt{-1}\psi},w,h)=
 \left( \tilde{y}_1^{\vartheta}\big(z,s,e^{\sqrt{-1}\psi},w,h\big) \, , \, \ldots \, , \,
 \tilde{y}_r^{\vartheta}\big(z,s,e^{\sqrt{-1}\psi},w,h \big) \right)
\]
of functions on $S_{\vartheta}\times\Delta_{\delta}$ for some $\delta>0$,
satisfying 
\begin{equation} 
 \frac{ d \tilde{Y}_{\vartheta}(z,s,e^{\sqrt{-1}\psi},w,h) } {dz} =
 -\frac{ A_{\infty}(z,\epsilon,w) + h \tilde{\Xi}^{(l)}_q(z,\epsilon,w) } {z^m-\epsilon^m}
 \; \tilde{Y}_{\vartheta}(z,s,e^{\sqrt{-1}\psi},w,h),
\end{equation}
such that the limit
\begin{equation} \label {equation: limit determining asymptotic}
 \lim_{t\to\infty}  
 \tilde{P}(z_{\theta}(t),h) \:
 \tilde{Y}_{\vartheta}(z_{\theta}(t),h) \;
 \mathrm{Diag}_{\big(\exp\big(\int_{t_0}^t 
 (\nu(\mu_k)(z_{\theta}(t)) + h \mu_k^l z_{\theta}(t)^q)
 e^{\sqrt{-1}\theta} dt\big)\big)} 
 =I_r
\end{equation}
is the identity matrix, where $z_{\theta}(t)$ is a flow of $v_{\epsilon,\theta}$
in $S_{\vartheta}=S^{(j)}_{\psi_0,\xi,p}$
and $\theta=\theta^{(j)}_{\psi_0,\xi}$ is determined from $\vartheta=(j,\psi_0,\xi,p)$.
We denote the restriction of 
$\tilde{Y}_{\vartheta}(z,s,e^{\sqrt{-1}\psi},h)$
to
$S_{\vartheta}[\bar{h}]$ by
$\tilde{Y}_{\vartheta}(z,\bar{h})$
and denote the restriction of
$\tilde{Y}_{\vartheta}(z,s,e^{\sqrt{-1}\psi},h)$
to
$S_{\vartheta}\times \{0\}$ by
$Y_{\vartheta}(z)$.
By (\ref {equation: limit determining asymptotic}), 
we have
\begin{equation} \label {equation: relative bounded limit}
 \lim_{t\to\infty}  
 \tilde{P}(z_{\theta}(t),\bar{h}) \:
 \tilde{Y}_{\vartheta}(z_{\theta}(t),\bar{h}) \;
 \mathrm{Diag}_{\big( \exp\big(\int_{t_0}^t (\nu(\mu_k)(z_{\theta}(t)))
 e^{\sqrt{-1}\theta}dt \big)
 \big(1+\bar{h}\int_{t_0}^t \mu_k^l z_{\theta}(t)^l e^{\sqrt{-1}\theta} dt)\big)\big)} 
 =I_r
\end{equation}
from which
$\tilde{Y}_{\vartheta}(z,\bar{h})
 \Diag_{ ( \exp (\int \nu(\mu_k)(z)\frac{dz}{z^m-\epsilon^m} ))}
 \big(I_r+\bar{h}\Diag_{(\mu_k^l z^q \frac{dz}{z^m-\epsilon^m})}\big)$
is bounded on $S_{\vartheta}[\bar{h}]$
and in particular
$Y_{\vartheta}(z)
 \Diag_{(\exp(\int \nu(\mu_k)(z)\frac{dz}{z^m-\epsilon^m}))}$
is  bounded on $S_{\vartheta}$.

Recall that we can write
$Y_{\infty}(z)=
\begin{pmatrix}
 y^{\infty}_1(z), & \ldots , & y^{\infty}_r(z)
\end{pmatrix}$
for $y^{\infty}_k(z):=\tilde{y}_k(z,0)$.


We take a family of loops
$\gamma\colon [0,1]\times\tilde{\mathcal V}'_b
\longrightarrow \big( \Delta\times\tilde{\mathcal V}'_b \big)
\setminus\Gamma_{\tilde{\mathcal V}'_b}$
satisfying $\gamma(0,w)=\gamma(1,w)$,
$p_2(\gamma(t,w))=w$
and that $\gamma(\bullet,w)\colon[0,1]\longrightarrow \Delta\times\{w\}$
is homotopic to $\tilde{\gamma}(\bullet,w)$  for any $w\in{\mathcal V}$.
From the analysis of flows in Proposition \ref  {proposition:open-covering},
we may assume that there are points
$0=t_1<t_2<\cdots<t_I<1$
such that
$t_i \in S_{\vartheta_i}$,
$\lim_{t\to\infty} z_{\theta_i}(t)=\epsilon \zeta_m^{j_i}$
and that either $j_{i+1}=j_i+1$ or
$j_{i+1}=j_i$
with $\epsilon\zeta_m^{j_i}\in \overline {S_{\vartheta_i}\cap S_{\vartheta_{i+1}}}$
holds.
Here in the case of
$\epsilon\zeta_m^{j_i}\in \overline {S_{\vartheta_i}\cap S_{\vartheta_{i+1}}}$,
we can further assume that
a flow $z_{\theta_i}(t)$ lie in
$S_{\vartheta_i}\cap S_{\vartheta_{i+1}}$
which is accumulated to $\epsilon\zeta_m^{j_i}$ and
a flow
$z_{\theta_{i+1}}(t)$ lie in $S_{\vartheta_i}\cap S_{\vartheta_{i+1}}$
which is accumulated to $\epsilon\zeta_m^{j_i}$.

\begin{lemma} \label {lemma:upper triangular matrix as comparison of asymptotic solutions}
Assume that
flows $z_{\theta}(t)$ (resp.\ $z_{\theta'}$)
of $v_{\epsilon,\theta}$ (resp.\ $v_{\epsilon,\theta'}$)
in $S_{\vartheta}$ (resp.\ $S_{\vartheta'}$)
lie in $S_{\vartheta}\cap S_{\vartheta'}$
for $\vartheta,\vartheta'$
and that
$\lim_{t\to\infty} z_{\theta}(t)=
\lim_{t\to\infty} z_{\theta'}(t)=\epsilon\zeta_m^j
\in \overline {S_{\vartheta}\cap S_{\vartheta'}}$.
We take a permutation $\sigma$ of $\{1,\ldots,r\}$ satisfying
\[
 \mathrm{Re} \big( e^{\sqrt{-1}\theta} \nu(\mu_{\sigma(1)})(\epsilon\zeta_m^j) \big)
 >\cdots>
 \mathrm{Re} \big( e^{\sqrt{-1}\theta} \nu(\mu_{\sigma(r)})(\epsilon\zeta_m^j) \big).
\]
Assume that
\[
 \tilde{Y}_{\vartheta'}(z,\bar{h})= \tilde{Y}_{\vartheta}(z,\bar{h}) 
 C_{\vartheta,\vartheta'}(\bar{h})
\]
holds under an analytic continuation along a path in
$S_{\vartheta}\cup S_{\vartheta'}$.
Then
\[
 (e_{\sigma(1)},\ldots, e_{\sigma(r)})^{-1} C_{\vartheta,\vartheta'}(\bar{h})
 (e_{\sigma(1)},\ldots, e_{\sigma(r)})
 \]
is an upper triangular matrix.
\end{lemma}

\begin{proof}
We put
\[
 \Lambda_k(z,\bar{h}):=\exp\left(\int(\nu(\mu_k)+\bar{h}\mu_k^lz^q)
 \frac{dz}{z^m-\epsilon^m}\right).
\]
If $k<k'$,
then $\Lambda_{\sigma(k)}(z,\bar{h})^{-1}\Lambda_{\sigma(k')}(z,\bar{h})$
tends to $0$ when $z$ tends to $\epsilon\zeta_m^j$.
Note that
\begin{align*}
 (\Diag_{(\Lambda_k(z,\bar{h}))} )^{-1}
 C_{\vartheta,\vartheta'}(\bar{h})  \:
 \Diag_{(\Lambda_k(z,\bar{h}))} 
 &=
 ( \tilde{Y}_{\vartheta}(z,\bar{h}) \Diag_{(\Lambda_k(z,\bar{h}))} )^{-1}
 \tilde{Y}_{\vartheta'}(z,\bar{h}) \: \Diag_{(\Lambda_k(z,\bar{h}))}
\end{align*}
tends to a matrix of bounded functions when $z$ tends to $\epsilon\zeta_m^j$
in $S_{\vartheta}\cap S_{\vartheta'}$.

If we put
\[
 C'(\bar{h}):=
 (e_{\sigma(1)},\ldots,e_{\sigma(r)})^{-1}
 C_{\vartheta,\vartheta'}(\bar{h}) (e_{\sigma(1)},\ldots,e_{\sigma(r)})
 =
 \begin{pmatrix}
  c_{1,1}(\bar{h}) & \cdots & c_{1,r}(\bar{h}) \\
  \vdots & \ddots & \vdots \\
  c_{r,1}(\bar{h}) & \cdots & c_{r,r}(\bar{h})
 \end{pmatrix}
\]
then we have
\begin{align*}
 &
 (e_{\sigma(1)},\ldots,e_{\sigma(r)})^{-1}
 (\Diag_{(\Lambda_k(z,\bar{h}))} )^{-1}
 C_{\vartheta,\vartheta'}(\bar{h})
 \Diag_{(\Lambda_k(z,\bar{h}))}
 (e_{\sigma(1)},\ldots,e_{\sigma(r)}) \\
 &=
 \begin{pmatrix}
  \Lambda_{\sigma(1)} & \cdots & 0 \\
  \vdots & \ddots & \vdots \\
  0 & \cdots & \Lambda_{\sigma(r)}
 \end{pmatrix}^{-1}
 \begin{pmatrix}
  c_{1,1}(\bar{h}) & \cdots & c_{1,r}(\bar{h}) \\
  \vdots & \ddots & \vdots \\
  c_{r,1}(\bar{h}) & \cdots & c_{r,r}(\bar{h})
 \end{pmatrix}
 \begin{pmatrix}
  \Lambda_{\sigma(1)} & \cdots & 0 \\
  \vdots & \ddots & \vdots \\
  0 & \cdots & \Lambda_{\sigma(r)}
 \end{pmatrix} \\
 &=
 \begin{pmatrix}
 c_{1,1}(\bar{h}) & \cdots & 
 \Lambda_{\sigma(1)}(z,\bar{h})^{-1}\Lambda_{\sigma(r)}(z,\bar{h}) c_{1,r}(\bar{h}) \\
 \vdots & \ddots & \vdots  \\
 \Lambda_{\sigma(1)}(z,\bar{h})\Lambda_{\sigma(r)}(z,\bar{h})^{-1} c_{r,1}(\bar{h})
 & \cdots & c_{r,r}(\bar{h})
 \end{pmatrix}.
 \end{align*}
Since $\Lambda_{\sigma(k)}(z,\bar{h})^{-1}\Lambda_{\sigma(k')}(z,\bar{h})$
is divergent for $k>k'$,
we should have
$c_{k',k}(\bar{h})=0$
for $k'>k$
\end{proof}

By an analytic continuation
we can write
\begin{align*}
 \tilde{Y}_{\vartheta_i}(z,\bar{h})
 &=
 \tilde{Y}_{\infty}(z,\bar{h}) C_{\infty,\vartheta_i}(\bar{h}) 
\end{align*}
from which we have
\[
 \tilde{Y}_{\vartheta_{i+1}}(z,\bar{h})=
 \tilde{Y}_{\vartheta_i}(z,\bar{h}) 
 C_{\infty,\vartheta_i}(\bar{h})^{-1} C_{\infty,\vartheta_{i+1}}(\bar{h}).
\]
If $j_i=j_{i+1}$, then
$(e_{\sigma(1)},\ldots, e_{\sigma(r)})^{-1}
C_{\infty,\vartheta_i}(\bar{h})^{-1} C_{\infty,\vartheta_{i+1}}(\bar{h})
 (e_{\sigma(1)},\ldots, e_{\sigma(r)})$
is an upper triangular matrix for a permutation $\sigma$
by Lemma \ref {lemma:upper triangular matrix as comparison of asymptotic solutions}.
The matrix $C_{\infty,\vartheta_i}(\bar{h})$
is  analogous to an unfolded Stokes matrix
given in \cite{Hurtubise-Lambert-Rousseau}
but we cannot say from its construction that
it is constant in $\bar{h}$.

We remark that the restriction 
$\tau_{\epsilon}(z)^{-1}B_{l,q}(z)|_{\epsilon^m=0}
=-(m-1)z^{m-1}B'_{0,l,q}(z)$
to the irregular singular locus $\epsilon^m=0$ is bounded around $z=0$
by its construction.
We can see that
\begin{align*}
 B_{l,q}(z)
 &=
 -\frac{\partial \tilde{Y}_{\infty}(z,\bar{h})} {\partial \bar{h}}
 \: Y_{\infty}(z)^{-1}  \\
 &=
 -\frac{\partial}{\partial\bar{h}}
 \left(\tilde{Y}_{\vartheta_i}(z,\bar{h}) C_{\infty,\vartheta_i}(\bar{h})^{-1}\right)
 (\tilde{Y}_{\vartheta_i}(z,0) C_{\infty,\vartheta_i}(0)^{-1})^{-1}  \\
 &=
 -\frac{\partial \tilde{Y}_{\vartheta_i}(z,\bar{h})} {\partial \bar{h}}
 \: Y_{\vartheta_i}(z)^{-1}
 +Y_{\vartheta_i}(z) \: C_{\infty,\vartheta_i}(0)^{-1}
 \frac{\partial C_{\infty,\vartheta_i}(\bar{h})} {\partial \bar{h}}
 \: Y_{\vartheta_i}(z)^{-1}.
\end{align*}
By the following proposition, we can say
that
$\tau_{\epsilon}(z)^{-1}\dfrac{\partial \tilde{Y}_{\vartheta_i}(z,\bar{h})} 
{\partial \bar{h}}
\: Y_{\vartheta_i}(z)^{-1}$
is bounded on $S_{\vartheta_i}$.
However,
$\tau_{\epsilon}(z)^{-1}
Y_{\vartheta_i}(z)  \: C_{\infty,\vartheta_i}(0)^{-1}
\dfrac{\partial C_{\infty,\vartheta_i}(\bar{h})} {\partial \bar{h}}Y_{\vartheta_i}(z)^{-1}$
 is not bounded unless
\[
 (e_{\sigma(1)},\ldots, e_{\sigma(r)})^{-1}
 C_{\infty,\vartheta_i}(0)^{-1}
 \dfrac{\partial C_{\infty,\vartheta_i}(\bar{h})} {\partial \bar{h}}
  (e_{\sigma(1)},\ldots, e_{\sigma(r)})
\]
is an upper triangular matrix.
So we can not say the boundedness of
$\tau_{\epsilon}(z)^{-1}B_{l,q}(z)$
on $S_{\vartheta}$. 
This is one of the reasons why we cannot get a canonical
global horizontal lift in section \ref  {section:construction of unfolding}.

\begin{proposition} \label{prop:boundedness by asymptotic}
$\tau_{\epsilon}(z)^{-1}\dfrac{\partial}{\partial\bar{h}} 
\tilde{Y}_{\vartheta}(z,\bar{h})
\: Y_{\vartheta}(z)^{-1}$
is bounded on $S_{\vartheta}$.
\end{proposition}

\begin{proof}
Since the limit in (\ref {equation: relative bounded limit}) is uniform in $\bar{h}$,
we can see that
\[
 T_{\vartheta}(z,\bar{h}):=
 \tilde{Y}_{\vartheta}(z,\bar{h}) \:
 \Diag_{\big(\exp\big(\int(\nu(\mu_k)+\bar{h}\mu_k^l z^q)
 \frac{dz}{z^m-\epsilon^m}\big)\big)}
\]
and its partial derivative in $\bar{h}$ is bounded on
$\Delta\times{\mathcal V}[\bar{h}]$.
So
\begin{align*}
 \frac {\partial T_{\vartheta}(z,\bar{h})} {\partial \bar{h}}  
 &=
 \frac{\partial \tilde{Y}_{\vartheta}(z,\bar{h})} {\partial \bar{h}} \:
 \Diag_{\big(\exp\big(\int\nu(\mu_k)\frac{dz}{z^m-\epsilon^m}\big)\big)}
 +
 Y_{\vartheta}(z) \:
 \frac{\partial}{\partial \bar{h}}
 \Diag_{\big(\exp\big(\int(\nu(\mu_k)+\bar{h}\mu_k^l z^q)\frac{dz}{z^m-\epsilon^m}\big)\big)}   \\
 &=
 \frac{\partial \tilde{Y}_{\vartheta}(z,\bar{h})} {\partial \bar{h}} 
 \: Y_{\vartheta}(z)^{-1} T_{\vartheta}(z)
 +
 T_{\vartheta}(z) \:
 \Diag_{\big( \int \mu_k^l z^q \frac{dz}{z^m-\epsilon^m}\big)}
\end{align*}
is bounded on $S_{\vartheta}$.
So it is sufficient to show that
$\tau_{\epsilon}(z)^{-1} \Diag_{\int( \mu_k^l z^q\frac{dz}{z^m-\epsilon^m})}$
is bounded.

If $\epsilon=0$, 
\begin{align*}
 \left| \tau_{\epsilon}(z)^{-1} \int \mu_k^l z^q \frac{dz}{z^m} \right|
 &=
 \left| \left(-\frac{1}{(m-1) z^{m-1}} \right)^{-1}
 \int 
 \frac{\mu_k^l}{z^{m-q}} dz \right|  \\
 &=
 \left|-(m-1) z^{m-1} \left( \frac{-\mu_k^l} {(m-q-1) z^{m-q-1}} 
 +(\text{constant}) \right) \right|  \\
 &\leq
 \frac{(m-1)|\mu_k^l| \: |z|^{q}} {m-q-1} +(\text{constant})
\end{align*}
is bounded on each
$S_{\vartheta}\cap (\Delta\times\varpi^{-1}(0)\times\Delta^s)$.

If $\epsilon\neq 0$, we can write
\[
 \mu_k^l z^q \dfrac{dz}{z^m-\epsilon^m}
 =
 \sum_{j=1}^m \frac{a^j_k} {z-\epsilon\zeta_m^j}dz
\]
for $0\leq q\leq m-2$.
Then
\begin{align*}
 \left| \tau_{\epsilon}(z)^{-1} \int \mu_k^l z^q \frac{dz}{z^m-\epsilon^m} \right|
 &=\left| \left(\sum_{j'=1}^m
 \frac{\log(z-\epsilon\zeta_m^{j'})}
 {\epsilon^{m-1}\prod_{j''\neq j'}(\zeta_m^{j''}-\zeta_m^{j'})}
 \right)^{-1}
 \int_{z_0}^z \sum_{j=1}^m a^j_k  
 \frac{dz}{z-\zeta_m^j\epsilon}  \right| \\
 &=\left| \left(\sum_{j'=1}^m
 \frac{\log(z-\epsilon\zeta_m^{j'}) }
 {\epsilon^{m-1}\prod_{j''\neq j'}(\zeta_m^{j''}-\zeta_m^{j'})} 
 \right)^{-1}
 \Big( \sum_{j=1}^m  a^j_k
 \log(z-\zeta_m^j\epsilon)   +(\text{constant}) \Big) \right| \\
 &\leq \sum_{j=1}^m
 \left| \sum_{j'=1}^m
 \frac{\log(z-\epsilon\zeta_m^{j'})} {\prod_{j''\neq j'}(\zeta_m^{j''}-\zeta_m^{j'})}
 \right| ^{-1} 
 |a^j_k| 
 \frac{ |\log (z-\zeta_m^j\epsilon) | } { |\epsilon|^{m-1} } +(\text{constant})
\end{align*}
is bounded on each $S_{\vartheta}\cap\{\epsilon\neq 0\}$.

Thus
$\tau_{\epsilon}(z)^{-1} \Diag_{(\int \mu_k^l z^q \frac{dz}{z^m-\epsilon^m})}$
is bounded on $S_{\vartheta}$ and
the proposition follows.
\end{proof}

In a precise setting in the paper \cite{Hurtubise-Rousseau}
by Hurtubise and Rousseau,
they consider a linear differential equation on $\mathbb{P}^1$
with poles along the unfolding  divisor and two regular singular points
$\infty^{\text{H-R}}$, $R^{\text{H-R}}$.
So we should associate a relative connection
\[
 \nabla'_{\mathbb{P}^1\times{\mathcal V}[\bar{h}],v_{l,q}}\colon
 ({\mathcal O}_{\mathbb{P}^1\times{\mathcal V}[\bar{h}]}^{hol})^{\oplus r}
 \longrightarrow 
 ({\mathcal O}_{\mathbb{P}^1\times{\mathcal V}[\bar{h}]}^{hol})^{\oplus r} \otimes 
 \Omega^1_{\mathbb{P}^1\times{\mathcal V}[\bar{h}]/
 {\mathcal V}[\bar{h}]}
 \left(D_{{\mathcal V}[\bar{h}]}\cup
 \big(\{\infty^{\rm H-R},R^{\rm H-R}\}\times{\mathcal V}[\bar{h}]\big)\right)^{hol}
\]
such that
$\nabla'_{\mathbb{P}^1\times{\mathcal V}[\bar{h}],v_{l,q}} 
\big|_{\Delta\times{\mathcal V}[\bar{h}]}
\cong
\nabla_{\mathbb{P}^1\times{\mathcal V}[\bar{h}],v_{l,q}} 
\big|_{\Delta\times{\mathcal V}[\bar{h}]}$.
In other words, we decompose the monodromy of
$\nabla_{\mathbb{P}^1\times{\mathcal V}[\bar{h}],v_{l,q}} $ along
$\infty\times{\mathcal V}[\bar{h}]$ to
the composition of the monodromy of
$\nabla'_{\mathbb{P}^1,\bar{h},v^{(l)}_q}$ around $\infty^{\text{H-R}}$
and that around a point  $R^{\text{H-R}}$ other than $\infty^{\text{H-R}}$.
The monodromy of $\nabla'$ around $R^{\text{H-R}}$
reflects the analytic continuation of fundamental solutions
of $\nabla_{\mathbb{P}^1\times{\mathcal V}[\bar{h}],v_{l,q}}$
along the `inner side' of the unfolded divisor
$D_{{\mathcal V}[\bar{h}]}$.
We can take a fundamental solution
$Y'_{\infty^{\text{H-R}}} (z,\bar{h})$
of $\nabla'_{\mathbb{P}^1\times{\mathcal V}[\bar{h}],v_{l,q}} $
near $\infty^{\text{H-R}}\times{\mathcal V}[\bar{h}]$.
Then we can write
\begin{align*}
 \tilde{Y}_{\vartheta_i}(z,\bar{h})
 &=
 Q(z,\bar{h})Y'_{\infty^{\text{H-R}}}(z,\bar{h})
 C'_{\infty^{\text{H-R}},\vartheta_i}(\bar{h}) 
\end{align*}
for an invertible matrix $Q(z,\bar{h})$ giving the isomorphism
$\nabla'_{\mathbb{P}^1\times{\mathcal V}[\bar{h}],v_{l,q}} |_{\Delta\times{\mathcal V}[\bar{h}]}
\cong
\nabla_{\mathbb{P}^1\times{\mathcal V}[\bar{h}],v_{l,q}} |_{\Delta\times{\mathcal V}[\bar{h}]}$.
Here the matrix $C'_{\infty^{\text{H-R}},\vartheta_i}(\bar{h})$
is a more close analogue of an unfolded Stokes matrix
in \cite{Hurtubise-Rousseau}.
Though there is an ambiguity in the choice of
$C'_{\infty^{\text{H-R}},\vartheta_i}(\bar{h})$
coming from the choices of
$\nabla'_{\mathbb{P}^1,{\mathcal V}[\bar{h}],v_{l,q}}$ and
$Y'_{\infty^{\text{H-R}}}(z,\bar{h})$,
we cannot say from its construction that 
$C'_{\infty^{\text{H-R}},\vartheta_i}(\bar{h})$ is constant in $\bar{h}$,
because we do not know the compatibility of the asymptotic properties
between $\tilde{Y}_{\vartheta_i}(z,\bar{h})$ and
$\tilde{Y}_{\vartheta_{i+1}}(z,\bar{h})$
when $j_i\neq j_{i+1}$.


We remark that in the general setting in \cite{Hurtubise-Lambert-Rousseau},
\cite{Hurtubise-Rousseau},
the asymptotic property of solutions of unfolded linear differential equations
is far more complicated than our one parameter deformation case.


\section {Construction of an unfolded
generalized isomonodromic deformation}
\label {section:construction of unfolding}

\subsection {Setting of the moduli space for an unfolded
generalized isomonodromic deformation}
\label {subsection:moduli setting}

In this subsection, we introduce the moduli theoretic setting for
describing an unfolding of the unramified irregular singular
generalized isomonodromic deformation.
Let us recall the independent variables of the usual
unramified irregular singular generalized isomonodromic deformation,
which basically comes from \cite{Jimbo-Miwa-Ueno}.
We consider unramified irregular singular connections
$\nabla\colon E\longrightarrow E\otimes\Omega^1_C(m_1t_1+\cdots +m_nt_n)$
and we take a certain \'etale covering
$U\longrightarrow {\mathcal M}_{g,n}^{\rm reg}$
of the moduli stack ${\mathcal M}_{g,n}^{\rm reg}$ of $n$-pointed smooth projective curves
of genus $g$
with a universal family $({\mathcal C},\tilde{t}_1,\ldots,\tilde{t}_n)$
over $U$.
Then
\[
 \left( \Omega^1_{{\mathcal C}/U}(m_1\tilde{t}_1+\cdots+m_n\tilde{t}_n)/
 \Omega^1_{{\mathcal C}/U}(\tilde{t}_1+\cdots+\tilde{t}_n) \right)^r
\]
becomes the space of independent variables
of the generalized isomonodromic deformation of $(E,\nabla)$.
We will give a certain perturbation of this space.


First we construct a smooth covering
${\mathcal H}\longrightarrow {\mathcal M}_{g,n}^{\rm reg}$
of the moduli stack of $n$-pointed smooth projective curves
of genus $g$ as follows.
If $g=0$, we put  $H:=\Spec\mathbb{C}$,
$Z:=\mathbb{P}^1$ and regard $Z$ as
a curve over $H$.
If $g=1$, we put
$H:=\left.\left\{ D\in | {\mathcal O}_{\mathbb{P}^2}(3) \right|
 \, | \, \text{$D$ is a smooth cubic curve} \right\}$
and we set
$Z \subset \mathbb{P}^2\times H$
as the universal family of smooth cubic curves.
Assume that $g\geq 2$.
Then we fix $l\geq 3$ and put
$N:=h^0(C,\omega_C^{\otimes l})-1$ for a smooth projective irreducible curve $C$ of genus $g$,
where $\omega_C$ is the canonical bundle of $C$.
We consider the locally closed subscheme
$H\subset\Hilb_{\mathbb{P}^N}$ of the Hilbert scheme
which parametrizes the closed subvarieties $C\subset \mathbb{P}^N$
isomorphic to the $l$-th canonical embeddings
$C\hookrightarrow\mathbb{P}(H^0(C,\omega_C^{\otimes l}))$
of smooth projective curves $C$ of genus $g$.
Let $Z\subset\mathbb{P}^N\times H$
be the universal family.
For any case $g\geq 0$, we define a Zariski open subset
\[
 {\mathcal H} := 
 \left.\left\{ (p_i) \in \prod_{i=1}^n Z \right| \text{$p_i\neq p_{i'}$ for $i\neq i'$} \right\} 
\]
of the fiber product $\prod_{i=1}^nZ$ of $n$ copies of $Z$ over $H$.
Similarly we define a Zariski open subset
\[
 {\mathcal P} :=
 \left.\left\{ (p_i),(p^{(i)}_j))\in \prod_{i=1}^n Z\times_H\prod_{i=1}^n\prod_{j=1}^{m_i} Z
 \right| \text{$p_i\neq p_{i'}$, $p_i\neq p^{(i')}_{j'}$ and
 $p^{(i)}_j\neq p^{(i')}_{j'}$ for $i\neq i'$}\right\}
\]
of the fiber product
$\prod_{i=1}^nZ\times_H\prod_{i=1}^n\prod_{j=1}^{m_i} Z$ 
of $n+\sum_{i=1}^n m_i$ copies of $Z$ over $H$.
Then there is a canonical projection
\[
 \pi_{{\mathcal P},{\mathcal H}}\colon{\mathcal P}\longrightarrow{\mathcal H}
\]
defined by $\pi_{{\mathcal P},{\mathcal H}}((p_i),(p^{(i)}_j))=(p_i)$
and there is a section
\[
 \tau_{{\mathcal H},{\mathcal P}}\colon{\mathcal H}\longrightarrow{\mathcal P}
\]
defined by $\tau_{{\mathcal H},{\mathcal P}}((p_i))=((p_i),(p_i))$.

We put ${\mathcal C}:=Z\times_H{\mathcal H}$ and
${\mathcal C}_{\mathcal P}=Z\times_H{\mathcal P}$.
Then there are universal sections
$\sigma_i\colon {\mathcal P}\longrightarrow{\mathcal C}_{\mathcal P}$ and
$\sigma^{(i)}_j\colon {\mathcal P}\longrightarrow {\mathcal C}_{\mathcal P}$
defined by
$\sigma_i((p_i),(p^{(i)}_j))=(p_i,(p_i),(p^{(i)}_j))$,
$\sigma^{(i)}_j((p_i),(p^{(i)}_j))=(p^{(i)}_j,(p_i),(p^{(i)}_j))$
which satisfy
$\sigma_i({\mathcal P})\cap\sigma_{i'}({\mathcal P})=\emptyset$,
$\sigma_i({\mathcal P})\cap\sigma^{(i')}_{j'}({\mathcal P})=\emptyset$ and
$\sigma^{(i)}_j({\mathcal P})\cap\sigma^{(i')}_{j'}({\mathcal P})=\emptyset$
for $i\neq i'$ and any $j,j'$.
We define divisors
${\mathcal D}_i$, ${\mathcal D}^{(i)}_j$, ${\mathcal D}^{(i)}$ and
${\mathcal D}$ on ${\mathcal C}_{\mathcal P}$ by putting
${\mathcal D}_i:=\sigma_i({\mathcal P})$,
${\mathcal D}^{(i)}_j:=\sigma^{(i)}_j({\mathcal P})$,
${\mathcal D}^{(i)}:=\sum_{j=1}^{m_i}{\mathcal D}^{(i)}_j$
and
${\mathcal D}:=\sum_{i=1}^n{\mathcal D}^{(i)}$.
We consider the closed subvariety
$\tau_{{\mathcal H},{\mathcal P}}({\mathcal H})\subset{\mathcal P}$
which can be written
\[
 \tau_{{\mathcal H},{\mathcal P}}({\mathcal H})
 =\left\{ ((p_i),(p^{(i)}_j))\in {\mathcal P}
 \left| \text{$p_i=p^{(i)}_j$ for any $i,j$} \right\}\right..
\]

It was necessary to set the differential form
(\ref {equation: fundamental differential form})
in subsection \ref {subsection:former main theorem}
for the formulation of the moduli space of $(\tilde{\bnu},\tilde{\bmu})$-connections.
For its construction, we use the following lemma.

\begin{lemma} \label {lem: choice of local coordinate}
 Let $f\colon X\longrightarrow S$ be a smooth morphism of algebraic schemes
 over $\Spec\mathbb{C}$
 such that all the geometric fibers of $X$ over $S$ are one dimensional.
 Assume that
 $X\longrightarrow S$ has a section $\sigma\colon S\longrightarrow X$.
 Consider the diagonals
 \begin{align*}
  \Delta_{1,2} &=\{ (x,y,z)\in X\times_S X\times_S X | x=y\} \\
  \Delta_{1,3} &=\{(x,y,z)\in X\times_S X\times_S X | x=z\} \\
  \Delta_{2,3} &=\{(x,y,z)\in X\times_S X\times_S X | y=z\} .
 \end{align*}
We denote the ideal sheaf of ${\mathcal O}_{X\times_SX\times_SX}$
defining $\Delta_{i,j}$ by $I_{\Delta_{i,j}}$.
Then for each closed point $p\in \sigma(S)\subset X$,
there exists an affine open neighborhood $W$ of $(p,p,p)$ in
$X\times_S X\times_S X$
such that
the ideal $I_{\Delta_{1,2}}|_W$ is generated by
a section $z_{1,2}\in H^0(W,I_{\Delta_{1,2}}|_W)$,
the ideal $I_{\Delta_{1,3}}|_W$ is generated by
a section $z_{1,3}\in H^0(W,I_{\Delta_{1,3}}|_W)$,
the ideal $I_{\Delta_{2,3}}|_W$ is generated by
$z_{1,2}-z_{1,3}$
and that
$z_{1,2}-z_{1,3}\in p_{2,3}^{-1}({\mathcal O}_V)$
for some open neighborhood $V$ of $(p,p)$ in $X\times_S X$.
\end{lemma}

\begin{proof}
If we put $s=f(p)$,
the stalk of $I_{\sigma(S)}\otimes{\mathcal O}_{X_s}=I_{\sigma(S)\cap X_s}$
at $p$ is a principal ideal of ${\mathcal O}_{X_s,p}$.
So there is an affine open neighborhood $U$ of $p$ in $X$
and a section $z\in H^0(U,I_{\sigma(S)}|_U)$
such that $z|_{U_s}$ is a generator of $I_{\sigma(S)\cap U_s}$.
By Nakayama's lemma, $z$ becomes a generator of $I_{\sigma(S)}|_U$
after shrinking $U$ if necessary.
Since 
\[
 \overline{z\otimes 1\otimes 1-1\otimes z\otimes 1}=\overline{dz\otimes 1}
 \in I_{\Delta_{1,2}}/I_{\Delta_{1,2}}^2|_{U\times_S U\times_SU}
 =\Omega^1_{U/S}\otimes_S{\mathcal O}_U
\]
is a generator after shrinking $U$,
Nakayama's lemma implies that
$z_{1,2}:=z\otimes 1\otimes 1-1\otimes z\otimes 1$ becomes
a generator of $I_{\Delta_{1,2}}|_W$ for some affine open neighborhood
$W$ of $(p,p,p)$ in $X\times_S X\times_S X$.
If we put
$z_{1,3}:=z\otimes 1\otimes 1-1\otimes 1\otimes z$,
then $z_{1,3}$ similarly becomes a generator of $I_{\Delta_{1,3}}|_W$
after shrinking $W$ again.
Since
\[
 z_{1,2}-z_{1,3}
 =(z\otimes 1\otimes 1-1\otimes z\otimes 1)-(z\otimes 1\otimes 1-1\otimes 1\otimes z)
 =1\otimes(1\otimes z-z\otimes 1)\in p_{2,3}^{-1}({\mathcal O}_{U\times_S U}),
\]
and $1\otimes(1\otimes z-z\otimes 1)$
becomes a generator of $I_{\Delta_{2,3}}$ after shrinking $W$,
the lemma is proved.
\end{proof}

\begin{remark} \label {remark: of lemma choice of local coordinate}  \rm
In the above lemma,  we may further assume that
$p_{2,3}^{-1}(V) \cap \Delta_{1,2} \subset W$
and 
$p_{2,3}^{-1}(V) \cap \Delta_{1,3} \subset W$.
\end{remark}

For each point $h_0 \in {\mathcal H}$,
we consider the fiber ${\mathcal C}_{h_0}$ of ${\mathcal C}_{\mathcal P}$ over
$\tau_{{\mathcal H},{\mathcal P}}(h_0)$.
If we put $p_0:=\sigma_i(\tau_{{\mathcal H},{\mathcal P}}(h_0))$,
then, by Lemma \ref {lem: choice of local coordinate}
and Remark \ref {remark: of lemma choice of local coordinate},
there is an affine open neighborhood $W$ of $p_0$ in 
${\mathcal C}_{\mathcal P}$ and sections $z^{(i)},z^{(i)}_j\in H^0(W,{\mathcal O}_W)$
such that
$z^{(i)}=0$ is a defining equation of ${\mathcal D}_i\cap W$,
$z^{(i)}_j=0$ is a defining equation of ${\mathcal D}^{(i)}_j\cap W$ for each $j$
and $z^{(i)}-z^{(i)}_j\in {\mathcal O}_{\mathcal P}$ for any $i,j$.
So we can take an affine open neighborhood ${\mathcal P'}$ of $p_0$ in
${\mathcal P}$ and an affine open covering
$\{{\mathcal U}_{\alpha}\}$ of ${\mathcal C}\times_{\mathcal P}{\mathcal P'}$
such that 
$\{\alpha \, | \ {\mathcal D}^{(i)}\times_{\mathcal P}{\mathcal P'}\subset
{\mathcal U}_{\alpha}\}
=\{\alpha \, | \ {\mathcal D}_i\times_{\mathcal P}{\mathcal P'}\subset
 {\mathcal U}_{\alpha}\}$
consists of a single element $\alpha_i$
for each $i$,
$\sharp\{ i \, | \  ({\mathcal D}_i\times_{\mathcal P}{\mathcal P'})\cap
{\mathcal U}_{\alpha}\neq\emptyset\} \leq 1$
and
$\sharp\{ i \, | \  ({\mathcal D}^{(i)}\times_{\mathcal P}{\mathcal P'})\cap
{\mathcal U}_{\alpha}\neq\emptyset\} \leq 1$
for each $\alpha$,
$({\mathcal D}_i)_{\mathcal P'}$ coincides with the zero scheme
of $z^{(i)}\in H^0({\mathcal U}_{\alpha_i},{\mathcal O}_{{\mathcal U}_{\alpha_i}})$,
$({\mathcal D}^{(i)}_j)_{\mathcal P'}$ coincides with the zero scheme of
$z^{(i)}_j\in  H^0({\mathcal U}_{\alpha_i},{\mathcal O}_{{\mathcal U}_{\alpha_i}})$,
$z^{(i)}_j-z^{(i)}\in {\mathcal O}_{\mathcal P'}$
and
$(z^{(i)}_j-z^{(i)})|_{\tau_{{\mathcal H},{\mathcal P}}({\mathcal H})
\times_{\mathcal P}{\mathcal P'}}=0$
for any $i,j$.
We denote the image of $z^{(i)}$ and $z^{(i)}_j$ in
${\mathcal O}_{2{\mathcal D}^{(i)}\times_{\mathcal P}{\mathcal P'}}$
by $\bar{z}^{(i)}$ and $\bar{z}^{(i)}_j$, respectively.
We put
\[
 \zeta_{m_i}:=\exp\left( \frac{2\pi \sqrt{-1}} {m_i} \right)
\]
and consider the locus
\[
 {\mathcal B}:=\left\{ h\in{\mathcal P'} \left|
 \begin{array}{l}
 \text{$(z^{(i)}_j-z^{(i)}) |_h =\zeta_{m_i}^j(z^{(i)}_{m_i}-z^{(i)}) |_h $ for any $i,j$} \\
 \text{and $(z^{(i)}_{m_i}-z^{(i)}) |_h = (z^{(i')}_{m_{i'}}-z^{(i')}) |_h $ for any $i,i'$}
 \end{array}
 \right\}\right.
\]
which is a smooth subvariety of ${\mathcal P'}$.
Note that we have $z^{(i)}_j-z^{(i)}\in H^0({\mathcal O}_{\mathcal P'})$
from the choice of ${\mathcal P'}$.
If we put $\epsilon(h):=(z^{(i)}_{m_i}-z^{(i)})(h)$
for $h\in {\mathcal B}$,
then 
$\epsilon \colon {\mathcal B}
 \longrightarrow \mathbb{A}^1=\mathbb{C}$
is an algebraic function.
There is a diagram
\[
 \xymatrix{
 {\mathcal C}_{\mathcal B} \ar[r] \ar[rd]
  & {\mathcal B} \ar[d] \ar[r]^-{\epsilon} & \mathbb{A}^1=\mathbb{C} \\
 & {\mathcal H} &
 }
\]
and we have
$z^{(i)}_j=z^{(i)}+\zeta_{m_i}^j\epsilon$
on ${\mathcal U}_{\alpha_i}\times_{\mathcal P}{\mathcal B}\subset{\mathcal C}_{\mathcal B}$.

Let  $(w_1,\ldots,w_s)$ 
be a holomorphic coordinate system
in a neighborhood of $h_0$ in ${\mathcal H}$.
Then we can see that
$(z^{(i)},\epsilon,w_1,\ldots,w_s)$ becomes a holomorphic coordinate system
in a neighborhood of $\sigma_i(\tau_{{\mathcal H},{\mathcal P}}(h_0))$
in ${\mathcal U}_{\alpha_i}\times_{\mathcal P'}{\mathcal B}$.
So we can take a disk $\Delta_{\epsilon_0}=\{z\in\mathbb{C} \,|\, |z|<\epsilon_0\}$
for small $\epsilon_0>0$,
an analytic open neighborhood ${\mathcal B'}$ of $\tau_{{\mathcal H},{\mathcal P}}(h_0)$ in
$\epsilon^{-1}(\Delta_{\epsilon_0})\subset{\mathcal B}$
and an analytic open neighborhood 
$U_i\subset{\mathcal U}_{\alpha_i}\times_{\mathcal P'}{\mathcal B'}$
of  $\sigma_i(\tau_{{\mathcal H},{\mathcal P}}(h_0))$ containing
${\mathcal D}^{(i)}\times_{\mathcal P}{\mathcal B'}$ such that
$U_i\cap U_{i'}=\emptyset$ for $i\neq i'$ and
\begin{equation} \label {equation:local biholomorphic}
 U_i \xrightarrow[\sim]{(z^{(i)},\epsilon,w_1,\ldots,w_s)}
 \Delta_a\times\Delta_{\epsilon_0}\times\Delta_r^s
\end{equation}
becomes biholomorphic for any $i$, 
where $a,r>0$, $\Delta_a=\{z\in\mathbb{C} \, | \, |z|<a\}$
and
$\Delta_r^s=\overbrace{\Delta_r\times\cdots\times\Delta_r}^s$
with $\Delta_r=\{z\in\mathbb{C} \, | \, |z|<r\}$.
We define a subset $\Gamma^{(i)}_{j,b}$
of the fiber ${\mathcal C}_{b}$
of ${\mathcal C}\times_{\mathcal H}{\mathcal B'}$
over $b\in {\mathcal B'}$ by setting
\[
 \Gamma^{(i)}_{j,b}:=
 \bigcup_{0\leq s\leq 1}\left\{
 x\in {\mathcal C}_{b}\cap U_i
 \left| (z^{(i)}+s\zeta_{m_i}^j\epsilon)(x)=0\right\}\right..
\]
Then $\Gamma^{(i)}_{j,b}$ becomes a simple path in 
${\mathcal C}_{b}$
joining the two points $({\mathcal D}_i)_{b}$
and $({\mathcal D}^{(i)}_j)_{b}$
for $\epsilon\in\Delta_{\epsilon_0}\setminus\{0\}$
because of the bijectivity of (\ref {equation:local biholomorphic}).
If we set
\[
 \Gamma^{(i)}_j :=\bigcup_{ b \in {\mathcal B'} }
 \Gamma^{(i)}_{j,b}, \quad
 \Gamma := \bigcup_{i,j}\Gamma^{(i)}_j,
\]
then $\Gamma^{(i)}_j$ and $\Gamma$ are closed subsets of
${\mathcal C}\times_{\mathcal P}{\mathcal B'}$
with respect to the analytic topology.

We fix distinct complex numbers
$\mu^{(i)}_1,\ldots,\mu^{(i)}_r\in\mathbb{C}$ for $i=1,\ldots,n$
and write $\bmu=(\mu^{(i)}_k)^{1\leq i\leq n}_{1\leq k\leq r}$.
Then we put
\[
 \varphi^{(i)}_{\bmu}(T):=(T-\mu^{(i)}_1)(T-\mu^{(i)}_2)\cdots(T-\mu^{(i)}_r)
 \in \mathbb{C} [T].
\]
We take an integer $a\in\mathbb{Z}$ and a tuple of complex numbers
$\blambda=(\lambda^{(i)}_k)\in\mathbb{C}^{nr}$
satisfying 
\begin{itemize}
\item[(i)]
 $ \displaystyle a+\sum_{i=1}^n\sum_{k=1}^r \lambda^{(i)}_k=0$,
\item[(ii)]
 $\lambda^{(i)}_k-\lambda^{(i)}_{k'}\notin\mathbb{Z}$ for $k\neq k'$.
\end{itemize}
We define an algebraic variety ${\mathcal T}_{\bmu,\blambda}$ over
${\mathcal B}$ whose set of $S$-valued points
is given by 
\[
 {\mathcal T}_{\bmu,\blambda}(S):=
 \left\{(\nu^{(i)}(T))_{1\leq i\leq n}
 \left|
 \begin{array}{l}
 \displaystyle \nu^{(i)}(T)=\sum_{l=0}^{r-1} \sum_{j=0}^{m_i-1}
 c^{(i)}_{l,j} \: (z^{(i)})^j \: T^l 
 \text{ with $c^{(i)}_{l,j}\in H^0({\mathcal O}_S)$}  \\
 \text{ satisfying the following (a) and (b)}
 \end{array}
 \right\}\right.
\]
for any noetherian scheme $S$ over ${\mathcal B}$;
\begin{itemize}
\item[(a)] $\displaystyle \lambda^{(i)}_k
 =\sum_{l=0}^{r-1} c^{(i)}_{l,m_i-1} (\mu^{(i)}_k)^l$
for each $i,k$
\item[(b)] $\nu^{(i)}(\mu^{(i)}_k)|_p \neq \nu^{(i)}(\mu^{(i)}_{k'})|_p$ for $k\neq k'$,
 $1\leq i\leq n$ and any $p\in {\mathcal D}^{(i)}_S$.
\end{itemize}
Here we intend to regard $(c^{(i)}_{l,j})^{1\leq i\leq n}_{0\leq l\leq r-1,0\leq j\leq m_i}$
with $c^{(i)}_{l,j}\in H^0(S,{\mathcal O}_S)$
as a precise data denoted by $(\nu^{(i)}(T))$.
We take a universal family
\[
 \tilde{\nu}^{(i)}(T) =\sum_{l=0}^{r-1} \sum_{l=0}^{m-1} c^{(i)}_{l,j} \: (z^{(i)})^j  \, T^l
\]
with $c^{(i)}_{l,j}\in H^0({\mathcal O}_{{\mathcal T}_{\bmu,\blambda}})$ 
and write $\tilde{\bnu}:=(\tilde{\nu}^{(i)}(T))$.
If we denote by $\tilde{\nu}^{(i)}_s$, $(c^{(i)}_{l,j})_s$
the restrictions of $\tilde{\nu}^{(i)}$, $c^{(i)}_{l,j}$
to $s\in {\mathcal T}_{\bmu,\blambda}$, respectively,
we can see by Lemma \ref{lemma:residue-identity} that
\begin{align*}
 \sum_{p\in {\mathcal D}^{(i)}_s}\res_{p}\left( \tilde{\nu}^{(i)}_s(\mu^{(i)}_k)
 \frac {d\bar{z}^{(i)}} {(\bar{z}^{(i)})^{m_i}-\epsilon^{m_i}}
 \right)
 &=
 \sum_{l=0}^{r-1} \sum_{j=0}^{m_i-1}  (c^{(i)}_{l,j})_s (\mu^{(i)}_k)^l
 \sum_{p\in {\mathcal D}^{(i)}_s} 
 \res_{z^{(i)}=p} \left( 
 \frac{ (z^{(i)})^j d z^{(i)} } {(z^{(i)})^{m_i}-\epsilon^{m_i}} \right) \\
 &=
 -\sum_{l=0}^{r-1} \sum_{j=0}^{m_i-1}  (c^{(i)}_{l,j})_s (\mu^{(i)}_k)^l
 \res_{z^{(i)}=\infty} \left(
 \frac{ (z^{(i)})^j d z^{(i)} } {(z^{(i)})^{m_i}-\epsilon^{m_i}} \right) \\
 &=\sum_{l=0}^{r-1} (c^{(i)}_{l,m_i-1})_s (\mu^{(i)}_k)^l .
\end{align*}
So the equality (a)
in the definition of ${\mathcal T}_{\bmu,\blambda}$ means the equality
\begin{equation} \label {equation:residue-identity}
 \lambda^{(i)}_k=\sum_{p\in {\mathcal D}^{(i)}_s}
 \res_p \left( \tilde{\nu}^{(i)}_s (\mu^{(i)}_k)
 \frac {d\bar{z}^{(i)}}  {(\bar{z}^{(i)})^{m_i}-\epsilon^{m_i}} \right)
\end{equation}
where $\sum_{p\in {\mathcal D}^{(i)}_s}$ runs over the set theoretical points $p$ of ${\mathcal D}^{(i)}_s$.
For each point $p=\epsilon\zeta_m^j\in {\mathcal D}^{(i)}_s$,
we have
\[
 \tilde{\nu}^{(i)}_s(\mu^{(i)}_k)\big|_p
 =\sum_{l=0}^{r-1} \sum_{j'=0}^{m_i-1}
 (c^{(i)}_{l,j'})_s(\epsilon(s) \zeta_m^j)^{j'}(\mu^{(i)}_k)^l
\]
for $1\leq k\leq r$.
The condition (b) in the definition of ${\mathcal T}_{\bmu,\blambda}$
is that
\[
 \sum_{l=0}^{r-1} \sum_{j'=0}^{m_i-1} (c^{(i)}_{l,j'})_s (\epsilon(s)\zeta_m^j)^{j'} (\mu^{(i)}_k)^l \neq
 \sum_{l=0}^{r-1} \sum_{j'=0}^{m_i-1} (c^{(i)}_{l,j'})_s (\epsilon(s)\zeta_m^j)^{j'} (\mu^{(i)}_{k'})^l 
\]
for $k\neq k'$,
when $\epsilon(s)\neq 0$ and that
\[
 \sum_{l=0}^{r-1} (c^{(i)}_{l,0})_s (\mu_k)^l \neq
 \sum_{l=0}^{r-1} (c^{(i)}_{l,0})_s (\mu_{k'})^l
\]
for $k\neq k'$ when $\epsilon(s)=0$.

By Theorem \ref{theorem:algebraic-moduli-unfolding}, there is a relative moduli space
\begin{equation}\label{equation:relative-moduli-scheme}
 \pi_{{\mathcal T}_{\bmu,\blambda}} \colon 
 M^{\balpha}_{{\mathcal C},{\mathcal D}}(\tilde{\bnu},\bmu)\longrightarrow 
 {\mathcal T}_{\bmu,\blambda}
\end{equation}
of $(\tilde{\bnu},\bmu)$-connections over ${\mathcal T}_{\bmu,\blambda}$.
Note that the morphism $\pi_{{\mathcal T}_{\bmu,\blambda}}$
in (\ref{equation:relative-moduli-scheme})
is an algebraic smooth morphism of quasi-projective schemes.
We consider the pull-back diagram
\[
 \begin{CD}
 M^{\balpha}_{{\mathcal C},{\mathcal D}}(\tilde{\bnu},\bmu)\times_{\mathcal B} {\mathcal B'}
 @>>>  M^{\balpha}_{{\mathcal C},{\mathcal D}}(\tilde{\bnu},\bmu) \\
 @VVV   @VVV \\
 {\mathcal B'} @>>>  {\mathcal B}
 \end{CD}
\]
where the horizontal arrows are open immersions as analytic spaces.

\subsection{Unramified irregular singular generalized isomonodromic deformation}
\label {subsection:irregular singular generalized isomonodromic deformation}

The unramified irregular singular generalized isomonodromic deformation
is the well-known theory by Jimbo, Miwa and Ueno,
which is completely given in \cite{Jimbo-Miwa-Ueno}, \cite{Jimbo-Miwa-2},
\cite{Jimbo-Miwa-3}
with explicit calculations using formal solutions based on the
Malgrange-Sibuya theorem (\cite[Theorem 4.5.1]{Babbitt-Varadarajan}).
We recall here a moduli theoretic construction of the unramified irregular singular
generalized isomonodromic deformation given in \cite {Inaba-Saito},
which is valid in a higher genus case.

Recall that there are compositions of morphisms
$M^{\balpha}_{{\mathcal C},{\mathcal D}}(\tilde{\bnu},\bmu) 
\longrightarrow
{\mathcal T}_{\bmu,\blambda}\longrightarrow {\mathcal B}
\stackrel{\epsilon}\longrightarrow \Delta_{\epsilon_0}$.
We consider the fibers
\[
 {\mathcal B}_{\epsilon=0}
 :=
 {\mathcal B}\times_{\Delta_{\epsilon_0}} 
 \{0\} ,   \quad
 {\mathcal T}_{\bmu,\blambda,\epsilon=0}
 :=
 {\mathcal T}_{\bmu,\blambda}\times_{\mathcal B} 
 {\mathcal B}_{\epsilon=0}, \quad
 M^{\balpha}_{{\mathcal C},{\mathcal D}}(\tilde{\bnu},\bmu)_{\epsilon=0}
 :=
 M^{\balpha}_{{\mathcal C},{\mathcal D}}(\tilde{\bnu},\bmu) 
 \times_{\mathcal B} {\mathcal B}_{\epsilon=0}
\]
over $\epsilon=0\in\Delta_{\epsilon_0}$.
Then
$\pi_{{\mathcal T}_{\bmu,\blambda,\epsilon=0}}\colon
M^{\balpha}_{{\mathcal C},{\mathcal D}}(\tilde{\bnu},\bmu)_{\epsilon=0}
\longrightarrow {\mathcal T}_{\bmu,\blambda,\epsilon=0}$
is the relative moduli space of unramified irregular singular connections.
In our moduli theoretic setting,
the unramified irregular singular generalized isomonodromic deformation
is given in \cite[Theorem 6.2]{Inaba-Saito} as an algebraic splitting
\[
 \Psi_0\colon
 \pi_{{\mathcal T_{\bmu,\blambda,\epsilon=0}}}^*
 T_{{\mathcal T}_{\bmu,\blambda,\epsilon=0}}
 \longrightarrow
 T_{M^{\balpha}_{{\mathcal C},{\mathcal D}}(\tilde{\bnu},\bmu)_{\epsilon=0}}
\]
of the canonical surjection
$T_{M^{\balpha}_{{\mathcal C},{\mathcal D}}(\tilde{\bnu},\bmu)_{\epsilon=0}}
\xrightarrow{d\pi_{{\mathcal T_{\bmu,\blambda,\epsilon=0}}}}
(\pi_{{\mathcal T_{\bmu,\blambda,\epsilon=0}}})^*
 T_{{\mathcal T}_{\bmu,\blambda,\epsilon=0}}$.
Here we use the symbol $\Psi_0$ instead of the symbol $D$ used in \cite{Inaba-Saito},
for the purpose of avoiding confusion with the divisor
of singularity of the connection.

Let us recall the construction of $\Psi_0$.
For each Zariski open subset
${\mathcal T}'_0\subset 
{\mathcal T}_{\bmu,\blambda,\epsilon=0}$
and for each vector field
$v\in H^0({\mathcal T}'_0,
T_{{\mathcal T}_{\bmu,\blambda,\epsilon=0}}|_{{\mathcal T}'_0})$,
let
${\mathcal T}'_0[v]:={\mathcal T}'_0\times\Spec\mathbb{C}[h]/(h^2)
\xrightarrow{I_v} {\mathcal T}'_0$
be the corresponding morphism
satisfying $I_v\otimes\mathbb{C}[h]/(h)=\mathrm{id}_{{\mathcal T}'_0}$.
If we put
\begin{align*}
 \nu^{(i)}_{0,hor}(T)
 &:=
 \sum_{l=0}^{r-1} \sum_{j=0}^{m_i-1}
 (I_v^*(c^{(i)}_{l,j})_{{\mathcal T}'_0}-\bar{h}v((c^{(i)}_{l.j})_{{\mathcal T}'_0}))
 (z^{(i)})^j T^l  \\
 \nu^{(i)}_{0,v}(T)
 &:=
 \sum_{l=0}^{r-1} \sum_{j=0}^{m_i-2}
  v((c^{(i)}_{l,j})_{{\mathcal T}'_0}) (z^{(i)})^j T^l,
\end{align*}
then we have
$I_v^*(\tilde{\nu}^{(i)}(T))=\nu^{(i)}_{0,hor}(T)+\bar{h}\nu^{(i)}_{0,v}(T)$
and $\nu^{(i)}_{0,hor}(T)$ is the pullback of $\tilde{\nu}^{(i)}(T)$
via the trivial projection
${\mathcal T}'_0[v]={\mathcal T}'_0\times\Spec\mathbb{C}[h]/(h^2)
\longrightarrow {\mathcal T}'_0\hookrightarrow {\mathcal T}_{\bmu,\blambda}$.
We consider the fiber product
${\mathcal C}_{{\mathcal T}'_0[v]}=
{\mathcal C}_{{\mathcal T}_0}\times_{{\mathcal T}'_0}
{\mathcal T}'_0[v]$
with respect to $I_v\colon {\mathcal T}'_0[v] \longrightarrow {\mathcal T}'_0$
and the trivial projection
${\mathcal C}_{{\mathcal T}'_0} \longrightarrow {\mathcal T}'_0$.
We denote the pullback of $z^{(i)}$
under the morphism
${\mathcal C}_{{\mathcal T}'[v]}
={\mathcal C}_{\mathcal T'}\times_{\mathcal T'}
({\mathcal T}'\times\Spec\mathbb{C}[h]/(h^2))
\longrightarrow {\mathcal C}_{\mathcal T'}$
by $\tilde{z}^{(i)}$.

For some \'etale surjective morphism
$\tilde{M}\longrightarrow
M^{\balpha}_{{\mathcal C},{\mathcal D}}(\tilde{\bnu},\bmu)$,
there is a universal family
$(\tilde{E},\tilde{\nabla},\{\tilde{N}^{(i)}\})$
on ${\mathcal C}_{\tilde{M}}$.
We put
$\tilde{M}'_0:= \tilde{M}\times_{{\mathcal T}_{\bmu,\blambda}}{\mathcal T}'_0$,
$\tilde{M}'_0[v]:=\tilde{M}\times_{{\mathcal T}_{\bmu,\blambda}}
{\mathcal T}'_0[v]$
and denote the restriction of
$(\tilde{E},\tilde{\nabla},\{\tilde{N}^{(i)}\})$
to ${\mathcal C}_{\tilde{M}'_0}$ by
$(\tilde{E}_{\tilde{M}'_0},\tilde{\nabla}_{\tilde{M}'_0},
\{\tilde{N}^{(i)}_{\tilde{M}'_0}\})$.
In the following definition, 
${\mathcal C}_{\tilde{M}'_0[v]}$ means the fiber product
${\mathcal C}_{{\mathcal T}'_0} \times_{{\mathcal T}'_0}
\tilde{M}'_0[v]$
with respect to the canonical morphism
${\mathcal C}_{{\mathcal T}'_0}\longrightarrow {\mathcal T}'_0$
and the composition
$\tilde{M}'_0[u] \longrightarrow {\mathcal T}'_0[u]
\xrightarrow{I_v} {\mathcal T}'_0$.
On the other hand, relative differentials in
$\Omega^1_{{\mathcal C}_{\tilde{M}'_0[v]}/\tilde{M}'_0 }$
are with respect to the composition
${\mathcal C}_{\tilde{M}'_0[v]}
\longrightarrow \tilde{M}'_0[v]
=\tilde{M}'_0\times\Spec\mathbb{C}[h]/(h^2)
\longrightarrow \tilde{M}'_0$
of the trivial projections.

\begin{definition} \label {def:horizontal lift for irregular case}
\rm
$\big({\mathcal E}^v_0,\nabla^v_0,\{{\mathcal N}^{(i)}_{0,v}\}\big)$
is a horizontal lift of
$\big(\tilde{E}_{\tilde{M}'_0},\tilde{\nabla}_{\tilde{M}'_0},
\{\tilde{N}^{(i)}_{\tilde{M}'_0}\}\big)$
with respect to $v$ if
\begin{itemize}
\item[(1)] ${\mathcal E}^v_0$ is an algebraic vector bundle on
${\mathcal C}_{\tilde{M}'_0[v]}$ of rank $r$,
\item[(2)]
$\nabla^v_0\colon {\mathcal E}^v_0 \longrightarrow
 {\mathcal E}^v_0\otimes
 \Omega^1_{{\mathcal C}_{\tilde{M}'_0[v]}/\tilde{M}'_0 }
 ({\mathcal D}_{\tilde{M}'_0[v]})$
 is a morphism of sheaves satisfying $\nabla^v(fa)=a\otimes df+f\nabla^v(a)$
 for $f\in{\mathcal O}^{hol}_{{\mathcal C}_{\tilde{M}'_0[v]}}$
 and $a\in {\mathcal E}^v$,
\item[(3)]
$\nabla^v_0$ is integrable in the sense that
the restriction of $\nabla^v_0$ to any open set
$U[v]\subset {\mathcal C}_{\tilde{M}'_0[v]}\setminus{\mathcal D}_{\tilde{M}'_0[v]}$
satisfying
${\mathcal E}^v |_{U[v]}\cong \left({\mathcal O}_{U[v]}\right)^{\oplus r}$
is expressed by
\[
 \left({\mathcal O}_{U[v]}\right)^{\oplus r}
 \ni \begin{pmatrix} f_1 \\ \vdots \\ f_r \end{pmatrix}
 \mapsto
 \begin{pmatrix} df_1 \\ \vdots \\ df_r \end{pmatrix}
 +
 \left( \tilde{A} d\tilde{z}
 +B d\overline{h}\right)
 \begin{pmatrix} f_1 \\ \vdots \\ f_r \end{pmatrix}
 \in \left( {\mathcal O}_{U[v]} \right)^{\oplus r}
 \otimes
 \Omega^1_{{\mathcal C}_{\tilde{M}'_0[v]}/\tilde{M}'_0 }
 ({\mathcal D}_{\tilde{M}'_0[v]})
\]
satisfying
$d \left( \tilde{A} d\tilde{z}
 +B d\overline{h}\right)
+ \left[ \left( \tilde{A} d\tilde{z}
 +B d\overline{h}\right),
 \left( \tilde{A} d\tilde{z}
 +B d\overline{h}\right) \right]
 =0$
in 
$\Omega^2_{{\mathcal C}_{\tilde{M}'_0[v]}/\tilde{M}'_0 }
 (2{\mathcal D}_{\tilde{M}'_0[v]})$,
\item[(4)]
${\mathcal N}^{(i)}_{0,v}\colon {\mathcal E}^v_0 |_{{\mathcal D}^{(i)}_{\tilde{M}'_0[v]}}
\longrightarrow  {\mathcal E}^v_0 |_{{\mathcal D}^{(i)}_{\tilde{M}'_0[v]}}$
is an endomorphism satisfying
$\varphi^{(i)}_{\bmu}({\mathcal N}^{(i)}_{0,v})=0$,
\item[(5)]
the relative connection
$\overline{\nabla^v_0}$  defined by the composition
\[
 \overline{\nabla^v_0}\colon
 {\mathcal E}^v_0 \xrightarrow{\nabla^v_0}
 {\mathcal E}^v_0\otimes \Omega^1_{{\mathcal C}_{\tilde{M}'_0[v]}/\tilde{M}'_0 }
 ({\mathcal D}_{\tilde{M}'_0[v]})
 \longrightarrow
 {\mathcal E}^v_0\otimes\Omega^1_{{\mathcal C}_{\tilde{M}'_0[v]}/\tilde{M}'_0[v]}
 ({\mathcal D}_{\tilde{M}'_0[v]})
\]
satisfies
\[
 \displaystyle (\nu^{(i)}_{0,hor}+\bar{h}\nu^{(i)}_{0,v})({\mathcal N}^{(i)}_{0,v})
 \;
 \frac{d\tilde{z}^{(i)}}{(\tilde{z}^{(i)})^{m_i}}
 =\overline{\nabla^v_0} \big|_{{\mathcal D}^{(i)}_{\tilde{M}'_0[v]}}
\]
for any $i$ and
\item[(6)]
$\big({\mathcal E}^v_0,\overline{\nabla^v_0},\{ {\mathcal N}^{(i)}_{0,v} \}\big)
\otimes {\mathcal O}_{\tilde{M}'_0[v]}/
\bar{h}{\mathcal O}_{\tilde{M}'_0[v]}
\cong
\big(\tilde{E}_{\tilde{M}'_0},\tilde{\nabla}_{\tilde{M}'_0},
\{ \tilde{N}^{(i)}_{\tilde{M}'_0}\}\big)$.
\end{itemize}
\end{definition}

The following proposition is essentially given in the proof of
\cite[Theorem 6.2]{Inaba-Saito} and we omit its proof here.
\begin{proposition} \label {prop:horizontal lift for irregular case}
There exists a unique horizontal lift
$\big({\mathcal E}^v_0,\nabla^v_0,\{{\mathcal N}^{(i)}_{0,v}\}\big)$
of
$\big(\tilde{E}_{\tilde{M}'_0},\tilde{\nabla}_{\tilde{M}'_0},
\{\tilde{N}^{(i)}_{\tilde{M}'_0}\}\big)$
with respect to $v$
\end{proposition}

For each vector field
$v\in H^0({\mathcal T}'_0,
T_{{\mathcal T}_{\bmu,\blambda,\epsilon=0}}|_{{\mathcal T}'_0})$,
the horizontal lift of
$\big(\tilde{E}_{\tilde{M}'_0},\tilde{\nabla}_{\tilde{M}'_0},
\{\tilde{N}^{(i)}_{\tilde{M}'_0}\}\big)$
with respect to $v$ induces a relative connection
$\big({\mathcal E}^v_0,\overline{\nabla^v_0},\{{\mathcal N}^{(i)}_{0,v}\}\big)$
which gives a morphism
$\tilde{M}'_0[v] \longrightarrow M^{\balpha}_{{\mathcal C},{\mathcal D}}
(\tilde{\bnu},\bmu)$
making the diagram
\[
 \begin{CD}
   \tilde{M}'_0[v] & \longrightarrow &
   M^{\balpha}_{{\mathcal C},{\mathcal D}}(\tilde{\bnu},\bmu) \\
   @VVV     @VVV \\
   {\mathcal T}'_0[v] & \xrightarrow{I_v} {\mathcal T}'_0 \hookrightarrow &
   {\mathcal T}_{\bmu,\blambda}
 \end{CD}
\]
commutative.
This morphism corresponds to a section of
$T_{M^{\balpha}_{{\mathcal C},{\mathcal D}}
(\tilde{\bnu},\bmu)_{\epsilon=0}}
\otimes{\mathcal O}_{\tilde{M}'_0}$
over $\tilde{M}'_0$
which descends to a vector field
$\Phi_0(v)\in 
H^0 \big(
\pi_{{\mathcal T}_{\bmu,\blambda,\epsilon=0}}^{-1}({\mathcal T}'_0),
T_{M^{\balpha}_{{\mathcal C},{\mathcal D}}
(\tilde{\bnu},\bmu)_{\epsilon=0}} \big|
_{ \pi_{{\mathcal T}_{\bmu,\blambda,\epsilon=0}}^{-1}({\mathcal T}'_0) } \big)$.
We can show that the correspondence
\[
 T_{{\mathcal T}_{\bmu,\blambda,\epsilon=0}}
 \ni v \mapsto \Phi_0(v)
 \in
 (\pi_{{\mathcal T}_{\bmu,\blambda,\epsilon=0}})_*
 T_{M^{\balpha}_{{\mathcal C},{\mathcal D}}
 (\tilde{\bnu},\bmu)_{\epsilon=0}}
\]
is an
${\mathcal O}_{{\mathcal T}_{\bmu,\blambda,\epsilon=0}}$-homomorphism.
We omit its proof because it is the same as that of Proposition 
\ref {prop:splitting is a homomorphism}
which is given later.
So $\Phi_0$ is equivalent to the morphism
\begin{equation} \label {equation:irregular generalized isomonodromic splitting}
 \Psi_0 \colon
 (\pi_{{\mathcal T}_{\bmu,\blambda,\epsilon=0}})^*
 T_{{\mathcal T}_{\bmu,\blambda,\epsilon=0}}
 \longrightarrow
 T_{M^{\balpha}_{{\mathcal C},{\mathcal D}}
 (\tilde{\bnu},\bmu)_{\epsilon=0}}.
\end{equation}

We devote the rest of this subsection
to the proof of  the integrability of the subbundle
$\im\Psi_0\subset
T_{M^{\balpha}_{{\mathcal C},{\mathcal D}}
(\tilde{\bnu},\bmu)_{\epsilon=0}}$.
The integrability of the irregular singular generalized isomonodromic deformation
in the zero genus case is proved by Jimbo, Miwa and Ueno in
\cite[Theorem 4.2]{Jimbo-Miwa-Ueno},
which is extended by Bremer and Sage in
\cite[Theorem 5.1]{Bremer-Sage-2}.
Although the integrability is almost a consequence of the Malgrange-Sibuya isomorphism
\cite[Theorem 4.5.1]{Babbitt-Varadarajan} in a general case
as in \cite{Boalch-2},
it will be worth giving a proof of the integrability of $\Psi_0$,
because the situation in an unfolded case is different.

For the proof of the integrability condition of $\Psi_0$,
we extend the definition of horizontal lift given
in Definition \ref{def:horizontal lift for irregular case}.
We consider  a morphism
\[
 u\colon
 {\mathcal T}'_0[u]:=
 {\mathcal T}'_0\times\Spec\mathbb{C}[h_1,h_2]/(h_1^2,h_2^2)
 \longrightarrow {\mathcal T}'_0\subset
 {\mathcal T}_{\bmu,\blambda,\epsilon=0}
\]
satisfying
$u\otimes\mathbb{C}[h_1,h_2]/(h_1,h_2)=\mathrm{id}_{{\mathcal T}'_0}$
and write
\[
 u^*\tilde{\nu}^{(i)}(T)
 =\nu_{hor}^{(i)}(T)+\nu_1^{(i)}(T)\bar{h}_1+\nu^{(i)}_2(T)\bar{h}_2
 +\nu^{(i)}_{1,2}(T)\bar{h}_1\bar{h}_2
\]
where $\nu_{hor}^{(i)}(T)$
is the pullback of
$\tilde{\nu}^{(i)}(T)$ by the composition
${\mathcal T}'_0\times\Spec\mathbb{C}[h_1,h_2]/(h_1^2,h_2^2)
\longrightarrow {\mathcal T}'_0\hookrightarrow
{\mathcal T}_{\bmu,\blambda}$
of the trivial projection and the inclusion
and $\nu^{(i)}_1(T),\nu^{(i)}_2(T),\nu^{(i)}_{1,2}(T)$
are pullbacks of polynomials in ${\mathcal O}_{{\mathcal D}^{(i)}_{{\mathcal T}'_0}}[T]$
via the trivial projection
${\mathcal T}'_0\times\Spec\mathbb{C}[h_1,h_2]/(h_1^2,h_2^2)
\longrightarrow {\mathcal T}'_0$.

We consider the fiber product
$\tilde{M}'_0[u]:=
\tilde{M}'_0\times_{{\mathcal T}'_0} {\mathcal T}'_0[u]$
with respect to the canonical morphism
$\tilde{M}'_0\longrightarrow {\mathcal T}'_0$
and
${\mathcal T}'_0[u]\xrightarrow{u}{\mathcal T}'_0$.
We can extend the notion of horizontal lift given in
Definition \ref {def:horizontal lift for irregular case}
to the morphism
$u\colon {\mathcal T}'_0\times\Spec\mathbb{C}[h_1,h_2]/(h_1^2,h_2^2)
\longrightarrow {\mathcal T}'_0$.

We say that a tuple
$\big({\mathcal E}^u_0,\nabla^u_0,
\{ {\mathcal N}^{(i)}_{0,u}\}\big)$
is a horizontal lift of 
$\big(\tilde{E}_{\tilde{M}'_0},\tilde{\nabla}_{\tilde{M}'_0},
\{\tilde{N}^{(i)}_{\tilde{M}'_0}\}\big)$
with respect to $u$ if
${\mathcal E}^u_0$ is a locally free sheaf on
${\mathcal C}_{\tilde{M}'_0[u]}$,
$\nabla^u_0\colon
 {\mathcal E}^u_0 \longrightarrow
 {\mathcal E}^u_0 \otimes
 \Omega^1_{{\mathcal C}_{\tilde{M}'_0[u]/\tilde{M}'_0}}({\mathcal D}_{\tilde{M}'_0[u]})$
is an integrable connection
and
${\mathcal N}^{(i)}_{0,u}\colon {\mathcal E}^u_0 \big|_{{\mathcal D}^{(i)}_{\tilde{M}'_0[u]}}
\longrightarrow  {\mathcal E}^u_0 \big|_{{\mathcal D}^{(i)}_{\tilde{M}'_0[u]}}$
is an endomorphism 
such that the conditions (3), (4), (5) and (6)
of Definition \ref {def:horizontal lift for irregular case}
hold after replacing $v$ by $u$.
Then we have the following:

\begin{lemma}  \label {lem:horizontal lift h_1h_2 in irregular case}
There exists a unique horizontal lift
$\big({\mathcal E}^u_0,\nabla^u_0,
\{ {\mathcal N}^{(i)}_{0,u}\}\big)$
of 
$\big(\tilde{E}_{\tilde{M}'_0},\tilde{\nabla}_{\tilde{M}'_0},
\{\tilde{N}^{(i)}_{\tilde{M}'_0}\}\big)$
with respect to $u$.
\end{lemma}
 
\begin{proof}
We consider the restriction of $\tilde{\nabla}_{\tilde{M}'_0}$
to an affine open neighborhood
$U^{(i)}$ of ${\mathcal D}^{(i)}_{\epsilon=0}$
such that
$\tilde{E}_{\tilde{M}'_0}\big|_{U^{(i)}}\cong {\mathcal O}_{U^{(i)}}^{\oplus r}$.
It can be written
\[
 \tilde{\nabla}_{\tilde{M}'_0} \big|_{U^{(i)}} \colon
 {\mathcal O}_{U^{(i)}}^{\oplus r} \ni
 \begin{pmatrix} f_1 \\ \vdots \\ f_r \end{pmatrix}
 \mapsto
 \begin{pmatrix} df_1 \\ \vdots \\ df_r \end{pmatrix}
 +
 A(z^{(i)})\frac { dz^{(i)} } { (z^{(i)})^{m_i} }
  \begin{pmatrix} f_1 \\ \vdots \\ f_r \end{pmatrix}
 \in 
 {\mathcal O}_{U^{(i)}}^{\oplus r} \otimes 
 \Omega^1_{U^{(i)}/\tilde{M}'_0}({\mathcal D}^{(i)}_{\tilde{M}'_0}).
\]
Here we may assume that
\[
 A(z^{(i)}) \big|_{3{\mathcal D}^{(i)}_{\tilde{M}'_0}}
 =
 \Diag_{(\tilde{\nu}(\mu_k))} \big|_{3{\mathcal D}^{(i)}_{\tilde{M}'_0}}.
\]
We can take a lift
$A(\tilde{z}^{(i)})$ of $A(z^{(i)})$ as a matrix of algebraic functions
on $U^{(i)}_{\tilde{M}'_0[u]}$
satisfying
\begin{equation} \label {equation:constant in h_1,h_2}
 \frac{\partial A(\tilde{z}^{(i)})} {\partial \bar{h}_1}=
 \frac{\partial A(\tilde{z}^{(i)})} {\partial \bar{h}_2}=0.
\end{equation}
Indeed for an arbitrary lift
$\tilde{A}(\tilde{z}^{(i)})$ of $A(z^{(i)})$,
we can write
$d\tilde{A}=A_0 d\tilde{z}^{(i)} +A_1d\bar{h}_1+A_2d\bar{h}_2$
with respect to the identification
\[
 \Omega^1_{U^{(i)}_{\tilde{M}'_0[u]}/\tilde{M}'_0}
 =
 {\mathcal O}_{U^{(i)}_{\tilde{M}'_0[u]}} d\tilde{z}^{(i)}
 \oplus
 {\mathcal O}_{U^{(i)}_{\tilde{M}'_0[u]}} d\bar{h}_1
 \oplus
 {\mathcal O}_{U^{(i)}_{\tilde{M}'_0[u]}} d\bar{h}_2.
\]
Here relative differential forms  in
$\Omega^1_{U^{(i)}_{\tilde{M}'_0[u]}/\tilde{M}'_0}$
are with respect to the composition of the trivial projections
$U^{(i)}_{\tilde{M}'_0[u]} \longrightarrow \tilde{M}'_0[u]
=\tilde{M}'_0\times\Spec\mathbb{C}[h_1,h_2]/(h_1^2,h_2^2)
\longrightarrow \tilde{M}'_0$.
Then the replacement
\[
 A(\tilde{z}^{(i)}):=
 \tilde{A}-\bar{h}_1A_1-\bar{h}_2A_2
 +\bar{h}_1\bar{h}_2\frac{\partial A_1}{\partial \bar{h}_2}
\]
satisfies the condition (\ref{equation:constant in h_1,h_2}) because of the equalities
\[
 \frac{\partial A_1}{\partial \bar{h}_2}
 =
 \frac{\partial^2 \tilde{A}}{\partial \bar{h}_1\partial\bar{h}_2}
 =
 \frac{\partial A_2}{\partial \bar{h}_1}.
\]
We put
\begin{align*}
 &B_1(\tilde{z}^{(i)})
 :=
 \Diag_{(\int \nu^{(i)}_1(\mu_k)
 \frac{d\tilde{z}^{(i)}} {(\tilde{z}^{(i)})^{m_i}})}, \quad
 B_2(\tilde{z}^{(i)})
 :=
 \Diag_{(\int \nu^{(i)}_2(\mu_k)
 \frac{d\tilde{z}^{(i)}} {(\tilde{z}^{(i)})^{m_i}})}, \\
 &B_{1,2}(\tilde{z}^{(i)})
 :=
 \Diag_{(\int \nu^{(i)}_{1,2}(\mu_k) 
 \frac{d\tilde{z}^{(i)}} {(\tilde{z}^{(i)})^{m_i}})}.
\end{align*}
Note that $(\tilde{z}^{(i)})^{m_i-1}B_1(\tilde{z}^{(i)})$, $(\tilde{z}^{(i)})^{m_i-1}B_2(\tilde{z}^{(i)})$
and
$(\tilde{z}^{(i)})^{m_i-1}B_{1,2}(\tilde{z}^{(i)})$ are
matrices of polynomials in $\tilde{z}^{(i)}$,
because 
$\nu^{(i)}_1(\mu_k)
 \dfrac{d\tilde{z}^{(i)}} {(\tilde{z}^{(i)})^{m_i}}$,
$\nu^{(i)}_2(\mu_k)
 \dfrac{d\tilde{z}^{(i)}} {(\tilde{z}^{(i)})^{m_i}}$
and
$\nu^{(i)}_{1,2}(\mu_k)
\dfrac{d\tilde{z}^{(i)}} {(\tilde{z}^{(i)})^{m_i}}$
have no residue part.
If we define
\begin{align*}
 C_1(\tilde{z}^{(i)}) \frac{d\tilde{z}^{(i)}} {(\tilde{z}^{(i)})^{m_i}}
 &:=
 d B_1(\tilde{z}^{(i)})
 +
 \big[ A(\tilde{z}^{(i)}) , B_1(\tilde{z}^{(i)}) \big] \frac{d\tilde{z}^{(i)}} {(\tilde{z}^{(i)})^{m_i}} \\
 C_2(\tilde{z}^{(i)}) \frac{d\tilde{z}^{(i)}} {(\tilde{z}^{(i)})^{m_i}}
 &:=
 d B_2(\tilde{z}^{(i)})
 +
 \big[ A(\tilde{z}^{(i)}) , B_2(\tilde{z}^{(i)}) \big] \frac{d\tilde{z}^{(i)}} {(\tilde{z}^{(i)})^{m_i}},
\end{align*}
we have
$C_1(\tilde{z}^{(i)})\big|_{2{\mathcal D}^{(i)}_{\tilde{M}'_0[u]}}
=\Diag_{ ( \nu^{(i)}_1(\mu_k) ) }\big|_{2{\mathcal D}^{(i)}_{\tilde{M}'_0[u]}}$
and
$C_2(\tilde{z}^{(i)})\big|_{2{\mathcal D}^{(i)}_{\tilde{M}'_0[u]}}
=\Diag_{ ( \nu^{(i)}_2(\mu_k) ) }\big|_{2{\mathcal D}^{(i)}_{\tilde{M}'_0[u]}}$.
Since $B_1(\tilde{z}^{(i)})$, $B_2(\tilde{z}^{(i)})$, $dB_1(\tilde{z}^{(i)})$
and $dB_2(\tilde{z}^{(i)})$
commute to each other, we have
\begin{align*}
 \big[ C_1(\tilde{z}^{(i)}), B_2(\tilde{z}^{(i)}) \big] 
 \frac{d\tilde{z}^{(i)}} {(\tilde{z}^{(i)})^{m_i}}
 &=
 \left[ dB_1(\tilde{z}^{(i)})+ \big[ A(\tilde{z}^{(i)}), B_1(\tilde{z}^{(i)}) \big]
 \frac{d\tilde{z}^{(i)}} {(\tilde{z}^{(i)})^{m_i}} , B_2(\tilde{z}^{(i)}) \right] \\
 &=
 \big[\big[ A(\tilde{z}^{(i)}), B_1(\tilde{z}^{(i)}) \big] , B_2(\tilde{z}^{(i)})\big] 
 \frac{d\tilde{z}^{(i)}} {(\tilde{z}^{(i)})^{m_i}} \\
 &=
 \big[\big[ A(\tilde{z}^{(i)}), B_2(\tilde{z}^{(i)}) \big] , B_1(\tilde{z}^{(i)})\big] 
 \frac{d\tilde{z}^{(i)}} {(\tilde{z}^{(i)})^{m_i}} \\
 &=
 \big[ C_2(\tilde{z}^{(i)}), B_1(\tilde{z}^{(i)}) \big] 
 \frac{d\tilde{z}^{(i)}} {(\tilde{z}^{(i)})^{m_i}}.
\end{align*}
If we put
\[
 C(\tilde{z}^{(i)}):= \big[ C_1(\tilde{z}^{(i)}),B_2(\tilde{z}^{(i)}) \big]
 =
  \big[ C_2(\tilde{z}^{(i)}), B_1(\tilde{z}^{(i)}) \big],
\]
then we can see that $C(\tilde{z}^{(i)})$
is a matrix of algebraic functions on $U^{(i)}_{\tilde{M}'_0[u]}$
such that
$C(\tilde{z}^{(i)})\big|_{{\mathcal D}^{(i)}_{\tilde{M}'_0[u]}}=0$.
We can check the integrability of
\[
 \eta=
 (A+\bar{h}_1C_1+\bar{h}_2 C_2+\bar{h}_1 \bar{h}_2 C) 
 \frac {d\tilde{z}^{(i)}} {(\tilde{z}^{(i)})^{m_i}}
 + B_1 d\bar{h}_1 + B_2 d\bar{h}_2
\]
by the calculation
\begin{align*}
 d\eta+[\eta,\eta]
 &=
 (C_1+\bar{h}_2C)d\bar{h}_1\wedge \frac{d\tilde{z}^{(i)}} {(\tilde{z}^{(i)})^{m_i}}
 +
 (C_2+\bar{h}_1C)d\bar{h}_2\wedge \frac{d\tilde{z}^{(i)}} {(\tilde{z}^{(i)})^{m_i}}
 +dB_1\wedge d\bar{h}_1+dB_2\wedge d\bar{h}_2 \\
 &\quad
 +\big([A,B_1]+\bar{h}_2[C_2,B_1]\big) 
 \frac{d\tilde{z}^{(i)}} {(\tilde{z}^{(i)})^{m_i}}\wedge d\bar{h}_1
 +\big([A,B_2]+\bar{h}_1[C_1,B_2]\big) \frac{d\tilde{z}^{(i)}} {(\tilde{z}^{(i)})^{m_i}}\wedge d\bar{h}_2 \\
 &=
 \left( dB_1+(-C_1+[A,B_1])\frac{d\tilde{z}^{(i)}} {(\tilde{z}^{(i)})^{m_i}} \right) \wedge d\bar{h}_1
 +
 \left( dB_2+(-C_2+[A,B_2])\frac{d\tilde{z}^{(i)}} {(\tilde{z}^{(i)})^{m_i}} \right) \wedge d\bar{h}_2 \\
 & \quad +
 \bar{h}_2 ( -C+[C_2,B_1] )\frac{d\tilde{z}^{(i)}} {(\tilde{z}^{(i)})^{m_i}}  \wedge d\bar{h}_1
 +
 \bar{h}_1 ( -C+[C_1,B_2] )\frac{d\tilde{z}^{(i)}} {(\tilde{z}^{(i)})^{m_i}}  \wedge d\bar{h}_2 \\
 &=0.
\end{align*}
If we put
\[
 C_{1,2}(\tilde{z}^{(i)}) \frac{d\tilde{z}^{(i)}} {(\tilde{z}^{(i)})^{m_i}} 
 :=
 dB_{1,2}(\tilde{z}^{(i)})
 +\big[ A(\tilde{z}^{(i)}),B_{1,2}(\tilde{z}^{(i)}) \big]
 \frac{d\tilde{z}^{(i)}} {(\tilde{z}^{(i)})^{m_i}} ,
\]
then the connection matrix
\[
 \tilde{\eta}:=\eta+\bar{h}_1\bar{h}_2C_{1,2}\frac{d\tilde{z}^{(i)}} {(\tilde{z}^{(i)})^{m_i}}
 +\bar{h}_2 B_{1,2}(\tilde{z}^{(i)}) d\bar{h}_1
 +\bar{h}_1 B_{1,2}(\tilde{z}^{(i)}) d\bar{h}_2
\]
satisfies the integrability condition
\begin{align*}
 d\tilde{\eta}+[\tilde{\eta},\tilde{\eta}]
 &=
 d\eta+[\eta,\eta]
 +\bar{h}_2 C_{1,2} \: d\bar{h}_1\wedge \frac{d\tilde{z}^{(i)}} {(\tilde{z}^{(i)})^{m_i}}
 +\bar{h}_1 C_{1,2} \: d\bar{h}_2\wedge\frac{d\tilde{z}^{(i)}} {(\tilde{z}^{(i)})^{m_i}}  
 +\bar{h}_2 dB_{1,2}\wedge d\bar{h}_1  \\
 &\quad
 +\bar{h}_1 dB_{1,2} \wedge d\bar{h}_2 
 +\bar{h}_2 [A,B_{1,2}] \frac{d\tilde{z}^{(i)}} {(\tilde{z}^{(i)})^{m_i}} \wedge d\bar{h}_1
 +\bar{h}_1 [A,B_{1,2}] \frac{d\tilde{z}^{(i)}} {(\tilde{z}^{(i)})^{m_i}} \wedge d\bar{h}_2  \\
 &=
 \bar{h}_2\left( dB_{1,2}+(-C_{1,2}+[A,B_{1,2}])
 \frac{d\tilde{z}^{(i)}} {(\tilde{z}^{(i)})^{m_i}}\right) \wedge d\bar{h}_1 \\
 & \quad
 +\bar{h}_1\left( dB_{1,2}+(-C_{1,2}+[A,B_{1,2}]) 
 \frac{d\tilde{z}^{(i)}} {(\tilde{z}^{(i)})^{m_i}}\right) \wedge d\bar{h}_2 \\
 &=0.
\end{align*}
Then the connection
\[
 \nabla^u_{U^{(i)}}\colon
 {\mathcal O}_{U^{(i)}_{\tilde{M}'_0[u]}}^{\oplus r}
 \longrightarrow
 {\mathcal O}_{U^{(i)}_{\tilde{M}'_0[u]}}^{\oplus r}
 \otimes \Omega^1_{U^{(i)}_{\tilde{M}'_0[u]}/\tilde{M}'_0}({\mathcal D}^{(i)}_{\tilde{M}'_0[u]})
\]
given by the connection matrix
\begin{align*}
 \tilde{\eta}
 &=
 \left( A(\tilde{z}^{(i)})+\bar{h}_1C_1(\tilde{z}^{(i)})+\bar{h}_2C_2(\tilde{z}^{(i)})
 +\bar{h}_1\bar{h}_2 \big( C(\tilde{z}^{(i)})+C_{1,2}(\tilde{z}^{(i)}) \big) \right)
 \frac{d\tilde{z}^{(i)}} {(\tilde{z}^{(i)})^{m_i}} \\
 & \quad
 +\big(B_1(\tilde{z}^{(i)}) +\bar{h}_2 B_{1,2}(\tilde{z}^{(i)}) \big) d\bar{h}_1
 +\big(B_2(\tilde{z}^{(i)}) +\bar{h}_1 B_{1,2}(\tilde{z}^{(i)}) \big) d\bar{h}_2 
\end{align*}
becomes an integrable connection.
If we put
${\mathcal N}^{(i)}_{U^{(i)},u}:=\Diag_{(\mu_k)}$,
then 
$\Big( {\mathcal O}_{U^{(i)}_{\tilde{M}'_0[u]}}^{\oplus r}, 
\nabla^u_{U^{(i)}},{\mathcal N}^{(i)}_{U^{(i)},u} \Big)$
is a local horizontal lift of
$\big( \tilde{E}_{\tilde{M}'_0},\tilde{\nabla}_{\tilde{M}'_0},
\{\tilde{N}^{(i)}_{\tilde{M}'_0}\} \big)\big|_{U^{(i)}}$.

Assume that
$\Big( {\mathcal O}_{U^{(i)}_{\tilde{M}'_0[u]}}^{\oplus r}, \nabla',N' \Big)$
is another local horizontal lift given by a connection matrix
\begin{align*}
 &(A(\tilde{z}^{(i)})+\bar{h}_1C'_1(\tilde{z}^{(i)})
 +\bar{h}_2C'_2(\tilde{z}^{(i)})+\bar{h}_1\bar{h}_2C'_{1,2}(\tilde{z}^{(i)}))
 \frac{d\tilde{z}}{\tilde{z}^{m_i}} \\
 & \quad
 +B'_1(\tilde{z}^{(i)})d\bar{h}_1+B'_2(\tilde{z}^{(i)})d\bar{h}_2+B'_{1,2}(\tilde{z})\bar{h}_2d\bar{h}_1
 +B'_{2,1}(\tilde{z}^{(i)})\bar{h}_1d\bar{h}_2.
\end{align*}
We want to construct an isomorphism between
$\nabla^u_{U^{(i)}}$ and $\nabla'$.
Since $C'_1(\tilde{z}^{(i)})\big|_{{\mathcal D}^{(i)}_{\tilde{M}'_0}[u]}$,
$C'_2(\tilde{z}^{(i)})\big|_{{\mathcal D}^{(i)}_{\tilde{M}'_0}[u]}$
and  $C'_{1,2}(\tilde{z}^{(i)})\big|_{{\mathcal D}^{(i)}_{\tilde{M}'_0}[u]}$
are diagonal matrices by the assumption,
the integrability condition
\begin{align*}
 &-\left( (C'_1+\bar{h}_2C'_{1,2}) d\bar{h}_1 + (C'_2+\bar{h}_1C'_{1,2}) d\bar{h}_2 \right)
 \wedge \frac{d\tilde{z}^{(i)}} {(\tilde{z}^{(i)})^{m_i}}   \\
 &=
 \big( dB'_1(z^{(i)}) + \bar{h}_2 \, dB'_{1,2}(z^{(i)}) \big) \wedge d\bar{h}_1
 +\big( dB'_2(z^{(i)})+\bar{h}_1 \, dB'_{2,1}(z^{(i)}) \big) \wedge d\bar{h}_2 \\
 &\quad 
 +\Big( B'_{2,1}(z^{(i)})-B'_{1,2}(z^{(i)})+[B'_1(\tilde{z}^{(i)}),B'_2(\tilde{z}^{(i)})] \Big)
 d\bar{h}_1\wedge d\bar{h}_2 \\
 &\quad
 +
 \Big(
 \big[ A(\tilde{z}^{(i)}), B'_1(\tilde{z}^{(i)})+\bar{h}_2B'_{1,2}(z^{(i)}) \big] 
 +\bar{h}_2\big[ C'_2, B'_1(\tilde{z}^{(i)})\big] \Big)
 \frac{d\tilde{z}^{(i)}} {(\tilde{z}^{(i)})^{m_i}} \wedge d\bar{h}_1   \\
 & \quad +   
 \Big(
 \big[ A(\tilde{z}^{(i)}), B'_2(\tilde{z}^{(i)})+\bar{h}_1B'_{2,1}(z^{(i)}) \big] 
 +\bar{h}_1\big[ C'_1, B'_2(\tilde{z}^{(i)})\big] \Big)
 \frac{d\tilde{z}^{(i)}} {(\tilde{z}^{(i)})^{m_i}} \wedge d\bar{h}_2
\end{align*}
implies
$dB'_1(\tilde{z}^{(i)})\big|_{{\mathcal D}^{(i)}_{\tilde{M}'_0[u]}}
=\Diag_{\big(\nu^{(i)}_1(\mu_k)\frac{d\tilde{z}^{(i)}}{(\tilde{z}^{(i)})^{m_i}}\big)}
\Big|_{{\mathcal D}^{(i)}_{\tilde{M}'_0[u]}}$
and
$dB'_2(\tilde{z}^{(i)})\big|_{{\mathcal D}^{(i)}_{\tilde{M}'_0[u]}}
=\Diag_{\big(\nu^{(i)}_2(\mu_k)\frac{d\tilde{z}^{(i)}}{(\tilde{z}^{(i)})^{m_i}}\big)}
\Big|_{{\mathcal D}^{(i)}_{\tilde{M}'_0[u]}}$.
Then 
$B_1(\tilde{z}^{(i)})-B'_1(\tilde{z}^{(i)})$,
$B_2(\tilde{z}^{(i)})-B'_2(\tilde{z}^{(i)})$
are matrices of algebraic functions on $U^{(i)}[u]$
and applying the transform
$(I_r+\bar{h}_1 ( B_1(\tilde{z}^{(i)})-B'_1(\tilde{z}^{(i)}) )
+\bar{h}_2 ( B_2(\tilde{z}^{(i)})-B'_2(\tilde{z}^{(i)}) )$
to $\nabla'$,
we may assume that
$B'_1=B_1$, $B'_2=B_2$ and consequently,
$C'_1=dB_1+[A,B_1]\dfrac{d\tilde{z}^{(i)}}{(\tilde{z}^{(i)})^{m_i}}=C_1$
and
$C'_2=dB_2+[A,B_2]\dfrac{d\tilde{z}^{(i)}}{(\tilde{z}^{(i)})^{m_i}}=C_2$.
Since $[B_1,B_2]=0$, we have $B'_{1,2}=B'_{2,1}$
and
$C'_{1,2}=dB'_{1,2}+([A,B'_{1,2}]+[C_2,B_1])\dfrac{d\tilde{z}^{(i)}}{(\tilde{z}^{(i)})^{m_i}}$
implies that
$dB'_{1,2}\big|_{{\mathcal D}^{(i)}_{\tilde{M}'_0[u]}}
=\Diag_{\big(\nu^{(i)}_{1,2}(\mu_k)\frac{d\tilde{z}^{(i)}}{(\tilde{z}^{(i)})^{m_i}}\big)}$.
So we can see that $B_{1,2}-B'_{1,2}$ is a matrix of
regular functions on $U^{(i)}[u]$ and
the transform $I_r+\bar{h}_1\bar{h}_2(B_{1,2}-B'_{1,2})$ 
gives an isomorphism between
$\Big( {\mathcal O}_{U^{(i)}_{\tilde{M}'_0[u]}}^{\oplus r}, 
\nabla^u_{U^{(i)}},{\mathcal N}^{(i)}_{U^{(i)},u} \Big)$
and
$\Big( {\mathcal O}_{U^{(i)}_{\tilde{M}'_0[u]}}^{\oplus r}, \nabla',N' \Big)$.
We can see that such an isomorphism is unique
because it is determined by the coefficients of
$d\bar{h}_1$ and $d\bar{h}_2$.

If an affine open subset $U\subset {\mathcal C}_{\tilde{M}'_0}$
is disjoint from ${\mathcal D}_{\tilde{M}'_0}$,
then we can easily give a local horizontal lift
of  
$\big(\tilde{E}_{\tilde{M}'_0},\tilde{\nabla}_{\tilde{M}'_0},
\{\tilde{N}^{(i)}_{\tilde{M}'_0}\}\big)\big|_{U}$.
In that case $\{\tilde{N}^{(i)}_{\tilde{M}'_0}\}\big|_U$ is nothing.
Patching local horizontal lifts altogether,
we obtain a unique horizontal lift
$\big({\mathcal E}^u_0,\nabla^u_0,
\{ {\mathcal N}^{(i)}_{0,u}\}\big)$
of 
$\big(\tilde{E}_{\tilde{M}'_0},\tilde{\nabla}_{\tilde{M}'_0},
\{\tilde{N}^{(i)}_{\tilde{M}'_0}\}\big)$
with respect to $u$.
\end{proof}

\begin{theorem} \label {thm:integrability condition}
The subbundle
$\Psi_0((\pi_{{\mathcal T}_{\bmu,\blambda,\epsilon=0}})^*
T_{{\mathcal T}_{\bmu,\blambda,\epsilon=0}})
\subset
T_{M^{\balpha}_{{\mathcal C},{\mathcal D}}(\tilde{\bnu},\bmu)_{\epsilon=0}}$
determined by (\ref {equation:irregular generalized isomonodromic splitting})
satisfies the integrability condition
\[
  \left[ \Psi_0((\pi_{{\mathcal T}_{\bmu,\blambda,\epsilon=0}})^*
 T_{{\mathcal T}_{\bmu,\blambda,\epsilon=0}}) \, , \; 
 \Psi_0((\pi_{{\mathcal T}_{\bmu,\blambda,\epsilon=0}})^*
 T_{{\mathcal T}_{\bmu,\blambda,\epsilon=0}}) \right]
 \subset 
 \Psi_0((\pi_{{\mathcal T}_{\bmu,\blambda,\epsilon=0}})^*
 T_{{\mathcal T}_{\bmu,\blambda,\epsilon=0}}).
\]
\end{theorem}

\begin{proof}
Take a Zariski open set
${\mathcal T}'_0\subset {\mathcal T}_{\bmu,\blambda,\epsilon=0}$
and vector fields
$v_1,v_2\in H^0({\mathcal T}'_0,T_{{\mathcal T}'_0})$.
We will prove the equality
\begin{equation}\label{equation:integrability condition}
 [\Phi_0(v_1),\Phi_0(v_2)]=\Phi_0([v_1,v_2])
\end{equation}
from which the theorem follows immediately.
Let ${\mathcal T}'_0\times\mathbb{C}[h_1,h_2]/(h_1^2,h_2^2)\xrightarrow{\tilde{I}_{v_1}}
{\mathcal T}'_0\times\Spec\mathbb{C}[h_1,h_2]/(h_1^2,h_2^2)$
be the morphism over $\Spec\mathbb{C}[h_1,h_2]/(h_1^2,h_2^2)$
corresponding to the ring homomorphism
\begin{align*}
 \tilde{I}_{v_1}^* \colon
 {\mathcal O}_{{\mathcal T}'_0}[h_1,h_2]/(h_1^2,h_2^2)
 \ni & f+f_1\bar{h}_1+f_2\bar{h}_2+f_{1,2}\bar{h}_1\bar{h}_2 \\
 & \mapsto
 f+(f_1+v_1(f))\bar{h}_1+f_2\bar{h}_2+ (f_{1,2}+v_1(f_2)) \bar{h}_1 \bar{h}_2
 \in {\mathcal O}_{{\mathcal T}'_0}[h_1,h_2]/(h_1^2,h_2^2)
\end{align*}
and let
${\mathcal T}'_0\times\mathbb{C}[h_1,h_2]/(h_1^2,h_2^2)\xrightarrow{\tilde{I}_{v_2}}
{\mathcal T}'_0\times\Spec\mathbb{C}[h_1,h_2]/(h_1^2,h_2^2)$
be the morphism corresponding to the ring homomorphism
\begin{align*}
 \tilde{I}_{v_2}^*\colon
 {\mathcal O}_{{\mathcal T}'_0}[h_1,h_2]/(h_1^2,h_2^2)
 \ni & f+f_1\bar{h}_1+f_2\bar{h}_2+f_{1,2}\bar{h}_1\bar{h}_2 \\
 &\mapsto
 f+f_1\bar{h}_1+(f_2+v_2(f))\bar{h}_2+(f_{1,2}+v_2(f_1))\bar{h}_1 \bar{h}_2
 \in {\mathcal O}_{{\mathcal T}'_0}[h_1,h_2]/(h_1^2,h_2^2).
\end{align*}
By the calculation
\begin{align*}
 &f+f_1\bar{h}_1+f_2\bar{h}_2+f_{1,2}\bar{h}_1\bar{h}_2
  \stackrel{ I_{v_2}^* }\mapsto
 f+f_1\bar{h}_1+(f_2+v_2(f))\bar{h}_2 +(f_{1,2}+v_2(f_1))\bar{h}_1\bar{h}_2 \\
 & \stackrel { I_{v_1}^* } \mapsto
 f+(f_1+v_1(f))\bar{h}_1+(f_2+v_2(f))\bar{h}_2
 +(f_{1,2}+v_2(f_1)+v_1(f_2)+v_1v_2(f))\bar{h}_1\bar{h}_2 \\
 & \stackrel { I_{-v_2}^* } \mapsto
 f+(f_1+v_1(f))\bar{h}_1+f_2\bar{h}_2
 +(f_{1,2}+v_1(f_2)+v_1v_2(f)-v_2v_1(f))\bar{h}_1\bar{h}_2 \\
 & \stackrel { I_{-v_1}^* } \mapsto
 f+f_1\bar{h}_1+f_2\bar{h}_2+(f_{1,2}+(v_1v_2-v_2v_1)(f))\bar{h}_1\bar{h}_2,
\end{align*}
 we can see that the composition
$\tilde{I}_{-v_1}^*\tilde{I}_{-v_2}^*\tilde{I}_{v_1}^*\tilde{I}_{v_2}^*$
is given by
\begin{align} 
 \tilde{I}_{-v_1}^*\tilde{I}_{-v_2}^*\tilde{I}_{v_1}^*\tilde{I}_{v_2}^* \colon &
 {\mathcal O}_{{\mathcal T}'_0}[h_1,h_2]/(h_1^2,h_2^2)
 \ni 
 f+f_1\bar{h}_1+f_2\bar{h}_2+f_{1,2}\bar{h}_1\bar{h}_2 \notag \\
 & \mapsto 
 f+f_1\bar{h}_1+f_2\bar{h}_2+(f_{1,2}+(v_1v_2-v_2v_1)(f))\bar{h}_1\bar{h}_2
 \in {\mathcal O}_{\mathcal T'}[h_1,h_2]/(h_1^2,h_2^2).
  \label {equation:composition giving [v_1,v_2]}
\end{align}

The vector field $\Phi_0(v_1)$ corresponds to a morphism
$\tilde{M}'_0\times \Spec\mathbb{C}[h_1]/(h_1^2)
\longrightarrow \tilde{M}'_0$.
This morphism together with the second projection
$\tilde{M}'_0\times \Spec\mathbb{C}[h_1]/(h_1^2)\longrightarrow 
\Spec\mathbb{C}[h_1]/(h_1^2)$
gives  a morphism
\begin{equation} \label {equation:automorphism corresponding to tangent vector}
 \tilde{M}'_0\times \Spec\mathbb{C}[h_1]/(h_1^2)
 \longrightarrow 
 \tilde{M}'_0\times \Spec\mathbb{C}[h_1]/(h_1^2)
\end{equation}
over $\Spec\mathbb{C}[h_1]/(h_1^2)$.
Let
\begin{equation} \label {equation:definition of I_{v_1}}
 \tilde{I}_{\Phi_0(v_1)}\colon 
 \tilde{M}'_0\times\Spec\mathbb{C}[h_1,h_2]/(h_1^2,h_2^2)
 \longrightarrow \tilde{M}'_0\times\Spec\mathbb{C}[h_1,h_2]/(h_1^2,h_2^2).
\end{equation}
be the base change of 
(\ref {equation:automorphism corresponding to tangent vector})
under the projection
$\Spec\mathbb{C}[h_1,h_2]/(h_1^2,h_2^2)\longrightarrow
\Spec\mathbb{C}[h_1]/(h_1^2)$.
Similarly we can define a morphism
\begin{equation} \label {equation:definition of I_{v_2}}
 \tilde{I}_{\Phi_0(v_2)}\colon 
 \tilde{M}'_0\times\Spec\mathbb{C}[h_1,h_2]/(h_1^2,h_2^2)
 \longrightarrow \tilde{M}'_0\times\Spec\mathbb{C}[h_1,h_2]/(h_1^2,h_2^2)
\end{equation}
from the morphism
$\tilde{M}'_0\times \Spec\mathbb{C}[h_2]/(h_2^2)
\longrightarrow \tilde{M}'_0$
corresponding to $\Phi_0(v_2)$.
We can see by a similar calculation to that of (\ref  {equation:composition giving [v_1,v_2]})
that the composition
$\tilde{I}_{\Phi_0(-v_1)}^*\tilde{I}_{\Phi_0(-v_2)}^*
\tilde{I}_{\Phi_0(v_1)}^*\tilde{I}_{\Phi_0(v_2)}^*$
corresponds to the ring homomorphism
\begin{align} 
 &\tilde{I}_{\Phi_0(-v_1)}^*\tilde{I}_{\Phi_0(-v_2)}^*
 \tilde{I}_{\Phi_0(v_1)}^*\tilde{I}_{\Phi_0(v_2)}^*
 \colon 
 {\mathcal O}_{\tilde{M}'_0}[h_1,h_2]/(h_1^2,h_2^2)
 \ni 
 f+f_1\bar{h}_1+f_2\bar{h}_2+f_{1,2}\bar{h}_1\bar{h}_2 
  \label {equation:composition giving [Psi(v_1),Psi(v_2)]} \\
 & \quad  \mapsto 
 f+f_1\bar{h}_1+f_2\bar{h}_2
 +(f_{1,2}+(\Phi_0(v_1)\Phi_0(v_2)-\Phi_0(v_2)\Phi_0(v_1))(f))\bar{h}_1\bar{h}_2
 \in {\mathcal O}_{\tilde{M}'_0}[h_1,h_2]/(h_1^2,h_2^2).
 \notag
\end{align}

Let
$\pi_{{\mathcal T}'_0}\colon 
{\mathcal T}'_0\times\Spec\mathbb{C}[h_1,h_2]/(h_1^2,h_2^2)
\longrightarrow {\mathcal T}'_0$
be the first projection.
By Lemma \ref {lem:horizontal lift h_1h_2 in irregular case},
there exists a unique horizontal lift
$\Big( {\mathcal E}_0^{\pi_{{\mathcal T}'_0}\circ \tilde{I}_{v_2}},
\nabla_0^{\pi_{{\mathcal T}'_0}\circ \tilde{I}_{v_2}},
\big\{ {\mathcal N}^{(i)}_{0,\pi_{{\mathcal T}'_0}\circ \tilde{I}_{v_2}} 
\big\} \Big)$
of 
$\big(\tilde{E}_{\tilde{M}'_0},\tilde{\nabla}_{\tilde{M}'_0},\{ N^{(i)}_{\tilde{M}'_0}\}\big)$
with respect to the composition
$\pi_{{\mathcal T}'_0}\circ \tilde{I}_{v_2}\colon
{\mathcal T}'_0\times\Spec\mathbb{C}[h_1,h_2]/(h_1^2,h_2^2)
\longrightarrow {\mathcal T}'_0$.
Then we can see that
\[
 (\mathrm{id}\times\tilde{I}_{\Phi_0(-v_1)})^*
 (\mathrm{id}\times\tilde{I}_{\Phi_0(-v_2)})^*
 (\mathrm{id}\times\tilde{I}_{\Phi_0(v_1)})^*
 \Big({\mathcal E}_0^{\pi_{{\mathcal T}'_0}\circ \tilde{I}_{v_2}},
 \nabla_0^{\pi_{{\mathcal T}'_0}\circ \tilde{I}_{v_2}},
 \big\{ {\mathcal N}^{(i)}_{0,\pi_{{\mathcal T}'_0}\circ \tilde{I}_{v_2}}
 \big\} \Big)
\]
is a horizontal lift of 
$\big(\tilde{E}_{\tilde{M}'_0},\tilde{\nabla}_{\tilde{M}'_0},\{ N^{(i)}_{\tilde{M}'_0}\}\big)$,
in the sense of Lemma \ref{lem:horizontal lift h_1h_2 in irregular case},
with respect to the composition
$\pi_{{\mathcal T}'_0} \circ 
\tilde{I}_{v_2}\circ \tilde{I}_{v_1}\circ \tilde{I}_{-v_2}\circ \tilde{I}_{-v_1}
\colon {\mathcal T}'_0\times\Spec\mathbb{C}[h_1,h_2]/(h_1^2,h_2^2)
\longrightarrow {\mathcal T}'_0$.
Let
\[
 \rho\colon {\mathcal T}'_0\times\Spec\mathbb{C}[h_1,h_2]/(h_1^2,h_2^2)
 \longrightarrow {\mathcal T}'_0\times\Spec\mathbb{C}[h]/(h^2)
\]
be the morphism whose corresponding ring homomorphism
$\rho^*\colon {\mathcal O}_{{\mathcal T}'_0}[h]/(h^2)
\longrightarrow {\mathcal O}_{{\mathcal T}'_0}[h_1,h_2]/(h_1^2,h_2^2)$
is given by
$\rho^*(f+gh)=f+g\bar{h}_1\bar{h}_2$
for $f,g\in{\mathcal O}_{{\mathcal T}'_0}$.
Then we have
\[
 \pi_{{\mathcal T}'_0} \circ 
 \tilde{I}_{v_2}\circ \tilde{I}_{v_1}\circ \tilde{I}_{-v_2}\circ \tilde{I}_{-v_1}
 =I_{[v_1,v_2]}\circ \rho,
\]
where
$I_{[v_1,v_2]}\colon {\mathcal T}'_0\times\Spec\mathbb{C}[h]/(h^2)
\longrightarrow {\mathcal T}'$
means the morphism corresponding to the commutator
vector field $[v_1,v_2]=v_1v_2-v_2v_1$.
If we denote by
$\big({\mathcal E}_0^{[v_1,v_2]},\nabla_0^{[v_1,v_2]},\{ {\mathcal N}^{(i)}_{0,[v_1,v_2]} \}\big)$
the horizontal lift of
$\big(\tilde{E}_{\tilde{M}'_0},\tilde{\nabla}_{\tilde{M}'_0},\{ N^{(i)}_{\tilde{M}'_0}\}\big)$,
in the sense of Proposition \ref {prop:horizontal lift for irregular case},
with respect to the  the commutator vector field
$[v_1,v_2]\in
H^0({\mathcal T}'_0,T_{{\mathcal T}'_0})$,
we can see that
$(\mathrm{id}\times \rho)^*\big({\mathcal E}_0^{[v_1,v_2]},
\nabla_0^{[v_1,v_2]},\{ {\mathcal N}^{(i)}_{0,[v_1,v_2]} \}\big)$
is also a horizontal lift of
$\big(\tilde{E}_{\tilde{M}'_0},\tilde{\nabla}_{\tilde{M}'_0},\{ N^{(i)}_{\tilde{M}'_0}\}\big)$,
in the sense of Lemma \ref {lem:horizontal lift h_1h_2 in irregular case},
with respect to
$I_{[v_1,v_2]}\circ\rho=
\pi_{{\mathcal T}'_0} \circ \tilde{I}_{v_2}\circ \tilde{I}_{v_1}\circ \tilde{I}_{-v_2}\circ \tilde{I}_{-v_1}
\colon {\mathcal T}'_0\times\Spec\mathbb{C}[h_1,h_2]/(h_1^2,h_2^2)
\longrightarrow {\mathcal T}'_0$.
By the uniqueness of the horizontal lift in Lemma \ref {lem:horizontal lift h_1h_2 in irregular case},
we have an isomorphism
\begin{align}
 (\mathrm{id}\times\tilde{I}_{\Phi_0(-v_1)})^*
 (\mathrm{id}\times\tilde{I}_{\Phi_0(-v_2)})^*
 (\mathrm{id}\times\tilde{I}_{\Phi_0(v_1)})^*
 \Big({\mathcal E}_0^{\pi_{{\mathcal T}'_0}\circ \tilde{I}_{v_2}},
 \nabla_0^{\pi_{{\mathcal T}'_0}\circ \tilde{I}_{v_2}},
 \{{\mathcal N}^{(i)}_{0,\pi_{{\mathcal T}'_0}\circ \tilde{I}_{v_2}}\} \Big) \quad & 
 \label {equation:horizontal lifts coincides}   \\
 \cong
 (\mathrm{id}\times \rho)^*
 \big( {\mathcal E}_0^{[v_1,v_2]},
 \nabla_0^{[v_1,v_2]},\{ {\mathcal N}^{(i)}_{0,[v_1,v_2]} \} \big). & \notag 
\end{align}
The flat family
$(\mathrm{id}\times\tilde{I}_{\Phi_0(-v_1)})^*
 (\mathrm{id}\times\tilde{I}_{\Phi_0(-v_2)})^*
 (\mathrm{id}\times\tilde{I}_{\Phi_0(v_1)})^*
 \Big({\mathcal E}_0^{\pi_{{\mathcal T}'_0}\circ \tilde{I}_{v_2}},
 \overline { \nabla_0^{\pi_{{\mathcal T}'_0}\circ \tilde{I}_{v_2}} },
 \{{\mathcal N}^{(i)}_{0,\pi_{{\mathcal T}'_0}\circ \tilde{I}_{v_2}}\}\Big)$
associated to (\ref {equation:horizontal lifts coincides})
corresponds to
the composition
\[
 \pi_{\tilde{M}'_0} \circ \tilde{I}_{\Phi_0(v_2)}\circ \tilde{I}_{\Phi_0(v_1)}\circ 
 \tilde{I}_{\Phi_0(-v_2)}\circ \tilde{I}_{\Phi_0(-v_1)}
 \colon \tilde{M}'_0\times\Spec\mathbb{C}[h_1,h_2]/(h_1^2,h_2^2) 
 \longrightarrow \tilde{M}'_0,
\]
where
$\pi_{\tilde{M}'_0}\colon \tilde{M}'_0\times\Spec\mathbb{C}[h_1,h_2]/(h_1^2,h_2^2)
\longrightarrow \tilde{M}'_0$
is the first projection.
The same associated flat family
$(\mathrm{id}\times \rho)^*
\big({\mathcal E}_0^{[v_1,v_2]},\overline{\nabla_0^{[v_1,v_2]}},
\{ {\mathcal N}^{(i)}_{0,[v_1,v_2]} \} \big)$
induced by (\ref {equation:horizontal lifts coincides})
corresponds to the composition
\[
 \pi_{\tilde{M}'_0}\circ\tilde{I}_{\Phi_0([v_1,v_2])}\circ (\mathrm{id}\times \rho)
 \colon \tilde{M}'_0\times\Spec\mathbb{C}[h_1,h_2]/(h_1^2,h_2^2)
 \longrightarrow \tilde{M}'_0.
\]
Thus we have
$\pi_{\tilde{M}'_0} \circ \tilde{I}_{\Phi_0(v_2)} \circ \tilde{I}_{\Phi_0(v_1)} \circ 
\tilde{I}_{\Phi_0(-v_2)} \circ \tilde{I}_{\Phi_0(-v_1)}
=
\pi_{\tilde{M}'_0} \circ \tilde{I}_{\Phi_0([v_1,v_2])}\circ (\mathrm{id}\times \rho)$. 
We can see by (\ref {equation:composition giving [Psi(v_1),Psi(v_2)]})
that the morphism
$\pi_{\tilde{M}'_0} \circ \tilde{I}_{\Phi_0(v_2)} \circ \tilde{I}_{\Phi_0(v_1)} \circ 
\tilde{I}_{\Phi_0(-v_2)} \circ \tilde{I}_{\Phi_0(-v_1)}$
is given by the ring homomorphism
\[
 {\mathcal O}_{\tilde{M}'_0} \ni f \mapsto
 f+(\Phi_0(v_1)\Phi_0(v_2)-\Phi_0(v_2)\Phi_0(v_1))\bar{h}_1\bar{h}_2
 \in {\mathcal O}_{\tilde{M}'_0}[h_1,h_2]/(h_1^2,h_2^2).
\]
On the other hand, the morphism
$\pi_{\tilde{M}'_0} \circ \tilde{I}_{\Phi_0([v_1,v_2])}\circ (\mathrm{id}\times \rho)$
is given by the ring homomorphism
\[
 {\mathcal O}_{\tilde{M}'_0} \ni f \mapsto
 f+\Phi_0(v_1v_2-v_2v_1)\bar{h}_1\bar{h}_2 \in {\mathcal O}_{\tilde{M}'_0}[h_1,h_2]/(h_1^2,h_2^2).
\]
Hence we have $\Phi_0(v_1)\Phi_0(v_2)-\Phi_0(v_2)\Phi_0(v_1)=\Phi_0(v_1v_2-v_2v_1)$,
which is nothing but the equation
(\ref{equation:integrability condition})
and the proposition is proved.
\end{proof}


\subsection{Global horizontal lift in the unfolded case
and the proof of Theorem \ref {thm:holomorphic-splitting}}
\label {subsection:unfolded global horizontal lift}

In this subsection, 
we give an analytic local lift of the unramified irregular singular generalized
isomonodromic deformation given in subsection
\ref {subsection:irregular singular generalized isomonodromic deformation}.
The key point is to construct a global horizontal lift
via patching local horizontal lifts given in Proposition \ref {prop:local horizontal lift}.
The consequent global horizontal lift given 
in Proposition \ref {prop:existence-horizontal-lift}
produces the proof of Theorem \ref {thm:holomorphic-splitting}.


Take a point
$x \in
M^{\balpha}_{{\mathcal C},{\mathcal D}}(\tilde{\bnu},\bmu)_{\epsilon=0}
\times_{\mathcal B}{\mathcal B}'$
which corresponds to a $(\bnu,\bmu)$-connection
$(E,\nabla,\{N^{(i)}\})$.
Recall that we are given an analytic open subset
$U_i\subset{\mathcal C}_{\mathcal B'}$ with a biholomorphic map
\begin{equation} \label {equation:local biholomorphic again}
 U_i\xrightarrow{\sim} \Delta_a\times\Delta_{\epsilon_0}\times\Delta_r^s
\end{equation}
given by (\ref {equation:local biholomorphic})
in subsection \ref {subsection:moduli setting}.
We take a loop $\tilde{\gamma}_x$ in $(U_i)_x \subset {\mathcal C}_x$
which is a boundary of a disk containing ${\mathcal D}^{(i)}_x$.
We consider the morphism
\[
 \nabla^{\dag}\colon
 {\mathcal End}(E) \ni u \mapsto \nabla\circ u -u\circ \nabla
 \in {\mathcal End}(E)\otimes\Omega^1_{{\mathcal C}_x}({\mathcal D}_x)
\]
and assume the following:
\begin{assumption} \label {assumption:irreducibility for universal family} \rm
\begin{itemize}
\item[(1)]
The monodromy of
$\nabla \colon E \longrightarrow 
E \otimes\Omega^1_{{\mathcal C}_x}({\mathcal D}_x)$
along $\tilde{\gamma}_x$
has the $r$ distinct eigenvalues and
\item[(2)]
$H^0((U_i)_x,\ker\nabla^{\dag}|_{(U_i)_x})=\mathbb{C}$.
\end{itemize}
\end{assumption}

There is an \'etale morphism
$\tilde{M} \longrightarrow M^{\balpha}_{{\mathcal C},{\mathcal D}}(\tilde{\bnu},\bmu)$
whose image contains $x$
such that
there is a universal family
$(\tilde{E},\tilde{\nabla},\{\tilde{N}^{(i)}\})$
on $({\mathcal C},{\mathcal D})_{\tilde{M}}$ over $\tilde{M}$.
We can take an analytic open neighborhood
$M^{\circ}$ of $x$ in
$M^{\balpha}_{{\mathcal C},{\mathcal D}}(\tilde{\bnu},\bmu)
\times_{\mathcal B}{\mathcal B}'$
with a factorization
$M^{\circ}\longrightarrow \tilde{M} \longrightarrow
M^{\balpha}_{{\mathcal C},{\mathcal D}}(\tilde{\bnu},\bmu)$.
We denote by
$\big(\tilde{E}^{hol}_{M^{\circ}},\tilde{\nabla}^{hol}_{M^{\circ}},
\{\tilde{N}^{(i),hol}_{M^{\circ}}\}\big)$
the pullback of
$(\tilde{E},\tilde{\nabla},\{\tilde{N}^{(i)}\})$
to (${\mathcal C},{\mathcal D})_{M^{\circ}}$.

In the following, we successively replace $M^{\circ}$ by
its shrink till Definition \ref {def:block of local horizontal lift}.
After shrinking $M^{\circ}$, we may assume that the morphism
induced by (\ref {equation:local biholomorphic again})
\[
 (U_i)_{M^{\circ}}\xrightarrow{\sim}
 \Delta_a \times M^{\circ}
\]
is an isomorphism.
We denote the image of $M^{\circ}$ under the morphism
$M^{\balpha}_{{\mathcal C},{\mathcal D}}(\tilde{\bnu},\bmu)
\times_{\mathcal B} {\mathcal B}'
\longrightarrow {\mathcal T}_{\bmu,\blambda}\times_{\mathcal B}{\mathcal B}'$
by ${\mathcal T}^{\circ}$,
which is an analytic open subset of
${\mathcal T}_{\bmu,\blambda}\times_{\mathcal B}{\mathcal B}'$.
Then the inclusion
${\mathcal T}^{\circ}\hookrightarrow
{\mathcal T}_{\bmu,\blambda}\times_{\mathcal B} {\mathcal B'}
\hookrightarrow {\mathcal T}_{\bmu,\blambda}$
corresponds to a tuple of polynomials
$\bnu=(\nu^{(i)}(T))_{1\leq i\leq n}$ given by
\[
 \nu^{(i)}(T)=\sum_{l=0}^{r-1} \sum_{j=0}^{m_i-1} c^{(i)}_{l,j} \: (z^{(i)})^j \; T^l
\]
with
$c^{(i)}_{l,j}\in H^0({\mathcal T}^{\circ},{\mathcal O}^{hol}_{\mathcal T^{\circ}})$
satisfying (a) and (b) of the definition of ${\mathcal T}_{\bmu,\blambda}$.

We apply the process in subsection \ref {subsection:local horizontal lift}
to the restricted relative connection
$\big(\tilde{E}^{hol}_{M^{\circ}},\tilde{\nabla}^{hol}_{M^{\circ}},
\{\tilde{N}^{(i),hol}_{M^{\circ}}\}\big)
\big|_{(U_i)_{M^{\circ}}}$.
Using Proposition \ref {prop:normalization of local connection via extension to P^1},
there is an isomorphism
$\theta^{(i)}\colon
\tilde{E}^{hol}_{M^{\circ}} \big| _{(U_i)_{M^{\circ}}}
\xrightarrow{\sim}
({\mathcal O}^{hol}_{(U_i)_{M^{\circ}}})^{\oplus r}$
after shrinking $M^{\circ}$ such that
the connection
$(\theta^{(i)}\otimes\mathrm{id})\circ
\tilde{\nabla}^{hol}_{M^{\circ}} \big|_{(U_i)_{M^{\circ}}}
\circ(\theta^{(i)})^{-1}$
is canonically extended to a global relative connection
\[
 \nabla^{(i),\mathbb{P}^1} _{M^{\circ}} \colon
 ({\mathcal O}^{hol}_{\mathbb{P}^1\times M^{\circ}})^{\oplus r}
 \longrightarrow
 ({\mathcal O}^{hol}_{\mathbb{P}^1\times M^{\circ}})^{\oplus r}
 \otimes
 \Omega^1_{\mathbb{P}^1\times M^{\circ}/M^{\circ}}
 ({\mathcal D}_{M^{\circ}} \cup ( \{ \infty \} \times M^{\circ}) )^{hol},
\]
where we are assuming the identification
$(U_i)_{M^{\circ}}=\Delta_a \times M^{\circ} 
\hookrightarrow \mathbb{P}^1\times M^{\circ}$.
Let
\[
 A^{(i)}(z^{(i)},\epsilon) \frac {dz^{(i)} } {(z^{(i)})^{m_i} -\epsilon^{m_i} }
 =
 \sum_{j=0}^{m_i-1} A^{(i)}_j(\epsilon)  (z^{(i)})^j  \frac {dz^{(i)} } {(z^{(i)})^{m_i} -\epsilon^{m_i} }
\]
be the connection matrix of $\nabla^{(i),\mathbb{P}^1} _{M^{\circ}}$.
By Assumption \ref  {assumption:irreducibility for universal family},
we can see, after shrinking $M^{\circ}$, that
\[
 \bigcap_{j=0}^{m_i-1} \ker \left( \mathrm{ad}(A^{(i)}_j(\epsilon)) \right)
 = {\mathcal O}_{M^{\circ}}^{hol}
\]
in the same way as (\ref  {equation:intersection of adjoint kernel})
in subsection \ref {subsection:local horizontal lift}.
As in the argument in subsection \ref {subsection:local horizontal lift}
producing  (\ref {equation:definition of Xi}),
we can take matrices
$\Xi^{(i)}_{l,j}(z^{(i)})$ of polynomials in $z^{(i)}$
of degree less than $m_i$
satisfying
\begin{equation} \label {equation:giving Xi}
 A^{(i)}(z^{(i)},\epsilon)=\sum_{l=0}^{r-1} \sum_{j=0}^{m_i-1}
 c^{(i)}_{l,j} \: \Xi^{(i)}_{l,j}(z^{(i)})
\end{equation}
and
\[
 (z^{(i)})^j \: \theta^{(i)}\circ \big( \tilde{N}^{(i),hol} \big)^l \circ (\theta^{(i)})^{-1}
 \big|_{{\mathcal D}^{(i)}_{M^{\circ}}}
 =
 \Xi^{(i)}_{l,j}(z^{(i)})  \big|_{{\mathcal D}^{(i)}_{M^{\circ}}}.
\]
Indeed there is a polynomial
\begin{equation} \label {equation:polynomial psi}
 \psi^{(i)}(T)
 =a^{(i)}_{r-1}(z^{(i)})T^{r-1}+\cdots+a^{(i)}_{1}(z^{(i)})T+a^{(i)}_{0}(z^{(i)})
 \in {\mathcal O}_{\mathcal V}^{hol}[z^{(i)}][T]
\end{equation}
in $T$ of degree less than $r$ with each
$a^{(i)}_{k}(z^{(i)})\in{\mathcal O}_{\mathcal V}[z^{(i)}]$
a polynomial in $z^{(i)}$ of degree less than $m_i$
and $\Xi^{(i)}_{l,j}(z^{(i)})$ is obtained from
$(z^{(i)})^j \psi^{(i)}(A(z^{(i)},\epsilon))^l$
by substituting $\epsilon^{m_i}$ in $(z^{(i)})^{m_i}$.

By Lemma \ref {lemma:existence of adjusting data},
we can take an adjusting data
$\big( R^{(i),(l)}_{j,l'} \big)$ 
for the connection 
$\nabla^{(i),\mathbb{P}^1} _{M^{\circ}}$
after shrinking $M^{\circ}$.
If we put
\begin{align} 
 \tilde{\Xi}^{(i)}_{l,j}(z^{(i)})
 &:= 
 \Xi^{(i)}_{l,j}(z^{(i)}) 
 -\sum_{q=0}^{m_i-1} \sum_{0\leq l'\leq m_i-1-q } 
 \big[ A^{(i)}_{q}(\epsilon)  ,  R^{(i),(l)}_{j,l'} \big] \: (z^{(i)})^{q+l'}
  \label {equation:adjusting data}  \\
 & \hspace{100pt}
 - \sum_{q=0}^{m_i-1} \sum_{m_i-q\leq l'\leq m_i-1} 
 \big[ A^{(i)}_{q}(\epsilon)  , R^{(i),(l)}_{j,l'} \big] \: 
 \epsilon^{m_i} (z^{(i)})^{q+l'-m_i}, 
 \notag  
\end{align}
then we have
\[
 \res_{z^{(i)}=\infty} \left( \tilde{\Xi}^{(i)}_{l,j}(z^{(i)})
 \frac { dz^{(i)} } { (z^{(i)})^{m_i}-\epsilon^{m_i} } \right) = 0
\]
and
\[
 \tilde{\Xi}^{(i)}_{l,j}(z^{(i)})   \big|_{{\mathcal D}^{(i)}_{M^{\circ}}}
 =
 \Xi^{(i)}_{l,j}(z^{(i)})   \big|_{{\mathcal D}^{(i)}_{M^{\circ}}}
 -\Big[ A^{(i)}(z^{(i)}), \sum_{l'=0}^{m_i-1}R^{(i),(l)}_{j,l'} (z^{(i)})^{l'}  \Big]
  \Big|_{{\mathcal D}^{(i)}_{M^{\circ}}}.
\]

We consider
the relative connection
\begin{equation} 
 \nabla_{\mathbb{P}^1\times M^{\circ}[\bar{h}],v^{(i)}_{l,j}} \colon
 ({\mathcal O}_{\mathbb{P}^1\times M^{\circ}[\bar{h}]}^{hol})^{\oplus r}
 \longrightarrow
 ({\mathcal O}_{\mathbb{P}^1\times M^{\circ}[\bar{h}]}^{hol})^{\oplus r}
 \otimes\Omega^1_{\mathbb{P}^1\times M^{\circ}[\bar{h}]/M^{\circ}[\bar{h}]}
 \left({\mathcal D}^{(i)}_{M^{\circ}[\bar{h}]} \cup
 (\infty\times M^{\circ}[\bar{h}]) \right)^{hol}
\end{equation}
determined by the connection matrix
\[
  \Big(A(z^{(i)},\epsilon)+ \bar{h} \tilde{\Xi}^{(i)}_{l,j}(z^{(i)}) \Big)
  \frac{dz^{(i)}} {(z^{(i)})^{m_i}-\epsilon^{m_i}},
\]
where we write $M^{\circ}[\bar{h}]:=M^{\circ}\times\Spec\mathbb{C}[h]/(h^2)$.
By Proposition \ref {prop:local horizontal lift},
we can take a horizontal lift
\begin{equation} \label {equation:block of local horizontal lift} 
 \nabla_{\mathbb{P}^1\times M^{\circ}[\bar{h}],v^{(i)}_{l,j}}^{flat}  \colon
 ({\mathcal O}_{\mathbb{P}^1\times M^{\circ}[\bar{h}]}^{hol})^{\oplus r}
  \longrightarrow
 ({\mathcal O}_{\mathbb{P}^1\times M^{\circ}[\bar{h}]}^{hol})^{\oplus r}
 \otimes
 (\iota_{M^{\circ}[\bar{h}]})_*
 \Omega^1_{(\mathbb{P}^1\times M^{\circ}
 \setminus\Gamma_{M^{\circ}})[\bar{h}] \big/M^{\circ} }
 \left( \infty\times M^{\circ}[\bar{h}] \right) ^{hol}
\end{equation}
of $\nabla_{\mathbb{P}^1\times M^{\circ}[\bar{h}],v^{(i)}_{l,j}}$
given by a connection matrix
\[
 \big(A(z^{(i)},\epsilon)+\bar{h} \tilde{\Xi}^{(i)}_{l,j}(z^{(i)}) \big)
 \frac{dz^{(i)}}{(z^{(i)})^{m_i}-\epsilon^{m_i}}
 + B^{(i)}_{l,j}(z^{(i)}) d\bar{h},
\]
where
$\iota_{M^{\circ}[\bar{h}]}
\colon
(\mathbb{P}^1\times M^{\circ}\setminus\Gamma_{M^{\circ}})[\bar{h}]
\hookrightarrow
\mathbb{P}^1\times M^{\circ}[\bar{h}]$
is the canonical inclusion.
By the construction,
the restriction of
$\nabla_{\mathbb{P}^1\times M^{\circ}[\bar{h}],v^{(i)}_{l,j}}^{flat}$
to
$\mathbb{P}^1\times M^{\circ}[\bar{h}]
\times_{\Delta{\epsilon_0}}\Spec\mathbb{C}[\epsilon]/(\epsilon^{m_i})$
coincides with the horizontal lift
giving the unramified irregular singular generalized isomonodromic deformation.

\begin{definition} \label {def:block of local horizontal lift}
We call the collection
$\Big(\nabla_{\mathbb{P}^1\times M^{\circ}[\bar{h}],v^{(i)}_{l,j}}^{flat}\Big)
^{1\leq i\leq n}_{0\leq l\leq r-1, 0\leq j\leq m_i-1}$
of integrable connections determined by 
$\big( \tilde{\Xi}^{(i)}_{l,j}(z^{(i)}) \big)$, 
$\big( B^{(i)}_{l,j}(z^{(i)}) \big) $
in (\ref {equation:block of local horizontal lift})
a block of local horizontal lifts.
\end{definition}

Take an analytic open subset
${\mathcal T'}\subset {\mathcal T}^{\circ}\subset
{\mathcal T}_{\bmu,\blambda}\times_{\mathcal B}{\mathcal B}'$
and a  $\Delta_{\epsilon_0}$-relative holomorphic vector field
$v\in H^0({\mathcal T'},T^{hol}_{{\mathcal T'}/\Delta_{\epsilon_0}})$
on ${\mathcal T'}$.
Then $v$ corresponds to
an analytic morphism
\[
 I_v\colon {\mathcal T'}\times\Spec\mathbb{C}[h]/(h^2)
 \longrightarrow {\mathcal T'}\hookrightarrow
 {\mathcal T}_{\bmu,\blambda}\times_{\mathcal B} {\mathcal B'}
\]
over $\Delta_{\epsilon_0}$
satisfying
$I_v|_{{\mathcal T'}\times\Spec\mathbb{C}[h]/(h)}=\mathrm{id}_{\mathcal T'}$.
We put ${\mathcal T'}[v]:={\mathcal T'}\times\Spec\mathbb{C}[h]/(h^2)$
which is regarded as an analytic space over ${\mathcal T'}$ via $I_v$
and consider the fiber product
\[
 \begin{CD}
  {\mathcal C}_{{\mathcal T'}[v]}:=
  {\mathcal C}_{\mathcal T'}\times_{\mathcal T'}
  ({\mathcal T'}\times\Spec\mathbb{C}[h]/(h^2))
  @>>> {\mathcal C}_{\mathcal T'} := {\mathcal C}\times_{\mathcal P}{\mathcal T'} \\
  @VVV   @VVV \\
  {\mathcal T'}\times\Spec\mathbb{C}[h]/(h^2) @>I_v>> {\mathcal T'}
 \end{CD}
\]
of ${\mathcal C}_{\mathcal T'}\longrightarrow {\mathcal T'}$
and
${\mathcal T'}\times\Spec\mathbb{C}[h]/(h^2) \xrightarrow{I_v}
{\mathcal T'}$.
The morphism $I_v$ corresponds to an analytic morphism
\[
 I_{v_{\mathcal B'}}\colon
 {\mathcal T'}\times\Spec\mathbb{C}[h]/(h^2)
 \longrightarrow {\mathcal B'}
\]
over $\Delta_{\epsilon_0}$
and a tuple of polynomials
\begin{equation} \label{equation:deformation of local exponents}
 \bnu_{hor}+\bar{h}\bnu_v=(\nu^{(i)}_{hor}(T)+\bar{h}\nu^{(i)}_v(T))
\end{equation} 
where
$\nu^{(i)}_{hor}(T)\in 
{\mathcal O}^{hol}_{{\mathcal D}^{(i)}_{\mathcal T'}} [h]/(h^2)\: [T]$
and 
$\nu^{(i)}_v(T) \in  
{\mathcal O}^{hol}_{{\mathcal D}^{(i)}_{\mathcal T'}} [T]$
are given by
\begin{align*}
 \nu^{(i)}_{hor}(T)
 &=
 \sum_{l=0}^{r-1} \sum_{j=0}^{m_i-1}
 (I_v^*c^{(i)}_{l,j}-\bar{h} v(c^{(i)}_{l,j}))
 (\tilde{z}^{(i)})^j T^l  \\
 \nu^{(i)}_v(T)
 &=
 \sum_{l=0}^{r-1} \sum_{ j=0 } ^{m_i-1 }
 v(c^{(i)}_{l,j})
 (\tilde{z}^{(i)})^j T^l. 
\end{align*}
Here $\tilde{z}^{(i)}_j$ is the pull-back of $z^{(i)}_j$
under the morphism
${\mathcal C}_{{\mathcal T'}[v]}\xrightarrow{\mathrm{id}\times I_v}
{\mathcal C}_{\mathcal T'}\longrightarrow {\mathcal C}_{\mathcal B'}$
and 
$\nu^{(i)}_{hor}(T)+\bar{h}\nu^{(i)}_v(T)\in
 {\mathcal O}^{hol}_{{\mathcal D}^{(i)}_{\mathcal T'}} [h]/(h^2) \: [T]$
should satisfy (a) in the definition of
${\mathcal T}_{\bmu,\blambda}$
in subsection \ref {subsection:moduli setting}.
For an analytic open subset
$U\subset {\mathcal C}_{\mathcal T'}$,
we denote by $U[v]$ the open subspace of
${\mathcal C}_{{\mathcal T}'[v]}$ whose underlying set of points
is $U$.

We consider 
the sheaf of ${\mathcal T'}$-relative differential forms
$\big(
\Omega^1_{ \left. \left({\mathcal C}_{\mathcal T'}
\setminus\Gamma_{\mathcal T'}\right)[v]\right/{\mathcal T'}}
\big)^{hol}$
with respect to the composite of the trivial projections
\[
 {\mathcal C}_{{\mathcal T'}[v]}
 ={\mathcal C}\times_{\mathcal P}{\mathcal T'}\times\Spec\mathbb{C}[h]/(h^2)
 \longrightarrow {\mathcal T'}\times\Spec\mathbb{C}[h]/(h^2)
 \longrightarrow{\mathcal T'}
 \]
which is different from the structure of
${\mathcal C}_{{\mathcal T'}[v]}$ over ${\mathcal T'}$
coming from the fiber product structure.
Note that
$\big( \Omega^1_{ \left.\left({\mathcal C}_{\mathcal T'}\setminus
\Gamma_{\mathcal T'}\right)[v] \right/ {\mathcal T'}}\big)^{hol}$
is locally generated by $d\tilde{z}$ and $d\bar{h}$,
where $\tilde{z}$ is the pullback of a uniformizing parameter $z$ of 
${\mathcal C}_{\mathcal T'}$
via the first projection
${\mathcal C}_{\mathcal T'}\times_{\mathcal T'} {\mathcal T'}[v]
\longrightarrow {\mathcal C}_{\mathcal T'}$.
Let 
\[
 \iota_{({\mathcal C}_{\mathcal T'}\setminus\Gamma_{\mathcal T'})[v]}\colon
 ({\mathcal C}_{\mathcal T'}\setminus\Gamma_{\mathcal T'})[v]\hookrightarrow
 {\mathcal C}_{{\mathcal T'}[v]}
\]
be the inclusion morphism.
We denote by
$\iota_{{\mathcal C}_{\mathcal T'}\setminus\Gamma_{\mathcal T'}}
\colon
{\mathcal C}_{\mathcal T'}\setminus \Gamma_{\mathcal T'}
\hookrightarrow {\mathcal C}_{\mathcal T'}$
its restriction to the underlying sets of points.

\begin{definition} 
We define the ${\mathcal O}_{{\mathcal C}_{{\mathcal T}'[v]}}^{hol}$-subsheaf
$\Omega^1_{{\mathcal C}_{{\mathcal T}',v}}$
of
$(\iota_{({\mathcal C}_{\mathcal T'}\setminus\Gamma_{\mathcal T'})[v]})_*
\big(\Omega^1_{ \left.\left({\mathcal C}_{\mathcal T'}\setminus
\Gamma_{\mathcal T'}\right)[v] \right/ {\mathcal T'}}\big)^{hol}$
by the condition that
$\Omega^1_{{\mathcal C}_{{\mathcal T}',v}}$
is locally generated by
$\dfrac{d\tilde{z}^{(i)}} {(\tilde{z}^{(i)})^{m_i}-\epsilon^{m_i}}$ and
$\big(\iota_{{\mathcal C}_{\mathcal T'}\setminus\Gamma_{\mathcal T'}}\big)_*
\big( {\mathcal O}_{{\mathcal C}_{\mathcal T'}
\setminus \Gamma_{\mathcal T'}}^{hol} \big)
d\bar{h}$
around points in $\Gamma^{(i)}_{{\mathcal T}'[v]}$
and locally generated by
$d\tilde{z}$ and $d\bar{h}$ around points in
$({\mathcal C}_{\mathcal T'}\setminus\Gamma_{\mathcal T'})[v]$
where $z$ is a local holomorphic coordinate of
${\mathcal C}_{{\mathcal T}'}\setminus\Gamma_{{\mathcal T}'}$.
We denote by $\Omega^2_{{\mathcal C}_{\mathcal T'},v}$
the canonical image of
$\Omega^1_{{\mathcal C}_{{\mathcal T}',v}} \wedge
\Omega^1_{{\mathcal C}_{{\mathcal T}',v}}$
in
$(\iota_{({\mathcal C}_{\mathcal T'}\setminus\Gamma_{\mathcal T'})[v]})_*
\big(\Omega^2_{ \left.\left({\mathcal C}_{\mathcal T'}\setminus
\Gamma_{\mathcal T'}\right)[v] \right/ {\mathcal T'}}\big)^{hol}$.
\end{definition}

We put
$M':=M^{\circ}\times_{{\mathcal T}^{\circ}}{\mathcal T}'$
and consider the analytic space
$M'[v]:=M'\times\Spec\mathbb{C}[h]/(h^2)$
with the structure morphisms 
\[
 M'[v]:= M' \times\Spec\mathbb{C}[h]/(h^2)  \longrightarrow
 {\mathcal T'}\times\Spec\mathbb{C}[h]/(h^2) \xrightarrow{I_v} {\mathcal T'}.
\]
We denote the base change of ${\mathcal C}\times_{\mathcal P}{\mathcal T'}$,
${\mathcal D}\times_{\mathcal P}{\mathcal T'}$ and
${\mathcal D}^{(i)}\times_{\mathcal P}{\mathcal T'}$
via $M'[v]\longrightarrow {\mathcal T'}$ by
${\mathcal C}_{M'[v]}$,
${\mathcal D}_{M'[v]}$ and
${\mathcal D}^{(i)}_{M'[v]}$,
respectively.
We denote the pullback of a local holomorphic coordinate
$z$ of ${\mathcal C}_{\mathcal T'}$ under the morphism
${\mathcal C}_{M'[v]}\longrightarrow {\mathcal C}_{\mathcal T'}$
by $\tilde{z}$.

Let us consider the analytic open subspace
$(U_i)_{M'[v]}\subset {\mathcal C}_{M'[v]}
={\mathcal C}_{\mathcal T'}\times_{\mathcal T'}(M'\times\Spec\mathbb{C}[h]/(h^2))$.
Using (\ref {equation:local biholomorphic})
in subsection \ref{subsection:moduli setting},
we have an analytic isomorphism
\[
 (U_i)_{M'[v]} \cong \Delta_a \times M'[v]
 =\Delta_a \times M'\times \Spec\mathbb{C}[h]/(h^2)
\]
whose structure morphism over ${\mathcal T}_{\bmu,\blambda}$ is given by
\[
 \Delta_a \times M'\times\Spec\mathbb{C}[h]/(h^2)
 \longrightarrow M'\times\Spec\mathbb{C}[h]/(h^2)
 \longrightarrow {\mathcal T}'\times\Spec\mathbb{C}[h]/(h^2)
 \xrightarrow{I_v} {\mathcal T}'\hookrightarrow {\mathcal T}_{\bmu,\blambda}.
\]


We remark that the elements in
$\Omega^1_{{\mathcal C}_{\mathcal T'},v}\otimes_{O_{{\mathcal T}'[v]}}
 {\mathcal O}_{M'[v]}
 \subset
 (\iota_{({\mathcal C}_{M'}\setminus\Gamma_{M'})[v]})_*
 \big(\Omega^1_{\left.\left({\mathcal C}_{M'}
 \setminus\Gamma_{M'}\right)[v]\right/M'}\big)^{hol}$
are relative differentials with respect to the morphism
\[
 {\mathcal C}_{M'[v]}=
 {\mathcal C}_{\mathcal T'}\times_{\mathcal T'}
 (M'\times\Spec\mathbb{C}[h]/(h^2))
 \longrightarrow M'\times\Spec\mathbb{C}[h]/(h^2)
 \longrightarrow M',
\]
where the arrows are the trivial projections.
The restriction of the above morphism to $(U_i)_{M'[v]}$
is just the trivial projection
$(U_i)_{M'[v]}\cong \Delta_a \times M' \times \Spec\mathbb{C}[h]/(h^2)
\longrightarrow M'$.
The corresponding inclusion
${\mathcal O}_{M'}^{hol} \hookrightarrow
{\mathcal O}_{(U_i)_{M'[v]}}^{hol}$
induces the ring homomorphism
\[
 {\mathcal O}_{M'}^{hol}[\tilde{z}^{(i)}]
 \longrightarrow
 {\mathcal O}^{hol}_{(U_i)_{M'[v]}}
\]
from the polynomial ring.
We denote the image of a matrix $A(z^{(i)})$
of polynomials with coefficients in ${\mathcal O}_{M'}^{hol}$
under this ring homomorphism
by $A(\tilde{z}^{(i)})$. 

We denote the restriction of
$\big(\tilde{E}^{hol}_{M^{\circ}},\tilde{\nabla}^{hol}_{M^{\circ}},
\{\tilde{N}^{(i),hol}_{M^{\circ}}\}\big)$
to ${\mathcal C}_{M'}$ by
$\big(\tilde{E}^{hol}_{M'},\tilde{\nabla}^{hol}_{M'},
\{\tilde{N}^{(i),hol}_{M'}\}\big)$.

\begin{definition} \label {def:global-horizontal-lift}
\rm
We say that a tuple
$\big({\mathcal E}^v,\nabla^v,\{{\mathcal N}^{(i)}_v\}\big)$
is a horizontal lift of
$\big(\tilde{E}^{hol}_{M'},\tilde{\nabla}^{hol}_{M'},
\{\tilde{N}^{(i),hol}_{M'}\}\big)$
with respect to
$v\in H^0({\mathcal T'},T^{hol}_{{\mathcal T'}/\Delta_{\epsilon_0}})$
and with respect to blocks of local horizontal lifts
$\big(\nabla_{\mathbb{P}^1\times M'[\bar{h}],v^{(i)}_{l,j}}^{flat}\big)$
if
\begin{itemize}
\item[(1)] ${\mathcal E}^v$ is a rank $r$ holomorphic vector bundle on
${\mathcal C}_{M'[v]}$,
\item[(2)] $\nabla^v\colon {\mathcal E}^v
\longrightarrow
{\mathcal E}^v\otimes_{{\mathcal O}_{{\mathcal C}_{{\mathcal T}'[v]}}^{hol}}
\Omega^1_{{\mathcal C}_{\mathcal T'},v}$
is a morphism of sheaves satisfying $\nabla^v(fa)=a\otimes df+f\nabla^v(a)$
for $f\in {\mathcal O}_{{\mathcal C}_{M'[v]}}^{hol}$ and $a\in{\mathcal E}^v$,
\item[(3)]
$\nabla^v$ is integrable in the sense that
the restriction of $\nabla^v$ to any open set
$U[v]\subset ({\mathcal C}_{M'}\setminus\Gamma_{M'})[v]$
which is expressed by
\[
 {\mathcal E}^v |_{U[v]}\cong \left({\mathcal O}^{hol}_{U[v]}\right)^{\oplus r}
 \ni \begin{pmatrix} f_1 \\ \vdots \\ f_r \end{pmatrix}
 \mapsto
 \begin{pmatrix} df_1 \\ \vdots \\ df_r \end{pmatrix}
 +
 \left( A \, d\tilde{z}
 +B \, d\overline{h}\right)
 \begin{pmatrix} f_1 \\ \vdots \\ f_r \end{pmatrix}
 \in \left( {\mathcal O}^{hol}_{U[v]} \right)^{\oplus r}
 \otimes_{{\mathcal O}_{{\mathcal C}_{{\mathcal T}'[v]}}}
 \Omega^1_{{\mathcal C}_{\mathcal T'},v}
\]
satisfies
\[
 d \left( A \, d\tilde{z}
 +B \, d\overline{h}\right) 
 +\left( A\, d\tilde{z} +B \, d\overline{h}\right)
 \wedge \left( A\, d\tilde{z} + B \, d\overline{h}\right)=0
 \]
in
${\mathcal End}\big( ({\mathcal O}^{hol}_{U[v]})^{\oplus r} \big)
\otimes_{{\mathcal O}_{{\mathcal C}_{{\mathcal T}'[v]}}}
\Omega^2_{{\mathcal C}_{\mathcal T'},v}$,
\item[(4)]
${\mathcal N}^{(i)}_v\colon {\mathcal E}^v |_{{\mathcal D}^{(i)}_{M'[v]}}
\longrightarrow  {\mathcal E}^v |_{{\mathcal D}^{(i)}_{M'[v]}}$
is an endomorphism satisfying
$\varphi^{(i)}_{\bmu}({\mathcal N}^{(i)}_v)=0$,
\item[(5)]
the relative connection
$\overline{\nabla^v}$  defined by the composition
\[
 \overline{\nabla^v}\colon
 {\mathcal E}^v \xrightarrow{\nabla^v}
 {\mathcal E}^v\otimes_{{\mathcal O}_{{\mathcal C}_{{\mathcal T}'[v]}}^{hol}}
 \Omega_{{\mathcal C}_{\mathcal T'},v}^1
 \longrightarrow
 {\mathcal E}^v\otimes\Omega^1_{{\mathcal C}_{M'[v]}/M'[v]}
 ({\mathcal D}_{M'[v]})^{hol}
\]
satisfies
\[
 \displaystyle (\nu^{(i)}_{hor}+\bar{h}\nu^{(i)}_v)({\mathcal N}^{(i)}_v)
 \;
 \frac{d\tilde{z}^{(i)}}{(\tilde{z}^{(i)})^{m_i}-\epsilon^{m_i}}
 =\overline{\nabla^v} \big|_{{\mathcal D}^{(i)}_{M'[v]}}
\]
for any $i$,
\item[(6)]
$\big({\mathcal E}^v,\overline{\nabla^v},\{ {\mathcal N}^{(i)}_v \}\big)
\otimes {\mathcal O}_{M'[v]}^{hol}/
\overline{h}{\mathcal O}_{M'[v]}^{hol}
\cong
\big(\tilde{E}^{hol}_{M'},\tilde{\nabla}^{hol}_{M'},\{ \tilde{N}^{(i),hol}_{M'}\}\big)$,
\item[(7)] there is an isomorphism
$\theta^{(i),v}\colon {\mathcal E}^v\big|_{(U_i)_{M'[v]}}\xrightarrow{\sim}
({\mathcal O}^{hol}_{(U_i)_{M'[v]}})^{\oplus r}$
which is a lift of the restriction
$\theta^{(i)}|_{(U_i)_{M'}}$ of the given isomorphism
$\theta^{(i)}\colon \tilde{E}|_{(U_i)_{M^{\circ}}}\xrightarrow{\sim}
({\mathcal O}^{hol}_{(U_i)_{M^{\circ}}})^{\oplus r}$
such that the consequent connection matrix of
$(\theta^{(i),v}\otimes\mathrm{id})\circ \nabla^v
\circ ( \theta^{(i),v} )^{-1}$
is given by
\[
 \Big( A^{(i)}(z^{(i)},\epsilon) + \bar{h} \sum_{l=0}^{r-1}\sum_{j=0}^{m_i-1}
  v(c^{(i)}_{l,j}) \: \tilde{\Xi}^{(i)}_{l,j}(z^{(i)}) \Big)
 \frac { dz^{(i)} } { (z^{(i)})^{m_i}-\epsilon^{m_i} }
 +\sum_{l=0}^{r-1}\sum_{j=0}^{m_i-2} v(c^{(i)}_{l,j}) B^{(i)}_{l,j}(z^{(i)}) \: d\bar{h}.
\]
\end{itemize}
\end{definition}


The following proposition on the existence of a global horizontal lift is a key process
in the construction of an unfolded generalized isomonodromic deformation.

\begin{proposition} \label {prop:existence-horizontal-lift}
For any $\Delta_{\epsilon_0}$-relative holomorphic vector field
$v\in H^0({\mathcal T'},
T^{hol}_{{\mathcal T}^{\circ}/\Delta_{\epsilon_0}})$,
there exists a unique horizontal lift
$\big({\mathcal E}^v,\nabla^v,\{N^{(i)}_v\}\big)$
of
$\big(\tilde{E}^{hol}_{M'},\tilde{\nabla}^{hol}_{M'},\{ \tilde{N}^{(i),hol}_{M'}\}\big)$
with respect to $v$
and with respect to the blocks of local horizontal lifts
$\Big(\nabla_{\mathbb{P}^1\times M'[\bar{h}],v^{(i)}_{l,j}}^{flat}\Big)$.
\end{proposition}

\begin{proof}
We can take an analytic open covering
$\{ U_{\beta} \}$ of ${\mathcal C}_{M'}$
which is a refinement of
$\{ {\mathcal U}_{\alpha}\times_{\mathcal P'} M' \}$
such that $U_{\beta}$ is contractible and
$\tilde{E}_{M'}^{hol} \big|_{U_{\beta}} 
\cong ({\mathcal O}^{hol}_{U_{\beta}})^{\oplus r}$
for any $\beta$.
Moreover, we may assume that
$U_{\beta}\cap \Gamma^{(i)}_{M'} = \emptyset$
unless $U_{\beta}=(U_i)_{M'}$.
Recall that
$(\theta^{(i)}\otimes\mathrm{id})\circ
\tilde{\nabla}^{hol}_{M'} \big|_{(U_i)_{M'}}
\circ(\theta^{(i)})^{-1}$
is canonically  extended to 
a global connection
\[
 \nabla^{(i),\mathbb{P}^1}_{M^{\circ}}\big|_{\mathbb{P}^1\times M'}
 \colon
 ({\mathcal O}^{hol}_{\mathbb{P}^1\times M'})^{\oplus r}
 \longrightarrow
 ({\mathcal O}^{hol}_{\mathbb{P}^1\times M'})^{\oplus r}
 \otimes 
 \Omega^1_{\mathbb{P}^1\times M' / M' }
 ( {\mathcal D}_{M'} \cup ( \{\infty\} \times M' ) )^{hol}
\]
given by the connection matrix
\[
 A^{(i)}(z^{(i)},\epsilon) \frac{dz^{(i)}} {(z^{(i)})^{m_i}-\epsilon^{m_i}}.
\]
Here we use the identification
$(U_i)_{M'} = \Delta_a \times M' \hookrightarrow \mathbb{P}^1\times M'$.
As in Definition \ref {def:block of local horizontal lift},
there is a block
$\big(\nabla^{flat}_{\mathbb{P}^1\times M^{\circ}[\bar{h}],v^{(i)}_{l,j}}\big)$
of local horizontal lifts given by 
$\big(\tilde{\Xi}^{(i)}_{l,j}(z^{(i)})\big)$ and
$\big(B^{(i)}_{l,j}(z^{(i)})\big)$.
We put
\begin{align*}
 A^{(i)}_v(z^{(i)}) 
 &:=
 \sum_{l=0}^{r-2}\sum_{j=0}^{m_i-1}
 v(c^{(i)}_{l,j}) \: \tilde{\Xi}^{(i)}_{l,j}(z^{(i)})  \\
 B^{(i)}_v(z^{(i)})
 &:=
 \sum_{l=0}^{r-2}\sum_{j=0}^{m_i-1}
 v(c^{(i)}_{l,j})  B^{(i)}_{l,j}(z^{(i)})
\end{align*}
and 
denote by
$\iota_{M'[\bar{h}]} \colon
\big(\mathbb{P}^1\times M'\setminus\Gamma_{M'} \big)[\bar{h}]
\hookrightarrow \mathbb{P}^1\times M'[\bar{h}]$
the inclusion morphism.
Consider the connection
\begin{equation} 
 \nabla^{flat}_{\mathbb{P}^1\times M'[\bar{h}],v}
 \colon
 ({\mathcal O}_{ \mathbb{P}^1\times M'[\bar{h}] }^{hol})^{\oplus r}
 \longrightarrow
 ({\mathcal O}_{ \mathbb{P}^1\times M'[\bar{h}] }^{hol})^{\oplus r}
 \otimes 
 (\iota_{M'[\bar{h}]})_*
 \Omega^1_{(\mathbb{P}^1\times M'\setminus\Gamma_{M'})[\bar{h}]\big/M'}
 (\infty\times M')^{hol}
\end{equation}
determined by the connection matrix
\[
 \left(
 A^{(i)}(\tilde{z}^{(i)},\epsilon)+\bar{h}A_v(z^{(i)})\right) 
 \dfrac{d\tilde{z}^{(i)}} {(\tilde{z}^{(i)})^{m_i}-\epsilon^{m_i}}
 +B^{(i)}_v(z^{(i)}) d\bar{h}.
\]
Then we can see by the same calculation as in the proof of
Proposition \ref {prop:local horizontal lift} that
$\nabla^{flat}_{\mathbb{P}^1\times M'[\bar{h}],v}$
is an integrable connection.
We denote by ${\mathcal N}^{(i)}_v$
the substitution of $\epsilon^{m_i}$ for $(z^{(i)})^{m_i}$
in $\psi^{(i)} \big( A^{(i)}(\tilde{z}^{(i)},\epsilon)+\bar{h}A_v(z^{(i)}) \big)$,
where $\psi^{(i)}$ is given in (\ref {equation:polynomial psi}).
Then
$\big( {\mathcal O}_{ (U_i)_{M'[v]} }^{\oplus r},
 \nabla^{flat}_{\mathbb{P}^1\times M'[\bar{h}],v} \big|_{(U_i)_{M'[v]}},
\big\{ {\mathcal N}^{(i)}_v \big\} \big)$
gives a local horizontal lift of
$\big(\tilde{E}^{hol}_{M'},\tilde{\nabla}^{hol}_{M'},
\{ \tilde{N}^{(i),hol}_{M'}\}\big) \big|_{(U_i)_{M'}}$
with respect to $v$.

Assume that $U_{\beta}\cap {\mathcal D}^{(i)}_{M'}=\emptyset$
for any $i$.
Then the connection
$\tilde{\nabla}^{hol}_{M'}\big|_{U_{\beta}}$
is given by a connection matrix
$A(z) dz$,
for some local holomorphic coordinate $z$ of ${\mathcal C}_{\mathcal T'}$
over ${\mathcal T'}$.
We can take a matrix
$\tilde{A}(\tilde{z})$
with entries in ${\mathcal O}^{hol}_{U_{\beta}[v]}$
which is a lift of $A(z)$,
where $\tilde{z}$ is the pullback of $z$ under the morphism
${\mathcal C}_{\mathcal T'}[v]\xrightarrow{\mathrm{id}\times I_v}
{\mathcal C}_{\mathcal T'}$.
We can write
\[
 d\tilde{A}(\tilde{z})=
 C(\tilde{z}) d\tilde{z} + B(z) d\bar{h}. 
\]
If we put
$\tilde{A}'(\tilde{z}):=\tilde{A}(\tilde{z})-\bar{h} B(z)$,
then we have
$d \tilde{A}'(\tilde{z})\in
M_r({\mathcal O}^{hol}_{U_{\beta}[v]}) d\tilde{z}$
and
\begin{align*}
 \nabla^v_{\beta}
 & \colon  ({\mathcal O}^{hol}_{U_{\beta}[v]})^{\oplus r}
 \longrightarrow
 ({\mathcal O}^{hol}_{U_{\beta}[v]})^{\oplus r}
 \otimes \Omega^1_{{\mathcal C}_{\mathcal T'},v} \\
 &
 \begin{pmatrix} f_1 \\ \vdots \\ f_r \end{pmatrix}
 \mapsto
 \begin{pmatrix} df_1 \\ \vdots \\ df_r \end{pmatrix}
 +
 \tilde{A}'(\tilde{z}) d\tilde{z} \begin{pmatrix} f_1 \\ \vdots \\ f_r \end{pmatrix}
\end{align*}
becomes a flat connection.
So
$( ({\mathcal O}^{hol}_{U_{\beta}[v]})^{\oplus r} ,\nabla^v_{\beta} )$
becomes a local horizontal lift of
$\big(\tilde{E}^{hol}_{M'},\tilde{\nabla}^{hol}_{M'},
\{ \tilde{N}^{(i),hol}_{M'}\} \big) \big|_{U_{\beta}}$,
where $\{ \tilde{N}^{(i),hol}_{M'}\} \big|_{U_{\beta}}$
is nothing in this case.

From the above arguments, we obtain
a local horizontal lift 
$({\mathcal E}^v_{\beta}, \nabla^v_{\beta},
\{ {\mathcal N}^v_{\beta} \})$
of
$\big(\tilde{E}^{hol}_{M'},\tilde{\nabla}^{hol}_{M'},
\{ \tilde{N}^{(i),hol}_{M'}\}\big) \big|_{U_{\beta}}$
for each piece $U_{\beta}$ of the covering
${\mathcal C}_{M'}=\bigcup_{\beta} U_{\beta}$.
If $U_{\beta}\neq U_{\beta'}$,
then $\Gamma_{M'}\cap U_{\beta}\cap U_{\beta'}=\emptyset$
by the assumption.
Assume that  $\nabla^v_{\beta}$ is given by
\begin{align*}
 ({\mathcal O}^{hol}_{U_{\beta}[v]})^{\oplus r}
 &\xrightarrow{\sim}
 {\mathcal E}^v_{\beta}
 \xrightarrow{\nabla^v_{\beta}}
 {\mathcal E}^v_{\beta}\otimes \Omega^1_{{\mathcal C}_{\mathcal T'},v}
 \xrightarrow{\sim}
 ({\mathcal O}^{hol}_{U_{\beta}[v]})^{\oplus r}
 \otimes \Omega^1_{{\mathcal C}_{\mathcal T'},v} \\
 &
 \begin{pmatrix} f_1 \\ \vdots \\ f_r \end{pmatrix}
 \mapsto
 \begin{pmatrix} df_1 \\ \vdots \\ df_r \end{pmatrix}
 +
 (\tilde{A}_{\beta}(\tilde{z}) d\tilde{z} +B_{\beta}(z) d\bar{h} )
 \begin{pmatrix} f_1 \\ \vdots \\ f_r \end{pmatrix},
\end{align*}
where the integrability condition
\[
 -\frac{\partial \tilde{A}_{\beta}(\tilde{z})} {\partial\bar{h}} d\tilde{z}\wedge d\bar{h}
 + d B_{\beta}(z) \wedge d\bar{h}
 +(\tilde{A}_{\beta}(\tilde{z}) B_{\beta}(z)-B_{\beta}(z)\tilde{A}_{\beta}(\tilde{z}))
 d\tilde{z}\wedge d\bar{h}=0
\]
is satisfied and so for $\nabla^u_{\beta'}$.
There is an invertible matrix
$P_{\beta,\beta'}(z)$ of holomorphic functions on
$U_{\beta\beta'}=U_{\beta}\cap U_{\beta'}$
satisfying
\[
 P_{\beta,\beta'}(z)^{-1} d P_{\beta,\beta'}(z)
 +P_{\beta,\beta'}(z)^{-1}
 \tilde{A}_{\beta}(z) dz
 P_{\beta,\beta'}(z)
 = \tilde{A}_{\beta'}(z) dz
\]
coming from the isomorphism
$({\mathcal E}^v_{\beta},\overline{\nabla^v_{\beta}})\big|_{U_{\beta\beta'}}
\xrightarrow{\sim}
(\tilde{E}^{hol}_{M'},\tilde{\nabla}^{hol}_{M'})\big|_{U_{\beta\beta'}}
\xrightarrow{\sim}
({\mathcal E}^v_{\beta'},\overline{\nabla^v_{\beta'}})\big|_{U_{\beta\beta'}}$.
We can take a matrix
$\tilde{P}_{\beta\beta'}(\tilde{z},\bar{h})$ of
holomorphic functions on $U_{\beta\beta'}[\bar{h}]$
which is a lift of $P_{\beta\beta'}(z)$.
If we put
\begin{align*}
 \tilde{A}'_{\beta}(\tilde{z})d\tilde{z}+B'_{\beta}(z)d\bar{h} :=
 \tilde{P}_{\beta,\beta'}(\tilde{z},\bar{h})^{-1} d \tilde{P}_{\beta,\beta'}(\tilde{z},\bar{h})
 +\tilde{P}_{\beta,\beta'}(\tilde{z},\bar{h})^{-1}
 \big(\tilde{A}_{\beta}(\tilde{z}) d\tilde{z}
 +B_{\beta}(z) d\bar{h} \big)
 \tilde{P}_{\beta,\beta'}(\tilde{z},\bar{h}),
\end{align*}
then we can write
$\tilde{A}_{\beta'}(\tilde{z})=\tilde{A}'_{\beta}(\tilde{z})+\bar{h}C_{\beta}(z)$.
If we put
$Q_{\beta\beta'}(z):=B_{\beta'}(z)-B'_{\beta}(z)$,
then 
$Q_{\beta\beta'}(z)$ is holomorphic on 
$U_{\beta}\cap U_{\beta'}=
(U_{\beta}\cap U_{\beta'}) \setminus
(\Gamma_{M'}\cap U_{\beta}\cap U_{\beta'})$
and we have
\begin{align*}
 & (I_r+\bar{h}Q_{\beta\beta'}(z))^{-1} d (I_r+\bar{h} Q_{\beta\beta'}(z))
 +
 (I_r+\bar{h}Q_{\beta\beta'}(z))^{-1} 
 (\tilde{A}'_{\beta}(\tilde{z})d\tilde{z}+B'_{\beta}(z)d\bar{h} )
 (I_r+\bar{h}Q_{\beta\beta'}(z)) \\
 &=
 \bar{h} dQ_{\beta\beta'}+Q_{\beta\beta'}d\bar{h}
 +
 \tilde{A}_{\beta'}(\tilde{z})d\tilde{z}-\bar{h}C_{\beta}(z)d\tilde{z}
 +\bar{h} [\tilde{A}_{\beta'}(\tilde{z}),B_{\beta'}(z)-B'_{\beta}(z)]d\tilde{z}
 +B'_{\beta}(z) d\bar{h} \\
 &=
 \tilde{A}_{\beta'}(\tilde{z})d\tilde{z}-\bar{h}C_{\beta}(z)d\tilde{z}
 +\bar{h} \big( dB_{\beta'}(z)+ [\tilde{A}_{\beta'}(\tilde{z}),B_{\beta'}(z)]d\tilde{z} \big) \\
 &\quad
 -\bar{h} \big( dB'_{\beta}(z)+ [\tilde{A}_{\beta'}(\tilde{z}),B'_{\beta}(z)]d\tilde{z} \big)
 + \big( Q_{\beta\beta'}(z)+B'_{\beta}(z) \big) d\bar{h} \\
 &= 
 \tilde{A}_{\beta'}(\tilde{z})d\tilde{z}-\bar{h}C_{\beta}(z)d\tilde{z}
 +\bar{h}\left(\frac{\partial\tilde{A}_{\beta'}}{\partial\bar{h}}(\tilde{z})
 -\frac{\partial\tilde{A}'_{\beta}(\tilde{z})}{\partial\bar{h}}(\tilde{z})\right) d\tilde{z}
 +B_{\beta'}(z) d\bar{h} \\
 &=
 \tilde{A}_{\beta'}(\tilde{z})d\tilde{z}
 +B_{\beta'}(z) d\bar{h}
\end{align*}
Thus the composition of $P_{\beta,\beta'}(\tilde{z},\bar{h})$ with
$I_r+\bar{h}Q_{\beta\beta'}(z)$
gives an isomorphism between
$({\mathcal E}^v_{\beta},\nabla^v_{\beta})\big|_{U_{\beta\beta'}[v]}$
and
$({\mathcal E}^v_{\beta'},\nabla^v_{\beta'})\big|_{U_{\beta\beta'}[v]}$
whose restriction to
$U_{\beta\beta'}=U_{\beta\beta'}[v]\otimes\mathbb{C}[\bar{h}]/(\bar{h})$
is the identity.
By construction, we can see that this isomorphism is unique,
because it is essentially determined by the $d\bar{h}$-coefficients.
So we can patch 
$({\mathcal E}^v_{\beta}, \nabla^v_{\beta},
\{ {\mathcal N}^v_{\beta} \})$
together and obtain a global horizontal lift
$({\mathcal E}^v,\nabla^v,\{N^{(i)}_v\})$
of 
$\big(\tilde{E}^{hol}_{M'},
\tilde{\nabla}^{hol}_{M'},\{ \tilde{N}^{(i),hol}_{M'}\}\big)$
with respect to $v$ and with respect to
the blocks $\big(\nabla^{flat}_{\mathbb{P}^1\times M^{\circ}[\bar{h}],v^{(i)}_{l,j}}\big)$
of local horizontal lifts.
Since the local horizontal lift is unique up to a unique isomorphism,
we can see that a global  horizontal lift
$({\mathcal E}^v,\nabla^v,\{N^{(i)}_v\})$
is unique up to an isomorphism.
\end{proof}

For a vector field
$v\in H^0({\mathcal T'},
T^{hol}_{{\mathcal T}_{\bmu,\blambda}\times_{\mathcal B} {\mathcal B'}/\Delta_{\epsilon_0}} )$
over an analytic open subset
${\mathcal T'} \subset {\mathcal T}^{\circ} \subset
{\mathcal T}_{\bmu,\blambda}\times_{\mathcal B}{\mathcal B'}$,
we have by Proposition \ref  {prop:existence-horizontal-lift} a unique 
horizontal lift
$({\mathcal E}^v,\nabla^v,\{{\mathcal N}^{(i)}_v\})$
of the restriction
$\big(\tilde{E}^{hol}_{M'},\tilde{\nabla}^{hol}_{M'},
\{\tilde{N}^{(i),hol}_{M'}\}\big)$
of the universal family to
${\mathcal C}\times_{\mathcal H}M'$
with respect to $v$ and with respect to the blocks
$\big(\nabla^{flat}_{\mathbb{P}^1\times M^{\circ}[\bar{h}],v^{(i)}_{l,j}}\big)$
of local horizontal lifts.
Let
\[
 \overline{\nabla^v} \colon {\mathcal E}^v
 \xrightarrow{\nabla^v}
 {\mathcal E}^v \otimes \Omega^1_{{\mathcal C}_{\mathcal T'},v}
 \longrightarrow
 {\mathcal E}^v\otimes\Omega^1_{{\mathcal C}_{M'[v]}/M'[v]}({\mathcal D}_{M'[v]})
\]
be the relative connection induced by $\nabla^v$.
Then $({\mathcal E}^v,\overline{\nabla^v},\{ {\mathcal N}^{(i)}_v\})$
becomes a holomorphic flat family of $(\bnu,\bmu)$-connections on
${\mathcal C}_{M'[v]}$ over $M'[v]$,
which determines a morphism
$M'[v] \longrightarrow 
M^{\balpha}_{{\mathcal C},{\mathcal D}}(\tilde{\bnu},\bmu)
\times_{{\mathcal T}_{\bmu,\blambda}} {\mathcal T'}$
making the diagram
\begin{equation} \label {equation:commutative diagram of tangent morphism}
 \begin{CD}
  M'[v] @>>>
  M^{\balpha}_{{\mathcal C},{\mathcal D}}(\tilde{\bnu},\bmu)
  \times_{{\mathcal T}_{\bmu,\blambda}} {\mathcal T'}  \\
  @VVV    @VVV \\
  {\mathcal T}'[v]  @> I_v >> {\mathcal T}' 
 \end{CD}
\end{equation}
commutative.
This morphism corresponds to a vector field
$\Phi(v)
\in H^0\big( (\pi^{\circ})^{-1}({\mathcal T'}),
T^{hol}_{M^{\circ} /\Delta_{\epsilon_0} } \big|_{(\pi^{\circ})^{-1}({\mathcal T'})} \big)$,
where
$\pi^{\circ} \colon M^{\circ} \longrightarrow {\mathcal T}^{\circ}$
is the projection morphism.
We can see $d\pi^{\circ}(\Phi(v))=v$
by the commutative diagram (\ref {equation:commutative diagram of tangent morphism}),
where
$d\pi^{\circ} \colon
\pi^{\circ}_*T^{hol}_{M^{\circ}/\Delta_{\epsilon_0}}
\longrightarrow
T^{hol}_{{\mathcal T}^{\circ}/\Delta_{\epsilon_0}}$
is the differential of $\pi^{\circ}$.
Thus we have defined a map
\begin{equation} \label {equation: definition of pi_*D}
 \Phi \:
 \colon \: T^{hol}_{{\mathcal T}^{\circ}/\Delta_{\epsilon_0}} 
 \ni v \; \mapsto \; \Phi(v)
 \in
 (\pi^{\circ})_* T^{hol}_{M^{\circ}/\Delta_{\epsilon_0}}.
\end{equation}

In the rest of this subsection,
we will prove
that the correspondence (\ref {equation: definition of pi_*D})
defined above is an
${\mathcal O}^{hol}_{{\mathcal T}^{\circ}}$-homomorphism.
In order to prove it,
we extend the notion of horizontal lift. 


Let $\mathbb{C}[I]=\mathbb{C}\oplus I$ be a finite dimensional local algebra over
$\mathbb{C}$ with the maximal ideal $I$
satisfying $I^2=0$.
For a morphism
$u\colon
{\mathcal T}'\times\Spec\mathbb{C}[I]
\longrightarrow{\mathcal T'}$
over $\Delta_{\epsilon_0}$
satisfying
$u|_{{\mathcal T}'\times\Spec\mathbb{C}[I]/I}=\mathrm{id}_{\mathcal T'}$,
we write ${\mathcal T'}[u]:={\mathcal T}'\times\Spec\mathbb{C}[I]$
which is endowed with the structure morphism
$u\colon {\mathcal T'}[u] \longrightarrow {\mathcal T'}$.
We endow the fiber product
${\mathcal C}_{{\mathcal T'}[u]}
 :=
 {\mathcal C}\times_{\mathcal H}{\mathcal T'}\times\Spec\mathbb{C}[I]$
with the structure morphism
\[
 {\mathcal C}_{{\mathcal T'}[u]}
 =
 {\mathcal C}\times_{\mathcal H}{\mathcal T'}\times\Spec\mathbb{C}[I]
 \longrightarrow {\mathcal T'}\times\Spec\mathbb{C}[I]
 \xrightarrow{u}{\mathcal T'}.
\]
For an analytic open subset
$U\subset{\mathcal C}_{\mathcal T'}$,
we denote by $U[u]$ the open subspace of
${\mathcal C}_{{\mathcal T}'[u]}$
whose underlying set of points is $U$.

We consider 
the sheaf of differential forms
$\big(
\Omega^1_{\left.\left({\mathcal C}_{\mathcal T'}
\setminus\Gamma_{\mathcal T'}\right)[u]
\right/{\mathcal T'}} \big)^{hol}$
with respect to the composite of the trivial projections
\[
 {\mathcal C}_{{\mathcal T'}[u]}
 ={\mathcal C}\times_{\mathcal P}{\mathcal T'}\times\Spec\mathbb{C}[I]
 \longrightarrow {\mathcal T'}\times\Spec\mathbb{C}[I]
 \longrightarrow{\mathcal T'}
 \]
which is different from the structure of
${\mathcal C}_{{\mathcal T'}[u]}$ over ${\mathcal T'}$
coming from the fiber product structure.
We can consider the quotient sheaf
\[
 \big(\Omega^1_{ \left.\left({\mathcal C}_{\mathcal T'}
 \setminus\Gamma_{\mathcal T'}\right)[u]
 \right/ {\mathcal T'} } \big)^{hol} \big/
 \big(I{\mathcal O}_{ ({\mathcal C}_{\mathcal T'}
 \setminus\Gamma_{\mathcal T'})[u]} ^{hol} dI\big)
\]
and define a subsheaf $\Omega^1_{{\mathcal C}_{\mathcal T'},u}$ of
$\big(\iota_{({\mathcal C}_{\mathcal T'}
\setminus\Gamma_{\mathcal T'})[u]}\big)_*
\left(
\big(\Omega^1_{  \left.\left({\mathcal C}_{\mathcal T'}
\setminus\Gamma_{\mathcal T'}\right)[u]
\right/{\mathcal T'}}\big)^{hol}
\big/
(I{\mathcal O}_{ ({\mathcal C}_{\mathcal T'}
\setminus\Gamma_{\mathcal T'})[u] }^{hol} dI)
\right)$
locally generated by
\[
 \left\{ \frac{d\tilde{z}^{(i)}} {\big(\tilde{z}^{(i)}\big)^{m_i}-\epsilon^{m_i}} \right\}
 \cup
 \sum_{q=1}^{\kappa} 
 \big(\iota_{{\mathcal C}_{\mathcal T'}\setminus\Gamma_{\mathcal T'}}\big)_*
 \big( {\mathcal O}_{{\mathcal C}_{\mathcal T'}\setminus \Gamma_{\mathcal T'}}^{hol} \big)
 d\bar{h}_q
\]
around points $p\in (\Gamma^{(i)})_{{\mathcal T'}[u]}$
and locally generated by
$\{ d\tilde{z} \}\cup \{ d\bar{h}_j |\bar{h}_j \in I \}$
around points
$p\in \big( {\mathcal C}_{\mathcal T'}\setminus \Gamma_{\mathcal T'}\big) [u]$.
Here $\bar{h}_1,\ldots,\bar{h}_{\kappa}$
is a basis of $I$
and $z$ is a local holomorphic coordinate of
${\mathcal C}_{{\mathcal T'}}\setminus \Gamma_{{\mathcal T'}}$
over ${\mathcal T'}$.
We denote 
the image of
$\Omega^1_{{\mathcal C}_{\mathcal T'},u}\wedge \Omega^1_{{\mathcal C}_{\mathcal T'},u}$
in
$(\iota_{({\mathcal C}_{\mathcal T'}\setminus\Gamma_{\mathcal T'})[u]})_*
\left(
\big(\Omega^2_{  \left.\left({\mathcal C}_{\mathcal T'}
\setminus\Gamma_{\mathcal T'}\right)[u]
\right/{\mathcal T'}}\big)^{hol}
\big/
\big(I{\mathcal O}_{ ({\mathcal C}_{\mathcal T'}
\setminus\Gamma_{\mathcal T'})[u] }^{hol} dI \big)
\right)$
by $\Omega^2_{{\mathcal C}_{\mathcal T'},u}$.

For each $i=1,\ldots,n$,
we consider the sheaf of differential forms
$\Omega^1_{(U_i)_{M'[u]}/M'}$
with respect to 
\[
 (U_i)_{M'[u]}\hookrightarrow {\mathcal C}\times_{\mathcal P}
 (M'\times\Spec\mathbb{C}[I])
 \longrightarrow M'\times\Spec\mathbb{C}[I]
 \longrightarrow M',
\]
where the last two arrows are the trivial projections.
From the above projection, a ring homomorphism from the polynomial ring
\[
 {\mathcal O}^{hol}_{M'}[\tilde{z}^{(i)}]
 \longrightarrow
 {\mathcal O}^{hol}_{(U_i)_{M'[u]}}
\]
is induced.
We denote the image of a matrix $A(z^{(i)})$
of polynomials in $z^{(i)}$ with coefficients in ${\mathcal O}_{M'}^{hol}$
under this ring homomorphism by
$A(\tilde{z}^{(i)})$.

Note that we can write
\[
 u^*(\nu^{(i)}(T))
 =\nu^{(i)}_{hor}(T)+\sum_{q=1}^s \bar{h}_q \nu_{u,q}^{(i)}(T)
\]
with
\begin{align*}
 \nu^{(i)}_{hor}(T)
 &=
 \sum_{l=0}^{r-1} \sum_{j=0}^{m_i-1}
 c^{(i)}_{hor,l,j} (\tilde{z}^{(i)})^j T^l  \\
 \nu^{(i)}_{u,q}(T)
 &=
 \sum_{l=0}^{r-1} c^{(i)}_{u,q,l,j} (\tilde{z}^{(i)})^j T^l,
\end{align*}
where
$c^{(i)}_{hor,l,j}$ and $c^{(i)}_{u,q,l,j}$
are pullbacks of
$c^{(i)}_{l,j}, c^{(i)}_{u,q,l,j} \in {\mathcal O}_{M'}^{hol}$
under the composition of the trivial projections
$(U_i)_{M'[u]}\longrightarrow M'[u]\longrightarrow M'$.

\begin{definition}\label {def:horizontal-lift-general} \rm
Under the above notation,
we say that a tuple
$\big({\mathcal E}^u,\nabla^u,\{ {\mathcal N}^{(i)}_u \}\big)$
is a horizontal lift of
$\big(\tilde{E}_{M'},\tilde{\nabla}_{M'},\{\tilde{N}^{(i)}_{M'}\}\big)$
with respect to $u$ and with respect to
blocks of local horizontal lifts
$\Big(\nabla_{\mathbb{P}^1\times M'[\bar{h}],v^{(i)}_{l,j}}^{flat}\Big)$
if
\begin{itemize}
\item[(1)] ${\mathcal E}^u$ is a rank $r$ holomorphic vector bundle on
${\mathcal C}_{M'[u]}$,
\item[(2)] 
$\nabla^u\colon {\mathcal E}^u
\longrightarrow
{\mathcal E}^u\otimes_{{\mathcal O}_{{\mathcal C}_{{\mathcal T}'[u]}}^{hol}}
\Omega_{{\mathcal C}_{\mathcal T'},[u]}^1$
is a morphism of sheaves
satisfying $\nabla^u(fa)=a\otimes df+f\nabla^u(a)$
for $f\in {\mathcal O}_{{\mathcal C}_{M'[u]}}^{hol}$ and
$a\in {\mathcal E}^u$,
\item[(3)]
$\nabla^u$ is integrable in the sense that for each local expression
\[
 \begin{pmatrix} f_1 \\ \vdots \\ f_r \end{pmatrix}
 \mapsto
 \begin{pmatrix} df_1 \\ \vdots \\ df_r \end{pmatrix}
 +
 \left( A d\tilde{z} +\sum_{l=1}^{\kappa} B_l d\bar{h}_l \right)
 \begin{pmatrix} f_1 \\ \vdots \\ f_r \end{pmatrix}
\]
of $\nabla^u$ on ${\mathcal E}^u|_{U[u]}\cong {\mathcal O}_{U[u]}^{\oplus r}$
for an open subset 
$U[u]\subset ({\mathcal C}_{{M}'}\setminus \Gamma_{{M}'})[u]$,
the equality
\begin{align*}
 & d\Big(  A \, d\tilde{z} 
 +\sum_{l=1}^{\kappa} B_l \, d\bar{h}_l \Big) 
 + 
 \Big( A \, d\tilde{z} +\sum_{l=1}^{\kappa} B_l \, d\bar{h}_l \Big)
 \wedge
 \Big(  A \, d\tilde{z} +\sum_{l=1}^{\kappa} B_l \, d\bar{h}_l \Big)
 =0
\end{align*}
holds in $\Omega^2_{{\mathcal C}_{\mathcal T'},u}$,
where $\{ \bar{h}_1,\ldots,\bar{h}_{\kappa} \}$ is a basis of $I$ over $\mathbb{C}$.
\item[(4)]
${\mathcal N}^{(i)}_u\colon {\mathcal E}^u |_{{\mathcal D}^{(i)}_{M'[u]}}
\longrightarrow  {\mathcal E}^u |_{{\mathcal D}^{(i)}_{M'[u]}}$
is an endomorphism satisfying
$\varphi^{(i)}_{\bmu}({\mathcal N}^{(i)}_u)=0$,
\item[(5)]
the relative connection
$\overline{\nabla^u}$  defined by the composition
\[
 \overline{\nabla^u}\colon
 {\mathcal E}^u \xrightarrow{\nabla^u}
 {\mathcal E}^u\otimes\Omega_{{\mathcal C}_{\mathcal T'},u}^1
 \longrightarrow
 {\mathcal E}^u\otimes\Omega^1_{{\mathcal C}_{M'[u]}/M'[u]}
 ({\mathcal D}_{M'[u]})^{hol}
\]
satisfies
\[
 \displaystyle ( u^*\nu^{(i)})({\mathcal N}^{(i)}_u)
 \;
 \frac {d\tilde{z}^{(i)}} {\big(\tilde{z}^{(i)}\big)^{m_i}-\epsilon^{m_i}}
 =\overline{\nabla^u} \big|_{{\mathcal D}^{(i)}_{M'[u]}}
\]
for any $i$,
\item[(6)]
$\big({\mathcal E}^u,\overline{\nabla^u},\{ {\mathcal N}^{(i)}_u \}\big)
\otimes {\mathcal O}_{M'[u]}^{hol} / I {\mathcal O}_{M'[u]}^{hol}
\cong
\big(\tilde{E}^{hol}_{M'},\tilde{\nabla}^{hol}_{M'},\{ \tilde{N}^{(i),hol}_{M'}\}\big)$,
\item[(7)] there is an isomorphism
$\theta^{(i),u}\colon {\mathcal E}^u \big|_{(U_i)_{M'[u]}}\xrightarrow{\sim}
({\mathcal O}^{hol}_{(U_i)_{M'[u]}})^{\oplus r}$
which is a lift of  the given isomorphism
$\theta^{(i)}|_{(U_i)_{M'}}\colon 
\tilde{E}|_{(U_i)_{M'}}\xrightarrow{\sim}
({\mathcal O}^{hol}_{(U_i)_{M'}})^{\oplus r}$
such that the connection matrix of
$(\theta^{(i),u}\otimes\mathrm{id})\circ \nabla^u
\circ ( \theta^{(i),u} )^{-1}$
is given by
\[
 \Big( A^{(i)}(\tilde{z}^{(i)},\epsilon)
  + \sum_{q=1}^{\kappa} \bar{h}_q \sum_{l=0}^{r-1}\sum_{j=0}^{m_i-1}
  c^{(i)}_{u,q,l,j} \: \tilde{\Xi}^{(i)}_{l,j}(\tilde{z}^{(i)}) \Big)
 \frac {d\tilde{z}^{(i)}} {(\tilde{z}^{(i)})^{m_i}-\epsilon^{m_i}}  
 + \sum_{q=1}^{\kappa}  \sum_{l=0}^{r-1}\sum_{j=0}^{m_i-1}
  c^{(i)}_{u,q,l,j} B^{(i)}_{l,j}(\tilde{z}^{(i)}) d \bar{h}_q.
\]
\end{itemize}
\end{definition}

\begin{lemma} \label {lemma:horizontal lift general}
There exists a unique  horizontal lift
$\big({\mathcal E}^u,\nabla^u,\{{\mathcal N}^{(i)}_u\}\big)$ of
$\big(\tilde{E}_{M'},\tilde{\nabla}_{M'},\{\tilde{N}^{(i)}_{M'}\}\big)$
with respect to $u$
and with respect to blocks of local horizontal lifts
$\Big(\nabla_{\mathbb{P}^1\times M'[\bar{h}],v^{(i)}_{l,j}}^{flat}\Big)$.
\end{lemma}

\begin{proof}
The proof of this lemma is the same as
that of Proposition  \ref {prop:existence-horizontal-lift}
and we omit the detail.

We take the same open covering $\{U_{\beta}\}$ as in the proof of
Proposition \ref {prop:existence-horizontal-lift}.
We consider the connection
$\nabla^{flat}_{\mathbb{P}^1\times M'[\bar{h}],u}$
on $({\mathcal O}^{hol}_{\mathbb{P}^1\times M'[u]})^{\oplus r}$
given by the connection matrix
\[
 \Big( A^{(i)}(\tilde{z}^{(i)},\epsilon)
  + \sum_{q=1}^{\kappa} \bar{h}_q \sum_{l=0}^{r-1}\sum_{j=0}^{m_i-1}
  c^{(i)}_{u,q,l,j} \: \tilde{\Xi}^{(i)}_{l,j}(\tilde{z}^{(i)}) \Big)
 \frac {d\tilde{z}^{(i)}} {(\tilde{z}^{(i)})^{m_i}-\epsilon^{m_i}}  
 + \sum_{q=1}^{\kappa}  \sum_{l=0}^{r-1}\sum_{j=0}^{m_i-1}
  c^{(i)}_{u,q,l,j} B^{(i)}_{l,j}(\tilde{z}^{(i)}) d \bar{h}_q
\]
with respect to $u$.
Let
${\mathcal N}^{(i)}_u$ be the endomorphism obtained  by substituting
$\epsilon^{m_i}$ for $(z^{(i)})^{m_i}$ in
\[
 \psi^{(i)}\Big( 
 A^{(i)}(\tilde{z}^{(i)},\epsilon)
 + \sum_{q=1}^{\kappa} \bar{h}_q \sum_{l=0}^{r-1}\sum_{j=0}^{m_i-1}
 c^{(i)}_{u,q,l,j} \: \tilde{\Xi}^{(i)}_{l,j}(\tilde{z}^{(i)}) \Big),
\]
where $\psi^{(i)}$ is given in (\ref {equation:polynomial psi}).
Then 
$\big( ({\mathcal O}^{hol}_{\mathbb{P}^1\times M'[u]})^{\oplus r},
\nabla^{flat}_{\mathbb{P}^1\times M'[\bar{h}],u}\big|_{(U_i)_{M'[u]}},
\{ {\mathcal N}^{(i)}_u\} \big)$
becomes a local horizontal lift.
Patching the local horizontal lifts altogether,
we obtain a unique horizontal lift
in the same way as Proposition \ref {prop:existence-horizontal-lift}.
\end{proof}

\begin{proposition} \label {prop:splitting is a homomorphism}
The morphism
\[
 T^{hol}_{{\mathcal T}^{\circ}/\Delta_{\epsilon_0}}
 \ni v\mapsto \Phi(v)
 \in (\pi^{\circ})_* T^{hol}_{M^{\circ}/\Delta_{\epsilon_0}}
\]
defined in (\ref {equation: definition of pi_*D})
is an
${\mathcal O}^{hol}_{{\mathcal T}^{\circ}}$-homomorphism.
\end{proposition}

\begin{proof}
Take an open subset
${\mathcal T'}\subset {\mathcal T}^{\circ}$
and holomorphic vector fields
$v_1,v_2\in H^0 \big({\mathcal T'},T^{hol}_{{\mathcal T}^{\circ}/\Delta_{\epsilon_0}}\big)$.
Let
\[
 u\colon{\mathcal T'}\times\Spec\mathbb{C}[h_1,h_2]/(h_1^2,h_2^2,h_1h_2)
\longrightarrow {\mathcal T'}
\]
be the morphism such that the restriction
$u|_{{\mathcal T'}\times\Spec\mathbb{C}[h_i]/(h_i^2)}$ corresponds to $v_i$
for $i=1,2$.
Applying Lemma \ref{lemma:horizontal lift general}
to $\mathbb{C}[I]=\mathbb{C}[h_1,h_2]/(h_1^2,h_1h_2,h_2^2)$,
we can take a horizontal lift
$\big({\mathcal E}^u,\nabla^u,\{{\mathcal N}^{(i)}_u\}\big)$ of
$\big(\tilde{E}^{hol}_{M'},\tilde{\nabla}^{hol}_{M'},\{\tilde{N}^{(i),hol}_{M'}\}\big)$
with respect to $u$
and with respect to the blocks
$\big(\nabla_{\mathbb{P}^1\times M'[\bar{h}],v^{(i)}_{l,j}}^{flat}\big)$
of local horizontal lifts.
We can see by construction that the restriction
$\big({\mathcal E}^u,\nabla^u,\{{\mathcal N}^{(i)}_u\}\big)
\big|_{M' \times \Spec\mathbb{C}[h_i]/ (h_i^2)}$
coincides with the horizontal lift
$\big({\mathcal E}^{v_i},\nabla^{v_i},\{{\mathcal N}^{(i)}_{v_i}\}\big)$
of
$\big(\tilde{E}^{hol}_{M'},\tilde{\nabla}^{hol}_{M'},\{\tilde{N}^{(i),hol}_{M'}\}\big)$
with respect to $v_i$.
So the morphism
\[
 M' \times \Spec\mathbb{C}[h_1,h_2]/ (h_1^2, h_1h_2, h_2^2)
 \longrightarrow 
 M^{\balpha}_{{\mathcal C},{\mathcal D}}(\tilde{\bnu},\bmu)
 \times_{\mathcal B} {\mathcal B'}
 \]
determined by the flat family 
$\big({\mathcal E}^u,\overline{\nabla^u},\{{\mathcal N}^{(i)}_u\}\big)$
coincides with the one given by the pair $(\Phi(v_1),\Phi(v_2))$
of vector fields,
where 
$\overline{\nabla^u}\colon{\mathcal E}^u\longrightarrow{\mathcal E}^u\otimes
\Omega^1_{{\mathcal C}_{{\mathcal T'}[u]}/{\mathcal T'}[u]}
({\mathcal D}_{{\mathcal T'}[u]})^{hol}$
is the relative connection induced by $\nabla^u$.
From the definition of the addition of vector fields, the restriction
$(\Phi(v_1),\Phi(v_2))|
_{M'\times\Spec\mathbb{C}[h_1,h_2]/(h_1-h_2,h_1^2)}$
to the diagonal coincides with $\Phi(v_1)+\Phi(v_2)$.
On the other hand, we can see by the construction that the restriction
$\big({\mathcal E}^u,\nabla^u,\{{\mathcal N}^{(i)}_u\}\big) 
\big|_{M'\times\Spec\mathbb{C}[h_1,h_2]/(h_1-h_2,h_1^2)}$
is a horizontal lift of 
$\big(\tilde{E}^{hol}_{M'},\tilde{\nabla}^{hol}_{M'},\{\tilde{N}^{(i),hol}_{M'}\}\big)$
with respect to $v_1+v_2$
and with respect to the blocks of local horizontal lifts
$\big(\nabla_{\mathbb{P}^1\times M'[\bar{h}],v^{(i)}_{l,j}}^{flat}\big)$
in the sense of
Proposition \ref  {prop:existence-horizontal-lift}.
So we have
$\Phi(v_1+v_2)=\Phi(v_1)+\Phi(v_2)$.

Take a holomorphic function $f\in H^0({\mathcal T'},{\mathcal O}^{hol}_{\mathcal T'})$
and a holomorphic vector field
$v\in H^0 \big({\mathcal T'},T^{hol}_{{\mathcal T}^{\circ}/\Delta_{\epsilon_0}}\big)$.
Let 
\[
 \sigma_f\colon {\mathcal T'}\times\Spec\mathbb{C}[h]/(h^2)
 \longrightarrow {\mathcal T'}\times\Spec\mathbb{C}[h]/(h^2)
\]
be the morphism corresponding to the ring homomorphism
${\mathcal O}_{\mathcal T'}^{hol}[h]/(h^2)\ni a+b\bar{h} \mapsto a+b f\bar{h}
\in {\mathcal O}_{\mathcal T'}^{hol}[t]/(h^2)$
and let
\[
 \mathrm{id}\times\sigma_f \colon
 M'\times_{\mathcal T'} {\mathcal T'}\times \Spec\mathbb{C}[h]/(h^2)
 \longrightarrow
 M'\times_{\mathcal T'} {\mathcal T'}\times \Spec\mathbb{C}[h]/(h^2)
\]
be its base change.
If
$\big({\mathcal E}^v,\nabla^v,\{{\mathcal N}^{(i)}_v\}\big)$
is a horizontal lift of
$\big(\tilde{E}^{hol}_{M'},\tilde{\nabla}^{hol}_{M'},\{\tilde{N}^{(i),hol}_{M'}\}\big)$
with respect to $v$
and with respect to the blocks of local horizontal lifts
$\big(\nabla_{\mathbb{P}^1\times M'[\bar{h}],v^{(i)}_{l,j}}^{flat}\big)$,
then we can see by the construction that the pull back
$(1\times\sigma_f)^*
\big({\mathcal E}^v,\nabla^v,\{ {\mathcal N}^{(i)}_v\}\big)$
is a horizontal lift of
$\big(\tilde{E}^{hol}_{M'},\tilde{\nabla}^{hol}_{M'},\{\tilde{N}^{(i),hol}_{M'}\}\big)$
with respect to $fv$
and with respect to the blocks of local horizontal lifts
$\big(\nabla_{\mathbb{P}^1\times M'[\bar{h}],v^{(i)}_{l,j}}^{flat}\big)$.
By the definition of ${\mathcal O}^{hol}_{\mathcal T'}$-module
structure on the tangent bundle, we can see that
the pull-back
$\big((\mathrm{id}\times\sigma_f)^*{\mathcal E}^v,
(\mathrm{id}\times\sigma_f)^*\overline{\nabla^v},
\{(1\times\sigma_f)^*{\mathcal N}^{(i)}_v \} \big)$
of the flat family
$\big({\mathcal E}^v,\overline{\nabla^v},\{ {\mathcal N}^{(i)}_v \}\big)$
corresponds to $f  \Phi(v)$.
So we have $\Phi(fv)=f \Phi(v)$.
Hence
we have proved that $\Phi$ is an ${\mathcal O}^{hol}_{\mathcal T^{\circ}}$-homomorphism.
\end{proof}

By the adjoint bijection
\begin{equation} \label {equation:adjoint bijection}
 \Hom_ { {\mathcal O}_{M^{\circ}}^{hol} }
 \left( (\pi^{\circ})^*
 T^{hol}_{{\mathcal T}^{\circ}/\Delta_{\epsilon_0}} ,
 T^{hol}_{M^{\circ} /\Delta_{\epsilon_0}} \right) 
 \cong
 \Hom_{{\mathcal O}_{{\mathcal T}^{\circ}}^{hol}}
 \left( T^{hol}_{{\mathcal T}^{\circ}/\Delta_{\epsilon_0}} ,
 (\pi^{\circ})_* T^{hol}_{M^{\circ} /\Delta_{\epsilon_0}} \right),
\end{equation}
the ${\mathcal O}_{{\mathcal T}^{\circ}}^{hol}$-homomorphism
$\Phi\colon
T^{hol}_{{\mathcal T}^{\circ}/\Delta_{\epsilon_0}}
\longrightarrow
(\pi^{\circ})_* T^{hol}_{M^{\circ}/\Delta_{\epsilon_0}}$
given in (\ref {equation: definition of pi_*D})
corresponds to an
${\mathcal O}^{hol}_{M^{\circ}}$-homomorphism
$\Psi \colon 
(\pi^{\circ})^* T^{hol}_{{\mathcal T}^{\circ}/\Delta_{\epsilon_0}} 
\longrightarrow
T^{hol}_{M^{\circ}/\Delta_{\epsilon_0}}$.
Since $\Phi$ satisfies $d\pi^{\circ}\circ\Phi(v)=v$
for vector fields $v\in T^{hol}_{{\mathcal T}^{\circ}/\Delta_{\epsilon_0}}$,
the homomorphism $\Psi$ is a splitting of the surjection
$T^{hol}_{M^{\circ}/\Delta_{\epsilon_0}}
\xrightarrow{d\pi^{\circ}}
(\pi^{\circ})^* T^{hol}_{{\mathcal T}^{\circ}/\Delta_{\epsilon_0}}$
canonically induced by the smooth morphism
$\pi^{\circ} \colon M^{\circ} \longrightarrow {\mathcal T}^{\circ}$.
Furthermore we can see
$\Psi\big|
_{M^{\balpha}_{{\mathcal C},{\mathcal D}}(\tilde{\bnu},\blambda)_{\epsilon=0}\cap M^{\circ}}
=
\Psi_0\big|
_{M^{\balpha}_{{\mathcal C},{\mathcal D}}(\tilde{\bnu},\blambda)_{\epsilon=0}\cap M^{\circ}}$
from its construction.
Thus we have proved
Theorem \ref {thm:holomorphic-splitting}.

\begin{example} \label {example:hypergeometric confluence} \rm
Let us consider the case of $g=0$, $r=2$,
$n=2$, $m_1=2$, $m_2=1$ and $a=\deg E=0$.
So $C=\mathbb{P}^1$,
$D^{(1)}=\{z^2-\epsilon^2=0\}$
and we may assume $D^{(2)}=\{\infty\}$.
We choose $z^{(1)}=z$ and $z^{(2)}=w=1/z$.
We take the exponent $\bnu$ so generic that
$\res_{z=\infty}\left(\nu^{(1)}(\mu_{k_1})\dfrac{dz}{z^2-\epsilon^2}\right)
+\res_{w=\infty}\left(\nu^{(2)}(\mu_{k_2})\dfrac{dw}{w}\right)
\notin\mathbb{Z}$
for any choice of $k_1,k_2\in\{1,2\}$.
Then the $(\bnu,\bmu)$ connections are irreducible and
correspond to the classical hypergeometric equations.
The moduli space
$M^{\alpha}_{\mathbb{P}^1,D}(\tilde{\bnu},\bmu)$
consists of a single point because of the rigidity
of the hypergeometric equations.
For a $(\bnu,\bmu)$-connection
$(E,\nabla,\{N^{(i)}\})\in M^{\alpha}_{\mathbb{P}^1,D}(\tilde{\bnu},\bmu)$,
we have
$E\cong {\mathcal O}_{\mathbb{P}^1}^{\oplus 2}$
and
$\nabla|_U$ is given by a connection matrix
\begin{equation} \label {equation:connection matrix of hypergeometric equation}
 \frac{A_0(\epsilon)+A_1(\epsilon)z}{z^2-\epsilon^2}dz.
\end{equation}
The above connection matrix is uniquely determined by
$(E,\nabla)$ up to a constant conjugate and
the matrices 
$\Xi^{(1)}_{l,j}(z)$ ($l=0,1$, $j=0,1$)
given in (\ref  {equation:giving Xi})  are systematically determined.
We write
\[
 \Xi^{(1)}_{l,j}(z)=C^{(1)}_{l,j,0}(\epsilon)+C^{(1)}_{l,j,1}(\epsilon)z.
\]
If we take an adjusting data
$\big( R^{(1)}_{l,j,0},R^{(1)}_{l,j,1} \big)$,
we have
$C^{(1)}_{l,j,1}(\epsilon)
=\big[ A_0,R^{(1)}_{l,j,1} \big] + \big[ A_1,R^{(1)}_{l,j,0} \big]$
and we define
\[
 \tilde{\Xi}^{(1)}_{l,j}(z)=C^{(1)}_{l,j,0}-\big[ A_0,R^{(1)}_{l,j,0} \big]
 -\epsilon^2 \big[ A_1,R^{(1)}_{l,j,1} \big].
\]
There is an ambiguity in the choice of adjusting data
$\big( R^{(1)}_{l,j,0},R^{(1)}_{l,j,1} \big)$.
If $\big( R'^{(1)}_{l,j,0},R'^{(1)}_{l,j,1} \big)$ is another one, then
$C^{(1)}_{l,j,1}=
\big[ A_0,R^{(1)}_{l,j1} \big] + \big[ A_1,R^{(1)}_{l,j,0} \big]
=\big[ A_0,R'^{(1)}_{l,j,1} \big] + \big[ A_1,R'^{(1)}_{l,j,0} \big]$.
Since we are choosing $A_0,A_1$ generic,
the full matrix ring is generated by $A_0,A_1,[A_0,A_1], I_2$.
Furthermore,
$\im\mathrm{ad}(A_0)\cap\im\mathrm{ad}(A_1)$
is generated by $[A_0,A_1]$.
Since
$\big[ A_0,R^{(1)}_{l,j,1}-R'^{(1)}_{l,j,1} \big]
=-\big[ A_1,R^{(1)}_{l,j,0}-R'^{(1)}_{l,j,0} \big]
\in \im \mathrm{ad}(A_0)\cap\im\mathrm{ad}(A_1)$,
we can write
$R^{(1)}_{l,j,0}-R'^{(1)}_{l,j,0}=aA_0+bA_1$ and
$R^{(1)}_{l,j,1}-R'^{(1)}_{l,j,1}=cA_0+aA_1$
for some functions $a,b,c$ defined on an open subset of the moduli space
$M^{\alpha}_{\mathbb{P}^1,D}(\tilde{\bnu},\bmu)$.
If we put
$\tilde{\Xi}'^{(1)}_{l,j}(z):=C^{(1)}_{l,j,0}
-\big[ A_0,R'^{(1)}_{l,j,0}\big]
-\epsilon^2\big[ A_1,R'^{(1)}_{l,j,1}\big]$,
then
$\tilde{\Xi}^{(1)}_{l,j}(z)-\tilde{\Xi}'^{(1)}_{l,j}(z)
= \big[ A_0,R^{(1)}_{l,j,0}-R'^{(1)}_{l,j,0} \big]
-\epsilon^2 \big[ A_1,R^{(1)}_{l,j,1}-R'^{(1)}_{l,j,1} \big]
=(b-\epsilon^2c) \big[ A_0,A_1 \big]$.
So we have
\[
 \big( I_2-\bar{h}(b-\epsilon^2c)A_1 \big)^{-1}
 \frac{ A_0+A_1z+\bar{h}\tilde{\Xi}^{(1)}_{l,j}(z)}{z^2-\epsilon^2}dz \;
 \big( I_2-\bar{h}(b-\epsilon^2c)A_1 \big)
 =
 \frac{ A_0+A_1z+\bar{h}\tilde{\Xi}'^{(1)}_{l,j}(z)}{z^2-\epsilon^2}dz
\]
which means that there is no essential ambiguity in
the relative connection given by the connection matrix
\[
 \frac{ A_0+A_1z+\bar{h}\tilde{\Xi}^{(1)}_{l,j}(z)}{z^2-\epsilon^2}dz.
\] 
up to a global automorphism.
However, there is an ambiguity in the choice of 
$B^{(1)}_{l,j}$ such that
the connection matrix
\[
 \frac{A_0(\epsilon)+A_1(\epsilon)z
 +\bar{h}\, \tilde{\Xi}^{(1)}_{l,j}(z)} {z^2-\epsilon^2}dz
 +B^{(1)}_{l,j}(z) d\bar{h}
\]
gives a horizontal lift.
Indeed, for a fundamental solution
$Y_{\infty}(z,\epsilon)$ of $\nabla$ near $\infty$,
there is an ambiguity in
$Y_{\infty}(z,\epsilon)+\bar{h} B^{(1)}_{l,j}(z) Y_{\infty}(z,\epsilon)$
by an action of
$(I_2+\bar{h}(c_0I_2+c_1\mathrm{Mon}_{\infty}))$ from the right
with $c_0\equiv 0, c_1\equiv 0 \pmod{\epsilon^2}$,
where $\mathrm{Mon}_{\infty}$ is the monodromy matrix
of $Y_{\infty}(z,\epsilon)$ along a loop around $\infty$.
If we write
$Y_{\infty}(z,\epsilon)+\bar{h} B^{(1)}_{l,j}(z) Y_{\infty}(z,\epsilon)
=(\tilde{y}_1,\tilde{y}_2)$
with $\tilde{y}_1,\tilde{y}_2$ two independent hypergeometric solutions,
then the ambiguity is essentially given by a replacement
of $(\tilde{y}_1,\tilde{y}_2)$ with
$((1+\bar{h}b_1)\tilde{y}_1,(1+\bar{h}b_2)\tilde{y}_2)$,
where $b_1\equiv 0, b_2\equiv 0\pmod{\epsilon^2}$.
Notice that we can in fact assume $c_0=0$
after a normalization via applying a global automorphism,
but there is still an ambiguity arising from $\bar{h}c_1$.
\end{example}



{\bf Acknowledgment.}
The author would like to thank Professor Takeshi Abe for
giving him opportunities to have discussions on the subject
in this paper.
He also would like to thank Professor Hironobu Kimura
and Professor Yoshishige Haraoka
for giving him valuable comments
when the author tried to start the study
of the subject in this paper.
The author would like to thank Professor Indranil Biswas
and Professor Balasubramanian Aswin
for the hospitality at the conference
``Quantum Fields, Geometry and Representation Theory''
held at Bangalore in 2018.
He also thanks Professor Takuro Mochizuki and Professor
Masa-Hiko Saito who had discussion with him 
related to this paper.
The author would like to thank Professors Hiroyuki Inou and
Masayuki Asaoka for telling him
elementary textbooks for an introduction to dynamical systems,
though the author have not followed most of them yet.
The author would like to thank Professor Indranil Biswas
for the hospitality at the conference ``Bundles -2019''
at Mumbai in 2019.
He also would like to thank Professor Jacques Hurtubise
for giving him valuable comments.

The author is partly supported by JSPS KAKENHI:
Grant Numbers 26400043, 17H06127 and 19K03422.


\end{document}